\documentclass[10pt,a4paper]{article}
\usepackage[english]{babel}
\usepackage[utf8x]{inputenc}
\usepackage[T1]{fontenc}
\usepackage{fullpage}
\usepackage{cite}
\usepackage{listings}
\usepackage{enumerate}
\usepackage{array}
\usepackage{multirow}

\usepackage{xargs}
\usepackage{tikz}
\usetikzlibrary{matrix, cd, positioning}

\tikzcdset{%
  every label/.style={
    /tikz/auto,
    /tikz/inner sep=+0.5ex}}

\usepackage{multicol}

\usepackage{amsmath}
\usepackage{amssymb}
\usepackage{amsfonts}
\usepackage{amsthm}
\usepackage{xcolor}
\usepackage{catex}

\usepackage{float}
\restylefloat{table}

\DeclareMathOperator{\CC}{\mathbf{C}}
\DeclareMathOperator{\DD}{\mathbf{D}}
\DeclareMathOperator{\f}{\mathbf{f}}
\DeclareMathOperator{\g}{\mathbf{g}}

\DeclareMathOperator{\Ass}{\mathbf{Assoc}}

\DeclareMathOperator{\Cat}{\mathbf{Cat}}
\DeclareMathOperator{\Pol}{\mathbf{Pol}}

\DeclareMathOperator{\WPol}{\mathbf{WPol}}
\DeclareMathOperator{\sCat}{\mathbf{sCat}}
\DeclareMathOperator{\WCat}{\mathbf{WCat}}

\DeclareMathOperator{\sOrd}{\mathbf{sOrd}}

\DeclareMathOperator{\Graph}{\mathbf{Graph}}

\newcommand{\C}{\mathcal C}
\newcommand{\A}{\mathcal A}

\newcommand{\D}{\mathcal D}
\newcommand{\E}{\mathcal E}
\newcommand{\B}{\mathcal B}
\newcommand{\F}{\mathcal F}
\newcommand{\R}{\mathcal R}

\newcommand{\Q}{\mathcal Q}

\newcommand{\N}{\mathbb N}
\renewcommand{\L}{\mathcal L}
\renewcommand{\t}{\mathbf t}
\newcommand{\s}{\mathbf{s}}
\newcommand{\co}{\mathbf{co}}
\newcommand{\op}{\mathbf{op}}
\newcommand{\w}{\mathbf{w}}
\renewcommand{\l}{\ell}
\newcommand{\p}{\mathbf{p}}

\renewcommand{\v}{\mathbf{v}}

\newcommand{\Bi}{\mathbf{BiCat}}
\newcommand{\PF}{\mathbf{PFonct}}
\newcommand{\TPN}{\mathbf{PNTrans}}
\newcommand{\Alg}{\mathit{Alg}}
\renewcommand{\restriction}{\mathord{\upharpoonright}}

\usepackage[all]{xy}

\tikzset{>=Implies}
\newlength{\myline}
\newcommandx*{\doublearrow}[2]{
  \draw[line width=rule_thickness,double equal sign distance,#1] #2;
}
\newcommandx*{\triplearrow}[4][1=0, 2=1]{
  \draw[line width=\myline,double distance=5\myline,#3] #4;
  \draw[line width=\myline,shorten <=#1\myline,shorten >=#2\myline,#3] #4;
}
\newcommandx*{\quadarrow}[4][1=0, 2=2.5]{
  \draw[line width=\myline,double distance=8\myline,#3] #4;
  \draw[line width=\myline,double distance=2\myline,shorten <=#1\myline,shorten >=#2\myline,#3] #4;
}

\UseTips
\SelectTips{eu}{11}

\newdir{ >}{{}*!/-10pt/@{>}}
\newdir{ -}{{}*!/-10pt/@{}}
\newdir{> }{{}*!/+10pt/@{>}}

\makeatletter

\xyletcsnamecsname@{dir4{}}{dir{}}
\xydefcsname@{dir4{-}}{\line@ \quadruple@\xydashh@}
\xydefcsname@{dir4{.}}{\point@ \quadruple@\xydashh@}
\xydefcsname@{dir4{~}}{\squiggle@ \quadruple@\xybsqlh@}
\xydefcsname@{dir4{>}}{\Tttip@}
\xydefcsname@{dir4{<}}{\reverseDirection@\Tttip@}

\xydef@\quadruple@#1{%
	\edef\Drop@@{%
		\dimen@=#1\relax
		\dimen@=.5\dimen@
		\A@=-\sinDirection\dimen@
		\B@=\cosDirection\dimen@
		\setboxz@h{%
			\setbox2=\hbox{\kern3\A@\raise3\B@\copy\z@}%
			\dp2=\z@ \ht2=\z@ \wd2=\z@ \box2
			\setbox2=\hbox{\kern\A@\raise\B@\copy\z@}%
			\dp2=\z@ \ht2=\z@ \wd2=\z@ \box2
			\setbox2=\hbox{\kern-\A@\raise-\B@\copy\z@}%
			\dp2=\z@ \ht2=\z@ \wd2=\z@ \box2
			\setbox2=\hbox{\kern-3\A@\raise-3\B@ \noexpand\boxz@}%
			\dp2=\z@ \ht2=\z@ \wd2=\z@ \box2
		}%
		\ht\z@=\z@ \dp\z@=\z@ \wd\z@=\z@ \noexpand\styledboxz@
	}%
}

\xydef@\Tttip@{\kern2pt \vrule height2pt depth2pt width\z@
	\Tttip@@ \kern2pt \egroup
	\U@c=0pt \D@c=0pt \L@c=0pt \R@c=0pt \Edge@c={\circleEdge}%
	\def\Leftness@{.5}\def\Upness@{.5}%
	\def\Drop@@{\styledboxz@}\def\Connect@@{\straight@{\dottedSpread@\jot}}}
	
\xydef@\Tttip@@{%
	\dimen@=.25\dimen@
 	\B@=\cosDirection\dimen@
	\setboxz@h\bgroup\reverseDirection@\line@ \wdz@=\z@ \ht\z@=\z@ \dp\z@=\z@
	{\vDirection@(1,-1)\xydashl@ \xyatipfont\char\DirectionChar}%
	{\vDirection@(1,+1)\xydashl@ \xybtipfont\char\DirectionChar}%
}

\xydef@\ar@form{
	\ifx \space@\next \expandafter\DN@\space{\xyFN@\ar@form}%
	\else\ifx ^\next \DN@ ^{\xyFN@\ar@style}\edef\arvariant@@{\string^}%
	\else\ifx _\next \DN@ _{\xyFN@\ar@style}\edef\arvariant@@{\string_}%
	\else\ifx 0\next \DN@ 0{\xyFN@\ar@style}\def\arvariant@@{0}%
	\else\ifx 1\next \DN@ 1{\xyFN@\ar@style}\def\arvariant@@{1}%
	\else\ifx 2\next \DN@ 2{\xyFN@\ar@style}\def\arvariant@@{2}%
	\else\ifx 3\next \DN@ 3{\xyFN@\ar@style}\def\arvariant@@{3}%
	\else\ifx 4\next \DN@ 4{\xyFN@\ar@style}\def\arvariant@@{4}%
	\else\ifx \bgroup\next \let\next@=\ar@style
	\else\ifx [\next \DN@[##1]{\ar@modifiers{[##1]}}
	\else\ifx *\next \DN@ *{\ar@modifiers}%
	\else\addLT@\ifx\next \let\next@=\ar@slide
	\else\ifx /\next \let\next@=\ar@curveslash
	\else\ifx (\next \let\next@=\ar@curveinout 
	\else\addRQ@\ifx\next \addRQ@\DN@{\ar@curve@}%
	\else\addLQ@\ifx\next \addLQ@\DN@{\xyFN@\ar@curve}%
	\else\addDASH@\ifx\next \addDASH@\DN@{\defarstem@-\xyFN@\ar@}%
	\else\addEQ@\ifx\next \addEQ@\DN@{\def\arvariant@@{2}\defarstem@-\xyFN@\ar@}%
	\else\addDOT@\ifx\next \addDOT@\DN@{\defarstem@.\xyFN@\ar@}%
	\else\ifx :\next \DN@:{\def\arvariant@@{2}\defarstem@.\xyFN@\ar@}%
	\else\ifx ~\next \DN@~{\defarstem@~\xyFN@\ar@}%
	\else\ifx !\next \DN@!{\dasharstem@\xyFN@\ar@}%
	\else\ifx ?\next \DN@?{\ar@upsidedown\xyFN@\ar@}%
	\else \let\next@=\ar@error
	\fi\fi\fi\fi\fi\fi\fi\fi\fi\fi\fi\fi\fi\fi\fi\fi\fi\fi\fi\fi\fi\fi\fi \next@}

\makeatother

\makeatletter
\newtheorem*{rep@theorem}{\rep@title}
\newcommand{\newreptheorem}[2]{%
\newenvironment{rep#1}[1]{%
 \def\rep@title{#2 \ref{##1}}%
 \begin{rep@theorem}}%
 {\end{rep@theorem}}}
\makeatother

\newcommand{\comment}[1]{}

\newcommand{\qfl}{
  \xymatrix@1@C=10pt{\ar@4 [r] &}}

\author{Maxime LUCAS\footnote{Univ Paris Diderot, Sorbonne Paris Cit\'e, IRIF, UMR 8243 CNRS, PiR2, INRIA Paris-Rocquencourt, F-75205 Paris, France - \texttt{maxime.lucas@pps.univ-paris-diderot.fr}}}

\title{A coherence theorem for pseudonatural transformations}

\date{}

\begin{document}

\maketitle

\begin{abstract} 
We prove coherence theorems for bicategories, pseudofunctors and pseudonatural transformations. These theorems boil down to proving the coherence of some free $(4,2)$-categories. In the case of bicategories and pseudofunctors, existing rewriting techniques based on Squier's Theorem allow us to conclude. In the case of pseudonatural transformations this approach only proves the coherence of part of the structure, and we use a new rewriting result to conclude. To this end, we introduce the notions of white-categories and partial coherence.
\end{abstract}

\tableofcontents

\newpage

\theoremstyle{plain}
\newtheorem{thm}{Theorem}[subsection]
\newtheorem{prop}[thm]{Proposition}
\newtheorem{lem}[thm]{Lemma}
\newtheorem{cor}[thm]{Corollary}
\newtheorem{conj}[thm]{Conjecture}
\newtheorem{quest}[thm]{Question}
\newreptheorem{prop}{Proposition}
\newreptheorem{thm}{Theorem}

\definecolor{cyan}{RGB}{150,50,50}
\definecolor{red}{RGB}{200,255,200}
\definecolor{tr}{RGB}{175,175,175}

\theoremstyle{definition}
\newtheorem{defn}[thm]{Definition}
\newtheorem{ex}[thm]{Example}
\newtheorem{remq}[thm]{Remark}
\newtheorem{nota}[thm]{Notation}

\deftwocell[polygon, white]{prodC: 2 -> 1}
\deftwocell[polygon, black]{prodD: 2 -> 1}
\deftwocell[box, red]{fonctF: 2 -> 2}
\deftwocell[box, cyan]{fonctG: 2 -> 2}
\deftwocell[circle, white]{unitC: 0 -> 1}
\deftwocell[circle, black]{unitD: 0 -> 1}
\deftwocell[rectangle, tr]{transfo: 1 -> 2}

\deftwocell[polygon, white]{assocC: 3 -> 1}
\deftwocell[polygon, black]{assocD: 3 -> 1}
\deftwocell[lefthalfcircle, white]{lunitC: 1 -> 1}
\deftwocell[lefthalfcircle, black]{lunitD: 1 -> 1}
\deftwocell[righthalfcircle, white]{runitC: 1 -> 1}
\deftwocell[righthalfcircle, black]{runitD: 1 -> 1}
\deftwocell[polygon, red]{img_prodF: 3 -> 2}
\deftwocell[polygon, cyan]{img_prodG: 3 -> 2}
\deftwocell[circle, red]{img_unitF: 1 -> 2}
\deftwocell[circle, cyan]{img_unitG: 1 -> 2}
\deftwocell[box, tr]{transfo_nat: 2 -> 2}

\deftwocell[polygon, white]{pentaC: 4 -> 1}
\deftwocell[circle, white]{trianC: 2 -> 1}
\deftwocell[polygon, black]{pentaD: 4 -> 1}
\deftwocell[circle, black]{trianD: 2 -> 1}
\deftwocell[polygon, red]{img_assocF: 4 -> 2}
\deftwocell[polygon, cyan]{img_assocG: 4 -> 2}
\deftwocell[righthalfcircle, red]{img_runitF : 2 -> 2}
\deftwocell[lefthalfcircle, red]{img_lunitF: 2 -> 2}
\deftwocell[righthalfcircle, cyan]{img_runitG: 2 -> 2}
\deftwocell[lefthalfcircle, cyan]{img_lunitG: 2 -> 2} 
\deftwocell[polygon, tr]{transfo_prod: 3 -> 2}
\deftwocell[circle, tr]{transfo_unit: 1 -> 2}

\deftwocell[circle, white]{bicat_source: 0 -> 0}
\deftwocell[circle, black]{bicat_but: 0 -> 0}
\deftwocell[circle, red]{fonct_source: 0 -> 0}
\deftwocell[circle, cyan]{fonct_but: 0 -> 0}

\deftwocell[cap]{eta: 0 -> 2}
\deftwocell[cup]{epsilon : 2 -> 0}
\deftwocell[circle]{tau : 1 -> 1}
\deftwocell[box]{sigma: 1 -> 1}
\deftwocell[polygon, white]{inv_prodC : 1 -> 2}
\deftwocell[circle, white]{conjugC : 1 -> 1}
\deftwocell[circle, black]{inv_conjugC : 1 -> 1}
\deftwocell[box, white]{conjugD : 1 -> 1}
\deftwocell[box, black]{inv_conjugD : 1 -> 1}
\deftwocell[circle, gray]{erase : 1 -> 0}
\deftwocell[circle, gray]{appear : 0 -> 1}

\deftwocell[rectangle, white]{exchange_z: 4 -> 2}
\deftwocell[rectangle, white]{exchange_u: 5 -> 3}

\deftwocell[dots]{dots: 2 -> 2}

\section*{Introduction}

\subsection*{An overview of coherence theorems}

A mathematical structure, such as the notion of monoid or algebra, is often defined in terms of some data satisfying relations. In the case of monoids, the data is a set and a binary application, and the relations are the associativity and the unit axioms. In category theory, one often considers relations that only hold \emph{up to isomorphism}. One of the simplest example of such a structure is that of monoidal categories, in which the product is not associative, but instead there exist isomorphisms $\alpha_{A,B,C} : (A \otimes B) \otimes C \to A \otimes (B \otimes C)$. This additional data must also satisfy some relation, known as Mac-Lane's pentagon:
\[
\xymatrix @C = .5em @R = 4em {
&
(A \otimes (B \otimes C)) \otimes D
\ar [rr] ^-{\alpha_{A,B \otimes C,D }}
& 
\ar@{} [dd] |-{=}
&
A \otimes ((B \otimes C) \otimes D)
\ar [rd] ^-{A \otimes \alpha_{B,C,D}}
&
\\
((A \otimes B) \otimes C) \otimes D
\ar [ru] ^-{\alpha_{A,B,C} \otimes D}
\ar [rrd] _-{\alpha_{A \otimes B, C , D}}
& & & &
A \otimes (B \otimes (C \otimes D))
\\
& &
(A \otimes B) \otimes (C \otimes D)
\ar [rru] _{\alpha_{A, B, C \otimes D}}
}
\]

The intended purpose of this relation is that, between any two bracketings of $A_1 \otimes A_2 \otimes \ldots \otimes A_{n-1} \otimes A_n$, there exists a unique isomorphism constructed from the isomorphisms $\alpha_{A,B,C}$. This statement was made precise  and proved by Mac Lane in the case of monoidal categories \cite{ML63}.
In general a \emph{coherence theorem} contains a description of a certain class of diagrams that are to commute. Coherence theorems exist for various other structures, e.g. bicategories \cite{ML85}, or $V$-natural transformations for a symmetric monoidal closed category $V$ \cite{K72}.

Coherence results are often a consequence of (arguably more essential \cite{K74}) \emph{strictification theorems}. A strictification theorem states that a ``weak'' structure is equivalent to a ``strict'' (or at least ``stricter'') one. For example any bicategory is biequivalent to a $2$-category, and the same is true for pseudofunctors (this is a consequence of this general strictification result \cite{P89}). It does not hold however for pseudonatural transformations.

\subsection*{Free categories and rewriting}
Coherence theorems can also be proven through rewriting techniques. The link between coherence and rewriting goes back to Squier's homotopical Theorem \cite{S94}, and has since been expanded upon \cite{G09}. Squier's theory is constructive, which means that the coherence conditions can be calculated from the relations, in a potentially automatic way. It can also be expanded to higher dimensions \cite{GM12}, a feature that may prove useful when studying weaker structures. In \cite{G12}, the authors use Squier's theory to prove the coherence of monoidal categories. Let us give an outline of the proof in the case of categories equipped with an associative tensor product. 

Polygraphs are presentations for higher-dimensional categories and were introduced by Burroni \cite{B93}, and by Street under the name of computads \cite{S76} \cite{S87}. In this paper we use Burroni's terminology. For example, a $1$-polygraph is given by a graph $G$, and the free $1$-category it generates is the category of paths on $G$. If $\Sigma$ is an $n$-polygraph, we denote by $\Sigma^*$ the free $n$-category generated by $\Sigma$.

An \emph{$(n,p)$-category} is a category where all $k$-cells are invertible, for $k>p$. In particular, $(n,0)$-categories are commonly called $n$-groupoids, and $(n,n)$-categories are just $n$-categories. There is a corresponding notion of $(n,p)$-polygraph. If $\Sigma$ is an $(n,p)$-polygraph, we denote by $\Sigma^{*(p)}$ the free $(n,p)$-category generated by $\Sigma$.

The structure of category equipped with an associative tensor product is encoded into a $4$-polygraph $\Ass$, which generates a free $(4,2)$-category $\Ass^{*(2)}$. The $4$-polygraph $\Ass$ contains one generating $2$-cell $\twocell{prodC}$ coding for product, one generating $3$-cell $\twocell{assocC} :
\twocell{(prodC *0 1) *1 prodC} \Rrightarrow \twocell{(1 *0 prodC) *1 prodC}$ coding for associativity and one generating $4$-cell $\twocell{pentaC}$ corresponding to Mac Lane's pentagon:
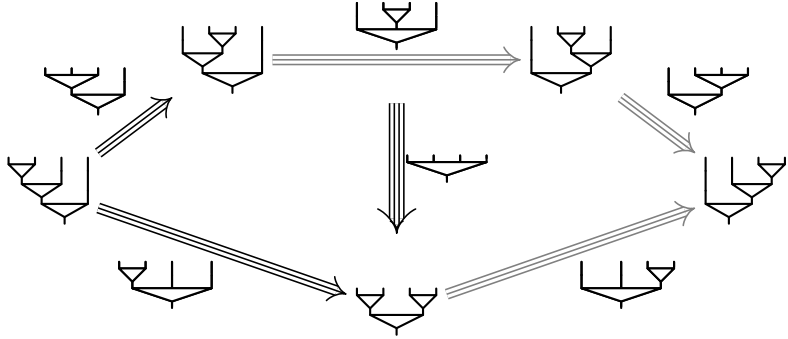
\[
\begin{tikzpicture}
\matrix (m) [matrix of math nodes, 
			nodes in empty cells,
			column sep = 1cm, 
			row sep = .6cm] 
{
&
\twocell{(1 *0 prodC *0 1) *1 (prodC *0 1) *1 prodC}
& 
&
\twocell{(1 *0 prodC *0 1) *1 (1 *0 prodC) *1 prodC}
&
\\
\twocell{(prodC *0 2) *1 (prodC *0 1) *1 prodC} 
& & & &
\twocell{(2 *0 prodC) *1 (1 *0 prodC) *1 prodC}
\\
& &
\twocell{(prodC *0 prodC) *1 prodC}
& & \\
};
\triplearrow{->, gray}
{(m-1-2) -- node [above] {$\twocell{(1 *0 prodC *0 1) *1 assocC}$} (m-1-4)}
\triplearrow{->}
{(m-2-1) -- node [above left] {$\twocell{(assocC *0 1) *1 prodC}$} (m-1-2)}
\triplearrow{->, gray}
{(m-1-4) -- node [above right] {$\twocell{(1 *0 assocC) *1 prodC}$} (m-2-5)}
\triplearrow{->}
{(m-2-1) -- node [below left] {$\twocell{(prodC *0 2) *1 assocC}$} (m-3-3)}
\triplearrow{->, gray}
{(m-3-3) -- node [below right] {$\twocell{(2 *0 prodC) *1 assocC}$} (m-2-5)}

\quadarrow{->}
{(m-1-3|-m-1-2.south) -- node [right] {$\twocell{pentaC}$} (m-3-3|-m-2-1.south)}
\end{tikzpicture}
\]
 
The coherence result for categories equipped with an associative product is now reduced to showing that, between every parallel $3$-cells $A$, $B$ in $\Ass^{*(2)}$, there exists a $4$-cell $\alpha: A \qfl B$ in $\Ass^{*(2)}$. A $4$-category satisfying this property is said to be \emph{$3$-coherent}. 

Let us denote by $\Ass^*$ the free $4$-category generated by $\Ass$. We have the following properties:
\begin{itemize}
\item Starting from any given $2$-cell in $\Ass^*$, it is impossible to form an infinite sequence of non-identity $3$-cells. This property is known as \emph{$3$-termination}.
\item If $A$ and $B$ are two $3$-cells in $\Ass^*$ with the same source, there exists $3$-cells $A'$ and $B'$ in $\Ass^*$ such that the composites $A \star_2 A'$ and $B \star_2 B'$  are well-defined and have the same target. This property is known as \emph{$3$-confluence}.
\end{itemize}
The conjunction of these two properties make $\Ass$ into a \emph{$3$-convergent} $4$-polygraph.

\subsection*{Squier's theory and coherence}

A generating $3$-cell composed with some lower dimensional context is called a \emph{rewriting step} of $\Ass$. A local branching in $\Ass$ is a pair of rewriting step of same source. Local branching are ordered by adjunction of context, that is a branching $(f,g)$ is smaller than a branching $(u \star_i f \star_i v,u \star_i g \star_i v)$ for any $2$-cells $u$ and $v$ and $i = 0,1$. There are three types of local branchings:
\begin{itemize}
\item A branching of the form $(f,f)$ is called \emph{aspherical} .
\item A branching of the form $(f \star_i \s(g),\s(f) \star_i g)$ for $i = 0$ or $1$ is called a \emph{Peiffer branching}.
\item Otherwise, $(f,g)$ is called an \emph{overlapping branching}.
\end{itemize}
Overlapping branchings that are also minimal are called \emph{critical branchings}.

There is exactly one critical branching in $\Ass$, of source $\twocell{(prodC *0 2) *1 (prodC *0 1) *1 prodC}$. Note that the critical pair appears as the source of the generating $4$-cell of $\Ass$. In particular there is a one-to-one correspondence between $4$-cells and critical pairs. A $3$-convergent $4$-polygraph that satisfies this property is said to satisfy the \emph{$3$-Squier condition}.

Proposition 4.3.4 in \cite{G09} states that a $4$-polygraph satisfying the $3$-Squier condition is $3$-coherent (and more generally, that any $(n+1)$-polygraph satisfying the $n$-Squier condition is $n$-coherent). In particular, the $4$-polygraph $\Ass$ satisfies the $3$-Squier condition, so it is $3$-coherent.

\comment{
We say that an $(n+1)$-polygraphs $\Sigma$ satisfies the \emph{$n$-Squier condition} if it is $n$-terminating, $n$-confluent and if $\Sigma_{n+1}$ consists of $(n+1)$-cells $A_{f,g}$ of the following shape, for any critical pair $(f,g)$:

\[
\xymatrix @C = 4em @R = 1.5em{
& 
\ar @/^/ [rd] ^{ f' }
\ar@2 [dd] ^{A_{f,g}}
& \\
\ar @/^/ [ru] ^{f}
\ar @/_/ [rd] _{g} 
& 
& \\
&  
\ar @/_/ [ru] _{g'}
& \\
}
\]

Proposition 4.3.4 in \cite{G09} states that an $(n+1)$-polygraph satisfying the $n$-Squier condition is $n$-coherent. In particular, the $4$-polygraph $\Ass$ satisfies the $3$-Squier condition, the only critical pair being the one corresponding to Mac Lane's pentagon:

As a consequence, the $4$-polygraph $\Ass$ satisfies the $3$-Squier condition, and $\Ass^{*(2)}$ is $3$-coherent.}

In Section \ref{sec:application}, we exhibit, for any sets $\CC$ and $\DD$ and any application $\f : \CC \to \DD$ two $4$-polygraphs $\Bi[\CC]$ and $\PF[\f]$ presenting respectively the structures of bicategory and pseudofunctor. Applying the reasoning we just presented, we prove our first two results:

\begin{repthm}{thm:BiCat_coh}[Coherence for bicategories]
Let $\CC$ be a set.

The $4$-polygraph $\Bi[\CC]$ is $3$-convergent and the free $(4,2)$-category $\Bi[\CC]^{*(2)}$ is $3$-coherent.
\end{repthm}

\begin{repthm}{thm:PFonct_coh}[Coherence for pseudofunctors]
Let $\CC$ and $\DD$ be sets, and $\f:\CC \to \DD$ an application.

The $4$-polygraph $\PF[\f]$ is $3$-convergent and the free $(4,2)$-category $\PF[\f]^{*(2)}$ is $3$-coherent.
\end{repthm}

However, this approach fails to work in the case of pseudonatural transformations, because the $(4,2)$-polygraph $\TPN[\f,\g]$ (where $\f$ and $\g$ are applications $\CC \to \DD$) encoding the structure of pseudonatural transformation is not $3$-confluent. 

\subsection*{The $2$-Squier condition of depth $2$}

In order to circumvent this difficulty, we introduce the notion of \emph{$2$-Squier condition of depth $2$}. We say that a $(4,2)$-polygraph $\Sigma$ satisfies the $2$-Squier condition of depth $2$ if it satisfies the $2$-Squier condition, and if the $4$-cells of $\Sigma$ correspond to the critical triples induced by the $2$-cells (with a prescribed shape). 

For example, the $4$-polygraph $\Ass$ satisfies the $2$-Squier condition of depth $2$: its underlying $2$-polygraph is both $2$-terminating and $2$-confluent. Moreover the only critical pair corresponds to the associativity $3$-cell. Finally, Mac Lane's pentagon can be written as follows, which shows that it corresponds to the only critical triple:

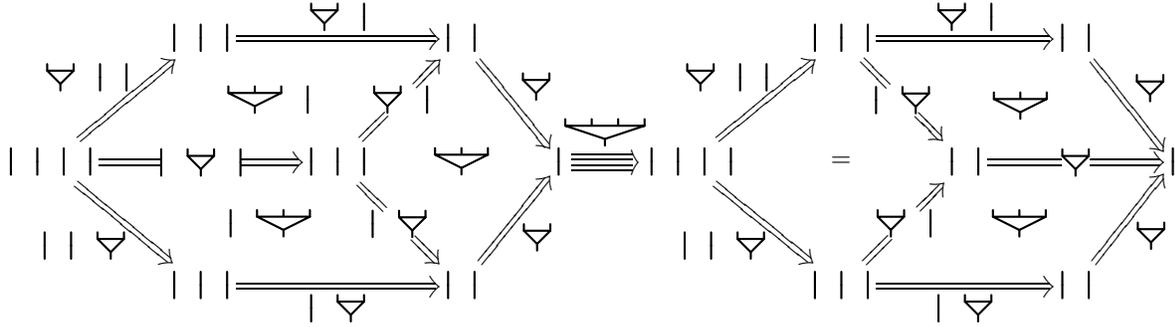
\[
\xymatrix @C = 2.5em @R = 3em {
&
\twocell{3}
\ar@2 [rr] ^-{\twocell{prodC *0 1}}
\ar@{} [rd] |-{\twocell{assocC *0 1}}
& &
\twocell{2}
\ar@2 [rd] ^-{\twocell{prodC}}
\ar@{} [dd] |-{\twocell{assocC}}
& & &
\twocell{3}
\ar@2 [rr] ^-{\twocell{prodC *0 1}}
\ar@{} [rrrd] |-{\twocell{assocC}}
\ar@2 [rd] |-{\twocell{1 *0 prodC}}
\ar@{} [dd] |-{=}
& &
\twocell{2}
\ar@2 [rd] ^-{\twocell{prodC}}
&
\\
\twocell{4}
\ar@2 [ru] ^-{\twocell{prodC *0 2}}
\ar@2 [rr] |-{\twocell{1 *0 prodC *0 1}}
\ar@2 [rd] _-{\twocell{2 *0 prodC}}
& &
\twocell{3}
\ar@{} [ld] |-{\twocell{1 *0 assocC}}
\ar@2 [ru] |-{\twocell{prodC *0 1}}
\ar@2 [rd] |-{\twocell{1 *0 prodC}}
&
&
\twocell{1} \,
\ar@4 [r] ^-*+{\twocell{pentaC}}
&
\, \twocell{4}
\ar@2 [ru] ^-{\twocell{prodC *0 2}}
\ar@2 [rd] _-{\twocell{2 *0 prodC}}
& & 
\twocell{2}
\ar@2 [rr] |-{\twocell{prodC}}
\ar@{} [rd] |-{\twocell{assocC}}
& &
\twocell{1}
\\
&
\twocell{3}
\ar@2 [rr] _-{\twocell{1 *0 prodC}}
& &
\twocell{2}
\ar@2 [ru] _-{\twocell{prodC}}
& & &
\twocell{3}
\ar@2 [rr] _-{\twocell{1 *0 prodC}}
\ar@2 [ru] |-{\twocell{prodC *0 1}}
& &
\twocell{2}
\ar@2 [ru] _-{\twocell{prodC}}
&
}
\]

We prove the following result about $(4,2)$-polygraph satisfying the $2$-Squier condition of depth $2$:
\begin{repthm}{thm:main_theory}
Let $\Sigma$ be a $(4,2)$-polygraph satisfying the $2$-Squier condition of depth $2$. 

For every parallel $3$-cells $A,B \in \Sigma_{3}^{*(2)}$ whose $1$-target is a normal form, there exists a $4$-cell $\alpha : A \qfl B$ in the free $(4,2)$-category $\Sigma_{4}^{*(2)}$.
\end{repthm}

Note in particular that the $2$-Squier condition of depth $2$ does not imply the $3$-coherence of the $(4,2)$-category generated by the polygraph, but only a partial coherence, "above the normal forms". For example in the case of $\Ass$, the only normal form is the $1$-cell $\twocell{1}$. So Theorem \ref{thm:main_theory} only expresses the coherence of the $3$-cells of $\Ass^{*(2)}$ whose $1$-target is $\twocell{1}$. On the other hand, Squier's Theorem as extended in \cite{G09} concerns all the $3$-cells of $\Ass^{*(2)}$, regardless of their $1$-target.

The $(4,2)$-polygraph $\TPN[\f,\g]$ does not satisfy the $2$-Squier condition. However, we identify in Section \ref{subsec:tpnat_coh_final} a sub-$(4,2)$-polygraph $\TPN^{++}[\f,\g]$ of $\TPN[\f,\g]$ that does. By Theorem \ref{thm:main_theory}, we get a partial coherence result in $\TPN^{++}[\f,\g]^{*(2)}$. The rest of Section \ref{sec:proof_of_coherence} is spent extending this partial coherence result to the rest of $\TPN[\f,\g]^{*(2)}$. To do so, we define a weight application from $\TPN[\f,\g]^{*(2)}$ to $\mathbb N$ to keep track of the condition on the $1$-targets of the $3$-cells considered. We thereby prove the following result:

\begin{repthm}{thm:TPNat_coh}[Coherence for pseudonatural transformations]
Let $\CC$ and $\DD$ be sets, and $\f,\g: \CC \to \DD$ applications.

Let $A,B \in \TPN[\f,\g]^{*(2)}_3$ be two parallel $3$-cells whose $1$-target is of weight $1$.

There is a $4$-cell $\alpha:A \qfl B \in \TPN[\f,\g]^{*(2)}_4$.
\end{repthm}

\subsection*{White-categories and partial coherence}

Let $j < k < n$ be integers. In an $n$-category $\C$, one can define the $j$-composition of $(k+1)$-cells $A$ and $B$ using the $k$-composition and whiskering by setting: 
\[
A \star_j B := (A \star_j \s_k(B)) \star_{k} (\t_k(A) \star_j B)
= (\s_k(A) \star_j B) \star_k (A \star_j \t_k(B)).
\] 
This is made possible by the exchange axiom between $\star_k$ and $\star_j$. An \emph{$n$-white-category} is an $n$-category in which the exchange axioms between $\star_k$ and $\star_0$ need not hold (even up to isomorphism) for any $k > 0$. As a result, $0$-composition is not defined for $k$-cells, for $k > 1$. The notion of $2$-white-category coincides with the notion of sesquicategory (see \cite{S96}).

Most concepts from rewriting have a straightforward transcription in the setting of white-categories. In particular in Section \ref{subsec:white_categories}, we define the notions of $(n,k)$-white-category and $(n,k)$-white-polygraph.  We also give an explicit description of the free $(n,k)$-white-category $\Sigma^{\w(k)}$ generated by an $(n,k)$-white-polygraph $\Sigma$. 

In this setting, we give a precise definition to the notion of partial coherence. Let $\C$ be a $(4,3)$-white-category and $S$ be a set of distinguished $2$-cells of $\C$. We call such a pair a \emph{pointed $(4,3)$-white-category}. We say that $\C$ is $S$-coherent if for any parallel $2$-cells $f, g \in S$ and any $3$-cells $A,B : f \Rrightarrow g \in \C$, there exists a $4$-cell $\alpha: A \qfl B \in \C$. In particular any $(4,3)$-white category is $\emptyset$-coherent, and a $(4,3)$-white category $\C$ is $\C_2$-coherent if and only if it is $3$-coherent (where $\C_2$ i sthe set of all the $2$-cells of $\C$). Theorem \ref{thm:main_theory} amounts to showing that the free $(4,2)$-category $\Sigma^{*(2)}$ is $S_\Sigma$-coherent, where $S_\Sigma$ is the set of all $2$-cells whose target is a normal form.

Finally, we give a way to modify partially coherent categories while retaining information about the partial coherence. Let $(\C,S)$ and $(\C',S')$ be pointed $(4,3)$-white-categories. We define a relation of \emph{strength} between pointed $(4,3)$-white-categories. We show that if $(\C,S)$ is stronger than $(\C',S')$, then the $S$-coherence of $\C$ implies the $S'$-coherence  of $\C'$.

\subsection*{Sketch of the proof of Theorem \ref{thm:main_theory}}

We now give an overview of the proof of Theorem \ref{thm:main_theory}. Let us fix a $(4,2)$-polygraph $\A$ satisfying the $2$-Squier condition of depth $2$, and denote by $S_\A$ the set of $2$-cells whose target is a normal form. In particular, $(\A^{*(2)},S_\A)$ is a pointed $(4,3)$-white-category. The first half of the proof (Section \ref{sec:transfo_polygraph}) consists in applying to $(\A^{*(2)},S_\A)$ a series of transformations. At each step, we verify that the new pointed $(4,3)$-white-category we obtain is stronger than the previous one. In the end, we get a pointed $(4,3)$-white-category $(\F^{\w(3)},S_\E)$, where $\F$ is a $4$-white-polygraph. In dimension $2$, the  $2$-cells of $\F$ consists of the union of the $2$-cells of $\A$ together with their formal inverses. We denote by $\bar f$ the formal inverse of a $2$-cell $f \in \A^*$. Let $\F_3$ be the set of $3$-cells of $\F$. It contains $3$-cells $C_{f,g}$ for any minimal local branching $(f,g)$, and cells $\eta_f$ for any $2$-cell $f \in \A$ of the following shape:
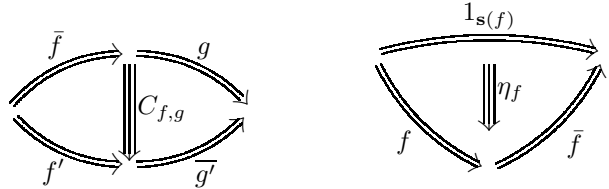
\[
\xymatrix @C = 4em @R = 1.5em{
& 
\ar@2 @/^/ [rd] ^{ g}
\ar@3 [dd] ^{C_{f,g}}
& \\
\ar@2 @/^/ [ur] ^{\bar f}
\ar@2 @/_/ [rd] _{f'}
& 
& 
\\
&  
\ar@2 @/_/ [ur] _{\overline{g'}} 
& \\
}
\qquad \qquad
\xymatrix @!C=4em @!R = 2em{
\ar@2 @/^/ [rr] ^{1_{\s(f)}}
\ar@2 @/_/ [rd] _f
&
\ar@3 [];[d]!<0pt,15pt> ^{\eta_f}
&
\\
& 
\ar@2 @/_/ [ru] _{\bar f}
&
}
\]
The purpose of this transformation is that in $\F^{\w(3)}$, for any $2$-cells $f,g \in S_E$, $3$-cells of the form $f \Rrightarrow g$ (and $4$-cells between them) are in one-to-one correspondence with $3$-cells of the form $\bar g \star_1 f \Rrightarrow 1_{\hat u}$ (and $4$-cells between them), where $\hat u$ is the common target of $f$ and $g$. More generally we study cells of the form $h \Rrightarrow 1_{\hat u}$, and $4$-cells between them.

We start by studying the rewriting system induced by the $3$-cells. Note that the $4$-white-polygraph $\F$ is not $3$-terminating, so we cannot use a Squier-like Theorem to conclude. However, let $\mathbb N[\A_1^*]$ be the free monoid on $\A_1^*$, the set of $1$-cells of $\A^*$. There is a well-founded ordering on $\A_1^*$ induced by the fact that $\A$ is $2$-terminating. This order induces a well-founded ordering on $\mathbb N[\A_1^*]$ called the \emph{multiset order}. We define an application $\p: \F_2^\w \to \mathbb N[\A_1^*]$ which induces a well-founded ordering on $\F_2^\w$, the set of $2$-cells of $\F^\w$, and show that the cells $C_{f,g}$ are compatible with this ordering (that is, the target of a cell $C_{f,g}$ is always smaller than the source). Thus, the fragment of $\F_3$ consisting of the cells $C_{f,g}$ is $3$-terminating. 

Thus the $\eta_f$ cells constitute the non-terminating part of $\F_3^\w$. To control their behaviour, we introduce a \emph{weight} application $\w_\eta : \F_3^\w \to \mathbb N[\A_1^*]$, that essentially counts the number of $\eta_f$ cells present in a $3$-cell. In section \ref{subsec:proof_final}, using the applications $\p$ and $\w_\eta$, we prove that for any $h \in \F_2^\w$ whose source and target are normal forms (for $\A_2$), and for any $3$-cells $A,B : h \Rrightarrow 1_{\hat u}$ in $\F_3^{\w}$, there is a $4$-cell $\alpha: A \qfl B$ in $\F^{\w(3)}$. Finally, we prove that this implies that $\F^{\w(3)}$ is $S_\E$-coherent, which concludes the proof.


\subsection*{Organisation}

In Section \ref{sec:rewriting}, we recall some classical definitions and results from rewriting theory, and we enunciate (without proof) Theorem \ref{thm:main_theory}. Section \ref{sec:white_category} contains the definitions of white-categories and white-polygraphs, together with the study of the notion of partial coherence. 
In Section \ref{sec:application}, we construct the free categories encoding the structures we want to study, and prove the coherence Theorems for bicategories (Theorem \ref{thm:BiCat_coh}) and pseudofunctors (Theorem \ref{thm:PFonct_coh}). The proof uses a lot of notions defined in Section \ref{sec:rewriting} and relies in particular on Squier's Theorem to conclude. There remains to show the coherence of pseudonatural transformations (Theorem \ref{thm:TPNat_coh}), which is done in Section \ref{sec:proof_of_coherence}. To prove Theorem \ref{thm:TPNat_coh}, we show that a fragment of the structure of pseudonatural transformations satisfies the hypotheses of Squier's Theorem while an other satisfies the hypotheses of Theorem \ref{thm:main_theory}, which is temporarily admitted. The following sections contain the proof of Theorem \ref{thm:main_theory}. The first half of the proof is contained in Section \ref{sec:transfo_polygraph} and consists in applying a series of transformations to a $(4,3)$-polygraph satisfying the hypotheses of Theorem \ref{thm:main_theory}. The combinatorics of the result of these transformations is analysed in Section \ref{sec:proof_final} where we conclude the proof.

\paragraph*{Acknowlegments}

This work was supported by the Sorbonne-Paris-Cité IDEX grant Focal and
the ANR grant ANR-13-BS02-0005-02 CATHRE.

\newpage
\section{Higher-dimensional rewriting}
\label{sec:rewriting}
We recall definitions and results from rewriting theory. Section \ref{subsec:poly} is devoted to polygraphs, which are presentations of higher-dimensional categories. In Section \ref{subsec:termination}, we define termination and enunciate Theorem \ref{thm:terminaison} which we will use throughout Sections \ref{sec:application} and \ref{sec:proof_of_coherence} in order to prove the $3$-termination of polygraphs. In Section \ref{subsec:branchings} we define the notion of branchings and classify them, which allows for a simple criterion to prove the $n$-confluence of a polygraph. Finally in Section \ref{subsec:coherence}, we define the $n$-Squier condition, and recall Squier's homotopical theorem, in a generalized form proven in \cite{G09}. We conclude this section by enunciating Theorem \ref{thm:main_theory}, whose proof will occupy Sections \ref{sec:transfo_polygraph} and \ref{sec:proof_final}. Except for Theorem \ref{thm:main_theory}, the proof of every result in this section can be found in \cite{G09}.

\subsection{Polygraphs}
\label{subsec:poly}

\begin{defn}
Let $n$ be a natural number. Let $\C$ be a (strict, globular) $n$-category. For $k \leq n$, we denote by $\C_k$ both the set of $k$-cells of $\C$ and the $k$-category obtained by deleting the cells of dimension greater than $k$. For $x \in \C_k$ and $i < k$, we denote by $\s_i(x)$ and $\t_i(x)$ respectively the $i$-source and $i$-target of $x$. Finally we write $\s(x)$ and $\t(x)$ respectively for $\s_{k-1}(x)$ and $\t_{k-1}(x)$.

For $\C$ a $2$-category, we denote by $\C^{\op}$ the $2$-category obtained by reversing the direction of the $1$-cells, and by $\C^\co$ the $2$-category obtained by reversing the direction of the $2$-cells.
\end{defn}

We recall the definition of Polygraphs from \cite{B93}. For $n \in \mathbb N$, we denote by $\Cat_n$ the category of $n$-categories and by $\Graph_n$ the category of $n$-graphs. The category of $n$-categories equipped with a cellular extension, denoted by $\Cat_n^+$, is the limit of the following diagram:

\[
\xymatrix{
  \Cat_n^+ \ar[r] \ar[d]
  \ar@{} [dr] |<<<<{\lrcorner}
&
  \Graph_{n+1} \ar[d]
\\
  \Cat_n \ar[r]
&
  \Graph_n
}
\]

where the functor $\Cat_n \to \Graph_n$ forgets the categorical structure and the functor $\Graph_{n+1} \to \Graph_n$ deletes the top-dimensional cells.

Hence an object of $\Cat_n^+$ is a couple $(\C,G)$ where $\C$ is an $n$-category and $G$ is a graph 
$
\xymatrix{
\C_n 
& 
S_{n+1}
\ar@<+0.5ex>[l]^-{\t}
\ar@<-0.5ex>[l]_-{\s}
}
$, such that for any $u,v \in S_{n+1}$, the following equations are verified:
\[
\s(\s(u))= \s(\t(u)) \qquad
\t(\s(u))= \t(\t(u))
\]

Let $\R_n$ be the functor from $\Cat_{n+1}$ to $\Cat_n^+$ that sends an $(n+1)$-category $\C$ on the couple $(\C_n,
\xymatrix{
\C_n 
& 
\C_{n+1}
\ar@<+0.5ex>[l]
\ar@<-0.5ex>[l]
}
)$. This functor admits a left-adjoint $\L_n : \Cat_n^+ \to \Cat_{n+1}$ (see \cite{M08}).

We now define by induction on $n$ the category $\Pol_n$ of $n$-polygraphs together with a functor $\Q_n : \Pol_n \to \Cat_n$.

\begin{itemize}
\item The category $\Pol_0$ is the category of sets, and $\Q_0$ is the identity functor.
\item Assume $\Q_n: \Pol_n \to \Cat_n$ is defined. Then $\Pol_{n+1}$ is the limit of the following diagram:
\[
\xymatrix{
  \Pol_{n+1} \ar[r] \ar[d]
  \ar@{} [dr] |<<<<{\lrcorner}
&
  \Cat^+_{n} \ar[d]
\\
  \Pol_n \ar[r] _{\Q_n}
&
  \Cat_n,
}
\]
and $\Q_{n+1}$ is the composite
\[
\xymatrix{
  \Pol_{n+1} \ar[r]
&
  \Cat^+_{n} \ar[r] ^{\L_n}
&
  \Cat_{n+1}
}
\]
\end{itemize}

\begin{defn}
Given an $n$-polygraph $\Sigma$, the $n$-category $\Q_n(\Sigma)$ is denoted by $\Sigma^*$ and is called \emph{the free $n$-category generated by $\Sigma$}.
\end{defn}

\begin{defn}
Let $\C$ be an $n$-category, and $0 \leq i < n$ and $A \in \C_{i+1}$. If it exists, we denote by $A^{-1}$ the inverse of $A$ for the $i$-composition.

For $k \leq n$, an $(n,k)$-category is an $n$-category which has every $(i+1)$-cell invertible for the $i$-composition, for $i \geq k$. We denote by $\Cat^{(k)}_n$ the full subcategory of $\Cat_n$ whose objects are the $(n,k)$-categories.
\end{defn}

In particular $\Cat^{(0)}_n$ is the category of $n$-groupoids, and $\Cat^{(n)}_n = \Cat_n$.

The functor $\R_n$ restricts to a functor $\R^{(n)}_n$ from  $\Cat^{(n)}_{n+1}$ to $\Cat^+_n$. Once again this functor admits a left-adjoint $\L^{(n)}_n : \Cat_n^+ \to \Cat_{n+1}^{(n)}$. We define categories $\Pol_n^{(k)}$ of $(n,k)$-polygraphs and functors $\Q^{(k)}_n: \Pol_n^{(k)} \to \Cat^{(k)}_n$ in a similar way to $\Pol_n$ and $\Q_n$. See 2.2.3 in \cite{GM12} for an explicit description of this construction.

\begin{defn}
Given an $(n,k)$-polygraph $\Sigma$, the $(n,k)$-category $\Q^{(k)}_n(\Sigma)$ is denoted by $\Sigma^{*(k)}$ and is called \emph{the free $(n,k)$-category generated by $\Sigma$}. For $j \leq n$, we denote by $\Sigma^{*(k)}_j$ both the the of $j$-cells of $\Sigma^{*(k)}$ and the $(j,k)$-category generated by $\Sigma$. Hence an $(n,k)$-polygraph $\Sigma$ consists of the following data:
\[
\begin{tikzpicture}[>=To]
\matrix (m) [matrix of math nodes, 
			nodes in empty cells,
			column sep = 1cm, 
			row sep = 1cm] 
{
\Sigma_0 & 
\Sigma_1 & 
\Sigma_2 & 
(\cdots) & 
\Sigma_k & 
\Sigma_{k+1} &
(\cdots) &
\Sigma_n
 \\ 
\Sigma_0 & 
\Sigma_1^* & 
\Sigma_2^* & 
(\cdots) & 
\Sigma_k^* & 
\Sigma_{k+1}^{*(k)} &
(\cdots) &
 \\ 
};
\doublearrow{-}
{(m-1-1) -- (m-2-1)}
\draw[>->] (m-1-2|- m-1-6.south) to (m-2-2|- m-2-6.north);
\draw[>->] (m-1-3|- m-1-6.south) to (m-2-3|- m-2-6.north);
\draw[>->] (m-1-5|- m-1-6.south) to (m-2-5|- m-2-6.north);
\draw[>->] (m-1-6|- m-1-6.south) to (m-2-6|- m-2-6.north);
\draw[transform canvas={yshift=0.05cm, xshift = -0.05cm}, ->] (m-1-2.west |- m-1-6.south) to (m-2-1.east |- m-2-6.north);
\draw[transform canvas={yshift=-0.05cm, xshift = 0.05cm}, ->] (m-1-2.west |- m-1-6.south) to (m-2-1.east |- m-2-6.north);
\draw[transform canvas={yshift=0.05cm, xshift = -0.05cm}, ->] (m-1-3.west |- m-1-6.south) to (m-2-2.east |- m-2-6.north);
\draw[transform canvas={yshift=-0.05cm, xshift = 0.05cm}, ->] (m-1-3.west |- m-1-6.south) to  (m-2-2.east |- m-2-6.north);
\draw[transform canvas={yshift=0.05cm, xshift = -0.05cm}, ->] (m-1-4.west |- m-1-6.south) to  (m-2-3.east |- m-2-6.north);
\draw[transform canvas={yshift=-0.05cm, xshift = 0.05cm}, ->] (m-1-4.west |- m-1-6.south) to  (m-2-3.east |- m-2-6.north);
\draw[transform canvas={yshift=0.05cm, xshift = -0.05cm}, ->] (m-1-5.west |- m-1-6.south) to  (m-2-4.east |- m-2-6.north);
\draw[transform canvas={yshift=-0.05cm, xshift = 0.05cm}, ->] (m-1-5.west |- m-1-6.south) to (m-2-4.east |- m-2-6.north);
\draw[transform canvas={yshift=0.05cm, xshift = -0.05cm}, ->] (m-1-6.west |- m-1-6.south) to  (m-2-5.east |- m-2-6.north);
\draw[transform canvas={yshift=-0.05cm, xshift = 0.05cm}, ->] (m-1-6.west |- m-1-6.south) to  (m-2-5.east |- m-2-6.north);
\draw[transform canvas={yshift=0.05cm, xshift = -0.05cm}, ->] (m-1-7.west |- m-1-6.south) to  (m-2-6.east |- m-2-6.north);
\draw[transform canvas={yshift=-0.05cm, xshift = 0.05cm}, ->] (m-1-7.west |- m-1-6.south) to  (m-2-6.east |- m-2-6.north);
\draw[transform canvas={yshift=0.05cm, xshift = -0.05cm}, ->] (m-1-8.west |- m-1-6.south) to (m-2-7.east |- m-2-6.north);
\draw[transform canvas={yshift=-0.05cm, xshift = 0.05cm}, ->] (m-1-8.west |- m-1-6.south) to (m-2-7.east |- m-2-6.north);
\draw[transform canvas={yshift=0.1cm}, ->] (m-2-2.west |- m-2-1) to (m-2-1);
\draw[transform canvas={yshift=-0.1cm}, ->] (m-2-2.west |- m-2-1) to (m-2-1);
\draw[transform canvas={yshift=0.1cm}, ->] (m-2-3.west |- m-2-1) to (m-2-2.east |- m-2-1);
\draw[transform canvas={yshift=-0.1cm}, ->] (m-2-3.west |- m-2-1) to (m-2-2.east |- m-2-1);
\draw[transform canvas={yshift=0.1cm}, ->] (m-2-4.west |- m-2-1) to (m-2-3.east |- m-2-1);
\draw[transform canvas={yshift=-0.1cm}, ->] (m-2-4.west |- m-2-1) to (m-2-3.east |- m-2-1);
\draw[transform canvas={yshift=0.1cm}, ->] (m-2-5.west |- m-2-1) to (m-2-4.east |- m-2-1);
\draw[transform canvas={yshift=-0.1cm}, ->] (m-2-5.west |- m-2-1) to (m-2-4.east |- m-2-1);
\draw[transform canvas={yshift=0.1cm}, ->] (m-2-6.west |- m-2-1) to (m-2-5.east |- m-2-1);
\draw[transform canvas={yshift=-0.1cm}, ->] (m-2-6.west |- m-2-1) to (m-2-5.east |- m-2-1);
\draw[transform canvas={yshift=0.1cm}, ->] (m-2-7.west |- m-2-1) to (m-2-6.east |- m-2-1);
\draw[transform canvas={yshift=-0.1cm}, ->] (m-2-7.west |- m-2-1) to (m-2-6.east |- m-2-1);
\end{tikzpicture}
\]
\end{defn}

\begin{remq}
Let $n, j$ and $k$ be integers, with $j \leq k \leq n$. Since an $(n,k)$-category is also an $(n,j)$-category, an $(n,k)$-polygraph gives rise to an $(n,j)$-polygraph. In particular, if $\Sigma$ is an $(n,k)$-polygraph, we denote by $\Sigma^{*(j)}$ the $(n,j)$-category it generates. 
\end{remq}

\begin{defn}
Let $\C$ be an $(n+1,k)$-category. We denote by $\bar \C$ the $(n,k)$-category $\C_n/\C_{n+1}$.

Let $\Sigma$ be an $(n+1,k)$-polygraph. We denote by $\bar \Sigma$ the $(n,k)$-category $\overline{\Sigma^{*(k)}}$ and call it the $(n,k)$-category presented by $\Sigma$.
\end{defn}

\subsection{Termination}
\label{subsec:termination}
\begin{defn}
Let $\Sigma$ be an $n$-polygraph. For $0 < k \leq n$, the binary relation $\rightarrow^{*}_k$ defined by $u \rightarrow^{*}_k v$ if there exists $f: u \to v$ in $ \Sigma_k^*$ is a preorder on $\Sigma_{k-1}^*$ (transitivity is given by composition, and reflexivity by the units). We say that the $n$-polygraph $\Sigma$ is \emph{$k$-terminating} if $\rightarrow^*_k$ is a well-founded ordering. We denote by $\rightarrow^+_k$ the strict ordering associated to $\rightarrow^*_k$.
\end{defn}

We recall Theorem 4.2.1 from \cite{G09}, which we will use in order to show the $3$-termination of some polygraphs.

\begin{defn}
Let $\sOrd$ be the 2-category with one object, whose $1$-cells are partially ordered sets, whose $2$-cells are monotonic functions and which $0$-composition is the cartesian product.
\end{defn}

\begin{defn}
Let $\C$ be a $2$-category, $X : \C_2 \to \sOrd$ and $Y : \C_2^{\co} \to \sOrd$ two $2$-functors, and $M$ a commutative monoid. An \emph{$(X,Y,M)$-derivation} on  $\C$ is given by, for every $2$-cell $f \in \C_2$, an application \[d(f) : X(\s(f)) \times Y(\t(f)) \to M,\]
such that for every $2$-cells $f_1,f_2 \in \C_2$, every $x$, $y$, $z$ and $t$ respectively in $X(\s(f_1))$, $Y(\t(f_1))$, $X(\s(f_2))$ and $Y(\t(f_2))$, the following equalities hold:
\[
d(f_1 \star_1 f_2)[x,t] = d(f_1)[x, Y(f_2)[y]] + d(f_2) [ X(f_1)[x],y] 
\]
\[
d(f_1 \star_0 f_2)[(x,z),(y,t)] = d(f_1)[x,y]  + d(f_2)[z,t] .
\]
\end{defn}

In order to show the $3$-termination of some polygraphs, we are going to use the following result (Theorem 4.2.1 from \cite{G09}).

\begin{thm}\label{thm:terminaison}
Let $\Sigma$ be an $n$-polygraph, $X : \Sigma^*_2 \to \sOrd$ and $Y : (\Sigma_2^*)^{\co} \to \sOrd$ two $2$-functors, and $M$ be a commutative monoid equipped with a well-founded ordering $\geq$, and whose addition is strictly monotonous in both arguments.

Suppose that for every $3$-cell $A \in \Sigma_3$, the following inequalities hold:
\[
X(\s(A)) \geq X(\t(A)) 
\qquad 
Y(\s(A)) \geq Y(\t(A))
\qquad 
d(\s(A)) > d(\t(A)).
\]

Then the $n$-polygraph $\Sigma$ is $3$-terminating.
\end{thm}

\subsection{Branchings and Confluence}
\label{subsec:branchings}
\begin{defn}
Let $\Sigma$ be an $n$-polygraph. A \emph{$k$-fold branching} of $\Sigma$ is a $k$-tuple $(f_1,f_2,\ldots,f_k)$ of $n$-cells in $\Sigma^*$ such that every $f_i$ has the same source $u$, which is called \emph{the source of the branching}.

The symmetric group $S_k$ acts on the set of all $k$-fold branchings of $\Sigma$. The equivalence class of a branching $(f_1,f_2,\ldots,f_k)$ under this action is denoted by $[f_1,f_2,\ldots,f_k]$. Such an equivalence class is called a \emph{$k$-fold symmetrical branching}, and $(f_1,f_2,\ldots,f_k)$ is called a \emph{representative} of $[f_1,f_2,\ldots,f_k]$
\end{defn}

\begin{defn}
Let $\Sigma$ be an $n$-polygraph. We denote by $\mathbb N$ the $n$-category with exactly one $k$-cell for every $k <n$, whose $n$-cells are the natural numbers and whose compositions are given by addition. 

We define an application $\l : \Sigma^* \to \mathbb N$ by setting $\l(f) = 1$ for every $f \in \Sigma_n$. For $f \in \Sigma_n^*$, we call $\l(f)$ the \emph{length} of a $f$.

An $n$-cell of length $1$ in $\Sigma^*_n$ is also called a \emph{rewriting step}.
\end{defn}

\begin{defn}
Let $\Sigma$ be an $n$-polygraph. A \emph{$k$-fold local branching} of $\Sigma$ is a $k$-fold branching $(f_1,f_2,\ldots,f_k)$ of $\Sigma$ where every $f_i$ is a rewriting step.

A $k$-fold local branching $(f_1,\ldots,f_k)$ of source $u$ is a \emph{strict aspherical branching} if there exists an integer $i$ such that $f_{i} = f_{i+1}$. We say that it is an \emph{aspherical branching} if it is in the equivalence class of a strict aspherical branching. 

A $k$-fold local branching $(f_1,\ldots,f_k)$ is a \emph{strict Peiffer branching} if it is not aspherical and there exist $v_1,v_2 \in \Sigma_{n-1}^*$ such that $u = v_1 \star_i v_2$, an integer $m < n$ and $f'_1,\ldots,f'_k \in \Sigma_n^*$ such that for every $j \leq m$, $f_j = f'_j \star_i  v_2$ and for every $j > m$, $f_j = v_1 \star_i f'_j$. It is a \emph{Peiffer branching} if it is in the equivalence class of a strict Peiffer branching.

A local branching that is neither aspherical nor Peiffer is \emph{overlapping}.
\end{defn}

Given an $n$-polygraph $\Sigma$, one defines an order $\subseteq$ on $k$-fold local branchings by saying that $(f_1,\ldots,f_k) \subseteq (u \star_i f_1 \star_i v,\ldots ,u \star_i f_k \star_i v)$ for every $u,v \in \Sigma_{n-1}^*$ and every $k$-fold local branching $(f_1,\ldots,f_k)$.

\begin{defn}
An overlapping branching that is minimal for $\subseteq$ is a \emph{critical branching}.

A $2$-fold (resp. $3$-fold) critical branching is also called a \emph{critical pair} (resp. \emph{critical triple}).
\end{defn}

\begin{defn}
Let $\Sigma$ be an $n$-polygraph. A $2$-fold branching $(f,g)$ is \emph{confluent} if there are $f',g' \in \Sigma_n^*$ of the following shape:

\[
\xymatrix @C = 4em @R = 1.5em{
& 
\ar @/^/ [rd] ^{f'}
& \\
\ar @/^/ [ru] ^{f}
\ar @/_/ [rd] _{g} 
& 
& \\
&  
\ar @/_/ [ru] _{g'}
& \\
}
\]
\end{defn}

\begin{defn}
An $n$-polygraph $\Sigma$ is \emph{$k$-confluent} if every $2$-fold branching of $\Sigma_k$ is confluent.
\end{defn}

\begin{defn}
An $n$-polygraph is \emph{$k$-convergent} if it is $k$-terminating and $k$-confluent.
\end{defn}

The following two propositions are proven in \cite{G09}.

\begin{prop}\label{prop:confluence_local}
Let $\Sigma$ be an $n$-terminating $n$-polygraph. It is $n$-confluent if and only if every $2$-fold critical branching is confluent.
\end{prop}

\begin{prop}
Let $\Sigma$ be a $k$-convergent $n$-polygraph. For every $u \in \Sigma^*_{k-1}$, there exists a unique $v \in \Sigma_{k-1}^*$ such that $u \rightarrow^*_k v$ and $v$ is minimal for $\rightarrow^*_k$.
\end{prop}

\begin{defn}
Let $\Sigma$ be an $n$-polygraph. \emph{A normal form} for $\Sigma$ is an $(n-1)$-cell minimal for $\rightarrow^*_n$.

If $\Sigma$ is $n$-convergent, for every $u \in \Sigma^*_{n-1}$, the unique normal form $v$ such that $u \rightarrow^*_n v$ is denoted by $\hat u$ and is called \emph{the normal form of $u$}.
\end{defn}

\subsection{Coherence}
\label{subsec:coherence}
\begin{defn}
Two $k$-cells are \emph{parallel} if they have the same source and the same target.

An $(n+1)$-category $\C$ is \emph{$n$-coherent} if, for each pair $(f,g)$ of parallel $n$-cells in $\C_n$, there exists an $(n+1)$-cell $A: f \to g$ in $\C_{n+1}$. 
\end{defn}

\begin{defn}
Let $\Sigma$ be an $(n+1)$-polygraph, and $(f,g)$ be a local branching of $\Sigma_n$. A \emph{filling} of $(f,g)$ is an $(n+1)$-cell $A \in \Sigma_{n+1}^{*(n)}$ of the shape: 
\[
\xymatrix @C = 4em @R = 1.5em{
& 
\ar @/^/ [rd] ^{  }
\ar@2 [dd] ^{A}
& \\
\ar @/^/ [ru] ^{f}
\ar @/_/ [rd] _{g} 
& 
& \\
&  
\ar @/_/ [ru] _{}
& \\
}
\]
\end{defn}

\begin{defn}
An $(n+1)$-polygraph $\Sigma$ satisfies the \emph{$n$-Squier condition} if:
\begin{itemize}
\item it is $n$-convergent,
\item there is a bijective application from $\Sigma_{n+1}$ to the set of all critical pairs of $\Sigma_n$ that associates to every $A \in \Sigma_{n+1}$, a critical pair $b$ of $\Sigma_n$ such that $A$ is a filling of a representative of $b$.
\end{itemize}
\end{defn}

The following Theorem is due to Squier for $n=2$ \cite{S94} and was extended to any integer $n\geq 2$ by Guiraud and Malbos \cite{G09}.

\begin{thm}\label{thm:squier}
Let $\Sigma$ be an $(n+1)$-polygraph satisfying the $n$-Squier condition. Then the free $(n+1,n-1)$-category $\Sigma^{*(n-1)}$ is $n$-coherent.
\end{thm}

In the proof of this Theorem appears the following result (Lemma 4.3.3 in \cite{G09}).

\begin{prop}\label{prop:lem_squier}
Let $\Sigma$ be an $(n+1)$-polygraph satisfying the $n$-Squier condition.

For every parallel $n$-cells $f,g \in \Sigma_{n}^*$ whose target is a normal form, there exists an $(n+1)$-cell $A : f \to g$ in $\Sigma_{n+1}^{*(n)}$.
\end{prop}

Let us compare those two last results. Let $\Sigma$ be an $(n+1)$-polygraph satifying the $n$-Squier relation, and let $f,g \in \Sigma_{n}^*$ be two parallel $n$-cells whose target is a normal form. According to Theorem \ref{thm:squier}, there exists an $(n+1)$-cell $A : f \to g$ in the free $(n+1,n-1)$-category $\Sigma_{n+1}^{*(n-1)}$. Proposition \ref{prop:lem_squier} shows that such an $A$ can be chosen in the free $(n+1,n)$-category $\Sigma_{n+1}^{*(n)}$, where the $n$-cells are not invertible. Hence for cells $f,g \in \Sigma_{n}^*$ whose target is a normal form, Proposition \ref{prop:lem_squier} is more precise than Theorem \ref{thm:squier}.

\begin{defn}
Let $\Sigma$ be an $(n+1)$-polygraph, and $(f,g)$ a local branching in $\Sigma_n$. Depending on the nature of $(f,g)$, we define the notion of \emph{canonical filling} of $(f,g)$.
\begin{itemize}
\item If $(f,g)$ is an aspherical branching, then its canonical filling is the identity $1_f$.
\item If $(f,g)$ is a Peiffer branching, if $(f,g) = (f' \star_i v_1, v_2 \star_i g')$ (resp. $(f,g) = (v_1 \star_i f', g' \star_i v_2)$), then its canonical filling is $1_{f' \star_i g'}$ (resp. $1_{g' \star_i f'}$). 
\item Assume that $\Sigma$ satisfies the $n$-Squier condition, and let $(f,g)$ be a critical pair. Let $A$ be the $(n+1)$-cell associated to $[f,g]$. If $A$ is a filling of $(f,g)$ then the canonical filling of $(f,g)$ is $A$. Otherwise, $A$ is a filling of $(g,f)$ and the canonical filling of $(f,g)$ is $A^{-1}$.
\item Assume that the branching $(f,g)$ admits a canonical filler $A$. Then the canonical filler of $(u \star_i f \star_i v,u \star_i f \star_i v)$ is $u \star_i A \star_i v$.
\end{itemize}
\end{defn}

\begin{defn}
Let $\Sigma$ be an $(n+2,n)$-polygraph satisfying the $n$-Squier condition, and $(f,g,h)$ be a local branching of $\Sigma_n$. A \emph{filling} of $(f,g,h)$ is an $(n+2)$-cell $\alpha \in \Sigma_{n+2}^{*(n)}$ of the shape: 
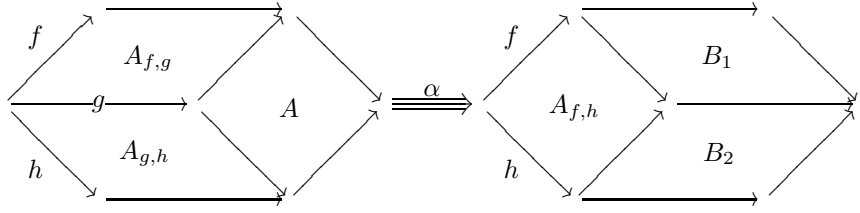
\[
\xymatrix @C = 3em @R = 3em {
&
\ar [rr]
\ar@{} [rd] |-{A_{f,g}}
& &
\ar [rd]
\ar@{} [dd] |-{A}
& & &
\ar [rr]
\ar@{} [rrrd] |-{B_1}
\ar [rd]
\ar@{} [dd] |-{A_{f,h}}
& &
\ar [rd]
&
\\
\ar [ru] ^f
\ar [rr] |g
\ar [rd] _h
& &
\ar@{} [ld] |-{A_{g,h}}
\ar [ru]
\ar [rd]
&
&
\ar@3 [r] ^{\alpha}
&
\ar [ru] ^f
\ar [rd] _h
& & 
\ar [rr]
\ar@{} [rd] |-{B_2}
& &
\\
&
\ar [rr]
& &
\ar [ru]
& & &
\ar [ru]
\ar [rr]
& &
\ar [ru]
&
}
\]
where $A,A_{f,g},A_{g,h},A_{f,h},B_1$ and $B_2$ are $(n+1)$-cells in $\Sigma^{*(n)}_{n+1}$, and $A_{f,g}$, $A_{g,h}$ and $A_{f,h}$ are the canonical fillings of respectively $(f,g)$, $(g,h)$ and $(f,h)$.
\end{defn}

\begin{defn}
An $(n+2,n)$-polygraph $\Sigma$ satisfies the \emph{$n$-Squier condition of depth $2$} if:
\begin{itemize}
\item it satisfies the $n$-Squier condition,
\item there is a bijective application from $\Sigma_{n+2}$ to the set of all critical triples of $\Sigma_n$ that associates to every $\alpha \in \Sigma_{n+2}$ a critical triple $b$ of $\Sigma_n$ such that $\alpha$ is a filling of a representative of $b$.
\end{itemize}
\end{defn}

We now enunciate the theorem whose proof will occupy Sections \ref{sec:transfo_polygraph} and \ref{sec:proof_final}.
\begin{thm}\label{thm:main_theory}
Let $\Sigma$ be a $(4,2)$-polygraph satisfying the $2$-Squier condition of depth $2$. 

For every parallel $3$-cells $A,B \in \Sigma_{3}^{*(2)}$ whose $1$-target is a normal form, there exists a $4$-cell $\alpha : A \qfl B$ in the free $(4,2)$-category $\Sigma_{4}^{*(2)}$.
\end{thm}
This Theorem can be compared with Proposition 4.4.4 in \cite{GM12}. There, for every parallel $A, B \in \Sigma_{3}^{*(1)}$, a $4$-cell $\alpha$ is constructed in the free $(4,1)$-category $\Sigma_{4}^{*(1)}$. Hence Theorem \ref{thm:main_theory} gives a more precise statement, at the cost of restricting the set of $3$-cells allowed.
\newpage

\section{Partial coherence and transformation of polygraphs}
\label{sec:white_category}

In Section \ref{subsec:white_categories} we define the notion of white-category together with the associated notion of white-polygraph. The $2$-white-categories are also known as \emph{sesquicategories} (see \cite{S96}). White-categories are strict categories in which the interchange law between the compositions $\star_0$ and $\star_i$ need not hold, for every $i > 0$. That is, strict $n$-categories are exactly the $n$-white-categories satisfying the additional condition that for every $i$-cells $f$ and $g$ of $1$-sources (resp. $1$-targets) $u$ and $v$ (resp. $u'$ and $v'$): $(f \star_0 v) \star_i (u' \star_0 g ) = (u \star_0 g ) \star_i (f \star_0 v)$.

In Section \ref{subsec:partial_coh}, we define a notion of partial coherence for $(4,3)$-white-categories. We show a simple criterion in order to deduce the partial coherence of a $(4,3)$-white-category from that of an other one. This criterion will be used throughout Section \ref{sec:transfo_polygraph}. We also adapt the notion of Tietze-transformation from \cite{GGM15} to our setting of partial coherence in white-categories, in preparation for Section \ref{subsec:retournement}.

In Section \ref{subsec:inj_functors}, we study injective functors between free white-categories. In particular, we give a sufficient condition for a morphism of white-polygraphs to yield an injective functor between the white-categories they generate. This result will be used in Section \ref{subsec:adjunc_2cell}.

Note that, although Sections \ref{subsec:partial_coh} and \ref{subsec:inj_functors} are expressed in terms of white-categories (since this is how they will be used throughout Section \ref{sec:transfo_polygraph}), all the definitions and results in these Sections also hold in terms of strict categories, \emph{mutatis mutandis}. 

\subsection{White-categories and White-polygraphs}
\label{subsec:white_categories}

\begin{defn}
Let $n \in \mathbb N$. An \emph{$(n+1)$-white-category} is given by:
\begin{itemize}
\item a set $\C_0$,
\item for every $x,y \in \C_0$, an $n$-category $\C(x,y)$. We denote by $\star_{k+1}$ the $k$-composition in this category,
\item for every $z \in \C_0$ and every $u:x \to y \in \C_1$, functors $u\star_0 \_ : \C(y,z)\to \C(x,z)$  and $\_ \star_0 u : \C(z,x) \to \C(z,y)$, so that for every composable $1$-cells $u,v \in \C_1$, their composite $u \star_0 v$ is defined in a unique way,
\item for every $x \in \C_0$, a $1$-cell $1_x \in \C(x,x)$.
\end{itemize}
Moreover, this data must satisfy the following axioms:
\begin{itemize}
\item For every $x \in \C_0$, and every $y \in \C_0$, the functors $1_x \star_0 \_ : \C(x,y) \rightarrow \C(x,y)$ and $\_ \star_0 1_y: \C(x,y) \rightarrow \C(x,y)$ are identities.
\item For every $u,v \in \C_1$, the following equalities hold:
\begin{itemize}
\item $u \star_0 (v \star_0 \_) = (u \star_0 v) \star_0 \_$,
\item $u \star_0 (\_ \star_0 v) = (u \star_0 \_) \star_0 v$,
\item $\_ \star_0 (u \star_0 v) = (\_ \star_0 u) \star_0 v$,
\end{itemize}
\end{itemize}

An \emph{$(n,k)$-white-category} is an $n$-white-category in which every $(i+1)$-cell is invertible for the $i$-composition, for every $i \geq k$.

\end{defn}

Let $n$ be a natural number. Let $\C$ be an $n$-white-category. For $k \leq n$, we denote by $\C_k$ both the set of $k$-cells of $\C$ and the $k$-white-category obtained by deleting the cells of dimension greater than $k$. For $x \in \C_k$ and $i < k$, we denote by $\s_i(x)$ and $\t_i(x)$ respectively the $i$-source and $i$-target of $x$. Finally we write $\s(x)$ and $\t(x)$ respectively for $\s_{k-1}(x)$ and $\t_{k-1}(x)$.

\begin{defn}
Let $\mathcal C$ and $\mathcal D$ be $n$-white-categories. An \emph{$n$-white-functor} is given by: 
\begin{itemize}
\item an application $F_0: \C_0 \to \D_0$,
\item for every $x,y \in \C_0$, a functor $F_{x,y}: \C(x,y) \to \D(F_0(x),F_0(y))$.
\end{itemize}
Moreover, this data must satisfy the following axioms:
\begin{itemize}
\item for every $x \in \C_0$, $F(1_x) = 1_{F_0(x)}$,
\item for every $z \in \C_0$ and $u:x \to y \in \C_1$, the following equalities hold between functors:
\begin{itemize}
\item $F(u) \star_0 F(\_) = F(u \star_0 \_) : \C(y,z) \to \D(F_0(x),F_0(z))$
\item $F(\_) \star_0 F(u) = F(\_ \star_0 u) : \C(z,x) \to \D(F_0(z),F_0(y))$
\end{itemize}
\end{itemize}

\end{defn}

This makes $n$-white-categories into a category, that we denote by $\WCat_n$.

\begin{remq}
\label{remq:struct_monoidale}
Let us define a structure of monoidal category $\otimes$ on $\Cat_n$, in such a way that $\WCat_{n+1}$ is the category of categories enriched over $(\Cat_n,\otimes)$.

Let $\C, \D$ be two $n$-categories. The $n$-categories $\C \times \D_0$ and $\C_0 \times \D$ are defined as follows:
\[
\C \times \D_0 := \bigsqcup_{y \in \D_0} \C, \qquad \C_0 \times \D := \bigsqcup_{x \in \C_0} \D 
\]

Let $\C_0 \times \D_0$ be the $n$-category whose $0$-cells are couples $(x,y) \in \C_0 \times \D_0$, and whose $i$-cells are identities for every $i>0$. Let $F: \C_0 \times \D_0 \to \C \times \D_0$ (resp. $G: \C_0 \times \D_0 \to \C_0 \times \D$) be the $n$-functor which is the identity on $0$-cells. Then $\C \otimes \D$ is the pushout $(\C \times \D_0 )\oplus_{\C_0 \times \D_0} (\C_0 \times \D)$:

\[
\xymatrix{
  \C_0 \times \D_0 \ar[r]^F \ar[d]_G
  \ar@{} [dr] |>>>>{\ulcorner}
&
  \C \times \D_0 \ar[d]
\\
  \C_0 \times \D \ar[r]
&
  \C \otimes \D.
}
\]

\end{remq}
\bigskip 

The category of $n$-white-categories equipped with a cellular extension, denoted by $\WCat_n^+$, is the limit of the following diagram:
\[
\xymatrix{
  \WCat_n^+ \ar[r] \ar[d]
  \ar@{} [dr] |<<<<{\lrcorner}
&
  \Graph_{n+1} \ar[d]
\\
  \WCat_n \ar[r]
&
  \Graph_n
}
\]
where the functor $\WCat_n \to \Graph_n$ forgets the white-categorical structure and the functor $\Graph_{n+1} \to \Graph_n$ deletes the top-dimensional cells.

Let $\R^\w_n$ be the functor from $\WCat_{n+1}$ to $\WCat_n^+$ that sends an $(n+1)$-white-category $\C$ on the couple $(\C_n,
\xymatrix{
\C_n 
& 
\C_{n+1}
\ar@<+0.5ex>[l]
\ar@<-0.5ex>[l]
}
)$.
\begin{prop}
The functor $\R_n^\w$ admits a left-adjoint $\L^\w_n : \WCat_n^+ \to \WCat_{n+1}$.
\end{prop}
\begin{proof}
Let $(\C,\Sigma) \in \WCat_n^+$ be an $n$-white-category equipped with a cellular extension. The construction of $\L^\w_n(\C,\Sigma)$ is split into three parts:
\begin{itemize}
\item First, we define a formal language $E_\Sigma$.
\item Then, we define a typing system $T_\C$ on $E_\Sigma$. We denote by $E^T_\Sigma$ the set of all typable expressions of $E_\Sigma$.
\item Finally, we define an equivalence relation $\equiv_\Sigma^*$ on $E_\Sigma^T$. The set of $(n+1)$-cell of $\L^\w_n(\C,\Sigma)$ is then the quotient $E_\Sigma^T / \equiv_\Sigma^*$.
\end{itemize}
Let $E_\Sigma$ be the formal language consisting of:
\begin{itemize}
\item For every $1$-cells $u, v \in \C_1$, and every $(n+1)$-cell $A \in \Sigma_{n+1}$, such that $\t_0(u) = \s_0(A)$ and $\t_0(A) = \s_0(v)$, a constant symbol $c_{uAv}$.
\item For every $n$-cell $f \in \C_n$, a constant symbol $i_f$.
\item For every $ 0 < i \leq n$, a binary function symbol $\star_i$.
\end{itemize}
Thus $E_\Sigma$ is the smallest set of expressions containing the constant symbols and such that $e \star_i f \in \Sigma$ whenever $e,f \in E_\Sigma$.

Let $T_\C$ be the set of all $n$-spheres of $\C$, that is of couples $(f,g)$ in $\C_n$ such that $\s(f)=\s(g)$ and $\t(f) = \t(g)$. For $e \in E_\Sigma$ and $t \in T_\C$, we define $e : t$ (read as "$e$ is of type $t$") as the smallest relation satisfying the following axioms:
\begin{itemize}
\item For every $1$-cells $u$ and $v$ in $\C_1$, and every $(n+1)$-cell $A \in \Sigma$, such that $\t_0(u) = \s_0(A)$ and $\t_0(A) = \s_0(v)$
\[ c_{uAv} : (u\s(A)v,u\t(A)v) \]
\item For every $n$-cell $f \in \C_n$  \[ i_f: (f,f) \]
\item For every $e_1,e_2 \in E_\Sigma$ and $i < n$, if $e_1 : (s_1,t_1)$, $e_2 : (s_2,t_2)$ and $\t_i(t_1) = \s_i(s_2)$, then \[e_1 \star_i e_2 : (s_1 \star_i s_2,t_1 \star_i t_2)\]
\item For every $e_1,e_2 \in E_\Sigma$, if $e_1 : (s_1,t_1)$, $e_2 : (s_2,t_2)$ and $t_1 = s_2$, then \[e_1 \star_n e_2 : (s_1 , t_2)\]
\end{itemize}
An expression $e \in E_\Sigma$ is said to be \emph{typable} if $e : (s,t)$ for some $n$-sphere $(s,t) \in T_\C$. Moreover there is only one such $n$-sphere, so the operations $\s(e) := s$ and $\t(e) := t$ are well defined. We denote by $E^T_\Sigma$ be the set of all typable expressions.

Let $\equiv_\Sigma$ be the symmetric relation generated by the following relations on   $E^T_\Sigma$:
\begin{itemize}
\item For every $A,B,C,D \in E^T_\Sigma$, and every $i_1,i_2 \leq n $ non-zero distinct natural numbers, \[(A \star_{i_1} B) \star_{i_2} (C \star_{i_1} D) \equiv_\Sigma (A \star_{i_2} C) \star_{i_1} (B \star_{i_2} D)\]
\item For every $A,B,C  \in E^T_\Sigma$, and every $0 < i \leq n $, \[(A \star_i B) \star_i C \equiv_\Sigma A \star_i (B \star_i C)\]
\item For every $A \in E^T_\Sigma$ and $f \in \C_n$:
\[i_{f} \star_n A \equiv_\Sigma A 
\qquad
A \star_n i_{f} \equiv_\Sigma A \]
\item For every $f_1,f_2 \in \C_n$ and every $i <n$, \[i_{f_1} \star_i i_{f_2} \equiv_\Sigma i_{f_1 \star_i f_2}\]
\item For every $A,A',B \in E^T_\Sigma$ and every $0 < i \leq n $, if $A \equiv_\Sigma A'$, then
\[
A \star_i B \equiv_\Sigma A' \star_i B
\]
\item For every $A,B,B' \in E^T_\Sigma$ and every $0 < i \leq n $, if $B \equiv_\Sigma B'$, then
\[
A \star_i B \equiv_\Sigma A \star_i B'
\]
\end{itemize} 

Let $\equiv^*_\Sigma$ be the reflexive closure of $\equiv_\Sigma$. The $(n+1)$-cells of $\L^\w_n(\C,\Sigma)$ are given by the quotient $E^T_\Sigma /\equiv_\Sigma^*$. The $i$-composition is given by the one of $E^T_\Sigma$, and identities by $i_f$.
\end{proof}
\begin{defn}
We now define by induction on $n$ the category $\WPol_n$ of $n$-white-polygraphs together with a functor $\Q^\w_n : \WPol_n \to \WCat_n$.

\begin{itemize}
\item The category $\WPol_0$ is the category of sets, and $\Q^\w_0$ is the identity functor.
\item Assume $\Q^\w_n: \WPol_n \to \WCat_n$ defined. Then $\WPol_{n+1}$ is the limit of the following diagram:
\[
\xymatrix{
  \WPol_{n+1} \ar[r] \ar[d]
  \ar@{} [dr] |<<<<{\lrcorner}
&
  \WCat^+_{n} \ar[d]
\\
  \WPol_n \ar[r] _{\Q^\w_n}
&
  \WCat_n,
}
\]
and $\Q^\w_{n+1}$ is the composite
\[
\xymatrix{
  \WPol_{n+1} \ar[r]
&
  \WCat^+_{n} \ar[r] ^{\L^\w_n}
&
  \WCat_{n+1}
}
\]
\end{itemize}

Given an $n$-white-polygraph $\Sigma$, the $n$-white-category $\Q^\w_n(\Sigma)$ is denoted by $\Sigma^\w$ and is called \emph{the free $n$-white-category generated by $\Sigma$}.
\end{defn}

\begin{defn}
Let $\WCat_{n+1}^{\w(n)}$ be the category of $(n+1,n)$-white-categories. Once again we have a functor $\R_n^{\w(n)} : \WCat_{n+1}^{\w(n)} \to \WCat_n^+$, and we are going to describe its left-adjoint $\L_{n+1}^{\w(n)}$. Let $(\C,\Sigma)$ be an $n$-white-category together with a cellular extension. To construct $\L_{n+1}^{\w(n)}(\C,\Sigma)$, we adapt the construction of the free $n$-white-categories as follows:
\begin{itemize}
\item Let $F_\Sigma$ be the formal language  $E_{\Sigma \cup \bar \Sigma}$, where $\bar \Sigma$ consists of formal inverses to the elements of $\Sigma$ (that is their source and targets are reversed).
\item The type system is extended by setting, for every $1$-cells $u, v$ in $\C_1$ and every $(n+1)$-cell $A \in \Sigma$ such that $\t_0(u) = \s_0(v)$ and $\t_0(A) = \s_0(A)$:
\[ c_{u\bar Av} : (u\t(A)v,u\s(A)v). \]
We denote by $F_\Sigma^T$ the set of all typable expressions for this new typing system.
\item We extend $\equiv_\Sigma$ into a relation denoted by $\cong_\Sigma$ by adding the following relations:
\[
c_{uAv} \star_n c_{u \bar A v} \cong_\Sigma i_{u\s(A)v}
\qquad
c_{u\bar Av} \star_n c_{u A v} \cong_\Sigma i_{u\t(A)v}
\]
for every $u,v$ in $\C_1$ and every $(n+1)$-cell $A \in \Sigma$, such that $\t_0(u) = \s_0(A)$ and $\t_0(A) = \s_0(v)$.
\end{itemize}

We define categories $\WPol_n^{(k)}$ of $(n,k)$-white-polygraphs and functors $\Q^{\w(k)}_n: \WPol_n^{(k)} \to \WCat^{(k)}_n$ in a similar way to $\Pol^{(k)}_n$ and $\Q^{(k)}_n$.
\end{defn}

\begin{defn}
Given an $(n,k)$-white-polygraph $\Sigma$, the $(n,k)$-white-category $\Q^{\w(k)}_n(\Sigma)$ is denoted by $\Sigma^{\w(k)}$ and is called \emph{the free $(n,k)$-white-category generated by $\Sigma$}. For $j \leq n$, we denote by $\Sigma^{\w(k)}_j$ both the the of $j$-cells of $\Sigma^{\w(k)}$ and the $(j,k)$-category generated by $\Sigma$. Hence an $(n,k)$-polygraph $\Sigma$ consists of the following data:
\[
\begin{tikzpicture}[>=To]
\matrix (m) [matrix of math nodes, 
			nodes in empty cells,
			column sep = 1cm, 
			row sep = 1cm] 
{
\Sigma_0 & 
\Sigma_1 & 
\Sigma_2 & 
(\cdots) & 
\Sigma_k & 
\Sigma_{k+1} &
(\cdots) &
\Sigma_n
 \\ 
\Sigma_0 & 
\Sigma_1^\w & 
\Sigma_2^\w & 
(\cdots) & 
\Sigma_k^\w & 
\Sigma_{k+1}^{\w(k)} &
(\cdots) &
 \\ 
};
\doublearrow{-}
{(m-1-1) -- (m-2-1)}
\draw[>->] (m-1-2|- m-1-6.south) to (m-2-2|- m-2-6.north);
\draw[>->] (m-1-3|- m-1-6.south) to (m-2-3|- m-2-6.north);
\draw[>->] (m-1-5|- m-1-6.south) to (m-2-5|- m-2-6.north);
\draw[>->] (m-1-6|- m-1-6.south) to (m-2-6|- m-2-6.north);
\draw[transform canvas={yshift=0.05cm, xshift = -0.05cm}, ->] (m-1-2.west |- m-1-6.south) to (m-2-1.east |- m-2-6.north);
\draw[transform canvas={yshift=-0.05cm, xshift = 0.05cm}, ->] (m-1-2.west |- m-1-6.south) to (m-2-1.east |- m-2-6.north);
\draw[transform canvas={yshift=0.05cm, xshift = -0.05cm}, ->] (m-1-3.west |- m-1-6.south) to (m-2-2.east |- m-2-6.north);
\draw[transform canvas={yshift=-0.05cm, xshift = 0.05cm}, ->] (m-1-3.west |- m-1-6.south) to  (m-2-2.east |- m-2-6.north);
\draw[transform canvas={yshift=0.05cm, xshift = -0.05cm}, ->] (m-1-4.west |- m-1-6.south) to  (m-2-3.east |- m-2-6.north);
\draw[transform canvas={yshift=-0.05cm, xshift = 0.05cm}, ->] (m-1-4.west |- m-1-6.south) to  (m-2-3.east |- m-2-6.north);
\draw[transform canvas={yshift=0.05cm, xshift = -0.05cm}, ->] (m-1-5.west |- m-1-6.south) to  (m-2-4.east |- m-2-6.north);
\draw[transform canvas={yshift=-0.05cm, xshift = 0.05cm}, ->] (m-1-5.west |- m-1-6.south) to (m-2-4.east |- m-2-6.north);
\draw[transform canvas={yshift=0.05cm, xshift = -0.05cm}, ->] (m-1-6.west |- m-1-6.south) to  (m-2-5.east |- m-2-6.north);
\draw[transform canvas={yshift=-0.05cm, xshift = 0.05cm}, ->] (m-1-6.west |- m-1-6.south) to  (m-2-5.east |- m-2-6.north);
\draw[transform canvas={yshift=0.05cm, xshift = -0.05cm}, ->] (m-1-7.west |- m-1-6.south) to  (m-2-6.east |- m-2-6.north);
\draw[transform canvas={yshift=-0.05cm, xshift = 0.05cm}, ->] (m-1-7.west |- m-1-6.south) to  (m-2-6.east |- m-2-6.north);
\draw[transform canvas={yshift=0.05cm, xshift = -0.05cm}, ->] (m-1-8.west |- m-1-6.south) to (m-2-7.east |- m-2-6.north);
\draw[transform canvas={yshift=-0.05cm, xshift = 0.05cm}, ->] (m-1-8.west |- m-1-6.south) to (m-2-7.east |- m-2-6.north);
\draw[transform canvas={yshift=0.1cm}, ->] (m-2-2.west |- m-2-1) to (m-2-1);
\draw[transform canvas={yshift=-0.1cm}, ->] (m-2-2.west |- m-2-1) to (m-2-1);
\draw[transform canvas={yshift=0.1cm}, ->] (m-2-3.west |- m-2-1) to (m-2-2.east |- m-2-1);
\draw[transform canvas={yshift=-0.1cm}, ->] (m-2-3.west |- m-2-1) to (m-2-2.east |- m-2-1);
\draw[transform canvas={yshift=0.1cm}, ->] (m-2-4.west |- m-2-1) to (m-2-3.east |- m-2-1);
\draw[transform canvas={yshift=-0.1cm}, ->] (m-2-4.west |- m-2-1) to (m-2-3.east |- m-2-1);
\draw[transform canvas={yshift=0.1cm}, ->] (m-2-5.west |- m-2-1) to (m-2-4.east |- m-2-1);
\draw[transform canvas={yshift=-0.1cm}, ->] (m-2-5.west |- m-2-1) to (m-2-4.east |- m-2-1);
\draw[transform canvas={yshift=0.1cm}, ->] (m-2-6.west |- m-2-1) to (m-2-5.east |- m-2-1);
\draw[transform canvas={yshift=-0.1cm}, ->] (m-2-6.west |- m-2-1) to (m-2-5.east |- m-2-1);
\draw[transform canvas={yshift=0.1cm}, ->] (m-2-7.west |- m-2-1) to (m-2-6.east |- m-2-1);
\draw[transform canvas={yshift=-0.1cm}, ->] (m-2-7.west |- m-2-1) to (m-2-6.east |- m-2-1);
\end{tikzpicture}
\]
\end{defn}

\subsection{Partial coherence in pointed $(4,3)$-white-categories}
\label{subsec:partial_coh}
\begin{defn}
A \emph{pointed} $(4,3)$-white-category is a couple $(\C,S)$, where $\C$ is a $4$-white-category, and $S$ is a subset of $\C_2$.
\end{defn}

\begin{defn}
Let $(\C,S)$ be a pointed $(4,3)$-white-category. The \emph{restriction of $\C$ to $S$}, denoted by $\C \restriction S$, is the following $(2,1)$-category:
\begin{itemize}
\item its $0$-cells are the  $2$-cells of $\C_2$ that lie in $S$,
\item its $1$-cells are the $3$-cells of $\C_3$ with source and target in $S$,
\item its $2$-cells are the $4$-cells of $\C_4$ with $2$-source and $2$-target in $S$,
\item its $0$-composition and $1$-composition are respectively induced by the compositions $\star_2$ and $\star_3$ of $\C$.
\end{itemize}
\end{defn}

\begin{defn}
Let $(\C,S)$ be a pointed $(4,3)$-white-category. We say that $\C$ is \emph{$S$-coherent} if for every parallel $1$-cells $A$, $B$ in the $(2,1)$-category $\C \restriction S$, there exists a $2$-cell $\alpha : A \Rightarrow B \in \C \restriction S$.
\end{defn}

\begin{ex}
Every $(4,3)$-white-category is $\emptyset$-coherent. A $(4,3)$-white-category $\C$ is $\C_2$-coherent if and only if it is $3$-coherent.
\end{ex}

We now rephrase Theorem \ref{thm:main_theory} in the setting of partial coherence.

\begin{repthm}{thm:main_theory}
Let $\A$ be a $(4,2)$-polygraph satisfying  the $2$-Squier condition of depth $2$, and let $S_\A$ be the set of all $2$-cells whose target is a normal form.

Then $\A$ is $S_\A$-coherent.
\end{repthm}

\begin{defn}
Let $\C$ and $\D$ be two $2$-categories, $F : \C \to \D$ a $2$-functor.

We say that $F$ is \emph{$0$-surjective} if the application $F : \C_0 \to \D_0$ is surjective.

Let $0 < k < 2$. We say that $F$ is \emph{$k$-surjective} if, for every $(k-1)$-parallel cells $s,t \in \C_{k-1}$, the application $F : \C_k(s,t) \to \D_k(F(s),F(t))$ is surjective.
\end{defn}

\begin{defn}
Let $(\C,S)$ and $(\C',S')$ be two pointed $(4,3)$-categories. We say that $(\C',S')$ is \emph{stronger} than $(\C,S)$ if there is a functor  $F : \overline{\C' \restriction S'} \to \overline{\C \restriction S}$ which is $0$-surjective and $1$-surjective.
\end{defn}

\begin{lem}\label{lem:strength_lemma}
Let $(\C,S)$, $(\C',S')$ be two pointed $(4,3)$-white-categories. If there exists a $2$-functor $F : \C' \restriction S' \to \C \restriction S$ which is $0$-surjective and $1$-surjective, then $(\C',S')$ is stronger than $(\C,S)$.
\end{lem}
\begin{proof}
The functor $F$ induces a functor $\bar F :  \overline{\C' \restriction S'} \to \overline{\C \restriction S}$. Since it is equal to $F$ on objects, it is $0$-surjective. On $1$-cells $\bar F$ is the composition of $F$ with the canonical projection associated to the quotient, hence it is $1$-surjective, and so $(\C',S')$ is stronger than $(\C,S)$.
\end{proof}

\begin{lem}\label{lem:translation_coh}
Let $(\C,S)$,  $(\C',S')$ be two pointed $(4,3)$-white-categories, and assume $(\C',S')$ is stronger than $(\C,S)$.

If $\C'$ is $S'$-coherent, then $\C$ is $S$-coherent.
\end{lem}

\begin{proof}
Let $F: \overline{\C' \restriction S'} \to \overline{\C \restriction S}$ be a functor that is $0$-surjective and $1$-surjective.
Let $A,B: f \to g \in (\C \restriction S)_1$ be parallel $1$-cells, and $\bar A$, $\bar B$ be their projections in $\overline{\C \restriction S}$.

Since $F$ is $0$-surjective, there exists $f', g' \in (\C'\restriction S')_0$ in the preimage of $f$ and $g$ under $F$. Since $F$ is $1$-surjective, there exists $ A', B' \in (\C' \restriction S')_1$ of source $f'$ and of target $g'$ such that $F(\bar A') = \bar A$ and $F(\bar B') = \bar B$. 

Since $\C' \restriction S'$ is $2$-coherent, there exists $\alpha': A' \Rightarrow B' \in (\C' \restriction S')_2$. Thus $\bar A' = \bar B'$ and $\bar A = \bar B$. Hence there exists $\alpha: A \Rightarrow B \in \C \restriction S$. This shows that $\C \restriction S$ is $1$-coherent, and therefore that $\C$ is $S$-coherent.
\end{proof}

We are going to define four families of Tietze-transformations on $(4,3)$-white-polygraphs. Tietze transformations originates from combinatorial group theory \cite{LS15}, and was adapted for $(3,1)$-categories in \cite{GGM15}, as a way to modify a $(3,1)$-polygraph without modifying the $2$-categories it presents. In particular, they preserve the $2$-coherence.  Here we adapt these transformations to our setting of $(4,3)$-white-polygraphs and show that they preserve the partial coherence. This will be used in Section \ref{subsec:retournement}. We fix a $4$-white-polygraph $\A$.

\begin{defn}
Let $A \in \A_3^{\w(3)}$. We define a $4$-white-polygraph $\A(A)$ by adding to $\A$ a $3$-cell $B$ and a $4$-cell $\alpha$, whose sources and targets are given by:
\begin{itemize}
\item $\s(B) = \s(A)$,
\item $\t(B) = \t(A)$,
\item $\s(\alpha) = A$,
\item $\t(\alpha) = B$.
\end{itemize}

The inclusion induces a functor between $(4,3)$-white-categories $\iota_A : \A^{\w(3)} \to (\A(A))^{\w(3)}$. We call this operation the \emph{adjunction of a $3$-cell with its defining $4$-cell}.
\end{defn}

\begin{defn}
Let $\alpha \in \A_4$ and $A \in \A_3$ such that:
\begin{itemize}
\item $\t(\alpha) =A$
\item $\s(\alpha) \in (\A\setminus{\{\t(\alpha)\}})^{\w(3)}_3$.
\end{itemize}

The $4$-cell $\alpha$ induces an application $\A_3 \to (\A_3\setminus \{\t(\alpha)\})^{\w(3)}$, by sending $\t(\alpha)$ on $\s(\alpha)$ and that is the identity on the other cells of $\A_3$. This application extends into a $3$-functor $\pi_\alpha: \A_3^{\w} \to (\A_3\setminus \{\t(\alpha)\})^\w$. 

Let $\A/(A;\alpha)$ be the following $4$-white-polygraph:
\[
\xymatrix{
\A_0 
& 
\A_1^{\w(3)}
\ar@<+0.5ex>[l]^-{\t}
\ar@<-0.5ex>[l]_-{\s}
&
\A_2^{\w(3)}
\ar@<+0.5ex>[l]^-{\t}
\ar@<-0.5ex>[l]_-{\s}
&
(\A_3\setminus \{\t(\alpha)\})^{\w(3)}
\ar@<+0.5ex>[l]^-{\t}
\ar@<-0.5ex>[l]_-{\s}
&
\A_4\setminus \{\alpha\}
\ar@<+0.5ex>[l]^-{\pi_\alpha \circ \t}
\ar@<-0.5ex>[l]_-{\pi_\alpha \circ \s}
}
\]

Then $\pi_\alpha$ induces a functor $\A^{\w(3)} \to (\A/(A;\alpha))^{\w(3)}$, which sends $\alpha$ on the identity of $\s(\alpha)$, and which is the identity on the other cells of $\A_4$. We call this operation the \emph{removal of a $3$-cell with its defining $4$-cell}.
\end{defn}

\begin{defn}
Let $\alpha$ be a $4$-cell in $\A_4^{\w(3)}$. We define a $4$-white-polygraph $\A(\alpha)$ by adding to $\A$ a $4$-cell $\beta:\s(\alpha) \qfl \t(\alpha)$. The inclusion of $\A$ into $\A(\alpha)$ induces a functor $\iota_\alpha : \A^{\w(3)} \to \A(\alpha)^{\w(3)}$. We call this operation the \emph{adjunction of a superfluous $4$-cell}.
\end{defn}

\begin{defn}
Let $\beta \in \A_4$ such that there exists a $4$-cell $\alpha \in (\A\setminus\{\beta\})^{\w(3)}$ parallel to $\beta$. Let $\A/\beta$ be the $4$-white-polygraph obtained by removing $\beta$ from $\A$. There exists a functor $\pi_\beta : \A^{\w(3)} \to (\A/\beta)^{\w(3)}$, that sends $\beta$ on $\alpha$ and which is the identity on the other cells of $\A$. We call this operation the \emph{removal of a superfluous $4$-cell}.
\end{defn}

\begin{remq}
Note that, in those four cases, the set of $2$-cells is left unchanged. In particular, let $\A$ be a $4$-white-polygraph, and $\B$ a $4$-white-polygraph  constructed from $\A$ through a series of Tietze-transformations. If $S$ is a sub-set of $\A^{\w}_2$, then $S$ still is a subset of $\B^\w_2$. 
\end{remq}

\begin{prop}\label{prop:tietze_invariance}
Let $\A$ be a $4$-white-polygraph, $S$ a sub-set of $\A^\w_2$, and $\B$ a $4$-white-polygraph constructed from $\A$ through a series of Tietze-transformations. 

If $\B^{\w(3)}$ is $S$-coherent, then $\A^{\w(3)}$ is $S$-coherent.
\end{prop}
\begin{proof}
We check that if $\B$ is constructed from $\A$ through a Tietze-transformation, then the $3$-white-categories presented by $\A$ and $\B$ are isomorphic. 

Suppose now that $\B$ is $S$-coherent, and let $A,B \in \A^\w_3$ be parallel $3$-cells, whose source and target are in $S$. Since $\B^{\w(3)}$ is $S$-coherent, the images of $A$ and $B$ in the $3$-white-category presented by $\B$ are equal. Since it is isomorphic to the $3$-white-category presented by $\A$, there exists a $4$-cell $\alpha: A \qfl B \in \A^{\w(3)}_4$, which proves that $\A$ is $S$-coherent.
\end{proof}

\subsection{Injective functors between white-categories}
\label{subsec:inj_functors}

\begin{defn}
Let $\Sigma$ and $\Gamma$ be two $(n,k)$-polygraphs (resp. $(n,k)$-white-polygraphs), and let $F : \Sigma \to \Gamma$ be a morphism of $(n,k)$-polygraphs (resp. $(n,k)$-white-polygraphs). We say that $F$ is \emph{injective} if for all $j \leq n$ it induces an injective application from $\Sigma_n$ to $\Gamma_n$.
\end{defn}

\begin{defn}
Let $\C$ and $\D$ be two $n$-white-categories, and let $F : \C \to \D$ be a morphism of $n$-white-categories. We say that $F$ is \emph{injective} if for all $j \leq n$ it induces an injective application from $\C$ to $\D$.
\end{defn}

\begin{remq}
An injective morphism between $(n,k)$-polygraphs does not always induce an injective functor between the free $(n,k)$-categories they generate. To show that, we are going to define two $2$-polygraphs $\Sigma$ and $\Gamma$,  an injective morphism of $2$-polygraphs $F : \Sigma \to \Gamma$, and two distinct $2$-cells $f,g \in \Sigma^{*(1)}$ such that $F^{*(1)}(f) = F^{*(1)}(g)$.

Let $\Sigma$ be the following $2$-polygraph:
\[
\Sigma_0 = \{*\} 
\qquad
\Sigma_1 = \{\twocell{1} : * \to *\}
\qquad
\Sigma_2 := \{ \twocell{conjugC} , \twocell{conjugD} : \twocell{1} \Rightarrow \twocell{1}\}
\]
and $\Gamma$:
\[
\Gamma_0 = \{*\} 
\qquad
\Gamma_1 = \{\twocell{1} : * \to *\}
\qquad
\Gamma_2 := \{ \twocell{conjugC} , \twocell{conjugD}: \twocell{1} \Rightarrow \twocell{1} , \twocell{erase} :  \twocell{1} \to 1_*\}
\] 

Let $F$ be the inclusion of $\Sigma$ into $\Gamma$, $f = \twocell{conjugC *1 conjugD}$ and $g = \twocell{conjugD *1 conjugC}$. They are distinct elements of $\Sigma^{*(1)}_2$. However using the exchange law, the following equality holds in $\Gamma^{*(1)}_2$, where $\twocell{appear}$ denotes the inverse of $\twocell{erase}$:

\[
F(f)=
\twocell{conjugC *1 conjugD} = 
\twocell{erase *1 appear *1 conjugC *1 erase *1 appear *1 conjugD *1 erase *1 appear} =
\twocell{erase *1 ((appear *1 conjugC *1 erase) *0 (appear *1 conjugD *1 erase)) *1 appear} =
\twocell{erase *1 appear *1 conjugD *1 erase *1 appear *1 conjugC *1 erase *1 appear} =
\twocell{conjugD *1 conjugC} = F(g)
\]
\end{remq}

\bigskip

In what follows, we prove some sufficient conditions so that a morphism between two $(n,k)$-white-polygraphs induces an injective fucntor between the $(n,k)$-categories they present. This is achieved in Proposition \ref{prop:inject_libres_gen}. This result will be used in Section \ref{subsec:adjunc_2cell}.

To prove this result, we start by studying the more general case of an injective morphism $I$ between $(n,k)$-white-categories equipped with a cellular extension. When its image is \emph{closed by divisors} (see Definition \ref{def:closed_by_divisors}), we show a simple sufficient condition so that $I$ induces an injective $(n+1)$-white-functor. We also show that the image of the $(n+1)$-white-functor induced by $I$ is then automatically closed by divisors. Hence this hypothesis disappears when we go back to morphisms of $(n,k)$-white-polygraphs. In particular we show that every injective morphism of $n$-white-polygraphs induces an injective white-functor between $n$-white-categories.

For the rest of this section, we fix two $n$-white-categories equipped with cellular extensions $(\C,\Sigma), (\C',\Sigma') \in \WCat^+$, and a morphism $I : (\C,\Sigma) \to (\C',\Sigma') \in \WCat^+$. That is, $I$ is given by an $n$-white-functor $I: \C \to \C'$ together with an application $I_{n+1}: \Sigma \to \Sigma'$ such that the following squares are commute:
\[
\xymatrix @R=5em @C=5em{
\Sigma
\ar@1 [r] ^{I_{n+1}} 
\ar@1 [d] _{\s} 
\ar@{} [rd] |-{=}
&
\Sigma'
\ar@1 [d] ^{\s} 
\\
\C
\ar@1 [r] _{I} &
\C'
}
\qquad
\xymatrix @R=5em @C=5em{
\Sigma
\ar@1 [r] ^{I_{n+1}} 
\ar@1 [d] _{\t} 
\ar@{} [rd] |-{=}
&
\Sigma'
\ar@1 [d] ^{\t} 
\\
\C
\ar@1 [r] _{I} &
\C'
}
\]

We denote by $I^\w$ (resp. $I^{\w(n)}$) the $(n+1)$-white-functor $\L^\w(I)$ (resp. $\L^{\w(n)}(I)$). By definition, $I^\w$ (resp. $I^{\w(n)}$) is induced by an application from  $E^T_{\Sigma}$ to $E^T_{\Sigma'}$ (resp. from $F^T_{\Sigma}$ to $F^T_{\Sigma'}$), that we again denote by $I^\w$ (resp. $I^{\w(n)}$).

Using their explicit definitions, the following properties of $I^\w$ (resp.  $I^{\w(n)}$) hold:
\begin{itemize}
\item Any element of $E^T_\Sigma$ (resp. $F^T_\Sigma$) whose image is an $i$-composite is an $i$-composite.
\item Any element of $E^T_\Sigma$ (resp. $F^T_\Sigma$) whose image is a an identity is an identity.
\item Any element of $E^T_\Sigma$ (resp. $F^T_\Sigma$) whose image is a $c_{u'A'v'}$ is a $c_{uAv}$.
\item Any element of $F^T_\Sigma$ whose image by $I^{\w(n)}$ is a $c_{u'\bar A'v'}$ is a $c_{u\bar Av}$.
\end{itemize} 

\begin{lem}\label{lem:equiv_inject_lg_0}
Assume that the application $I_{n+1}$ is injective, and that $I$ induces an injection on $\C$.

Then the applications $I^\w : E^T_{\Sigma} \to E^T_{\Sigma'}$ and $I^{\w(n)}:  F^T_{\Sigma} \to F^T_{\Sigma'}$ are injective.
\end{lem}
\begin{proof}
Let $a_1,a_2 \in E^T_{\Sigma}$ such that $I^\w(a_1) = I^\w(a_2)$. We reason by induction on the structure of $I^\w(a_1)$.

If $I^\w(a_1) = c_{u'A'v'}$, with $u',v' \in \C'_1$ and $A' \in \Sigma'$. Then there are $u_1,v_1,u_2,v_2 \in \C_1$ and $A_1,A_2 \in \Sigma$ such that $a_1 = c_{u_1A_1v_1}$ and $a_2 = c_{u_2A_2v_2}$, and so:
\[
I(u_1) = I(u_2) = u'
\qquad 
I_{n+1}(A_1) = I_{n+1}(A_2) = A'
\qquad 
I(v_1) = I(v_2) = v'.
\]
Since $I$ and $I_{n+1}$ are injective, we get:
\[
u_1 = u_2
\qquad
A_1 = A_2
\qquad
v_1 = v_2,
\]
which proves that $a = b$.

If $I^\w(a_1) = i_f$, with $f' \in \C'_n$. Then there exist $f_1,f_2 \in \C_n$ such that:
\[
a_1 = i_{f_1}
\qquad
a_2 = i_{f_2}
\qquad 
I(f_1) = f'
\qquad
I(f_2) = f.
\]
Since $I$ is injective, $f_1 = f_2$, and so $a_1 = a_2$.

If $I^\w(a_1) = A' \star_i B'$, with $i < n$, and $A',B' \in E^T_{\Sigma'}$. Then there exist $A_1,A_2,B_1,B_2 \in E^T_{\Sigma}$ such that:
\[
a_1 = A_1 \star_i B_1
\qquad
a_2 = A_2 \star_i B_2
\qquad 
I^\w(A_1) = I^\w(A_2) = A'
\qquad 
I^\w(B_1) = I^\w(B_2) = B'.
\]
Using the induction hypothesis, we get that $A_1 = A_2$ and $B_1 = B_2$, and so $a_1 = a_2$.

In the case of $I^{\w(n)}$, we reason as previously, and we have one more case to check: if $I^{\w(n)}(a_1) = c_{u'\bar A'v'}$, with $u',v' \in \C'_1$ and $A' \in \Sigma'$. Then there are $u_1,v_1,u_2,v_2 \in \C_1$ and $A_1,A_2 \in \Sigma$ such that $a_1 = c_{u_1\bar A_1v_1}$ and $a_2 = c_{u_2\bar A_2v_2}$, and so:
\[
I(u_1) = I(u_2) = u'
\qquad 
I_{n+1}(A_1) = I_{n+1}(A_2) = A'
\qquad 
I(v_1) = I(v_2) = v'.
\]
Using the injectivity of $I$ and $I_{n+1}$, we get:
\[
u_1 = u_2
\qquad
A_1 = A_2
\qquad
v_1 = v_2,
\]
and finally  $a_1 = a_2$.
\end{proof}

\begin{defn}\label{def:closed_by_divisors}
Let $\C$ be an $n$-white-category, and $E$ be a subset of $\C_n$. We say that $E$ is \emph{closed by divisors} if, for any $f \in E$, if $f = f_1 \star_i f_2$, then $f_1$ and $f_2$ are in $E$.
\end{defn}

\begin{lem}\label{lem:equiv_inject_lg_1}
Assume the image of $I$ in $\C_n$ is closed by divisors, and that $I$ and $I_{n+1}$ are injective.

Then, for every $a',b' \in E^T_{\Sigma'}$ such that $a' \equiv_{\Sigma'} b'$, and for every $a \in E^T_\Sigma$ such that $I^{\w}(a) = a'$, there exists $b \in E^T_\Sigma$ such that
\[
I^\w(b) = b'
\qquad
a \equiv_\Sigma b.
\]

Assume moreover that the application $I_{n+1}$ is bijective and that $I$ is bijective on the $1$-cells of $\C$.

Then, for every $a',b' \in F^T_{\Sigma'}$ such that $a'\cong_{\Sigma'} b'$, and for every $a \in F^T_\Sigma$ such that $I^{\w(n)}(a) = a'$, there exists $b \in F^T_\Sigma$ such that \[I^{\w(n)}(b)= b' \qquad a \cong_\Sigma b\].
\end{lem}
\begin{proof}
To show the result on $I^\w$ we reason by induction on the structure of $a'$.

If there exist $A',B',C',D' \in E_{\Sigma'}^T$, $0<i_1<i_2\leq n$ and $a \in E^T_{\Sigma}$ such that:
\[
a' = (A' \star_{i_1} B') \star_{i_2} (C' \star_{i_1} D')
\qquad
b' = (A' \star_{i_2} C') \star_{i_1} (B' \star_{i_2} D')
\qquad
I^\w( a) = a',
\]
then, $a = ( A \star_{i_1}  B) \star_{i_2} ( C \star_{i_1}  D)$, with $ A,  B, C,  D \in E^T_{\Sigma}$. Let $b := ( A \star_{i_2}  C)\star_{i_1} ( B \star_{i_2}  D)$: by construction, we have $I^\w( b) = b'$ and $a \equiv_{\Sigma}  b$. The case where the roles of $a'$ and $b'$ are reversed is symmetrical.

If there exist $A',B',C' \in E_{\Sigma'}^T$, $0<i\leq n$ and $a \in E^T_{\Sigma}$ such that:
\[
a' = (A' \star_{i} B') \star_{i} C'
\qquad
b' = A' \star_{i} (B' \star_{i} C')
\qquad
I^\w( a) = a',
\]
then, $a = ( A \star_{i}  B) \star_{i}  C$, with $A,  B, C \in E^T_{\Sigma}$. Let $ b :=  A \star_{i}  (B\star_{i}  C)$: By construction, we have $I^\w( b) = b'$ and $a \equiv_{\Sigma}  b$. The case where the roles of $a'$ and $b'$ are reversed is symmetrical.

If there exist $A' \in E_\Sigma^T$, $f' \in \C'_n$ and $a \in E^T_\Sigma$ such that
\[
a' = i_{f'} \star_n A'
\qquad
b' = A'
\qquad 
I^\w(a) = a',
\]
then $a = i_f \star_n A$, with $f \in \C_n$ and $A \in E_\Sigma^T$. Let $b := A'$: by construction, we have $I^\w( b) = b'$ and $a \equiv_{\Sigma}  b$.

If there exist $A' \in E_\Sigma^T$, $f' \in \C'_n$ and $a \in E^T_\Sigma$ such that
\[
a' = A'
\qquad
b' =  i_{f'} \star_n A'
\qquad
I^\w(a) = a',
\]
let $b := i_{\s(A)} \star_n a$. Since $b'$ is well typed, we have $f' = \s(A')$, hence $I(\s(A)) = \s(I^\w(A)) = \s(A') = f'$, and so $I^\w( b) = b'$ and $a \equiv_{\Sigma}  b$. The case of the right-unit is symmetrical.

If there are $f'_1,f'_2 \in \C'_n$, $i <n$ and $a \in E_\Sigma^T$ such that:
\[
a' = i_{f'_1} \star_i i_{f'_2} 
\qquad
b' = i_{f'_1 \star_i f'_2}
\qquad
I^\w(a) = a',
\]
then $a = i_{f_1} \star_i i_{f_2}$, with $f_1,f_2 \in \C_n$. Let $b := i_{f_1 \star_i f_2}$: by construction, we have $I^\w( b) = b'$ and $a \equiv_{\Sigma}  b$.

If there are $f'_1,f'_2 \in \C'_n$, $i <n$ and $a \in E_\Sigma^T$ such that:
\[
a' =  i_{f'_1 \star_i f'_2}
\qquad
b' = i_{f'_1} \star_i i_{f'_2}
\qquad
I^\w(a) = a'
\]
then $a = i_{f}$, with $f \in \C_n$. Since the image of $I$ in $\C_n$ is closed by divisors, there exist $f_1,f_2 \in \C_n$ such that 
\[
I(f_1) = f'_1 
\qquad
I( f_2) = f'_2
\qquad
 f =  f_1 \star_i  f_2.
\]
Let us define $b' := i_{f_1} \star_i i_{f_2}$: By construction, we have
$I^\w( b) = b'$ and $a \equiv_{\Sigma}  b$.

If there are $A'_1,A'_2,B' \in E^T_{\Sigma'}$, $i \leq n$ and $a \in E^T_{\Sigma}$ such that:
\[
a' = A'_1 \star_i B'
\qquad
A'_1 \equiv_{\Sigma'} A'_2
\qquad
b' = A'_2 \star_i B'
\qquad
I^\w(a)=a'
\]
then $a = A_1 \star_i B$, with $A_1,B \in E^T_\Sigma$. Using the induction hypothesis, there exist $A_2 \in E^T_\Sigma$ such that $I^\w(A_2) = A'_2$ and $A_1 \equiv_{\Sigma} A_2$. Let us define $b := A_2 \star_i B$: by construction, we have $I^\w( b) = b'$ and $a \equiv_{\Sigma}  b$. The last case is symmetric.

In the case of $I^{\w(n)}$, we reason as previously, and we have two more cases to check. If there exist $u',v' \in \C'_1$, $A \in \Sigma'$ and $a \in F^T_{\Sigma}$ such that:
\[
a' = c_{u'A'v'} \star_n c_{u' \bar A' v'}
\qquad
b' = i_{u'\s(A')v'}
\qquad
I^{\w(n)}(a) = a'
\]
then $a = c_{u_1A_1v_2} \star_n c_{u_2 \bar A_2 v_2}$, with $u_1,u_2,v_1,v_2 \in \C_1$ and $A_1,A_2 \in \Sigma$ such that:
\[
I(u_1) = I(u_2) = u'
\qquad
I(v_1) = I(v_2) = v'
\qquad
I_{n+1}(A_1) = I_{n+1}(A_2) = A.
\]
Let $b:= i_{u_1\s(A_1)v_1}$. Since $I$ and $I_{n+1}$ are injective, we have $I^\w( b) = b'$ and $a \cong_{\Sigma}  b$.

If there exist $u',v'\in \C'_1$,  $A \in \Sigma'$ and $a \in F^T_{\Sigma}$ such that
\[
a' = i_{u'\s(A')v'}
\qquad
b' = c_{u'A'v'} \star_n c_{u' \bar A' v'}
\qquad
I^{\w(n)}(a) = a'
\]
Then $a = i_f$, with $f \in \C_n$. Let $b' := c_{uAv} \star_i c_{u\bar A v}$, with $u = I^{-1}(u')$, $v = I^{-1}(v')$ and $A = I_{n+1}^{-1}(A')$: by construction, we have $I^\w( b) = b'$ and $a \cong_{\Sigma}  b$. The final case is symmetrical.
\end{proof}

\begin{lem}\label{lem:inject_pour_libres}
Assume that $I_{n+1}$ and $I$ are injective, and that the image of $I$ in $\C_n$ is closed by divisors. Then the functor $I^\w : L^\w(\C,\Sigma) \to L^\w(\C',\Sigma')$ is injective, and its image is closed by divisors.

Assume moreover that $I_{n+1}$ is bijective, and that $I$ is bijective on the $1$-cells of $\C$. Then the functor $I^{\w(n)} : L^{\w(n)}(\C,\Sigma) \to L^{\w(n)}(\C',\Sigma')$ is injective and its image is closed by divisors.
\end{lem}
\begin{proof}
Let $f_1,f_2 \in  L^\w(\C,\Sigma)$ and $a_1,a_2 \in E^T_{\Sigma}$ such that: 
\[
I^\w(f_1) = I^\w(f_2)
\qquad
[a_1] = f_1
\qquad
[a_2] = f_2.
\]
Then $[I^\w(a_1)] = [I^\w(a_2)]$, that is $I^\w(a_1) \equiv^*_{\Sigma'} I^\w(a_2)$.
 Hence by definition, there exist $n >0$ and $t'_1, \ldots , t'_n \in E^T_{\Sigma'}$ such that:
\[
t'_1 = I^\w(a_1)
\qquad
t'_{i} \equiv_{\Sigma'} t'_{i+1}
\qquad
t'_n = I^\w(a_2).
\]

Applying Lemma \ref{lem:equiv_inject_lg_1} successively, we get $t_1,\ldots,t_n \in E^T_{\Sigma'}$ such that:
\[
t_1 = a_1
\qquad
t_i \equiv_{\Sigma} t_{i+1}
\qquad
I^\w(t_i) = t'_i.
\]
In particular $a_1 \equiv^*_\Sigma t_n$ and $I^\w(t_n) = t'_n = I^\w(a_2)$. Using Lemma \ref{lem:equiv_inject_lg_0}, this implies that $t_n = a_2$, and so $a_1 \equiv^*_\Sigma a_2$, which proves that $f_1 = [a_1] = [a_2] = f_2$.

It remains to show that the image of $I^\w$ is closed by divisors. Let $f',f'_1,f'_2 \in L^\w(\C',\Sigma')$ and $i \leq n$ such that $f' = f'_1 \star_i f'_2$, and assume that there is an $f \in L^\w(\C',\Sigma')$ such that $I^\w(f) = f'$. Let $a \in E^T_\Sigma$ and $b'_1,b'_2 \in E^T_{\Sigma'}$ such that:
\[
[a] = f
\qquad
[b'_1] = f'_1
\qquad
[b'_2] = f'_2.
\]

In particular, we have $I^\w(a) \equiv^*_{\Sigma'} b'_1 \star_i b'_2$. Using both Lemmas \ref{lem:equiv_inject_lg_0} and \ref{lem:equiv_inject_lg_1} as before, we get an element $b \in E^T_\Sigma$ such that:
\[
a \equiv_\Sigma^* b
\qquad
I^\w(b) = b'_1 \star_i b'_2.
\]
Since the image of $I^\w$ is closed by divisors, there exists $b_1,b_2 \in E^T_\Sigma$ such that $b = b_1 \star_i b_2$. Let $f_1 := [b_1]$ and $f_2 := [f_2]$: by construction we have:
\[
I^\w(f_1) = f'_1
\qquad
I^\w(f_2) = f'_2
\qquad
f_1 \star_i f_2 = f.
\]

The case of $I^{\w(n)}$ is identical, the only difference lying in the hypothesis needed to apply Lemma \ref{lem:equiv_inject_lg_1}.
\end{proof}

\begin{prop}\label{prop:inject_libres_gen}
Let $\Sigma$ and $\Gamma$ be two $(n,k)$-white-polygraphs and $I : \Sigma \to \Gamma$ be an injective morphism of $(n,k)$-polygraphs. Then for every $j \leq k$ the functor $I_j^\w : \Sigma_j^\w \to \Gamma_j^\w$ is injective, and its image is closed by divisors.

Assume moreover that $I_0$ and $I_1$ are bijections, and that
for every $j > k$ the application $I_j  :\Sigma_j \to \Gamma_j$ is bijective. Then for every $j$ the functor $I_j^{\w(k)} : \Sigma_j^{\w(k)} \to \Gamma_j^{\w(k)}$ is injective, and its image is closed by divisors.
\end{prop}
\begin{proof}
We reason by induction on $j$. The case $j = 0$ is true by hypothesis.

Let $1 \leq j \leq k$. By hypothesis, the application $I_j$ is injective, and by induction hypothesis, the functor $I^\w_{j-1}$ is injective with image closed by divisors. Hence $I_j$ satisfies the hypothesis of Lemma \ref{lem:inject_pour_libres}, and $I^\w_{j}$ is injective with image closed by divisors.

Let $j > k$. Again, using the hypothesis and induction hypothesis, we get that $I_j$ satisfies the hypotheses of Lemma \ref{lem:inject_pour_libres}. Hence $I^{\w(k)}_{j}$ is injective and its image is closed by divisors.
\end{proof}

In what follows, we use the fact that the image of a functor generated by a morphism of polygraphs is closed by divisors in order to prove a characterisation of the image of such a functor. 

\begin{defn}
Let $\C, \D$ be two $n$-white-categories, $F : \C \to \D$ be an $n$-functor and $f$ be an $n$-cell of $\D$. We say that $F$  \emph{$k$-discriminates} $f$ if the following are equivalent:
\begin{enumerate}
\item \label{enum:source} The $k$-source of $f$ is in the image of $F$.
\item \label{enum:but} The $k$-target of $f$ is in the image of $F$.
\item \label{enum:image} The $n$-cell $f$ is in the image of $F$.
\end{enumerate}

Given a subset $D$ of $\D_n$, we say that $F$ is \emph{$k$-discriminating on $D$} if for every  $n$-cell $f$ in $D$, $F$ $k$-discriminates $f$.
\end{defn}

\begin{lem}\label{lem:img_foncteur_libre}
Assume that the image of $I$ is closed by divisors, that the application $I_n$ is injective, and that $I$ is $n$-discriminating on $\Sigma'$.

Then, $I^\w$ (resp. $I^{\w(n)}$) is $n$-discriminating on  $L^\w(\C',\Sigma')$ (resp. $L^{\w(n)}(\C',\Sigma')$).

\end{lem}
\begin{proof}
Let us start with $I^\w$. Let $E$ be the set all $(n+1)$-cells of $L^\w(\C',\Sigma')$ which $I^\w$ discriminates. Let us show that $E = L^\w(\C',\Sigma')$. Since $I^\w$ commutes with the source and target applications, the implications $\eqref{enum:image} \Rightarrow  \eqref{enum:source}$ and $\eqref{enum:image} \Rightarrow  \eqref{enum:but}$ hold for any cell in $L^\w(\C',\Sigma')$. So in order to show that a cell is in $E$, it remains to show that it verifies the implications $\eqref{enum:source} \Rightarrow \eqref{enum:image}$ and $\eqref{enum:but} \Rightarrow \eqref{enum:image}$.

The set $E$ contains all units. Indeed, let $A' = 1_{f'}$, with $f' \in \C'$. If $\s(A') = f'$ is in the image of $I^\w$, there exists $f \in \C$ such that $I(f) = f'$. Let us define $A = 1_f \in L^\w(\C,\Sigma)$: by construction we have $I^\w(A) = 1_{I(f)} = 1_{f'} = A'$, hence the implication $\eqref{enum:source} \Rightarrow \eqref{enum:image}$ holds for $A'$. Moreover since $\t(A') = \s(A')$, the implication $\eqref{enum:but} \Rightarrow \eqref{enum:image}$ also holds for $A$.

The set $E$ contains all cells of length $1$. Indeed, given such a cell $A'$, there exist $f'_k,g'_k \in \C'_k$ and $A'_0 \in \Sigma'$ such that 
\[A' = f'_n \star_{n-1} (f'_{n-1} \star_{n-2} \ldots \star_2 (f'_1 A'_0 g'_1) \star_2 \ldots \star_{n-2} g'_{n-1}) \star_{n-1} g'_n.\]
Let $A'_k := f'_k \star_{k-1} (f'_{k-1} \star_{k-2} \ldots \star_2 (f'_1 A'_0 g'_1) \star_2 \ldots \star_{k-2} g'_{k-1}) \star_{k-1} g'_k$. Suppose that the source (resp. target) of $A'$ is in the image of $I$, and let us show that $A'$ is in the image of $I^\w$. Since the image of $I$ is closed by divisors, we get first that $f'_n$, $g'_n$ and $\s(A'_{n-1})$ (resp. $\t(A'_{n-1})$) are in the image of $I$. By iterating this reasoning, we get that, for all $i$, $f'_i$, $g'_i$ and $\s(A'_{i-1})$ (resp. $\t(A'_{i-1})$) are in the image of $I$. Since $I^\w$ discriminates $\Sigma'$, there exist $f_k,g_k \in \C_k$ and $A_0 \in \Sigma$ such that:
\[
I(f_k) = f'_k
\qquad
I(g_k) = g'_k
\qquad
I_{n+1}(A_0) = A'_0.
\]
By induction on $k$ we show that  $A_k := f_k \star_{k-1} A_{k-1} \star_{k-1} g_k$ is well defined and that $I^\w(A_k) = A'_k$. Indeed, assume that it is true at rank $k-1$. Then we have the equalities:
\[
I(\t(f_k)) = \t(f'_{k}) = \s_{k-1}(A'_{k-1}) = I(\s_{k-1}(A_{k-1}))
\qquad
I(\t_{k-1}(A_{k-1})) = \t_{k-1}(A'_{k-1}) = \s(g'_k) = I(\s(g_k))
\]
Using the injectivity of $I$ we get that $\t(f_k) = \s_{k-1}(A_{k-1})$ and $\t_{k-1}(A_{k-1}) = \s(g_k)$, which shows that $A_k$ is well defined, and finally:
\[
I^\w(A_k) = f'_k \star_{k-1} A'_{k-1} \star_{k-1} g'_k = A'_k.
\]
In particular, we have $A'_n = I^\w(A_n)$.

The set $E$ is stable by $n$-composition. Indeed let $A',B' \in E$, and assume that the source of $A' \star_n B'$ is in the image of $I$. Let us show that $A' \star_n B'$ is in the image of $I^\w$. The source of $A' \star_n B'$ is none other that the one of $A'$. Since $A'$ is in $E$, there exists $A \in \L^\w(\C,\Sigma)$ such that $I^\w(A) = A'$. Hence the source of $B'$ is in the image of $I$, and since $B' \in E$, there exists $B \in \L^\w(\C,\Sigma)$ such that $I^\w(B) = B'$. Moreover we have $I(\t(A)) = \t(A') = \s(B') = I(\s(B))$, so using the injectivity of $I$ we get $\t(A) = \s(B)$. Hence the cell $A \star_n B$ is well defined and satisfies:
\[
I^\w(A \star_n B) = I^\w(A) \star_n I^\w(B) = A' \star_n B'.
\]
The case where the target of  $A' \star_n B'$ is in the image of $I$ is symmetrical.

This concludes the proof for $I^\w$. Concerning $I^{\w(n)}$, the reasoning is the same except that we also have to show that $E$ is stable under inversion. Indeed let $A' \in E$ and assume that the source (resp. target) of $(A')^{-1}$ is in the image of $I$. Then the target (resp. source) of $A'$ is in the image of $I$ and since $A'$ is in  $E$, there exists $A \in \L^{\w(n)}(\C,\Sigma)$ such that $I^{\w(n)}(A) = A'$, and so $I^{\w(n)}(A^{-1}) = (A')^{-1}$.
\end{proof}

\begin{prop}\label{prop:img_foncteur_gen}
Let $\Sigma$ and $\Gamma$ be two $(n,k)$-white-polygraphs, and $I : \Sigma \to \Gamma$ be a morphism of polygraphs. Let $k_0$ such that for every $j > k_0$, $I_j$ is a bijection.

Assume that $I$ satisfies the hypothesis of Proposition \ref{prop:inject_libres_gen}, and that, for every $j > k_0$, $I_j$ is $k_0$-discriminating on $\Gamma_j$. Then for every $j \geq k_0$, $I_j^{\w(k)}$ is $k_0$-discriminating on $\Gamma_j^{\w(k)}$.
\end{prop}
\begin{proof}
Since $I$ satisfies the hypotheses of Proposition \ref{prop:inject_libres_gen}, we know that for every $j$, the functor $I^{\w(k)}_j$ is injective, and that its image is closed by divisors.

We reason by induction on $j > k_0$. For $j = k_0 +1$, the result is a direct application of Lemma \ref{lem:img_foncteur_libre}.

Let $j > k_0+1$: let us show that $I_j^{\w(k)}$ is $(j-1)$-discriminating on $\Gamma_j$. Let $A \in \Gamma_j$. If $\s(A)$ (resp. $\t(A)$) is in the image of $I_{j-1}^{\w(k)}$ then in particular, the $k_0$-source (resp. $k_0$-target) of $A$ is in the image of $I_{k_0}^{\w(k)}$. Since $I_{j}^{\w(k)}$ is $k_0$-discriminating on $\Gamma_j$, $A$ is in the image of $I_j^{\w(k)}$. Hence we can use Lemma \ref{lem:img_foncteur_libre}, and we get that $I_j^{\w(k)}$ is $(j-1)$-discriminating on $\Gamma_j^{\w(k)}$. Let $A \in \Gamma_j^{\w(k)}$. If its $k_0$-source (resp. $k_0$-target) is in the image of $I^{\w(k)}_{k_0}$ then, by induction hypothesis, the source (rep. target) of $A$ is in the image of $I^{\w(k)}_{j-1}$, and so $A$ is in the image of $I_{j}^{\w(k)}$, which proves that $I_{j}^{\w(k)}$ is $k_0$-discriminating.
\end{proof}

\section{Application to the coherence of pseudonatural transformations}
\label{sec:application}

We now study the coherence problem successively for bicategories, pseudofunctors and pseudonatural transformations. In Section \ref{subsec:bicat}, we start by recalling the usual definition of bicategories (see \cite{B67}). We then give an alternative description of bicategories in terms of algebras over a certain $4$-polygraph $\Bi[\CC]$, and show that the two definitions coincide. The coherence problem for bicategories is now reduced to showing the $3$-coherence of $\Bi[\CC]$, and we use the techniques introduced in the previous section (especially Theorems \ref{thm:terminaison} and \ref{thm:squier}) to conclude. In Section \ref{subsec:pseudo_functors} and \ref{subsec:pseudo_nat_transfo}, we apply the same reasoning to pseudofunctors and pseudonatural transformations. However in the case of pseudonatural transformations, we get a $(4,3)$-polygraph $\TPN[\f,\g]$ which is not $3$-confluent, and so we cannot directly apply Theorem \ref{thm:squier}. The proof of the coherence theorem for pseudonatural transformations will take place in Section \ref{sec:proof_of_coherence} and will make use of Theorem \ref{thm:main_theory}.

\subsection{Coherence for bicategories}
\label{subsec:bicat}


Let $\Cat$ be the category of (small) categories. We denote by $\top$ the terminal category in $\Cat$. Let $\sCat$ be the $3$-category with one $0$-cell, (small) categories as $1$-cells, functors as $2$-cells, and natural transformations as $3$-cells, where $0$-composition is given by the cartesian product, $1$-composition by functor composition, and $2$-composition by composition of natural transformations.

\begin{defn}
A \emph{bicategory} $\B$ is given by:
\begin{itemize}
\item A set $\B_0$.
\item For every $a,b \in \B_0$, a category  $\B(a,b)$. The objects and arrows of $\B(a,b)$ are respectively called the $1$-cells $\B$ and $2$-cells of $\B$.
\item For every $a,b,c \in \B_0$, a functor $\star_{a,b,c} :  \B(a,b) \times \B(b,c) \to \B(a,c)$.
\item For every $a \in \B_0$, a functor $I_a : \top \to \B(a,a)$, that is to say a  $1$-cell $I_a : a \to a$.
\item For every $a,b,c,d \in \B_0$, a natural isomorphism $\alpha_{a,b,c,d}$:
\[
\xymatrix @R = 1em  @C = 3em{
\B(a,b) \times \B(b,c) \times \B(c,d) 
  \ar [rrr] ^{\B(a,b) \times \star_{b,c,d}}
  \ar [ddd] _{\star_{a,b,c} \times \B(c,d)}
& & &
\B(a,b) \times \B(b,d)
  \ar [ddd] ^{\star_{a,b,d}}
\\
& & &
\\
& 
  \ar@2 [ru] ^-{\alpha_{a,b,c,d}}
& &
\\
\B(a,c) \times \B(c,d)
  \ar [rrr] _{\star_{a,c,d}}
& & &
\B(a,d)
}
\]
of components $\alpha_{f,g,h} : (f \star g) \star h \Rightarrow f \star (g \star h)$, for every triple $(f,g,h) \in \B(a,b) \times \B(b,c) \times \B(c,d)$.
\item For every $a,b \in \B_0$, natural isomorphisms $R_{a,b}$ and $L_{a,b}$:
\[
\xymatrix @R=6em @C=6em{
\B(a,b) 
  \ar@2{-} [rd] |-{}="eq"
  \ar [d] _{I_a \times \B(a,b)}
  \ar [r] ^{\B(a,b) \times I_b}
& 
\B(a,b) \times \B(b,b)
  \ar [d] ^{\star_{a,b,b}}
  \ar@2 []!<-5pt,10pt>;"eq"!<5pt,5pt> ^-{R_{a,b}}
\\
\B(a,a) \times \B(a,b)
  \ar [r] _{\star_{a,a,b}}
  \ar@2 []!<5pt,-10pt>;"eq"!<-5pt,-5pt> ^-{L_{a,b}}
&
\B(a,b)
}
\]
of components $L_f : I_a \star f \Rightarrow f$ and $R_f : f \star I_b \Rightarrow f$ for every $1$-cell $f \in \B(a,b)$.
\end{itemize}

This data must also satisfy the following axioms:
\begin{itemize}
\item For every composable $2$-cells $f,g,h,i$ in $\B$:

\begin{equation} \label{eq:penta}
\begin{gathered}
\xymatrix {
& 
((f \star g) \star h) \star i
  \ar@2 [rd] ^-{\alpha_{f,g,h} \star i}
  \ar@2 [ldd] _-{\alpha_{f \star g, h , i}}
  \ar@{} [dddd] |-{=}
& \\
& & 
(f \star (g \star h)) \star i
  \ar@2 [dd] ^-{\alpha_{f,g \star h,i }}
\\
(f \star g) \star (h \star i)
  \ar@2 [rdd] _{\alpha_{f, g, h \star i}}
& & \\
& & 
f \star ((g \star h) \star i)
  \ar@2 [ld] ^-{f \star \alpha_{g,h,i}} 
\\
& 
f \star (g \star (h \star i))
& 
}
\end{gathered}
\end{equation}

\item For every couple $(f,g) \in \B(a,b) \times \B(b,c)$:
\begin{equation} \label{eq:trian}
\begin{gathered}
\begin{tikzpicture}
\matrix (m) [matrix of math nodes, 
			nodes in empty cells,
			column sep = 1cm, 
			row sep = .7cm] 
{
& 
(f \star I_b) \star g
& \\
f \star (I_b \star g)  & & f \star g \\
};
\doublearrow{->}
{(m-1-2) -- node [above left] {$\alpha_{f,I_b,g}$}(m-2-1)}
\doublearrow{->}
{(m-1-2) -- node [above right] {$R_f \star g$}(m-2-3)}
\doublearrow{->}
{(m-2-1) -- node [below] {$f \star L_g$}(m-2-3)}
\path (m-1-2) -- node {$=$} (m-2-2);
\end{tikzpicture}
\end{gathered}
\end{equation}
\end{itemize}
\end{defn}

\begin{defn}\label{def:poly_bicat}
Let $\CC$ be a set. Let us describe dimension by dimension a  $4$-polygraph $\Bi[\CC]$, so that bicategories correspond to algebras on $\Bi[\CC]$, that is to $4$-functors from $\Bi[\CC]$ to $\sCat$ (see Proposition \ref{prop:alg_bicat}).

\paragraph*{Dimension $0$:}
Let $\Bi[\CC]_0$ be the set $\CC$. 
\paragraph*{Dimension $1$:}
The set $\Bi[\CC]_1$ contains, for every $a,b \in \CC$, a $1$-cell ${}_a\twocell{1}{}_b : a \to b$.
\paragraph*{Dimension $2$:}

The set $\Bi[\CC]_2$ contains the following $2$-cells:
\begin{itemize}
\item For every $a,b,c \in \CC$, a $2$-cell $\twocell{prodC}_{a,b,c}: {}_a\twocell{1}{}_b \twocell{1} {}_c \Rightarrow {}_a \twocell{1} {}_c$.
\item For every $a \in \CC$, a $2$-cell $\twocell{unitC}_a: 1_a \Rightarrow {}_a \twocell{1} {}_a$.
\end{itemize}

Note that the indices are redundant with the source of a generating $2$-cell. In what follows, we will therefore omit them when the context is clear. For example, the $2$-cell $\twocell{(1 *0 prodC) *1 prodC}$ of source ${}_a\twocell{1}{}_b\twocell{1}{}_c\twocell{1}{}_d$ designates the composite $({}_a\twocell{1}{}_b \twocell{prodC}_{b,c,d}) \star_1 \twocell{prodC}_{a,b,d}$. We will use the same notation for higher-dimensional cells.

\paragraph*{Dimension $3$:}
The set $\Bi[\CC]_3$ contains the following $3$-cells:
\begin{itemize}
\item For every $a,b,c,d \in \CC$, a $3$-cell $\twocell{assocC}_{a,b,c,d} :
\twocell{(prodC *0 1) *1 prodC} \Rrightarrow \twocell{(1 *0 prodC) *1 prodC} $ of $1$-source ${}_a \twocell{1} {}_b \twocell{1} {}_c \twocell{1} {}_d$.

\item For every $a,b \in \CC$, $3$-cells $\twocell{lunitC}_{a,b} : \twocell{(unitC *0 1) *1 prodC} \Rrightarrow \twocell{1}$ and $\twocell{runitC}_{a,b} : \twocell{(1 *0 unitC *0) *1 prodC} \Rrightarrow \twocell{1}$ of $1$-source ${}_a \twocell{1} {}_b$.
\end{itemize}

\paragraph*{Dimension $4$:}
The set $\Bi[\CC]_4$ contains the following $4$-cells:
\begin{itemize}
\item For every $a,b,c,d,e \in \CC$, a $4$-cell $\twocell{pentaC}_{a,b,c,d,e}$ of $1$-source ${}_a \twocell{1} {}_b \twocell{1} {}_c \twocell{1} {}_d \twocell{1} {}_e$.

\[
\begin{tikzpicture}
\matrix (m) [matrix of math nodes, 
			nodes in empty cells,
			column sep = 1cm, 
			row sep = .7cm] 
{
&
\twocell{(1 *0 prodC *0 1) *1 (prodC *0 1) *1 prodC}
& 
&
\twocell{(1 *0 prodC *0 1) *1 (1 *0 prodC) *1 prodC}
&
\\
\twocell{(prodC *0 2) *1 (prodC *0 1) *1 prodC} 
& & & &
\twocell{(2 *0 prodC) *1 (1 *0 prodC) *1 prodC}
\\
& &
\twocell{(prodC *0 prodC) *1 prodC}
& & \\
};
\triplearrow{->}
{(m-1-2) -- node [above] {$\twocell{(1 *0 prodC *0 1) *1 assocC}$} (m-1-4)}
\triplearrow{->}
{(m-2-1) -- node [above left] {$\twocell{(assocC *0 1) *1 prodC}$} (m-1-2)}
\triplearrow{->}
{(m-1-4) -- node [above right] {$\twocell{(1 *0 assocC) *1 prodC}$} (m-2-5)}
\triplearrow{->}
{(m-2-1) -- node [below left] {$\twocell{(prodC *0 2) *1 assocC}$} (m-3-3)}
\triplearrow{->}
{(m-3-3) -- node [below right] {$\twocell{(2 *0 prodC) *1 assocC}$} (m-2-5)}

\quadarrow{->}
{(m-1-3|-m-1-2.south) -- node [right] {$\twocell{pentaC}$} (m-3-3|-m-2-1.south)}
\end{tikzpicture}
\]

\item For every $a,b,c \in \CC$, a $4$-cell $\twocell{trianC}_{a,b,c}$ of $1$-source ${}_a \twocell{1}{}_b \twocell{1} {}_c$.

\[
\begin{tikzpicture}
\matrix (m) [matrix of math nodes, 
			nodes in empty cells,
			column sep = 1cm, 
			row sep = .7cm] 
{
&
\twocell{(1 *0 unitC *0 1) *1 (1 *0 prodC) *1 prodC}
& 
\\
\twocell{(1 *0 unitC *0 1) *1 (prodC *0 1) *1 prodC}
& & 
\twocell{prodC}
\\
};
\node (start) [below =.1cm of m-1-2] {};
\triplearrow{->, bend left}
{(m-1-2) to node [above right] {$\twocell{(1 *0 lunitC) *1 prodC}$}(m-2-3)}
\triplearrow{->, bend left}
{(m-2-1) to node [above left] {$\twocell{(1 *0 unitC *0 1) *1 assocC}$}(m-1-2)}
\triplearrow{->, bend right}
{(m-2-1) to node [below] {$\twocell{(runitC *0 1) *1 prodC}$} (m-2-3)}
\quadarrow{->}
{(start) -- node [right] {$\twocell{trianC}$} (m-2-2)}
\end{tikzpicture}
\]

\end{itemize}
\end{defn}

\begin{defn}
We denote by $\Alg(\Bi)$ the set of all couples $(\CC,\Phi)$:
\begin{itemize}
\item where $\CC$ is a set,
\item where $\Phi$ is a functor from $\overline{\Bi}[\CC]$ to $\sCat$.
\end{itemize} 
\end{defn}

\begin{prop}\label{prop:alg_bicat}
There is a one-to-one correspondence between (small) bicategories and $\Alg(\Bi)$.
\end{prop}
\begin{proof}

The correspondence between a bicategory $\B$ and an algebra $(\CC,\Phi)$ over $\Bi$ is given by:

\begin{itemize}
\item At the level of sets: $\CC = \B_0$.
\item For every $a,b \in \B_0$, $\Phi({}_a \twocell{1} {}_b) = \B(a,b)$.
\item For every $a,b,c \in \B_0$, $\Phi(\twocell{prodC}_{a,b,c}) = \star_{a,b,c}$.
\item For every $a \in \B_0$, $\Phi(\twocell{unitC}_a) = I_a$.
\item For every $a,b,c,d \in \B_0$, $\Phi(\twocell{assocC}_{a,b,c,d}) = \alpha_{a,b,c,d}$.
\item For every $a,b \in \B_0$, $\Phi(\twocell{runitC}_{a,b}) = R_{a,b}$ and $\Phi(\twocell{lunitC}_{a,b}) = L_{a,b}$.
\item The axioms that a bicategory must satisfy correspond to the fact that $\Phi$ is compatible with the quotient by the $4$-cells $\twocell{pentaC}$ and $\twocell{trianC}$.
\end{itemize}

This correspondence between the structures of bicategory and of algebra over $\Bi$ is summed up by the following table: 

\begin{table}[H]
\centering
\begin{tabular}{|c|c||c|c|}
  \hline
\multicolumn{2}{|c||}{Bicategory} & \multicolumn{2}{c|}{$\Alg(\Bi)$} \\
\hline
Sets & $\B_0$ & $\CC$ & $0$-cells \\
Categories & $\B(\_,\_)$ & $\twocell{1}$ & $1$-cells \\
Functors & $\star$, $I$ & $\twocell{prodC}$,  $\twocell{unitC}$ & $2$-cells \\
Natural transformations & $\alpha$, $L$, $R$ & $\twocell{assocC}$, $\twocell{lunitC}$, $\twocell{runitC}$ & $3$-cells \\
Equalities & \eqref{eq:penta} \eqref{eq:trian} & $\twocell{pentaC}$ $\twocell{trianC}$ & $4$-cells \\
\hline
\end{tabular}
\caption{Correspondence for bicategories} 
\label{table:corresp_bicat}
\end{table}
\end{proof}

We are going to show the coherence theorem for bicategories, using Theorem \ref{thm:squier}.

\begin{prop}\label{prop:Bicat_term}
For every set $\CC$, the $4$-polygraph $\Bi[\CC]$ $3$-terminates.
\end{prop}
\begin{proof}
In order to apply Theorem \ref{thm:terminaison} we construct two functors $X_{\CC}: \Bi[\CC]^*_2 \to \sOrd$ and $Y_{\CC} : (\Bi[\CC]^*_2)^\co \to \sOrd$ by setting, for every $a,b \in \CC$:
\[
X_{\CC}({}_a \twocell{1} {}_b) = Y_{\CC}({}_a \twocell{1} {}_b) = \mathbb N^*
\]
and, for every $i,j \in \mathbb N^*$:
\[
X_{\CC}(\twocell{prodC})[i,j] =  i+j,
\qquad
X_{\CC}(\twocell{unitC})  = 1,
\qquad
Y_{\CC}(\twocell{prodC})[i] = (i,i).
\]

We now define an $(X_{\CC},Y_{\CC},\N)$-derivation $d_{\CC}$ on $\Bi[\CC]^*_2$ by setting, for every $i,j,k \in \N^*$:

\[
d_{\CC}(\twocell{prodC})[i,j,k] = i+k+1,
\qquad
d_{\CC}(\twocell{unitC})[i] = i,
\]

It remains to show that the required inequalities are satisfied. Concerning $X_{\CC}$ and $Y_{\CC}$, we have for every $i,j,k \in \N^*$:

\[
X_{\CC}(\twocell{(prodC *0 1) *1 prodC})[i,j,k] = i+j+k \geq i+j+k = X_{\CC}(\twocell{(1 *0 prodC) *1 prodC})[i,j,k] 
\]

\[
X_{\CC}(\twocell{(unitC *0 1) *1 prodC})[i] = i+1 \geq i = X_{\CC}(\twocell{1})[i]
\qquad
X_{\CC}(\twocell{(1 *0 unitC) *1 prodC})[i] = i+1 \geq i = X_{\CC}(\twocell{1})[i]
\]

\[
Y_{\CC}(\twocell{(prodC *0 1) *1 prodC})[i] = (i,i,i) \geq (i,i,i) = Y_C(\twocell{(1 *0 prodC) *1 prodC})[i] 
\]

\[
Y_{\CC}(\twocell{(unitC *0 1) *1 prodC})[i] = i \geq i = Y_{\CC}(\twocell{1})[i]
\qquad
Y_{\CC}(\twocell{(1 *0 unitC) *1 prodC})[i] = i \geq i = Y_{\CC}(\twocell{1})[i].
\]

Concerning $d_{\CC}$, we have for every $i,j,k,l \in \N^*$:

\[
d_{\CC}(\twocell{(prodC *0 1) *1 prodC})[i,j,k,l] = 2i+j+2l+2 > i+j+2l+2 = d_{\CC}(\twocell{(1 *0 prodC) *1 prodC})[i,j,k,l] 
\]

\[
d_{\CC}(\twocell{(unitC *0 1) *1 prodC})[i,j] = 2j+2 > 0 = d_{\CC}(\twocell{1})[i,j]
\qquad
d_{\CC}(\twocell{(1 *0 unitC) *1 prodC})[i,j] = i+2j+1 > 0 = d_{\CC}(\twocell{1})[i,j].
\]
\end{proof}

The following Theorem is a rephrasing of Mac Lane's coherence Theorem (\cite{ML85}) in our setting.
\begin{thm}\label{thm:BiCat_coh}
Let $\CC$ be a set.

The $4$-polygraph $\Bi[\CC]$ is $3$-convergent and the free $(4,2)$-category $\Bi[\CC]^{*(2)}$ is $3$-coherent.
\end{thm}
\begin{proof}
We already know that $\Bi[\CC]$ is $3$-terminating.  Using Proposition \ref{prop:confluence_local} and Theorem \ref{thm:squier}, it remains to show that every critical pair admits a filling.

There are five families of critical pairs, of sources:

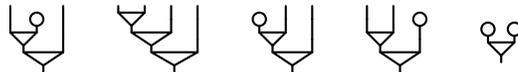
\[
\twocell{(1 *0 unitC *0 1) *1 (prodC *0 1) *1 prodC}
\qquad
\twocell{(prodC *0 2) *1 (prodC *0 1) *1 prodC} 
\qquad
\twocell{(unitC *0 2) *1 (prodC *0 1) *1 prodC}
\qquad
\twocell{(2 *0 unitC) *1 (prodC *0 1) *1 prodC}
\qquad
\twocell{(unitC *0 unitC) *1 prodC} 
\]

The first two families are filled by the $4$-cells $\twocell{trianC}$ and $\twocell{pentaC}$, whereas the last three are filled by $4$-cells $\omega_i \in \Bi[\CC]_4^{*(2)}$, which are constructed in a similar fashion as in the case of monoidal categories (see Proposition 3.5 in \cite{G12}).
\end{proof}

\subsection{Coherence for pseudofunctors}
\label{subsec:pseudo_functors}
\begin{defn}
A \emph{pseudofunctor} $F$ is given by:
\begin{itemize}
\item Two bicategories $\B$ and $\B'$.
\item A function $F_0 : \B_0 \to \B'_0$.
\item For every $a,b \in \B_0$, a functor $F_{a,b} : \B(a,b) \to \B'(F_0(a),F_0(b))$.
\item For every $a,b,c \in \B_0$, a natural isomorphism $\phi_{a,b,c}$:
\[
\xymatrix @R = 1em  @C = 3em{
\B(a,b) \times \B(b,c)
  \ar [rrr] ^-{\star_{a,b,c}}
  \ar [ddd] _-{F_{a,b} \times F_{b,c}}
& & &
\B(a,c)
  \ar [ddd] ^-{F_{a,c}}
\\
& & &
\\
& 
  \ar@2 [ru];[] _-{\phi_{a,b,c}}
& &
\\
\B'(F_0(a),F_0(b)) \times \B'(F_0(b),F_0(c))
  \ar [rrr] _-{\star'_{F_0(a),F_0(b),F_0(c)}}
& & &
\B'(F_0(a),F_0(c))
}
\]
of components $\phi_{f,g} : F(f \star g) \Rightarrow F(f) \star' F(g)$, for every couple $(f,g) \in \B(a,b) \times \B(b,c)$.

\item For every $a \in \B_0$, a natural isomorphism $\psi_a$:
\[
\xymatrix @R = 1em {
\top
  \ar [rrr] ^-{I_a}
  \ar@2{-} [ddd] 
& & &
\B(a,a)
  \ar [ddd] ^-{F_{a,a}}
\\
& & 
\ar@2 [ld] _-{\psi_{a}}
&
\\
& 
& &
\\
\top
  \ar [rrr] _-{I'_{F_0(a),F_0(a)}}
& & &
\B'(F_0(a),F_0(a))
}
\]
of components $\psi_a : F(I_a) \Rightarrow I'_{F_0(a)}$, for every $a \in \B_0$
\end{itemize}

This data must satisfy the following axioms:
\begin{itemize}
\item For every composable $1$-cells $f,g$ and $h$ in $\B$:
\begin{equation} \label{eq:img_assoc}
\begin{gathered}
\xymatrix{
& 
F((f \star g) \star h)
  \ar@2 [rd] ^-{\phi_{f \star g,h}}
  \ar@2 [ld] _-{F(\alpha_{f,g,h})}
  \ar@{} [ddd] |-{=}
& \\
F(f \star (g \star h))
  \ar@2 [d] _-{\phi_{f,g \star h}}
& & 
F(f \star g) \star' F(h)
  \ar@2 [d] ^-{\phi_{f,g} \star' F(h)}
\\
F(f) \star' F(g \star h)
  \ar@2 [rd] _-{F(f) \star' \phi_{g,f}}
& & 
(F(f) \star' F(g)) \star' F(h)
  \ar@2 [ld] ^-*+{\alpha'_{F(f),F(g),F(h)}}
\\
& 
F(f) \star' (F(g) \star' F(h))
&
}
\end{gathered}
\end{equation}

\item For every $1$-cell $f: a \to b$ in $\B$:
\begin{equation} \label{eq:img_lunit}
\begin{gathered}
\xymatrix @C = 4em @R = 3em {
&
F(I_a) \star' F(f)
\ar@2 [r] ^-{\psi_a \star' F(f)}
\ar@{} [rd] |-{=}
&
I'_{F_0(a)} \star' F(f)
\ar@2 [rd] ^-{L'_{F(f)}}
&
\\
F(I_a \star f)
\ar@2 [ru] ^-{\phi_{I_a,f}}
\ar@2 @/_/ [rrr] _-{F(L_f)}
& & &
F(f)
}
\end{gathered}
\end{equation}
\item For every $1$-cell $f: a \to b$ in $\B$:
\begin{equation} \label{eq:img_runit}
\begin{gathered}
\xymatrix @C = 4em @R = 3em {
&
F(f) \star' F(I_b)
\ar@2 [r] ^-{F(f) \star' \psi_b}
\ar@{} [rd] |-{=}
&
F(f) \star' I'_{F_0(b)} 
\ar@2 [rd] ^-{R'_{F(f)}}
&
\\
F(f \star I_b)
\ar@2 [ru] ^-{\phi_{f,I_b}}
\ar@2 @/_/ [rrr] _-{F(R_f)}
& & &
F(f)
}
\end{gathered}
\end{equation}
\end{itemize}
\end{defn}

\begin{defn}
Let $\CC$ and $\DD$ be sets, and $\f$ an application from $\CC$ to $\DD$. Let us describe dimension by dimension a $4$-polygraph $\PF[\f]$. We will prove in Proposition \ref{prop:alg_pfonct} that pseudofunctors correspond to algebras over $\PF[\f]$. 

The polygraph $\PF[\f]$ contains the union of:
\begin{itemize}
\item the polygraph $\Bi[\CC]$, whose cells are denoted by $\twocell{prodC}$, $\twocell{unitC}$, $\twocell{assocC}$, $\twocell{runitC}$, $\twocell{lunitC}$, $\twocell{pentaC}$ and $\twocell{trianC}$, defined as in Definition \ref{def:poly_bicat},
\item the polygraph $\Bi[\DD]$, whose cells are denoted by $\twocell{prodD}$, $\twocell{unitD}$, $\twocell{assocD}$, $\twocell{runitD}$, $\twocell{lunitD}$, $\twocell{pentaD}$ and $\twocell{trianD}$, defined as in Definition \ref{def:poly_bicat},
\end{itemize}
together with the following cells:

\paragraph*{Dimension $1$:}
For every $a \in \CC$, the set $\PF[\f]_1$ contains a $1$-cell ${}_a\twocell{1} {}_{\f(a)} : a \to \f(a)$.

\paragraph*{Dimension $2$:}
For every $a,b \in \CC$, the set $\PF[\f]_2$ contains a $2$-cell $\twocell{fonctF}_{a,b}: {}_a \twocell{1} {}_b \twocell{1} {}_{\f(b)} \Rightarrow {}_a \twocell{1} {}_{\f(a)} \twocell{1} {}_{\f(b)}$.
\paragraph*{Dimension $3$:}
The set $\PF[\f]_3$ contains the following $3$-cells:
\begin{itemize}
\item For every  $a,b,c \in \CC$, a $3$-cell $\twocell{img_prodF}_{a,b,c} :\twocell{(prodC *0 1) *1 fonctF}  \Rrightarrow  \twocell{(1 *0 fonctF) *1 (fonctF *0 1) *1 (1 *0 prodD)}$ of $1$-source ${}_a \twocell{1} {}_b \twocell{1} {}_c \twocell{1} {}_{\f(c)}$.

\item For every $a \in \CC$, a $3$-cell $\twocell{img_unitF}_{a} : \twocell{(unitC *0 1) *1 fonctF} \Rrightarrow  \twocell{1 *1 (1 *0 unitD)} $ of $1$-source ${}_a \twocell{1} {}_{\f(a)}$.
\end{itemize}

\paragraph*{Dimension $4$:}
The $\PF[\f]_4$ contains the following $4$-cells:
\begin{itemize}
\item For every $a,b,c,d \in \CC$, a $4$-cell $\twocell{img_assocF}_{a,b,c,d}$ of $1$-source ${}_a \twocell{1} {}_b \twocell{1} {}_c \twocell{1} {}_d \twocell{1} {}_{\f(d)}$
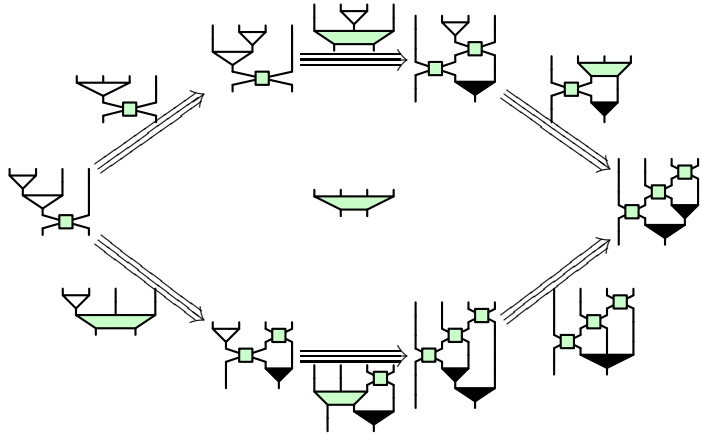
\[
\xymatrix @C = 4em @R = 1.5em {
&
\twocell{(1 *0 prodC *0 1) *1 (prodC *0 1) *1 fonctF}
\ar@3 [r] ^-{\twocell{(1 *0 prodC *0 1) *1 img_prodF}}
&
\twocell{(1 *0 prodC *0 1) *1 (1 *0 fonctF) *1 (fonctF *0 1) *1 (1 *0 prodD)}
\ar@3 [rd] ^-{\twocell{(1 *0 img_prodF) *1 (fonctF *0 1) *1 (1 *0 prodD)}}
&
\\
\twocell{(prodC *0 2) *1 (prodC *0 1) *1 fonctF}
\ar@3 [rd] _-{\twocell{(prodC *0 2) *1 img_prodF}}
\ar@3 [ru] ^-{\twocell{(assocC *0 1) *1 fonctF}}
\ar@{} [rrr] |-{\twocell{img_assocF}}
& & &
\twocell{(2 *0 fonctF) *1 (1 *0 fonctF *0 1) *1 (fonctF *0 prodD) *1 (1 *0 prodD)}
\\
&
\twocell{(prodC *0 fonctF) *1 (fonctF *0 1) *1 (1 *0 prodD)}
\ar@3 [r] _-{\twocell{(2 *0 fonctF) *1 (img_prodF *0 1) *1 (1 *0 prodD)}}
&
\twocell{(2 *0 fonctF) *1 (1 *0 fonctF *0 1) *1 (fonctF *0 2) *1 (1 *0 prodD *0 1) *1 (1 *0 prodD)}
\ar@3 [ru] _-{\twocell{(2 *0 fonctF) *1 (1 *0 fonctF *0 1) *1 (fonctF *0 2) *1 (1 *0 assocD)}}
&
}
\]
\item For every $a,b \in \CC$, $4$-cells $\twocell{img_runitF}_{a,b}$ and $\twocell{img_lunitF}_{a,b}$ of $1$-source $ {}_a \twocell{1} {}_b \twocell{1} {}_{\f(b)}$
\end{itemize}
\[
\xymatrix @C = 4em @R = 1.5em {
&
\twocell{(1 *0 unitC *0 1) *1 (1 *0 fonctF) *1 (fonctF *0 1) *1 (1 *0 prodD)}
\ar@3 [r] ^-{\twocell{(1 *0 img_unitF) *1 (fonctF *0 1) *1 (1 *0 prodD)}}
\ar@{} [rd] |-{\twocell{img_runitF}}
&
\twocell{(fonctF *0 unitD) *1 (1 *0 prodD)}
\ar@3 [rd] ^-{\twocell{fonctF *1 (1 *0 runitD)}}
&
\\
\twocell{(1 *0 unitC *0 1) *1 (prodC *0 1) *1 fonctF}
\ar@3 [ru] ^-{\twocell{(1 *0 unitC *0 1) *1 img_prodF}}
\ar@3 @/_/ [rrr] _-{\twocell{(runitC *0 1) *1 fonctF}}
& & &
\twocell{fonctF}
}
\quad
\xymatrix @C = 4em @R = 1.5em {
&
\twocell{(unitC *0 fonctF) *1 (fonctF *0 1) *1 (1 *0 prodD)}
\ar@3 [r] ^--{\twocell{fonctF *1 (img_unitF *0 1) *1 (1 *0 prodD)}}
\ar@{} [rd] |-{\twocell{img_lunitF}}
&
\twocell{fonctF *1 (1 *0 unitD *0 1) *1 (1 *0 prodD)}
\ar@3 [rd] ^-{\twocell{fonctF *1 (1 *0 lunitD)}}
&
\\
\twocell{(unitC *0 2) *1 (prodC *0 1) *1 fonctF}
\ar@3 [ru] ^-{\twocell{(unitC *0 2) *1 img_prodF}}
\ar@3 @/_/ [rrr] _-{\twocell{(lunitC *0 1) *1 fonctF}}
& & &
\twocell{fonctF}
}
\]
\end{defn}

\begin{defn}
Let $\Alg(\PF)$ be the set of all tuples $(\CC,\DD,\f, \Phi)$:
\begin{itemize}
\item where $\CC$ and $\DD$ are sets,
\item where $\f$ is an application from $\CC$ to $\DD$,
\item where $\Phi$ is a functor from $\overline{\PF}[\f]$ to $\sCat$ such that, for every $c \in \CC$ the following equality holds:
\[
\Phi({}_c \twocell{1} {}_{\f(c)}) = \top
\] 
\end{itemize} 
\end{defn}


\begin{remq}
Let $\f : \CC \to \DD$ be an application. Since $\Bi[\CC]$ (resp. $\Bi[\DD]$) is a sub-$4$-polygraph of $\PF[\f]$, every functor $\Phi: \overline{\PF}[\f] \to \sCat$ induces by restriction two functors: \[\Phi_{\twocell{bicat_source}}: \overline{\Bi}[\CC] \to \sCat \qquad \Phi_{\twocell{bicat_but}}: \overline{\Bi}[\DD] \to \sCat\]
\end{remq}

\begin{prop}\label{prop:alg_pfonct}
Pseudofunctors between (small) categories are in one to one correspondence with elements of $\Alg(\PF)$.
\end{prop}
\begin{proof}

The proof is similar to the case of bicategories, using the following correspondence table:

\begin{table}[H]
\centering
\begin{tabular}{|c|c||c|c|}
  \hline
\multicolumn{2}{|c||}{Pseudofunctors} & \multicolumn{2}{c|}{$\Alg(\PF)$} \\
\hline
Source and target & $\B$ and $\B'$ & $(\CC, \Phi_{\twocell{bicat_source}})$ and $(\DD,\Phi_{\twocell{bicat_but}})$ & Restrictions \\
Function & $F_0$ & $\f$ & Function \\
Functors & $F$ & $\twocell{fonctF}$ & $2$-cells \\
Natural transformations & $\psi$, $\phi$ & $\twocell{img_unitF}$, $\twocell{img_prodF}$ & $3$-cells \\
Equalities & \eqref{eq:img_assoc} \eqref{eq:img_lunit} \eqref{eq:img_runit}  & $\twocell{img_assocF}$ $\twocell{img_lunitF}$ $\twocell{img_runitF}$  & $4$-cells \\
\hline
\end{tabular}
\caption{Correspondence for pseudofunctors} 
\end{table}
\end{proof}

\begin{prop}\label{prop:PFonct_term}
For every sets $\CC,\DD$ and every application $\f : \CC \to \DD$,  the $4$-polygraph $\PF[\f]$ $3$-terminates.
\end{prop}
\begin{proof}
In order to apply Theorem \ref{thm:terminaison}, we define functors $X_{\f} : \PF[\f]_2^* \to \sOrd$ and $Y_{\f} : (\PF[\f]_2^*)^\co \to \sOrd$ as extensions of the functors $X_{\CC}$, $X_{\DD}$, $Y_{\CC}$ and $Y_{\DD}$ from Proposition \ref{prop:Bicat_term}, and by setting for every $a \in \CC$:
\[
X_{\f}({}_a\twocell{1}{}_{\f(a)}) = Y_{\f}({}_a\twocell{1}{}_{\f(a)}) = \top,
\]
where $\top$ is the terminal ordered set, and for every $i \in \mathbb N^*$:
\[
X_{\f}(\twocell{fonctF})[i] = i \qquad Y_{\f}(\twocell{fonctF})[i]= 2i+1.
\]

We now define an $(X_{\f},Y_{\f},\mathbb N)$-derivation $d_{\f}$ on $\PF[\f]_2^*$ as an extension of $d_{\CC}$, by setting for every $i,j,k \in \N^*$:
\[
d_{\f}(\twocell{prodD})[i,j,k] = i+k
\qquad
d_{\f}(\twocell{unitD})[i] = i
\qquad
d_{\f}(\twocell{fonctF})[i,j] = i+j+1 
\]

It remains to show that the inequalities required to apply Theorem \ref{thm:terminaison} are satisfied. Since $X_{\f}$ (resp. $Y_{\f}$) extends $X_{\CC}$ and $X_{\DD}$ (resp. $Y_{\CC}$ and $Y_{\DD}$), the only inequalities that need to be checked are those corresponding to the $3$-cells $\twocell{img_prodF}$ and $\twocell{img_unitF}$. Indeed for every $i,j \in \N^*$, we have:
\[
X_{\f}(\twocell{(unitC *0 1) *1 fonctF}) = 1 \geq 1 = X_{\f}(\twocell{1 *0 unitD})
\]
\[
X_{\f}(\twocell{(prodC *0 1) *1 fonctF})[i,j] = i+j \geq i+j = X_{\f}(\twocell{(1 *0 fonctF) *1 (fonctF *0 1) *1 (1 *0 prodD)})[i,j]
\]
\[
Y_{\f}(\twocell{(prodC *0 1) *1 fonctF})[i] = (2i+1,2i+1) \geq (2i+1,2i+1) = Y_{\f}(\twocell{(1 *0 fonctF) *1 (fonctF *0 1) *1 (1 *0 prodD)})[i]
\]

Concerning $d_{\f}$, the $3$-cells from $\Bi[\CC]$ have already been checked in Proposition \ref{prop:Bicat_term}. For the other $3$-cells, we have, for every $i,j,k \in \N^*$:
\[
d_{\f}(\twocell{(unitD *0 1) *1 prodD})[i,j] = 2j+1 > 0 = d_{\f}(\twocell{1})[i,j]
\qquad
d_{\f}(\twocell{(1 *0 unitD) *1 prodD})[i,j] = i+2j > 0 = d_{\f}(\twocell{1})[i,j]
\]
\[
d_{\f}(\twocell{(unitC *0 1) *1 fonctF})[i] = 3i+2 > i = d_{\f}(\twocell{1 *0 unitD})
\]
\[
d_{\f}(\twocell{(prodC *0 1) *1 fonctF})[i,j,k] = 2i+j+3k+3 > 2i+j+3k+2 = d_{\f}(\twocell{(1 *0 fonctF) *1 (fonctF *0 1) *1 (1 *0 prodD)})[i,j,k].
\]
\end{proof}

\begin{thm}\label{thm:PFonct_coh}
Let $\CC$ and $\DD$ be sets, and $\f:\CC \to \DD$ an application.

The $4$-polygraph $\PF[\f]$ is $3$-convergent and the free $(4,2)$-category $\PF[\f]^{*(2)}$ is $3$-coherent.
\end{thm}
\begin{proof}
We have shown that it is $3$-terminating, so using Proposition \ref{prop:confluence_local} and Theorem \ref{thm:squier}, it remains to show  that every critical pair admits a filler in $\PF[\f]_4$. 

There are thirteen families of critical pairs. Among them, ten come from $\Bi[\CC]$ or $\Bi[\DD]$, and were already dealt with in Theorem \ref{thm:BiCat_coh}. The remaining three have the following sources:

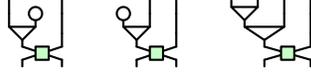
\[
\twocell{(1 *0 unitC *0 1) *1 (prodC *0 1) *1 fonctF}
\qquad
\twocell{(unitC *0 2) *1 (prodC *0 1) *1 fonctF}
\qquad
\twocell{(prodC *0 2) *1 (prodC *0 1) *1 fonctF}
\]

and they are filled respectively by the $4$-cells $\twocell{img_runitF}$, $\twocell{img_lunitF}$ and $\twocell{img_assocF}$.
\end{proof}

\subsection{Coherence for pseudonatural transformations}
\label{subsec:pseudo_nat_transfo}
\begin{defn}
A pseudonatural transformation  $\tau$ consists of the following data:
\begin{itemize}
\item Two pseudofunctors $F,F' : \B \to \B'$, where $\B$ and $\B'$ are bicategories.
\item For every $a \in \B_0$, a functor $\tau_a : \top \to \B'(F_0(a),F'_0(a))$, that is a $1$-cell $\tau_a: F_0(a) \to F'_0(a)$ in $\B'$.
\item For every $a,b \in \B_0$, a natural isomorphism $\sigma_{a,b}$:

\[
\xymatrix @!C=6em {
& 
\B(a,b)
  \ar [dr] ^{F'_{a,b}}
  \ar [ld] _{F_{a,b}}
& \\
\B'(F_0(a),F_0(b))
  \ar [d] _{\B'(F_0(a),F_0(b)) \times \tau_b} ^{\qquad}="tgt"
& & 
\B'(F'_0(a),F'_0(b))
  \ar [d] ^{\tau_a \times \B'(F'_0(a),F'_0(b))} _{\qquad}="src"
\\
\B'(F_0(a),F_0(b)) \times \B'(F_0(b),F'_0(b))
  \ar [rd] _{\star'_{F_0(a),F_0(b),F'_0(b)}}
& & 
\B'(F_0(a),F'_0(a)) \times \B'(F'_0(a),F'_0(b))
  \ar [dl] ^-*+{\star'_{F_0(a),F'_0(a),F'_0(b)}}
\\
& 
\B'(F_0(a),F'_0(b))
& 
\ar@2 "tgt";"src" ^{\sigma_{a,b}}
}
\]
of components $\sigma_f : F(f) \star' \tau_b \Rightarrow  \tau_a \star' F'(f) $, for every $f \in \B(a,b)$.
\end{itemize}

This data must satisfy the following axioms:
\begin{itemize}
\item For every $(f,g) \in \B(a,b) \times \B(b,c)$:
\begin{equation} \label{eq:transfo_prod}
\begin{gathered}
\xymatrix @R=3em{
& 
\tau_a \star' F'(f \star g)
  \ar@2 [rd] ^{\tau_a \star' \phi'_{f,g}}
  \ar@2 [ld];[] ^{\sigma_{f \star g}}
  \ar@{} [dddd] |-{=}
& \\
F(f \star g) \star' \tau_c
  \ar@2 [d] _{\phi_{f,g} \star' \tau_c}
& & 
\tau_a \star' (F'(f) \star' F'(g)) 
  \ar@2 [d];[] _{\alpha'_{\tau_a,F'(f),F'(g)}}
\\
(F(f) \star' F(g)) \star' \tau_c
  \ar@2 [d] _{\alpha'_{F(f),F(g),\tau_c}}
& & 
(\tau_a \star' F'(f)) \star' F'(g)
  \ar@2 [d];[] _{\sigma_f \star' F'(g)}
\\
F(f) \star' (F(g) \star' \tau_c)
  \ar@2 [rd] _*+{F(f) \star' \sigma_g} 
& & 
(F(f) \star' \tau_b) \star' F'(g)
  \ar@2 [ld] ^*+{\alpha'_{F(f),\tau_b,F(g)}}
\\
& 
F(f) \star' (\tau_b \star' F'(g))
&
}
\end{gathered}
\end{equation}

\item For every $a \in \B_0$:

\begin{equation} \label{eq:transfo_unit}
\begin{gathered}
\begin{tikzpicture}
\matrix (m) [matrix of math nodes, 
			nodes in empty cells,
			column sep = 1cm, 
			row sep = .7cm] 
{
& F(I_a) \star' \tau_a & \\
& & \tau_a \star' F'(I_a) \\
I'_{F_0(a)} \star' \tau_a & & \\
& & \tau_a \star' I'_{F'_0(a)} \\
& \tau_a & \\
};
\doublearrow{->}
{(m-1-2) -- node [above right] {$\sigma_{I_a}$} (m-2-3)}
\doublearrow{->}
{(m-1-2) -- node [above left] {$\psi_a \star' \tau_a$} (m-3-1)}
\doublearrow{->}
{(m-2-3) -- node [right] {$\tau_a \star' \psi'_a$} (m-4-3)}
\doublearrow{->}
{(m-3-1) -- node [below left] {$L'_{\tau_a}$} (m-5-2)}
\doublearrow{->}
{(m-4-3) -- node [below right] {$R'_{\tau_a}$} (m-5-2)}

\path (m-1-2) -- node {$=$} (m-5-2);
\end{tikzpicture}
\end{gathered}
\end{equation}
\end{itemize}
\end{defn}
\begin{defn}
Let $\CC$ and $\DD$ be sets, and $\f,\g$ be applications from $\CC$ to $\DD$. Let us define dimension by dimension a $(4,2)$-polygraph $\TPN[\f,\g]$. We will see in Proposition \ref{prop:alg_tpn} that pseudonatural transformations correspond to algebras over $\TPN[\f,\g]$.

The polygraph $\TPN[\f,\g]$ contains the union of the  polygraphs $\PF[\f]$ and $\PF[\g]$. In particular, the following cells are in $\TPN[\f,\g]$:
\begin{itemize}
\item  the cells $\twocell{prodC}$, $\twocell{unitC}$, $\twocell{assocC}$, $\twocell{runitC}$, $\twocell{lunitC}$, $\twocell{pentaC}$ and $\twocell{trianC}$ coming from $\Bi[\CC]$,
\item the cells $\twocell{prodD}$, $\twocell{unitD}$, $\twocell{assocD}$, $\twocell{runitD}$, $\twocell{lunitD}$, $\twocell{pentaD}$ and $\twocell{trianD}$ coming from $\Bi[\DD]$,
\item the cells $\twocell{fonctF}$, $\twocell{img_prodF}$, $\twocell{img_unitF}$, $\twocell{img_assocF}$, $\twocell{img_runitF}$ and $\twocell{img_lunitF}$ coming from $\PF[\f]$,
\item the cells $\twocell{fonctG}$, $\twocell{img_prodG}$, $\twocell{img_unitG}$, $\twocell{img_assocG}$, $\twocell{img_runitG}$ and $\twocell{img_lunitG}$ coming from $\PF[\g]$.
\end{itemize}

Together with the union of $\PF[\f]$ and $\PF[\g]$, $\TPN[\f,\g]$ contains the following cells:
\paragraph*{Dimension $2$:}
For every  $a \in \CC$, the set $\TPN[\f,\g]_2$ contains a $2$-cell $\twocell{transfo}_a: {}_a \twocell{1} {}_{\g(a)} \Rightarrow {}_a \twocell{1} {}_{\f(a)} \twocell{1} {}_{\g(a)}$.

\paragraph*{Dimension $3$:}
For every $a,b \in \CC$, the set $\TPN[\f,\g]_3$ contains a $3$-cell:
 $\twocell{transfo_nat}_{a,b} :
\twocell{(1 *0 transfo) *1 (fonctF *0 1) *1 (1 *0 prodD)}  \Rrightarrow
\twocell{fonctG *1 (transfo *0 1) *1 (1 *0 prodD)} $ of $1$-source ${}_a \twocell{1} {}_b \twocell{1} {}_{\g(b)}$.
\paragraph*{Dimension $4$:}
The set $\TPN[\f,\g]_4$ contains the following $4$-cells:
\begin{itemize}
\item For every $a \in \CC$, a $4$-cell $\twocell{transfo_unit}_a$ of $1$-source ${}_a \twocell{1} {}_{\g(a)}$
\[
\xymatrix @C = 1.5em @R = 2em {
&
\twocell{(unitC *0 1) *1 fonctG *1 (transfo *0 1) *1 (1 *0 prodD)} 
\ar@3 [rr] ^-{\twocell{img_unitG *1 (transfo *0 1) *1 (1 *0 prodD)}}
& 
\ar@{} [dd] |-{\twocell{transfo_unit}}
&
\twocell{(transfo *0 unitD) *1 (1 *0 prodD)}
\ar@3 [rd] ^{\twocell{transfo *1 (1 *0 runitD)}}
&
\\
\twocell{(unitC *0 transfo) *1 (fonctF *0 1) *1 (1 *0 prodD)}
\ar@3 [rrd] _-{\twocell{transfo *1 (img_unitF *0 1) *1 (1 *0 prodD)}}
\ar@3 [ru] ^-{\twocell{(unitC *0 1) *1 transfo_nat}}
& & & &
\twocell{transfo}
\\
& &
\twocell{transfo *1 (1 *0 unitD *0 1) *1 (1 *0 prodD)}
\ar@3 [rru] _-{\twocell{transfo *1 (1 *0 lunitD)}}
}
\]
\item For every $a,b,c \in \CC$, a $4$-cell $\twocell{transfo_prod}_{a,b,c}$ of $1$-source ${}_a \twocell{1} {}_b \twocell{1} {}_c \twocell{1} {}_{\g(c)}$
\[
\xymatrix @!C = 5em @R = 1.5em {
&
\twocell{(1 *0 fonctG) *1 (fonctG *0 1) *1 (transfo *0 prodD) *1 (1 *0 prodD)}
& 
\twocell{(1 *0 fonctG) *1 (fonctG *0 1) *1 (transfo *0 2) *1 (1 *0 prodD *0 1) *1 (1 *0 prodD)}
\ar@3 [l] _{\twocell{(1 *0 fonctG) *1 (fonctG *0 1) *1 (transfo *0 2) *1 (1 *0 assocD)}}
& 
\twocell{(1 *0 fonctG) *1 (1 *0 transfo *0 1) *1 (fonctF *0 2) *1 (1 *0 prodD *0 1) *1 (1 *0 prodD)}
\ar@3 [r] ^{\twocell{(1 *0 fonctG) *1 (1 *0 transfo *0 1) *1 (fonctF *0 2) *1 (1 *0 assocD)}}
\ar@3 [l] _{\twocell{(1 *0 fonctG) *1 (transfo_nat *0 1) *1 (1 *0 prodD)}}
&
\twocell{(1 *0 fonctG) *1 (1 *0 transfo *0 1) *1 (fonctF *0 prodD) *1 (1 *0 prodD)}
&
\\
\twocell{(prodC *0 1) *1 fonctG *1 (transfo *0 1) *1 (1 *0 prodD)}
\ar@3 [ru] ^{\twocell{img_prodG *1 (transfo *0 1) *1 (1 *0 prodD)}}
\ar@{} [rrrrr] |-{\twocell{transfo_prod}}
& & 
& &
&
\twocell{(2 *0 transfo) *1 (1 *0 fonctF *0 1) *1 (fonctF *0 prodD) *1 (1 *0 prodD)}
\ar@3 [ul] _{\twocell{(1 *0 transfo_nat) *1 (fonctF *0 1) *1 (1 *0 prodD)}}
\\
& 
& 
\twocell{(prodC *0 transfo) *1 (fonctF *0 1) *1 (1 *0 prodD)}
\ar@3 [r] _{\twocell{(2 *0 transfo) *1 (img_prodF *0 1) *1 (1 *0 prodD)}}
\ar@3 [llu] ^{\twocell{(prodC *0 1) *1 transfo_nat}}
&
\twocell{(2 *0 transfo) *1 (1 *0 fonctF *0 1) *1 (fonctF *0 2) *1 (1 *0 prodD *0 1) *1 (1 *0 prodD)}
\ar@3 [rru] _{\twocell{(2 *0 transfo) *1 (1 *0 fonctF *0 1) *1 (fonctF *0 2) *1 (1 *0 assocD)}}
&
&
}
\]
\end{itemize}
\end{defn}

\begin{defn}
Let $\Alg(\TPN)$  be the set of tuples $(\CC,\DD,\f,\g,\Phi)$ :
\begin{itemize}
\item where $\CC$ and $\DD$ are sets,
\item where $\f,\g:\CC \to \DD$ are applications,
\item where $\Phi$ is a functor from $\overline{\TPN}[\f,\g]$  to $\sCat$, such that for every $c \in \CC$,  $d \in \DD$ and $1$-cell $\twocell{1} : c \to d$: 
\[
\Phi({}_c\twocell{1}{}_d) = \top
\]
\end{itemize} 
\end{defn}

\begin{remq}
Since $\PF[\f]$ (resp. $\PF[\g]$) is a sub-$4$-polygraph of $\TPN[\f,\g]$, every functor $\Phi: \overline{\TPN}[\f,\g] \to \sCat$ induces by restriction two functors \[\Phi_{\twocell{fonct_source}}: \overline{\PF}[\f] \to \sCat \qquad \Phi_{\twocell{fonct_but}}: \overline{\PF}[\g] \to \sCat\]
\end{remq}

\begin{prop}\label{prop:alg_tpn}
Pseudonatural transformations between pseudofuncteurs are in one to one correspondence with elements of $\Alg(\TPN)$.
\end{prop}
\begin{proof}
The proof is similar to that of bicategories, using the following correspondence table:

\begin{table}[H]
\centering
\begin{tabular}{|c|c||c|c|}
  \hline
\multicolumn{2}{|c||}{Pseudonatural transformations} & \multicolumn{2}{c|}{$\Alg(\TPN)$} \\
\hline
Source and target & $F$ and $F'$ & $\Phi_{\twocell{fonct_source}}$ and $\Phi_{\twocell{fonct_but}}$ & Restrictions \\
Functors & $\tau$ & $\twocell{transfo}$ & $2$-cells \\
Natural transformations & $\sigma$ & $\twocell{transfo_nat}$ & $3$-cells \\
Equalities & \eqref{eq:transfo_prod} \eqref{eq:transfo_unit}  & $\twocell{transfo_prod}$ $\twocell{transfo_unit}$ & $4$-cells \\
\hline
\end{tabular}
\caption{Correspondence for pseudonatural transformations} 
\end{table}
\end{proof}

This result induces the following classification of the cells of the $(4,2)$-polygraph $\TPN[f,g]$, depending on which structure they come from. We also distinguish two types of cells: \emph{product cells} and \emph{unit cells}. Moreover, in the following table, every line corresponds to a dimension.

\begin{table}[H]
\centering
\begin{tabular}{|c|c|c c|}
  \hline
  Origin & Dimension & Product cells & Unit cells \\
  \hline
  \hline
  \multirow{3}{*}{Source bicategory} & $2$-cells & $\twocell{prodC}$ & $\twocell{unitC}$  \\
   & $3$-cells & $\twocell{assocC}$ & $\twocell{runitC}$, $\twocell{lunitC}$   \\
   & $4$-cells & $\twocell{pentaC}$ & $\twocell{trianC}$  \\
 \hline
 \multirow{3}{*}{Target bicategory} & $2$-cells & $\twocell{prodD}$ & $\twocell{unitD}$  \\
   & $3$-cells & $\twocell{assocD}$ & $\twocell{runitD}$, $\twocell{lunitD}$   \\
   & $4$-cells & $\twocell{pentaD}$ & $\twocell{trianD}$  \\
  \hline
  \multirow{3}{*}{Source pseudofunctor} & $2$-cells & $\twocell{fonctF}$ &    \\
   & $3$-cells & $\twocell{img_prodF}$ & $\twocell{img_unitF}$  \\
   & $4$-cells & $\twocell{img_assocF}$ & $\twocell{img_runitF}$, $\twocell{img_lunitF}$  \\
  \hline
  \multirow{3}{*}{Target pseudofunctor} & $2$-cells &  $\twocell{fonctG}$ &    \\
   & $3$-cells & $\twocell{img_prodG}$ & $\twocell{img_unitG}$  \\
   & $4$-cells & $\twocell{img_assocG}$ & $\twocell{img_runitG}$, $\twocell{img_lunitG}$ \\
  \hline
  \multirow{3}{*}{Pseudonatural transformation} & $2$-cells & $\twocell{transfo}$ &    \\
   & $3$-cells &$\twocell{transfo_nat}$&   \\
   &  $4$-cells & $\twocell{transfo_prod}$ & $\twocell{transfo_unit}$  \\
  \hline
\end{tabular}
\caption{Classification of the cells of $\TPN[f,g]$} 
\label{table:Liste_cellule}
\end{table}

\begin{prop}\label{prop:TPN_termine}
Let $\f,\g: \CC \to \DD$ be two applications. The $(4,2)$-polygraph $\TPN[\f,\g]$ $3$-terminates.
\end{prop}
\begin{proof}
We apply Theorem \ref{thm:terminaison}. To construct the functors $X_{\f,\g} : \TPN[f,g]_2^* \to \sOrd$ and $Y_{\f,\g} : (\TPN[\f,\g]_2^*)^\co \to \sOrd$, we extend the functors $X_{\f}$, $X_{\g}$, $Y_{\f}$ and $Y_{\g}$ from Proposition \ref{prop:PFonct_term}, by setting:

\[
X_{\f,\g}(\twocell{transfo}) = 1
\]

We now define an $(X_{\f,\g},Y_{\f,\g},\mathbb N)$-derivation $d_{\f,\g}$ of the $2$-category $\TPN[\f,\g]_2^*$ as the extension of $d_{\f}$ satisfying, for every $i,j \in \N^*$:

\[
d_{\f,\g}(\twocell{fonctG})[i,j] = i+j
\qquad
d_{\f,\g}(\twocell{transfo})[i] = i
\]

It remains to show that the required inequalities are satisfied. Since $X_{\f,\g}$ (resp. $Y_{\f,\g}$) is an extension $X_{\f}$ and $X_{\g}$ (resp. $Y_{\f}$ and $Y_{\g}$), it only remains to treat the case of the  $3$-cell $\twocell{transfo_nat}$. For every $i,j \in \N^*$, we have:

\[
X_{\f,\g}(\twocell{(1 *0 transfo) *1 (fonctF *0 1) *1 (1 *0 prodD)})[i]  = i+1 \geq i+1= X_{\f,\g}( \twocell{fonctG *1 (transfo *0 1) *1 (1 *0 prodD)}
)[i]
\qquad
Y_{\f,\g}(\twocell{(1 *0 transfo) *1 (fonctF *0 1) *1 (1 *0 prodD)})[i] = 2i+1 \geq 2i+1=   Y_{\f,\g}(\twocell{fonctG *1 (transfo *0 1) *1 (1 *0 prodD)})[i]
\]

Concerning $d_{\f,\g}$, the $3$-cells from $\PF[\f]$ were already treated in Proposition \ref{prop:Bicat_term}. For the others we have, for every $i,j,k \in \N^*$:

\[
d_{\f,\g}(\twocell{(prodC *0 1) *1 fonctG})[i,j,k] = 2i+j+3k+2 > 2i+j+3k = d_{\f,\g}(\twocell{(1 *0 fonctG) *1 (fonctG *0 1) *1 (1 *0 prodD)})[i,j,k]
\]
\[
d_{\f,\g}(\twocell{(unitC *0 1) *1 fonctG}) = 3i+1 > i = d_{\f,\g}(\twocell{1 *0 unitD})
\qquad
d_{\f,\g}(\twocell{(1 *0 transfo) *1 (fonctF *0 1) *1 (1 *0 prodD)})[i,j]  = 2i+3j+1 >  i+3j+1 = d_{\f,\g}(\twocell{fonctG *1 (transfo *0 1) *1 (1 *0 prodD)})[i,j]
\]
\end{proof}

\begin{defn}
We define a \emph{weight} application $\w$ as the $1$-functor from $\TPN[\f,\g]_1^*$ to $\mathbb N$,  defined as follows on  $\TPN[\f,\g]_1$:
\begin{itemize}
\item for all $a,b \in \CC$, $\w({}_a \twocell{1} {}_b) = 1$,
\item for all $a,b \in \DD$, $\w({}_a \twocell{1} {}_b) = 1$,
\item for all $a \in \CC$ and $b \in \DD$, $\w({}_a \twocell{1} {}_b) = 0$.
\end{itemize}
\end{defn}

\begin{thm}[Coherence for pseudonatural transformations]\label{thm:TPNat_coh}
Let $\CC$ and $\DD$ be sets, and $\f,\g: \CC \to \DD$ applications.

Let $A,B \in \TPN[\f,\g]^{*(2)}_3$ be two parallel $3$-cells whose $1$-target is of weight $1$.

There is a $4$-cell $\alpha:A \qfl B \in \TPN[\f,\g]^{*(2)}_4$.
\end{thm}

This theorem will be proven in Section \ref{sec:proof_of_coherence}. Contrary to the case of bicategories and pseudofunctors, we cannot directly apply Theorem \ref{thm:squier} to the $(4,2)$-polygraph $\TPN[\f,\g]$, because the following  critical pair is not confluent:

\[
\xymatrix @C = 4em @R = 1.5em {
&
\twocell{(prodC *0 1) *1 fonctG *1 (transfo *0 1) *1 (1 *0 prodD)}
\ar@3 [r] ^{\twocell{img_prodG *1 (transfo *0 1) *1 (1 *0 prodD)}}
& 
\twocell{(1 *0 fonctG) *1 (fonctG *0 1) *1 (transfo *0 prodD) *1 (1 *0 prodD)}
& 
\\
\twocell{(prodC *0 transfo) *1 (fonctF *0 1) *1 (1 *0 prodD)}
\ar@3 [dr] _{\twocell{(2 *0 transfo) *1 (img_prodF *0 1) *1 (1 *0 prodD)}}
\ar@3 [ru] ^{\twocell{(prodC *0 1) *1 transfo_nat}}
& & 
& 
\\
& 
\twocell{(2 *0 transfo) *1 (1 *0 fonctF *0 1) *1 (fonctF *0 2) *1 (1 *0 prodD *0 1) *1 (1 *0 prodD)}
\ar@3 [r] _{\twocell{(2 *0 transfo) *1 (1 *0 fonctF *0 1) *1 (fonctF *0 2) *1 (1 *0 assocD)}}
& 
\twocell{(2 *0 transfo) *1 (1 *0 fonctF *0 1) *1 (fonctF *0 prodD) *1 (1 *0 prodD)}
\ar@3 [r] _{\twocell{(1 *0 transfo_nat) *1 (fonctF *0 1) *1 (1 *0 prodD)}}
&
\twocell{(1 *0 fonctG) *1 (1 *0 transfo *0 1) *1 (fonctF *0 prodD) *1 (1 *0 prodD)}
}
\]
Theorem \ref{thm:main_theory} will be used in order to avoid this difficulty.
\newpage

\section{Proof of the coherence for pseudonatural transformations}
\label{sec:proof_of_coherence}

In this section we prove Theorem \ref{thm:TPNat_coh}. We fix for the rest of this section two sets $\CC$ and $\DD$, together with two applications $\f, \g: \CC \to \DD$. Let $A,B \in \TPN[\f,\g]^{*(2)}$ be $3$-cells whose $1$-target is of weight $1$. We want to build a $4$-cell $\alpha: A \qfl B \in \TPN[\f,\g]^{*(2)}$.

The $1$-cells of weight $1$ are of one of the following forms, with $a, a' \in \CC$ and $b, b' \in \DD$:
\[
{}_{a} \twocell{1} {}_{a'} 
\qquad 
{}_{b} \twocell{1} {}_{b'}
\qquad 
{}_{a} \twocell{1} {}_{a'} \twocell{1} {}_{\f(a')}
\qquad
{}_{a} \twocell{1} {}_{a'} \twocell{1} {}_{\g(a')}
\qquad
{}_{a} \twocell{1} {}_{a'} \twocell{1} {}_{\f(a')}
\qquad
{}_{a} \twocell{1} {}_{\g(a)} \twocell{1} {}_{b}
\qquad
{}_{a} \twocell{1} {}_{\f(a)} \twocell{1} {}_{b}
\]

In Section \ref{subsec:cas_convergent}, we show that if the common $1$-target of $A$ and $B$ is not of the last form, then they are generated by a sub-$4$-polygraph $\PF[\f,\g]$ of $\TPN[\f,\g]$. We then show using Theorem \ref{thm:squier} that this $4$-polygraph is coherent.

There remains to treat the case where the $1$-target of $A$ and $B$ is of the last form. We define two sub-$(4,2)$-polygraphs of $\TPN[\f,\g]$: $\TPN^{++}[\f,\g]$ and $\TPN^{+}[\f,\g]$. The $(4,2)$-polygraph $\TPN^{++}[\f,\g]$ contains all the structure of pseudonatural transformations, except for the axioms concerning the units $\twocell{unitC}$ and $\twocell{unitD}$, while $\TPN^+[\f,\g]$ is an intermediary between $\TPN^{++}[\f,\g]$ and $\TPN[\f,\g]$, which contains the $2$-cells $\twocell{unitC}$ and $\twocell{unitD}$. The inclusions between the $(4,2)$-polygraphs can be seen as follows:
 \[
\TPN^{++}[\f,\g]_2 \subset \TPN^{+}[\f,\g]_2 = \TPN[\f,\g]_2
\]
\[
\TPN^{++}[\f,\g]_3 = \TPN^{+}[\f,\g]_3 \subset \TPN[\f,\g]_3
\] 

In Section \ref{subsec:tpnat_coh_final} we show that $\TPN^{++}[\f,\g]$ satisfies the $2$-Squier condition of depth $2$, which allows us to apply Theorem \ref{thm:main_theory}. Since $1$-cells of the form ${}_{a} \twocell{1} {}_{\f(a)} \twocell{1} {}_{b}$ are normal forms for $\TPN^{++}[\f,\g]$, this concludes the case where $A$ and $B$ are in $\TPN^{++}[\f,\g]$. 

We then define a sub-$3$-polygraph $\TPN^u[\f,\g]$ of $\TPN[\f,\g]$. The rewriting system induced by the $3$-cells $\TPN^u[\f,\g]$ corresponds to simplifying the units out.

Using the properties of this rewriting system, we extend the result of Section \ref{subsec:tpnat_coh_final}, first to $3$-cells $A$ and $B$ in $\TPN^+[\f,\g]$ in Section \ref{subsec:2_unit_elim}, and finally to general $A$ and $B$ whose $1$-target is ${}_{a} \twocell{1} {}_{\f(a)} \twocell{1} {}_{b}$ in Section \ref{subsec:3_unit_elim}, thereby concluding the proof.

\subsection{A convergent sub-polygraph of $\TPN[\f,\g]$}
\label{subsec:cas_convergent}

\begin{defn}
Let $\PF[\f,\g]$ be the $4$-polygraph containing every cell of $\TPN[\f,\g]$, except those corresponding to the pseudonatural transformation. Alternatively, $\PF[\f,\g]$ is the union of $\PF[\f]$ and $\PF[\g]$.
\end{defn}

\begin{lem}\label{lem:but_ou_PF}
For every $h \in \TPN[\f,\g]_2^{*}$, one of the following holds:
\begin{itemize}
\item The target of $h$ is of the form
\begin{equation}\label{eq:but_ok}
{}_{a_1} \twocell{dots} {}_{a_i} \twocell{1} {}_{\f(a_i)} \twocell{1} {}_{b_1} \twocell{dots} {}_{b_j},
\end{equation}

where $i$ and $j$ are non-zero integers, the $a_k$ are in $\CC$ and the $b_k$ are in $\DD$.
\item The $2$-cell $h$ is in $\PF[\f,\g]_2^*$.
\end{itemize}
\end{lem}
\begin{proof}
Let us show first that the set of all $1$-cells of the form \eqref{eq:but_ok} is stable when rewritten by $\PF[\f,\g]_2^*$. To prove this, we examine the case of every cell of $\PF[\f,\g]_2^*$ of length $1$:
\[
{}_{a_1} \twocell{dots} {}_{a_{k-1}} \twocell{prodC} {}_{a_{k+1}} \twocell{dots} {}_{b_j} : {}_{a_1} \twocell{dots} {}_{a_i} \twocell{1} {}_{\f(a_i)} \twocell{1} {}_{b_1} \twocell{dots} {}_{b_j} \Rightarrow {}_{a_1} \twocell{dots} {}_{a_{k-1}} \twocell{1} {}_{a_{k+1}}  \twocell{dots} {}_{b_j}
\]
\[
{}_{a_1}\twocell{dots} {}_{a_k} \twocell{unitC} {}_{a_k} \twocell{dots} {}_{b_j}: {}_{a_1} \twocell{dots} {}_{a_i} \twocell{1} {}_{\f(a_i)} \twocell{1} {}_{b_1} \twocell{dots} {}_{b_j} \Rightarrow {}_{a_1}\twocell{dots} {}_{a_k} \twocell{1} {}_{a_k} \twocell{dots} {}_{b_j}
\]
\[
{}_{a_1} \twocell{dots} {}_{b_{k-1}} \twocell{prodD} {}_{b_{k+1}} \twocell{dots} {}_{b_j} : {}_{a_1} \twocell{dots} {}_{a_i} \twocell{1} {}_{\f(a_i)} \twocell{1} {}_{b_1} \twocell{dots} {}_{b_j} \Rightarrow {}_{a_1} \twocell{dots} {}_{b_{k-1}} \twocell{1} {}_{b_{k+1}}  \twocell{dots} {}_{b_j}
\]
\[
{}_{a_1}\twocell{dots} {}_{b_k} \twocell{unitD} {}_{b_k} \twocell{dots} {}_{b_j}: {}_{a_1} \twocell{dots} {}_{a_i} \twocell{1} {}_{\f(a_i)} \twocell{1} {}_{b_1} \twocell{dots} {}_{b_j} \Rightarrow {}_{a_1} \twocell{dots} {}_{b_k} \twocell{1} {}_{b_k} \twocell{dots} {}_{b_j}
\]
\[
{}_{a_1}\twocell{dots} {}_{a_{i-1}} \twocell{fonctF} {}_{\f(a_i)} \twocell{dots} {}_{b_j}: {}_{a_1} \twocell{dots} {}_{a_i} \twocell{1} {}_{\f(a_i)} \twocell{1} {}_{b_1} \twocell{dots} {}_{b_j} \Rightarrow {}_{a_1} \twocell{dots} {}_{a_{i-1}} \twocell{1} {}_{\f(a_{i-1})} \twocell{1} {}_{\f(a_i)} \twocell{1} {}_{b_1} \twocell{dots} {}_{b_j}
\]

Let us now prove the lemma: we reason by induction on the length of $h$. If $h$ is of length $0$, it is an identity, so  $h$ is in $\PF[\f,\g]^*$.

If $h$ is of length $1$ and $h$ is not in $\PF[\f,\g]^*$, then $h$ has to be of the form $\twocell{dots *0 transfo *0 dots}$. So its target is of the form:
\[
{}_{a_1}\twocell{dots} {}_{a_k} \twocell{1} {}_{\f(a_k)} \twocell{1} {}_{\g(a_k)} \twocell{1} {}_{b_2}  \twocell{dots} {}_{b_j}
\]
which is indeed of the form \eqref{eq:but_ok}, with $b_1 = \g(a_k)$.

Let now $h$ be of length $n > 1$. We can write $h = h_1 \star_1 h_2$, where $h_2$ is of length $1$, and $h_1$ is strictly shorter than $h$. Let us apply the induction hypothesis to $h_2$. If the target of $h_2$ is of the form \eqref{eq:but_ok}, then so is the target of $h$, since $\t(h_2) = \t(h)$. Otherwise, then $h_1 \in \PF[\f,\g]^*$, and we can apply the induction hypothesis to $h_2$. If $h_2$ also is in $\PF[\f,\g]^*$, then so is $h$. 

It remains to treat the case where $\t(h_1)$ is of the form \eqref{eq:but_ok} , and $h_2$ is in $\PF[\f,\g]^*$. But we have shown that the $1$-cells of the form \eqref{eq:but_ok} are stable when rewritten by $\PF[\f,\g]^*$. Thus the target of $h_2$ (which is the target of $h$) is of the form \eqref{eq:but_ok}, which concludes the proof.
\end{proof}

\begin{lem}\label{lem:but_ou_PF_3}
For every $A \in \TPN[\f,\g]_3^{*(2)}$, one of the following holds:
\begin{itemize}
\item The $1$-target of $A$ is of the form \eqref{eq:but_ok}.
\item The $3$-cell $A$ is in $\PF[\f,\g]_3^{*(2)}$.
\end{itemize}
\end{lem}
\begin{proof}
Let us start by the case where $A$ is a $3$-cell of length $1$ in $\TPN[\f,\g]_3^*$. If the $1$-target of $A$ is not of the form \eqref{eq:but_ok} then, according to Lemma \ref{lem:but_ou_PF}, the $2$-source of $A$ is in $\PF[\f,\g]_2^{*}$. The only $3$-cell in $\TPN[\f,\g]_3$ which is not in $\PF[\f,\g]_3$ is the $3$-cell $\twocell{transfo_nat}$, whose $2$-source $\twocell{(1 *0 transfo) *1 (fonctF *0 1) *1 (1 *0 prodD)}$ is not in $\PF[\f,\g]^*_2$. Thus $A$ is in $\PF[\f,\g]_3^*$.

Suppose now that $A=B^{-1}$, where $B$ is a $3$-cell of $\TPN[\f,\g]_3^*$ of length $1$. The $1$-target of $B$ is the same as the one of $A$. If it is not of the form \eqref{eq:but_ok}, $B$ is in $\PF[\f,\g]_3^{*(2)}$, and so is $A$.

In the general case, $A$ is a composite of $3$-cells of one of the two previous forms, and all of them have the same $1$-target as $A$. Thus if the $1$-target of $A$ is not of the form \eqref{eq:but_ok}, all those $3$-cells are in $\PF[\f,\g]_3^{*(2)}$, and so is $A$.
\end{proof}

\begin{lem}\label{lem:PF_coh}
The $4$-polygraph $\PF[\f,\g]$ is $3$-coherent.
\end{lem}
\begin{proof}
It is a sub-$4$-polygraph of $\TPN[\f,\g]$ which is $3$-terminating, therefore it is also $3$-terminating. Moreover, every critical pair in $\PF[\f,\g]$ arises from one either in $\PF[\f]$ or $\PF[\g]$. Since those $4$-polygraphs are confluent and satisfy the Squier condition, so does $\PF[\f,\g]$.

Using Theorem \ref{thm:squier}, this means that $\PF[\f,\g]$ is $3$-coherent.
\end{proof}

\begin{prop}\label{prop:TPN_partial_coh}
Let $\f,\g: \CC \to \DD$ be two applications. 

For every parallel $3$-cells $A,B \in \TPN[\f,\g]^{*(2)}$ whose $1$-target is not of the form \eqref{eq:but_ok}, there exists a $4$-cell $\alpha: A \qfl B \in \TPN[\f,\g]_4^{*(2)}$.

In particular, for every parallel $3$-cells $A,B \in \TPN[\f,\g]^{*(2)}$ whose $1$-target is of weight $1$ and is not of the form ${}_a \twocell{1} {}_{\f(a)} \twocell{1} {}_b$, there exists a $4$-cell $\alpha: A \qfl B \in \TPN[\f,\g]_4^{*(2)}$.
\end{prop}
\begin{proof}
Let $A,B \in \TPN[\f,\g]_3^{*(2)}$ whose $1$-target is not of the form \eqref{eq:but_ok}. We want to build a $4$-cell $\alpha : A \qfl B \in \TPN[\f,\g]_4^{*(2)}$. According to Lemma \ref{lem:but_ou_PF_3}, $A$ and $B$ are actually $3$-cells in $\PF[\f,\g]_3^{*(2)}$. In Lemma \ref{lem:PF_coh} we showed that $\PF[\f,\g]$ is $3$-coherent, hence there exists a $4$-cell $\alpha: A \qfl B \in \PF[\f,\g]_4^{*(2)} \subset \TPN[\f,\g]_4^{*(2)}$.

Moreover, the only $1$-cells of weight $1$ and of the form \eqref{eq:but_ok} are the $1$-cells ${}_a \twocell{1} {}_{\f(a)} \twocell{1} {}_b$, which proves the second part of the Proposition.
\end{proof}

\subsection{A sub-polygraph of $\TPN[\f,\g]$ satifying $2$-Squier condition of depth $2$}

\label{subsec:tpnat_coh_final}

\begin{defn}
Let $\TPN^{++}[\f,\g]$ be the sub-$(4,2)$-polygraph of $\TPN[\f,\g]$ containing every product cell from Table \ref{table:Liste_cellule}.
\end{defn}

\begin{lem}\label{lem:verif_thm_hypothese}
The $(4,2)$-polygraph $\TPN^{++}[\f,\g]$ satisfies the $2$-Squier condition of depth $2$. 
\end{lem}
\begin{proof}
\subsubsection*{The $2$-Squier condition}

Let us start by showing the $2$-termination of the $(4,2)$-polygraph $\TPN^{++}[\f,\g]$.

We define a functor  $\tau: \TPN[\f,\g]_1^* \to \mathbb N^3$, where compositions in $\mathbb N^3$ are given by component-wise addition, by defining:
\begin{itemize}
\item For all $a,b \in \CC$, $\tau({}_a \twocell{1} {}_b) = (1,0,0)$.
\item For all $a \in \CC$, $\tau({}_a \twocell{1} {}_{\f(a)}) = (0,1,0)$.
\item For all $a \in \CC$, $\tau({}_a \twocell{1} {}_{\g(a)}) = (0,2,0)$.
\item For all $a,b \in \DD$, $\tau({}_a \twocell{1} {}_b) = (0,0,1)$.
\end{itemize}
The lexicographic order on $\mathbb N^3$ induces a noetherian ordering on  $\TPN[\f,\g]^*_1$. Moreover the $2$-cells are indeed decreasing for this order:

\[
 \tau(\s(\twocell{prodC})) = (2,0,0) > (1,0,0) = \tau(\t(\twocell{prodC}))
\qquad
 \tau(\s(\twocell{prodD})) = (0,0,2) > (0,0,1) = \tau(\t(\twocell{prodD}))
\]

\[
 \tau(\s(\twocell{fonctF})) = (1,1,0) > (0,1,1) = \tau(\t(\twocell{fonctF}))
\qquad
 \tau(\s(\twocell{fonctG})) = (1,2,0) > (0,2,1) = \tau(\t(\twocell{fonctG}))
\]

\[
 \tau(\s(\twocell{transfo}))= (0,2,0) > (0,1,1) = \tau(\t(\twocell{transfo}))
\]

The following diagrams show both the $2$-confluence of $\TPN^{++}[\f,\g]$ and the correspondence between critical pairs and $3$-cells:

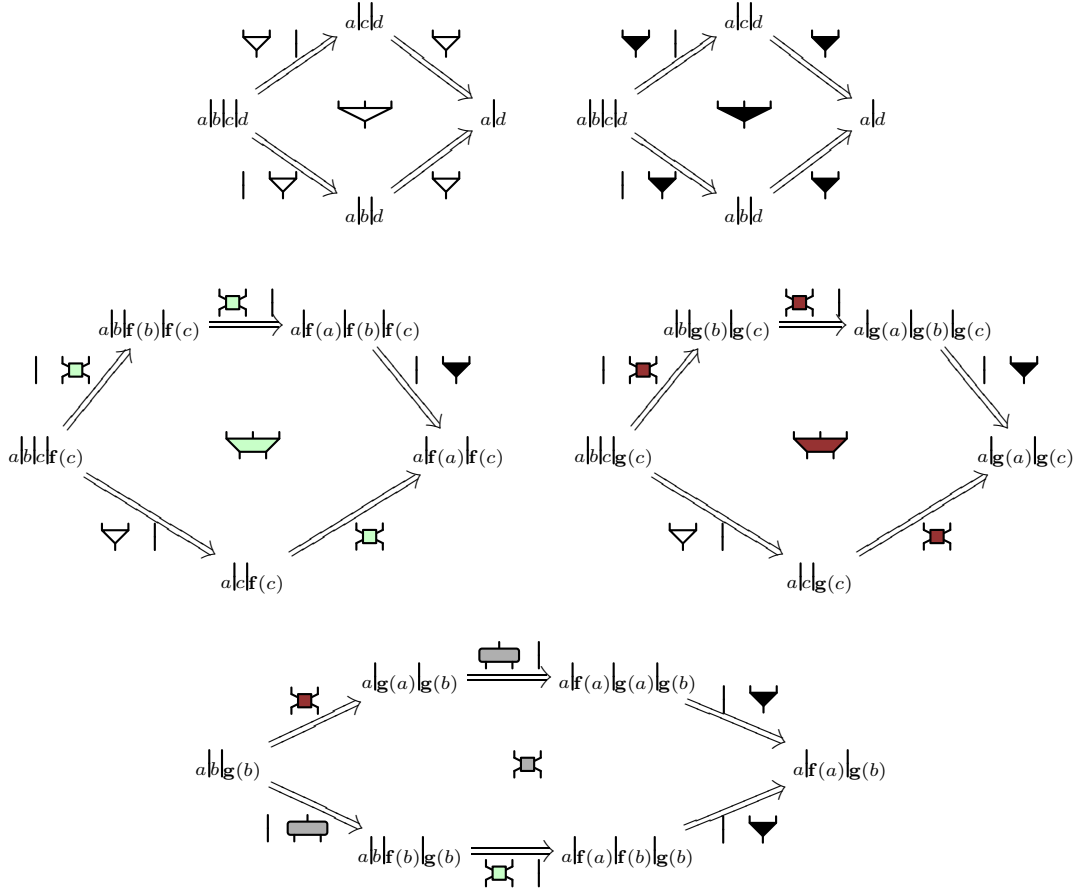
\[
\xymatrix @C = 3em @R = 2em {
& 
{}_a \twocell{1} {}_c \twocell{1} {}_d 
\ar@2 [rd] ^{\twocell{prodC}}
&
\\
{}_a \twocell{1} {}_b \twocell{1} {}_c \twocell{1} {}_d
\ar@2 [ru] ^{\twocell{prodC *0 1}}
\ar@2 [rd] _{\twocell{1 *0 prodC}} 
\ar@{} [rr] |-{\twocell{assocC}}
& &
{}_a \twocell{1} {}_d
\\
&
{}_a \twocell{1} {}_b \twocell{1} {}_d
\ar@2 [ru] _{\twocell{prodC}}
&
}
\qquad
\xymatrix @C = 3em @R = 2em {
& 
{}_a \twocell{1} {}_c \twocell{1} {}_d 
\ar@2 [rd] ^{\twocell{prodD}}
&
\\
{}_a \twocell{1}{}_ b \twocell{1} {}_c \twocell{1} {}_d
\ar@2 [ru] ^{\twocell{prodD *0 1}}
\ar@2 [rd] _{\twocell{1 *0 prodD}} 
\ar@{} [rr] |-{\twocell{assocD}}
& &
{}_a \twocell{1} {}_d
\\
&
{}_a \twocell{1} {}_b \twocell{1} {}_d
\ar@2 [ru] _{\twocell{prodD}}
&
}
\]

\[
\xymatrix @!C = 1.5em @R = 3em {
&
{}_a \twocell{1} {}_b \twocell{1} {}_{\f(b)} \twocell{1} {}_{\f(c)}
\ar@2 [rr] ^-{\twocell{fonctF *0 1}}
& 
\ar@{} [dd] |-{\twocell{img_prodF}}
&
{}_a \twocell{1} {}_{\f(a)} \twocell{1} {}_{\f(b)} \twocell{1} {}_{\f(c)}
\ar@2 [rd] ^-{\twocell{1 *0 prodD}}
&
\\
{}_a \twocell{1} {}_b \twocell{1} {}_c \twocell{1} {}_{\f(c)}
\ar@2 [ru] ^-{\twocell{1 *0 fonctF}}
\ar@2 [rrd] _-{\twocell{prodC *0 1}}
& & & &
{}_a \twocell{1} {}_{\f(a)} \twocell{1} {}_{\f(c)}
\\
& &
{}_a \twocell{1} {}_c \twocell{1} {}_{\f(c)}
\ar@2 [rru] _{\twocell{fonctF}}
}
\qquad
\xymatrix @!C = 1.5em @R = 3em {
&
{}_a \twocell{1} {}_b \twocell{1} {}_{\g(b)} \twocell{1} {}_{\g(c)}
\ar@2 [rr] ^-{\twocell{fonctG *0 1}}
& 
\ar@{} [dd] |-{\twocell{img_prodG}}
&
{}_a \twocell{1} {}_{\g(a)} \twocell{1} {}_{\g(b)} \twocell{1} {}_{\g(c)}
\ar@2 [rd] ^-{\twocell{1 *0 prodD}}
&
\\
{}_a \twocell{1} {}_b \twocell{1} {}_c \twocell{1} {}_{\g(c)}
\ar@2 [ru] ^-{\twocell{1 *0 fonctG}}
\ar@2 [rrd] _-{\twocell{prodC *0 1}}
& & & &
{}_a \twocell{1} {}_{\g(a)} \twocell{1} {}_{\g(c)}
\\
& &
{}_a \twocell{1} {}_c \twocell{1} {}_{\g(c)}
\ar@2 [rru] _{\twocell{fonctG}}
}
\]

\[
\xymatrix @C = 3em @R = 1.5em {
&
{}_a \twocell{1} {}_{\g(a)} \twocell{1} {}_{\g(b)}
\ar@2 [r] ^-{\twocell{transfo *0 1}}
&
{}_a \twocell{1} {}_{\f(a)} \twocell{1} {}_{\g(a)} \twocell{1} {}_{\g(b)}
\ar@2 [rd] ^-{\twocell{1 *0 prodD}}
&
\\
{}_a \twocell{1} {}_b \twocell{1} {}_{\g(b)}
\ar@2 [ru] ^-{\twocell{fonctG}}
\ar@2 [rd] _-{\twocell{1 *0 transfo}}
\ar@{} [rrr] |-{\twocell{transfo_nat}}
& & &
{}_a \twocell{1} {}_{\f(a)} \twocell{1} {}_{\g(b)}
\\
&
{}_a \twocell{1} {}_b \twocell{1} {}_{\f(b)} \twocell{1} {}_{\g(b)} 
\ar@2 [r] _-{\twocell{fonctF *0 1}}
&
{}_a \twocell{1} {}_{\f(a)} \twocell{1} {}_{\f(b)} \twocell{1} {}_{\g(b)} 
\ar@2 [ru] _-{\twocell{1 *0 prodD}}
&
}
\]

\subsubsection*{The $2$-Squier condition of depth $2$}

The following diagrams show the bijection between critical triples and $4$-cells.

\[
\xymatrix @C = 2.5em @R = 3em {
&
{}_a \twocell{1} {}_c \twocell{1} {}_d \twocell{1} {}_e
\ar@2 [rr] ^-{\twocell{prodC *0 1}}
\ar@{} [rd] |-{\twocell{assocC *0 1}}
& &
{}_a \twocell{1} {}_d \twocell{1} {}_e
\ar@2 [rd] ^-{\twocell{prodC}}
\ar@{} [dd] |-{\twocell{assocC}}
& & &
{}_a \twocell{1} {}_c \twocell{1} {}_d \twocell{1} {}_e
\ar@2 [rr] ^-{\twocell{prodC *0 1}}
\ar@{} [rrrd] |-{\twocell{assocC}}
\ar@2 [rd] |-{\twocell{1 *0 prodC}}
\ar@{} [dd] |-{=}
& &
{}_a \twocell{1} d \twocell{1} {}_e
\ar@2 [rd] ^-{\twocell{prodC}}
&
\\
{}_a \twocell{1} {}_b \twocell{1} {}_c \twocell{1} {}_d \twocell{1} {}_e
\ar@2 [ru] ^-{\twocell{prodC *0 2}}
\ar@2 [rr] |-{\twocell{1 *0 prodC *0 1}}
\ar@2 [rd] _-{\twocell{2 *0 prodC}}
& &
{}_a \twocell{1} {}_b \twocell{1} {}_d \twocell{1} {}_e
\ar@{} [ld] |-{\twocell{1 *0 assocC}}
\ar@2 [ru] |-{\twocell{prodC *0 1}}
\ar@2 [rd] |-{\twocell{1 *0 prodC}}
&
&
{}_a \twocell{1} {}_e
\ar@4 [r] ^-{\twocell{pentaC}}
&
{}_a \twocell{1} {}_b \twocell{1} {}_c \twocell{1} {}_d \twocell{1} {}_e
\ar@2 [ru] ^-{\twocell{prodC *0 2}}
\ar@2 [rd] _-{\twocell{2 *0 prodC}}
& & 
{}_a  \twocell{1} {}_c \twocell{1} {}_e
\ar@2 [rr] |-{\twocell{prodC}}
\ar@{} [rd] |-{\twocell{assocC}}
& &
{}_a  \twocell{1} {}_e
\\
&
{}_a  \twocell{1} {}_b  \twocell{1} {}_c  \twocell{1} {}_e
\ar@2 [rr] _-{\twocell{1 *0 prodC}}
& &
{}_a  \twocell{1} {}_b  \twocell{1} {}_e
\ar@2 [ru] _-{\twocell{prodC}}
& & &
{}_a  \twocell{1} {}_b  \twocell{1} {}_c  \twocell{1} {}_e
\ar@2 [rr] _-{\twocell{1 *0 prodC}}
\ar@2 [ru] |-{\twocell{prodC *0 1}}
& &
{}_a  \twocell{1} {}_b  \twocell{1} {}_e
\ar@2 [ru] _-{\twocell{prodC}}
&
}
\]

\[
\xymatrix @C = 2.5em @R = 3em {
&
{}_a \twocell{1} {}_c \twocell{1} {}_d \twocell{1} {}_e
\ar@2 [rr] ^-{\twocell{prodD *0 1}}
\ar@{} [rd] |-{\twocell{assocD *0 1}}
& &
{}_a \twocell{1} {}_d \twocell{1} {}_e
\ar@2 [rd] ^-{\twocell{prodD}}
\ar@{} [dd] |-{\twocell{assocD}}
& & &
{}_a \twocell{1} {}_c \twocell{1} {}_d \twocell{1} {}_e
\ar@2 [rr] ^-{\twocell{prodD *0 1}}
\ar@{} [rrrd] |-{\twocell{assocD}}
\ar@2 [rd] |-{\twocell{1 *0 prodD}}
\ar@{} [dd] |-{=}
& &
{}_a \twocell{1} {}_d \twocell{1} {}_e
\ar@2 [rd] ^-{\twocell{prodD}}
&
\\
{}_a \twocell{1} {}_b \twocell{1} {}_c \twocell{1} {}_d \twocell{1} {}_e
\ar@2 [ru] ^-{\twocell{prodD *0 2}}
\ar@2 [rr] |-{\twocell{1 *0 prodD *0 1}}
\ar@2 [rd] _-{\twocell{2 *0 prodD}}
& &
{}_a \twocell{1} {}_b \twocell{1} {}_d \twocell{1} {}_e
\ar@{} [ld] |-{\twocell{1 *0 assocD}}
\ar@2 [ru] |-{\twocell{prodD *0 1}}
\ar@2 [rd] |-{\twocell{1 *0 prodD}}
&
&
{}_a \twocell{1} {}_e
\ar@4 [r] ^-{\twocell{pentaD}}
&
{}_a \twocell{1} {}_b \twocell{1} {}_c \twocell{1} {}_d \twocell{1} {}_e
\ar@2 [ru] ^-{\twocell{prodD *0 2}}
\ar@2 [rd] _-{\twocell{2 *0 prodD}}
& & 
{}_a  \twocell{1} {}_c \twocell{1} {}_e
\ar@2 [rr] |-{\twocell{prodD}}
\ar@{} [rd] |-{\twocell{assocD}}
& &
{}_a  \twocell{1} {}_e
\\
&
{}_a  \twocell{1} {}_b  \twocell{1} {}_c  \twocell{1} {}_e
\ar@2 [rr] _-{\twocell{1 *0 prodD}}
& &
{}_a  \twocell{1} {}_b  \twocell{1} {}_e
\ar@2 [ru] _-{\twocell{prodD}}
& & &
{}_a  \twocell{1} {}_b  \twocell{1} {}_c  \twocell{1} {}_e
\ar@2 [rr] _-{\twocell{1 *0 prodD}}
\ar@2 [ru] |-{\twocell{prodD *0 1}}
& &
{}_a  \twocell{1} {}_b  \twocell{1} {}_e
\ar@2 [ru] _-{\twocell{prodD}}
&
}
\]

\[
\xymatrix @C = 2em @R = 2.4em {
&
{}_a \twocell{1} {}_c \twocell{1} {}_d \twocell{1} {}_{\f(d)}
  \ar@2 [rr] ^{\twocell{prodC *0 1}}
&&
{}_a \twocell{1} {}_d \twocell{1} {}_{\f(d)}
  \ar@2 @/^/ [rrd] ^{\twocell{fonctF}}
&&&
\\ 
{}_a \twocell{1} {}_b \twocell{1} {}_c \twocell{1} {}_d \twocell{1} {}_{\f(d)}
  \ar@2 [ur] ^{\twocell{prodC *0 2}}
  \ar@2 [rr] ^{\twocell{1 *0 prodC *0 1}}
  \ar@2 [d] _{\twocell{2 *0 fonctF}}
  \ar@{} [urrr] |-{\twocell{assocC *0 1}}
  \ar@{} [ddrrr]|-{\twocell{1 *0 img_prodF}}
&&
{}_a \twocell{1} {}_b \twocell{1} {}_d \twocell{1} {}_{\f(d)}
  \ar@2 [ru] ^{\twocell{prodC *0 1}}
  \ar@2 @/^/ [rdd] ^{\twocell{1 *0 fonctF}}
  \ar@{} [rrr] |-{\twocell{img_prodF}}
&&&
{}_a \twocell{1} {}_{\f(a)} \twocell{1} {}_{\f(d)}
\\ 
{}_a \twocell{1} {}_b \twocell{1} {}_c \twocell{1} {}_{\f(c)} \twocell{1} {}_{\f(d)}
  \ar@2 [rd] _{\twocell{1 *0 fonctF *0 1}}
&&&&&
{}_a \twocell{1} {}_{\f(a)} \twocell{1} {}_{\f(b)} \twocell{1} {}_{\f(d)}
  \ar@2 [u] _{\twocell{1 *0 prodD}}
&
\\ 
&
{}_a \twocell{1} {}_b \twocell{1} {}_{\f(b)} \twocell{1} {}_{\f(c)} \twocell{1} {}_{\f(d)}
  \ar@2 [rr] ^{\twocell{2 *0 prodD}}
&
\ar@4 []!<0pt,-10pt>;[d]!<0pt,20pt> ^{\twocell{img_assocF}}
&
{}_a \twocell{1} {}_b \twocell{1} {}_{\f(b)} \twocell{1} {}_{\f(d)}
  \ar@2 [rru] _{\twocell{fonctF *0 1}}
&&&
\\
&
{}_a \twocell{1} {}_c \twocell{1} {}_d \twocell{1} {}_{\f(d)}
  \ar@2 [rr] ^{\twocell{prodC *0 1}}
  \ar@2 [rd] _{\twocell{1 *0 fonctF}}
  \ar@{} [rrrd] ^-{\twocell{img_prodF}}
&&
{}_a \twocell{1} {}_d \twocell{1} {}_{\f(d)}
  \ar@2 @/^/ [rrd] ^{\twocell{fonctF}}
&&& \\ 
{}_a \twocell{1} {}_b \twocell{1} {}_c \twocell{1} {}_d \twocell{1} {}_{\f(d)}
  \ar@2 [ur] ^{\twocell{prodC *0 2}}
  \ar@2 [d] _{\twocell{2 *0 fonctF}}
  \ar@{} [rr] |-{=}
&&
{}_a \twocell{1} {}_c \twocell{1} {}_{\f(c)} \twocell{1} {}_{\f(d)}
  \ar@2 [r] ^{\twocell{fonctF *0 1}}
&
{}_a \twocell{1} {}_{\f(a)} \twocell{1} {}_{\f(c)} \twocell{1} {}_{\f(d)}
  \ar@2 [rr] ^{\twocell{1 *0 prodD}}
&&
{}_a \twocell{1} {}_{\f(a)} \twocell{1} {}_{\f(d)} 
\\ 
{}_a \twocell{1} {}_b \twocell{1} {}_c \twocell{1} {}_{\f(c)} \twocell{1} {}_{\f(d)}
  \ar@2 [rd] _{\twocell{1 *0 fonctF *0 1}}
  \ar@2 [rru] ^<<<<<{\twocell{prodC *0 2}}
&&
{}_a \twocell{1} {}_{\f(a)} \twocell{1} {}_{\f(b)} \twocell{1} {}_{\f(c)} \twocell{1} {}_{\f(d)}
  \ar@2 [ur]  _{\twocell{1 *0 prodD *0 1}}
  \ar@2 [rrr] ^{\twocell{2 *0 prodD}}
  \ar@{} [rrru] |-{\twocell{assocD}}
&&&
{}_a \twocell{1} {}_{\f(a)} \twocell{1} {}_{\f(d)}
  \ar@2 [u] _{\twocell{1 *0 prodD}}
& \\ 
  \ar@{} [rrruu] ^-{\twocell{img_prodF *0 1}}
&
{}_a \twocell{1} {}_b \twocell{1} {}_{\f(b)} \twocell{1} {}_{\f(c)} \twocell{1} {}_{\f(d)}
  \ar@2 [rr] _{\twocell{2 *0 prodD}}
  \ar@2 [ru] _{\twocell{fonctF *0 2}}
  \ar@{} [rrru] |-{=}
&&
{}_a \twocell{1} {}_b \twocell{1} {}_{\f(b)} \twocell{1} {}_{\f(d)}
  \ar@2 [urr] _{\twocell{fonctF *0 1}}
&&&
}
\]
\[
\xymatrix @C = 2em @R = 2.4em {
&
{}_a \twocell{1} {}_c \twocell{1} {}_d \twocell{1} {}_{\g(d)}
  \ar@2 [rr] ^{\twocell{prodC *0 1}}
&&
{}_a \twocell{1} {}_d \twocell{1} {}_{\g(d)}
  \ar@2 @/^/ [rrd] ^{\twocell{fonctG}}
&&&
\\ 
{}_a \twocell{1} {}_b \twocell{1} {}_c \twocell{1} {}_d \twocell{1} {}_{\g(d)}
  \ar@2 [ur] ^{\twocell{prodC *0 2}}
  \ar@2 [rr] ^{\twocell{1 *0 prodC *0 1}}
  \ar@2 [d] _{\twocell{2 *0 fonctG}}
  \ar@{} [urrr] |-{\twocell{assocC *0 1}}
  \ar@{} [ddrrr]|-{\twocell{1 *0 img_prodG}}
&&
{}_a \twocell{1} {}_b \twocell{1} {}_d \twocell{1} {}_{\g(d)}
  \ar@2 [ru] ^{\twocell{prodC *0 1}}
  \ar@2 @/^/ [rdd] ^{\twocell{1 *0 fonctG}}
  \ar@{} [rrr] |-{\twocell{img_prodG}}
&&&
{}_a \twocell{1} {}_{\g(a)} \twocell{1} {}_{\g(d)}
\\ 
{}_a \twocell{1} {}_b \twocell{1} {}_c \twocell{1} {}_{\g(c)} \twocell{1} {}_{\g(d)}
  \ar@2 [rd] _{\twocell{1 *0 fonctG *0 1}}
&&&&&
{}_a \twocell{1} {}_{\g(a)} \twocell{1} {}_{\g(b)} \twocell{1} {}_{\g(d)}
  \ar@2 [u] _{\twocell{1 *0 prodD}}
&
\\ 
&
{}_a \twocell{1} {}_b \twocell{1} {}_{\g(b)} \twocell{1} {}_{\g(c)} \twocell{1} {}_{\g(d)}
  \ar@2 [rr] ^{\twocell{2 *0 prodD}}
&
\ar@4 []!<0pt,-10pt>;[d]!<0pt,20pt> ^{\twocell{img_assocG}}
&
{}_a \twocell{1} {}_b \twocell{1} {}_{\g(b)} \twocell{1} {}_{\g(d)}
  \ar@2 [rru] _{\twocell{fonctG *0 1}}
&&&
\\
&
{}_a \twocell{1} {}_c \twocell{1} {}_d \twocell{1} {}_{\g(d)}
  \ar@2 [rr] ^{\twocell{prodC *0 1}}
  \ar@2 [rd] _{\twocell{1 *0 fonctG}}
  \ar@{} [rrrd] ^-{\twocell{img_prodG}}
&&
{}_a \twocell{1} {}_d \twocell{1} {}_{\g(d)}
  \ar@2 @/^/ [rrd] ^{\twocell{fonctG}}
&&& \\ 
{}_a \twocell{1} {}_b \twocell{1} {}_c \twocell{1} {}_d \twocell{1} {}_{\g(d)}
  \ar@2 [ur] ^{\twocell{prodC *0 2}}
  \ar@2 [d] _{\twocell{2 *0 fonctG}}
  \ar@{} [rr] |-{=}
&&
{}_a \twocell{1} {}_c \twocell{1} {}_{\g(c)} \twocell{1} {}_{\g(d)}
  \ar@2 [r] ^{\twocell{fonctG *0 1}}
&
{}_a \twocell{1} {}_{\g(a)} \twocell{1} {}_{\g(c)} \twocell{1} {}_{\g(d)}
  \ar@2 [rr] ^{\twocell{1 *0 prodD}}
&&
{}_a \twocell{1} {}_{\g(a)} \twocell{1} {}_{\g(d)} 
\\ 
{}_a \twocell{1} {}_b \twocell{1} {}_c \twocell{1} {}_{\g(c)} \twocell{1} {}_{\g(d)}
  \ar@2 [rd] _{\twocell{1 *0 fonctG *0 1}}
  \ar@2 [rru] ^<<<<<{\twocell{prodC *0 2}}
&&
{}_a \twocell{1} {}_{\g(a)} \twocell{1} {}_{\g(b)} \twocell{1} {}_{\g(c)} \twocell{1} {}_{\g(d)}
  \ar@2 [ur]  _{\twocell{1 *0 prodD *0 1}}
  \ar@2 [rrr] ^{\twocell{2 *0 prodD}}
  \ar@{} [rrru] |-{\twocell{assocD}}
&&&
{}_a \twocell{1} {}_{\g(a)} \twocell{1} {}_{\g(d)}
  \ar@2 [u] _{\twocell{1 *0 prodD}}
& \\ 
  \ar@{} [rrruu] ^-{\twocell{img_prodG *0 1}}
&
{}_a \twocell{1} {}_b \twocell{1} {}_{\g(b)} \twocell{1} {}_{\g(c)} \twocell{1} {}_{\g(d)}
  \ar@2 [rr] _{\twocell{2 *0 prodD}}
  \ar@2 [ru] _{\twocell{fonctG *0 2}}
  \ar@{} [rrru] |-{=}
&&
{}_a \twocell{1} {}_b \twocell{1} {}_{\g(b)} \twocell{1} {}_{\g(d)}
  \ar@2 [urr] _{\twocell{fonctG *0 1}}
&&&
}
\]

\[
\xymatrix @C = 2em @R = 2.5em {
&
{}_a \twocell{1} {}_c \twocell{1} {}_{\g(c)} 
  \ar@2 [r] ^-{\twocell{fonctG}}
  \ar@{} [dd] |-{\twocell{img_prodG}}
&
{}_a \twocell{1} {}_{\g(a)} \twocell{1} {}_{\g(c)} 
  \ar@2 [r] ^-{\twocell{transfo *0 1}}
  \ar@{} [rd] |-{=}
&
{}_a \twocell{1} {}_{\f(a)} \twocell{1} {}_{\g(a)} \twocell{1} {}_{\g(c)} 
  \ar@2 [rr] ^-{\twocell{1 *0 prodD}}
& &
{}_a \twocell{1} {}_{\f(a)} \twocell{1} {}_{\g(c)}
\\ 
&&
{}_a \twocell{1} {}_{\g(a)} \twocell{1} {}_{\g(b)} \twocell{1} {}_{\g(c)} 
  \ar@2 [u] ^{\twocell{1 *0 prodD}}
  \ar@2 [r] _-{\twocell{transfo *0 2}}
&
{}_a \twocell{1} {}_{\f(a)} \twocell{1} {}_{\g(a)} \twocell{1} {}_{\g(b)} \twocell{1} {}_{\g(c)} 
  \ar@2 [u] _{\twocell{2 *0 prodD}}
  \ar@2 [rd] _{\twocell{1 *0 prodD *0 1}}
&& \\ 
{}_a \twocell{1} {}_b \twocell{1} {}_c \twocell{1} {}_{\g(c)} 
  \ar@2 [ruu] ^{\twocell{prodC *0 1}}
  \ar@2 [r] ^{\twocell{1 *0 fonctG}}
  \ar@2 [rdd] _{\twocell{2 *0 transfo}}
&
{}_a \twocell{1} {}_b \twocell{1} {}_{\g(b)} \twocell{1} {}_{\g(c)} 
  \ar@2 [ru] ^{\twocell{fonctG *0 1}}
  \ar@2 [rd] _{\twocell{1 *0 transfo *0 1}}
  \ar@{} [rrr] |-{\twocell{transfo_nat}}
&&& 
{}_a \twocell{1} {}_{\f(a)} \twocell{1} {}_{\g(b)} \twocell{1} {}_{\g(c)} 
  \ar@2 [uur] _{\twocell{1 *0 prodD}}
  \ar@{} [rd] |-{\twocell{assocD}}
  \ar@{} [uu] |-{\twocell{assocD}}
&
\\ 
&&
{}_a \twocell{1} {}_b \twocell{1} {}_{\f(b)} \twocell{1} {}_{\g(b)} \twocell{1} {}_{\g(c)} 
  \ar@2 [r] ^-{\twocell{fonctF *0 2}}
  \ar@2 [rd] _{\twocell{2 *0 prodD}}
&
{}_a \twocell{1} {}_{\f(a)} \twocell{1} {}_{\f(b)} \twocell{1} {}_{\g(b)} \twocell{1} {}_{\g(c)} 
  \ar@2 [ru] ^{\twocell{1 *0 prodD *0 1}}
  \ar@2 @/_/ [rr] _{\twocell{2 *0 prodD}}
& &
{}_a \twocell{1} {}_{\f(a)} \twocell{1} {}_{\f(b)} \twocell{1} {}_{\g(c)}
  \ar@2 [uuu] _{\twocell{1 *0 prodD}} 
\\ 
&
{}_a \twocell{1} {}_b \twocell{1} {}_c \twocell{1} {}_{\f(c)} \twocell{1} {}_{\g(c)} 
  \ar@2 [r] _-{\twocell{1 *0 fonctF *0 1}}
  \ar@{} [ru] |-{\twocell{1 *0 transfo_nat}}
&
{}_a \twocell{1} {}_b \twocell{1} {}_{\f(b)} \twocell{1} {}_{\f(c)} \twocell{1} {}_{\g(c)} 
  \ar@2 [r] _-{\twocell{2 *0 prodD}}
&
{}_a \twocell{1} {}_b \twocell{1} {}_{\f(b)} \twocell{1} {}_{\g(c)} 
  \ar@2 @/_5ex/ [rru] _{\twocell{fonctF *0 1}}
  \ar@{} [ru] _{=}
& \\ 
&&&
\ar@4 [d] ^{\twocell{transfo_prod}}
&&\\
&&&&&\\
&
{}_a \twocell{1} {}_c \twocell{1} {}_{\g(c)} 
  \ar@2 [r] ^-{\twocell{fonctG}}
  \ar@2 [ddr] _{\twocell{1 *0 transfo}}
  \ar@{} [dddd] |-{=}
&
{}_a \twocell{1} {}_{\g(a)} \twocell{1} {}_{\g(c)} 
  \ar@2 [r] ^-{\twocell{transfo *0 1}}
&
{}_a \twocell{1} {}_{\f(a)} \twocell{1} {}_{\g(a)} \twocell{1} {}_{\g(c)} 
  \ar@2 [rr] ^-{\twocell{1 *0 prodD}}
& &
{}_a \twocell{1} {}_{\f(a)} \twocell{1} {}_{\g(c)}
\\ 
&&&&
{}_a \twocell{1} {}_{\f(a)} \twocell{1} {}_{\f(c)} \twocell{1} {}_{\g(c)}
  \ar@2 [ru] ^{\twocell{1 *0 prodD}}
& \\ 
{}_a \twocell{1} {}_b \twocell{1} {}_c \twocell{1} {}_{\g(c)} 
  \ar@2 [ruu] ^{\twocell{prodC *0 1}}
  \ar@2 [rdd] _{\twocell{2 *0 transfo}}
&
&
{}_a \twocell{1} {}_c \twocell{1} {}_{\f(c)} \twocell{1} {}_{\g(c)}
  \ar@2 [rur] ^{\twocell{fonctF *0 1}}
  \ar@{} [ruu] ^-{\twocell{transfo_nat}}
&&&
\\ 
&&
&&
{}_a \twocell{1} {}_{\f(a)} \twocell{1} {}_{\f(b)} \twocell{1} {}_{\f(c)} \twocell{1} {}_{\g(c)} 
  \ar@2 [uu] ^{\twocell{1 *0 prodD *0 1}}
  \ar@2 [r] ^{\twocell{2 *0 prodD}}
  \ar@{} [ruu] |-{\twocell{assocD}}
& 
{}_a \twocell{1} {}_{\f(a)} \twocell{1} {}_{\g(b)} \twocell{1} {}_{\g(c)} 
  \ar@2 [uuu] _{\twocell{1 *0 prodD}}
\\ 
&
{}_a \twocell{1} {}_b \twocell{1} {}_c \twocell{1} {}_{\f(c)} \twocell{1} {}_{\g(c)} 
  \ar@2 [r] _-{\twocell{1 *0 fonctF *0 1}}
  \ar@2 [uur] ^{\twocell{prodC *0 2}}
  \ar@{} [uuurrr] |-{\twocell{img_prodF}}
&
{}_a \twocell{1} {}_b \twocell{1} {}_{\f(b)} \twocell{1} {}_{\f(c)} \twocell{1} {}_{\g(c)} 
  \ar@2 [r] _-{\twocell{2 *0 prodD}}
  \ar@2 [rru] ^{\twocell{fonctF *0 2}}
  \ar@{} [rrru] _-{=}
&
{}_a \twocell{1} {}_b \twocell{1} {}_{\f(b)} \twocell{1} {}_{\g(c)} 
  \ar@2 @/_5ex/ [rru] _{\twocell{fonctF *0 1}}
&
}
\]
\end{proof}

\begin{prop}\label{prop:coh_tpn++}
For every $3$-cells $A, B \in \TPN^{++}[\f,\g]^{*(2)}$ whose $1$-target is of the form ${}_a \twocell{1} {}_{\f(a)} \twocell{1} {}_b$, there exists a $4$-cell $\alpha: A \qfl B  \in \TPN[\f,\g]^{*(2)}$.
\end{prop}
\begin{proof}
Thanks to Lemma \ref{lem:verif_thm_hypothese}, we can apply  Theorem \ref{thm:main_theory} to  $\TPN^{++}[\f,\g]^{*(2)}$, and there exists a $4$-cell $\alpha: A \qfl B$ in $\TPN^{++}[\f,\g]^{*(2)}$ for every $3$-cells $A,B \in \TPN^{++}[\f,\g]^{*(2)}$ whose $1$-target is a normal form. In particular the $1$-cells of the form ${}_a \twocell{1} {}_{\f(a)} \twocell{1} {}_b$ are normal forms. 
\end{proof}

\subsection{Adjunction of the units $2$-cells}
\label{subsec:2_unit_elim}

\begin{defn}
Let $\TPN^u[\f,\g]$ be the sub-$3$-polygraph of $\TPN[\f,\g]$ containing the same $1$- and $2$-cells, and whose only $3$-cells are the unit cells from Table \ref{table:Liste_cellule}.

A $2$-cell $h \in \TPN[\f,\g]^*_2$ is said \emph{unitary} if it is generated by the sub-$2$-polygraph of $\TPN[\f,\g]$ whose only $2$-cells are $\twocell{unitC}$ and $\twocell{unitD}$.
\end{defn}

\begin{lem}\label{lem:unit_normalisables}
Let $h \in \TPN[\f,\g]_2^*$ whose target is of the form ${}_a \twocell{1} {}_{\f(a)} \twocell{1} {}_b$, where $a \in \CC$ and $b \in \DD$. 

If there is a decomposition $h = h_1 \star_1 h_2$, where $h_1 \in \TPN^u[\f,\g]^*$ and $h_2 \in \TPN[\f,\g]^*$ are not identities, and $h_1$ is a unitary $2$-cell, then there is a $3$-cell $A \in \TPN^u[\f,\g]^*_3$ of source $h$ which is not an identity.
\end{lem}

\begin{proof}
Let us start with the case where $h_1$ is of length $1$. We reason by induction on the length of $h_2$. If $h_2$ is of length $1$, since the target of $h_2$ is of the form ${}_a \twocell{1} {}_{\f(a)} \twocell{1} {}_b$, $h_2$ is one of the following $2$-cells:
\[
\twocell{1 *0 prodD} \qquad \twocell{fonctF} \qquad \twocell{transfo}
\]
Hence $h$ is one of the following $2$-cells:
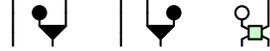
\[
\twocell{(1 *0 unitD *0 1) *1 (1 *0 prodD)} \qquad
\twocell{(2 *0 unitD) *1 (1 *0 prodD)} \qquad
\twocell{ (unitC *0 1) *1 fonctF}
\]
And all of these $2$-cells are indeed the sources of $3$-cells in $\TPN^u[\f,\g]^*_3$.

In the general case, let us write $h_2 = h_0 \star_1 h'_2$, where $h_0$ is of length $1$. Two cases can occur.
\begin{itemize}
\item If there exist $1$-cells $u, u', v$ and $v'$ and $2$-cells $h'_0:u \Rightarrow u' \in \TPN^u[\f,\g]^*$ and $h'_1: v \Rightarrow v' \in \TPN[\f,\g]^*$ such that $h_1 = h'_1 \star_0 u$ (resp. $h_1 = u \star_0 h'_1$) and $h_0 = v' \star_0 h'_0$ (resp. $h_0 = h'_0 \star_0 v'$).

Then $h = (h'_1 \star_0 h'_0) \star_1 h'_2$ (resp. $h = (h'_0 \star_0 h'_1 ) \star_1 h'_2$), and we can apply the induction hypothesis to $(h'_1 \star_0 u') \star_1 h'_2$ (resp. $( u' \star_0 h'_1 ) \star_1 h'_2$).

\item Otherwise, $h_1 \star_1 h_0$ is one of the following $2$-cells,
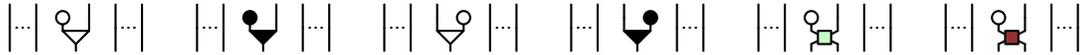
\[
\twocell{dots *0 ((unitC *0 1) *1 prodC) *0 dots} \qquad
\twocell{dots *0 ((unitD *0 1) *1 prodD) *0 dots} \qquad
\twocell{dots *0 ((1 *0 unitC) *1 prodC) *0 dots} \qquad
\twocell{dots *0 ((1 *0 unitD) *1 prodD) *0 dots} \qquad
\twocell{dots *0 ((unitC *0 1) *1 fonctF) *0 dots} \qquad
\twocell{dots *0 ((unitC *0 1) *1 fonctG) *0 dots}
\]
and all of them are sources of $3$-cells in $\TPN^u[\f,\g]^*$.
\end{itemize}

In the case general case where $h_1$ is of any length, let $h'_1,h''_1 \in \TPN[\f,\g]^*_2$ with $h''_1$ of length $1$ such that $h_1 = h'_1 \star_1 h''_1$. Then there is a non-empty $3$-cell $A' \in \TPN^u[\f,\g]_3^*$ of source $h''_1 \star_1 h_2$, and one can take the $3$-cell $h'_1 \star_1 A'$.
\end{proof}

\begin{lem}\label{lem:distinct_FN_TPN-}
Let $h$ be a $2$-cell in $\TPN[\f,\g]^*$ whose target is of the form ${}_a \twocell{1} {}_{\f(a)} \twocell{1} {}_b$, with $a \in \CC$ and $b \in \DD$. 

If $h$ is a normal form for $\TPN^u[\f,\g]$, then one of the following holds:
\begin{itemize}
\item The $2$-cell $h$ equals the composite  $\twocell{1 *0 unitD}$.
\item The $2$-cell $h$ is in $\TPN^{++}[\f,\g]^*$.
\end{itemize}
\end{lem}

\begin{proof}
We reason by induction on the length of $h$. If $h$ is of length $1$, the cells of $\TPN[\f,\g]_2^*$ of length $1$ and of target ${}_a \twocell{1} {}_{\f(a)} \twocell{1} {}_b$ are:
\[
\twocell{1 *0 prodD} \qquad
\twocell{transfo} \qquad
\twocell{fonctF} \qquad
\twocell{1 *0 unitD}
\]

Otherwise, let us write $h = h_1 \star_1 h_2$, where $h_1$ is of length $1$. We can apply the induction hypothesis to $h_2$, which leads us to distinguish three cases:
\begin{itemize}
\item If $h_2 = \twocell{1 *0 unitD}$, then $h_1$ is a $2$-cell in $\TPN[\f,\g]^*_2$ whose target is of the form ${}_a \twocell{1} {}_{\f(a)}$. The only such cell is the identity, and $h = h_2 =  \twocell{1 *0 unitD}$.
\item If $h_1$ and $h_2$ are in $\TPN^{++}[\f,\g]^*$, then $h$ is in $\TPN^{++}[\f,\g]^*$.
\item Lastly, if $h_2$ is in $\TPN^{++}[\f,\g]^*$ and $h_1$ is in $\TPN^u[\f,\g]^*$, then because of Lemma 
\ref{lem:unit_normalisables}, $h$ is the source of a $3$-cell in $\TPN^u[\f,\g]^*$ of length $1$, which is impossible since, by hypothesis, $h$ is a normal form for $\TPN^u[\f,\g]$.
\end{itemize}
\end{proof}

\begin{defn}
Let $\TPN^+[\f,\g]$ be the sub-$4$-polygraph of $\TPN[\f,\g]$ containing $\TPN^{++}[\f,\g]$, together with the $2$-cells $\twocell{unitC}$ and $\twocell{unitD}$.

In particular a $3$-cell in the free $(3,2)$-category $\TPN^+[\f,\g]^{*(2)}$ is in $\TPN^{++}[\f,\g]^{*(2)}$ if and only if its $2$-source is in $\TPN^{++}[\f,\g]^{*(2)}$ too.
\end{defn}

\begin{prop}\label{prop:restriction_TPN++}
For every parallel $3$-cells $A,B \in \TPN^{+}[\f,\g]^{*(2)}$ whose $1$-target is of the form ${}_a \twocell{1} {}_{\f(a)} \twocell{1} {}_b$ and whose $2$-source is a normal form for $\TPN^u[\f,\g]$, there exists a $4$-cell $\alpha: A \qfl B \in \TPN[\f,\g]^{*(2)}$.


\end{prop}
\begin{proof}

Given such $3$-cells $A$ and $B$, we use Lemma \ref{lem:distinct_FN_TPN-} to distinguish two cases:

If the source of $A$ and $B$ is $\twocell{1 *0 unitD}$, the only $3$-cell in $\TPN^{+}[\f,\g]^{*(2)}$ with source $\twocell{1 *0 unitD}$ is the identity. So $A=B$ and we can take $\alpha = 1_A$.

Otherwise, the source of $A$ and $B$ lies in $\TPN^{++}[\f,\g]_2^*$, so  $A$ and $B$  lie in $\TPN^{++}[\f,\g]^{*(2)}_3$. Proposition \ref{prop:coh_tpn++} allows us to conclude.
\end{proof}

\subsection{Adjunction of the units $3$-cells}
\label{subsec:3_unit_elim}

In this section, we consider the rewriting system formed by the $3$-cells of $\TPN^u[\f,\g]$. Since it is a sub-$3$-polygraph of $\TPN[\f,\g]$ (which $3$-terminates by Proposition \ref{prop:TPN_termine}), $\TPN^u[\f,\g]$ is $3$-terminating. The fact that it is $3$-confluent is a consequence of the following more general Lemma:

\begin{lem}\label{lem:norm_TPN}
Let $A \in \TPN[\f,\g]_3^{*}$ and $B \in \TPN^u[\f,\g]_3^*$.
There exist $3$-cells $A' \in \TPN[\f,\g]_3^{*}$ and $B' \in \TPN^u[\f,\g]^{*}$ and a $4$-cell $\alpha_{A,B} \in \TPN[\f,\g]_4^{*(2)}$ of the following shape:
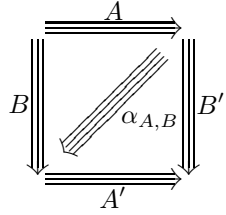
\[
\xymatrix @R=5em @C=5em{
\ar@3 [r] ^{A} 
\ar@3 [d] _{B} 
&
\ar@3 [d] ^{B'} 
\ar@4 []!<-10pt,-10pt>;[ld]!<10pt,10pt> ^{\alpha_{A,B}}
\\
\ar@3 [r] _{A'} &
}
\]
\end{lem}
\begin{proof}
Let us start by the case where $(A,B)$ is a critical pair of $\TPN[\f,\g]_3$. If $A$ and $B$ are in $\PF[\f,\g]^*_3$, the result holds because $\PF[\f,\g]$ is $3$-convergent. Otherwise, the only critical pair left is the following one:

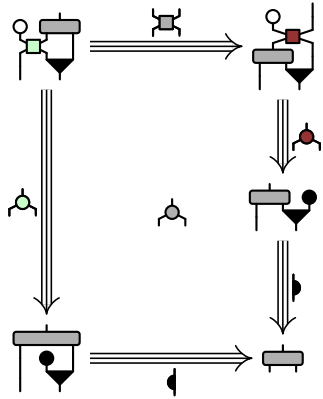
\[
\begin{tikzpicture}

\matrix (m) [matrix of math nodes, 
			nodes in empty cells,
			column sep = 2cm, 
			row sep = 1cm] 
{
\twocell{(unitC *0 transfo) *1 (fonctF *0 1) *1 (1 *0 prodD)}
& 
\twocell{(unitC *0 1) *1 fonctG *1 (transfo *0 1) *1 (1 *0 prodD)} 
\\
& 
\twocell{(transfo *0 unitD) *1 (1 *0 prodD)}
\\
\twocell{transfo *1 (1 *0 unitD *0 1) *1 (1 *0 prodD)}
& 
\twocell{transfo}
\\
};
\triplearrow{arrows = {-Implies}}
{(m-1-1) -- node [above] {$\twocell{transfo_nat}$} (m-1-2)}

\triplearrow{arrows = {-Implies}}
{(m-1-2) -- node [right] {$\twocell{img_unitG}$} (m-2-2)}

\triplearrow{arrows = {-Implies}}
{(m-2-2) -- node [right] {$\twocell{runitD}$} (m-3-2)}

\triplearrow{arrows = {-Implies}}
{(m-1-1) -- node [left] {$\twocell{img_unitF}$} (m-3-1)}

\triplearrow{arrows = {-Implies}}
{(m-3-1) -- node [below] {$\twocell{lunitD}$} (m-3-2)}

\path (m-1-1) -- node {$\twocell{transfo_unit}$} (m-3-2);
\end{tikzpicture}
\]

Let us now study the case where $(A,B)$ is a local branching of $\TPN[\f,\g]_3$. We distinguish three cases depending on the shape of the branching:
\begin{itemize}
\item If $(A,B)$ is an aspherical branching, then one can take identities for $A'$ and $B'$, and $\alpha=1_A$.
\item If $(A,B)$ is a Peiffer branching, let $A'$ and $B'$ be the canonical fillers of the confluence diagram of $(A,B)$, and $\alpha$ be an identity.
\item Lastly, if $(A,B)$ is an overlapping branching, let us write $(A,B) = (f \star_1 u A_1 v \star_1 g,f \star_1 u B_1 v \star_1 g)$, where $(A_1,B_1)$ is a critical pair. Let $A'_1$, $B'_1$ and $\alpha_1$ be the cells associated with $(A_1,B_1)$. We then define $A' := f \star_1 u A'_1 v \star_1 g$, $B' := f \star_1 u B'_1 v \star_1 g$ and $\alpha_1 := f \star_1 u \alpha_1 v \star_1 g$.
\end{itemize}

In the general case, we reason by noetherian induction on $h = \s(A) = \s(B)$, using the $3$-termination of $\TPN[\f,\g]$. 
\begin{itemize}
\item If $A$ or $B$ is an identity, then the result holds immediately.
\item Otherwise, we write $A = A_1 \star_2 A_2$ and $B = B_1 \star_2 B_2$, where $A_1$ and $B_1$ are of length $1$. We now build the following diagram:
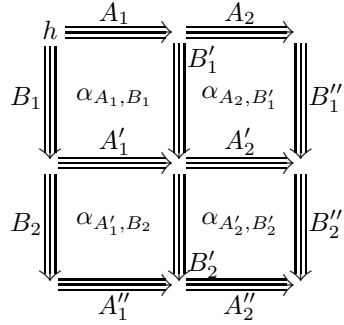
\[
\xymatrix @R=4em @C=4em{
h
\ar@3 [r] ^{A_1}
\ar@3 [d] _{B_1}
\ar@{} [rd] |{\alpha_{A_1,B_1}}
&
\ar@3 [r] ^{A_2}
\ar@3 [d] ^<<<{B'_1}
\ar@{} [rd] |{\alpha_{A_2,B'_1}}
&
\ar@3 [d] ^{B''_{1}}
\\
\ar@3 [r] ^{A'_1}
\ar@3 [d] _{B_2}
\ar@{} [rd] |{\alpha_{A'_1,B_2}}
&
\ar@3 [r] ^{A'_2}
\ar@3 [d] ^>>>{B'_2}
\ar@{} [rd] |{\alpha_{A'_2,B'_2}}
&
\ar@3 [d] ^{B''_{2}}
\\
\ar@3 [r] _{A''_1}
&
\ar@3 [r] _{A''_2}
&
}
\]
\end{itemize}
In this diagram, $\alpha_{A_1,B_1}$ is obtained thanks to our study of the local branchings. The existence of $\alpha_{A_2,B'_1}$ and $\alpha_{A'_1,B_2}$ (followed by $\alpha_{A'_2,B'_2}$) then follows from the induction hypothesis.
\end{proof}

\begin{lem}\label{lem:FN_stables}
Let $f,g$ be $2$-cells of $\TPN[\f,\g]^*$, and $A:f \Rrightarrow g$ a $3$-cell of $\TPN^+[\f,\g]^{*}$. If $f$ is a normal form for $\TPN^u[\f,\g]$, then so is $g$.
\end{lem}
\begin{proof}
We prove this result by contrapositive. We are going to show that for any $A \in \TPN^+[\f,\g]^*$ and $B \in \TPN^u[\f,\g]^*$ two $3$-cells of length $1$ such that $\t(A) = \s(B)$, there exists $B' \in \TPN^u[\f,\g]^*$ of length $1$ and of source $\s(A)$:
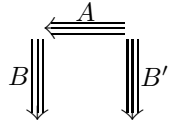
\[
\xymatrix @R=3em @C=3em {
\ar@3 [d] _{B} 
&
\ar@3 [l] _{A}
\ar@3 [d] ^{B'}
\\
&
}
\]

Two cases can occur depending on the shape of the branching $(A^{-1},B)$:
\begin{itemize}
\item If it is a Peiffer branching, then the required cell is provided by the canonical filling.
\item If it is an overlapping branching, then it is enough to check the underlying critical pair.
\end{itemize}
It remains to examine those critical pairs:

\[
\xymatrix @R=3em @C=3em{
\twocell{(1 *0 unitC *0 1) *1 (1 *0 prodC) *1 prodC}
\ar@3 [d] _{\twocell{lunitC}}
&
\twocell{(1 *0 unitC *0 1) *1 (prodC *0 1) *1 prodC}
\ar@3 [l] _{\twocell{assocC}}
\ar@3 [d] ^{\twocell{runitC}}
\\
\twocell{prodC}
& 
\twocell{prodC}
}
\qquad
\xymatrix @R=3em @C=3em{
\twocell{(1 *0 unitD *0 1) *1 (1 *0 prodD) *1 prodD}
\ar@3 [d] _{\twocell{lunitD}}
&
\twocell{(1 *0 unitD *0 1) *1 (prodD *0 1) *1 prodD}
\ar@3 [l] _{\twocell{assocD}}
\ar@3 [d] ^{\twocell{runitD}}
\\
\twocell{prodD}
& 
\twocell{prodD}
}
\]

\[
\xymatrix @R=3em @C=3em{
\twocell{(unitC *0 2) *1 (1 *0 prodC) *1 prodC}
\ar@3 [d] _{\twocell{lunitC}}
&
\twocell{(unitC *0 2) *1 (prodC *0 1) *1 prodC}
\ar@3 [d] ^{\twocell{lunitC}}
\ar@3 [l] _{\twocell{assocC}}
\\
\twocell{prodC}
& 
\twocell{prodC}
}
\qquad
\xymatrix @R=3em @C=3em{
\twocell{(unitD *0 2) *1 (1 *0 prodD) *1 prodD}
\ar@3 [d] _{\twocell{lunitD}}
&
\twocell{(unitD *0 2) *1 (prodD *0 1) *1 prodD}
\ar@3 [d] ^{\twocell{lunitD}}
\ar@3 [l] _{\twocell{assocD}}
\\
\twocell{prodD}
& 
\twocell{prodD}
}
\qquad
\xymatrix @R=3em @C=3em{
\twocell{(2 *0 unitC) *1 (1 *0 prodC) *1 prodC}
\ar@3 [d] _{\twocell{runitC}}
&
\twocell{(2 *0 unitC) *1 (prodC *0 1) *1 prodC}
\ar@3 [d] ^{\twocell{runitC}}
\ar@3 [l] _{\twocell{assocC}}
\\
\twocell{prodC}
& 
\twocell{prodC}
}
\qquad
\xymatrix @R=3em @C=3em{
\twocell{(2 *0 unitD) *1 (1 *0 prodD) *1 prodD}
\ar@3 [d] _{\twocell{runitD}}
&
\twocell{(2 *0 unitD) *1 (prodD *0 1) *1 prodD}
\ar@3 [d] ^{\twocell{runitD}}
\ar@3 [l] _{\twocell{assocD}}
\\
\twocell{prodD}
& 
\twocell{prodD}
}
\]

\[
\xymatrix @C = 3em @R = 3em {
\twocell{(1 *0 unitC *0 1) *1 (1 *0 fonctF) *1 (fonctF *0 1) *1 (1 *0 prodD)}
\ar@3 [d] _-{\twocell{img_unitF}}
&
\twocell{(1 *0 unitC *0 1) *1 (prodC *0 1) *1 fonctF}
\ar@3 [l] _-{\twocell{img_prodF}}
\ar@3 [d] ^-{\twocell{runitC}}
\\
\twocell{(fonctF *0 unitD) *1 (1 *0 prodD)}
&
\twocell{fonctF}
}
\qquad
\xymatrix @C = 3em @R = 3em {
\twocell{(1 *0 unitC *0 1) *1 (1 *0 fonctG) *1 (fonctG *0 1) *1 (1 *0 prodD)}
\ar@3 [d] _-{\twocell{img_unitG}}
&
\twocell{(1 *0 unitC *0 1) *1 (prodC *0 1) *1 fonctG}
\ar@3 [l] _-{\twocell{img_prodG}}
\ar@3 [d] ^-{\twocell{runitC}}
\\
\twocell{(fonctG *0 unitD) *1 (1 *0 prodD)}
&
\twocell{fonctG}
}
\qquad
\xymatrix @C = 3em @R = 3em {
\twocell{(unitC *0 fonctF) *1 (fonctF *0 1) *1 (1 *0 prodD)}
\ar@3 [d] _{\twocell{img_unitF}}
&
\twocell{(unitC *0 2) *1 (prodC *0 1) *1 fonctF}
\ar@3 [l] _-{\twocell{img_prodF}}
\ar@3 [d] ^-{\twocell{lunitC}}
\\
\twocell{fonctF *1 (1 *0 unitD *0 1) *1 (1 *0 prodD)}
&
\twocell{fonctF}
}
\qquad
\xymatrix @C = 3em @R = 3em {
\twocell{(unitC *0 fonctG) *1 (fonctG *0 1) *1 (1 *0 prodD)}
\ar@3 [d] _{\twocell{img_unitG}}
&
\twocell{(unitC *0 2) *1 (prodC *0 1) *1 fonctG}
\ar@3 [l] _-{\twocell{img_prodG}}
\ar@3 [d] ^-{\twocell{lunitC}}
\\
\twocell{fonctG *1 (1 *0 unitD *0 1) *1 (1 *0 prodD)}
&
\twocell{fonctG}
}
\]

\[
\xymatrix @C = 3em @R = 3em {
\twocell{(unitC *0 1) *1 fonctG *1 (transfo *0 1) *1 (1 *0 prodD)} 
\ar@3 [d] _-{\twocell{img_unitG}}
&
\twocell{(unitC *0 transfo) *1 (fonctF *0 1) *1 (1 *0 prodD)}
\ar@3 [d] ^-{\twocell{img_unitF}}
\ar@3 [l] _-{\twocell{transfo_nat}}
\\
\twocell{(transfo *0 unitD) *1 (1 *0 prodD)}
&
\twocell{transfo *1 (1 *0 unitD *0 1) *1 (1 *0 prodD)}
}
\]
\end{proof}

\begin{lem}\label{lem:sourceFN}
Let $A \in \TPN[\f,\g]^*_3$. If the source of $A$ is a formal form for $\TPN^u[\f,\g]$, then $A$ is in $\TPN^+[\f,\g]^*_3$.
\end{lem}
\begin{proof}
We reason by induction on the length of $A$:
\begin{itemize}
\item If $A$ is an identity, then it is in $\TPN^+[\f,\g]$.
\item Otherwise, let us write $A = A_1 \star_2 A_2$, where $A_1$ is of length $1$. Since the source of $A$ is a normal form for $\TPN^u[\f,\g]$, the $3$-cell $A_1$ can only be in $\TPN^+[\f,\g]^*$. 

According to Lemma \ref{lem:FN_stables}, the normal forms for $\TPN^u[\f,\g]$ are stable when rewritten by $\TPN^+[\f,\g]^*$. Hence the source $A_2$ is a normal form for $\TPN^u[\f,\g]$, and by induction hypothesis, $A_2$ is in $\TPN^+[\f,\g]^*$. By composition, so is $A$.
\end{itemize}
\end{proof}

\begin{lem}\label{lem:nom_TPN_32}
Let $A$ be a $3$-cell in $\TPN[\f,\g]^{*(2)}$. There exist $C_1,C_2 \in \TPN^u[\f,\g]_3^{*}$ whose target is a normal form for $\TPN^u[\f,\g]$, a $3$-cell $A' \in \TPN^+[\f,\g]_3^{*(2)}$ and a $4$-cell $\alpha \in \TPN[\f,\g]_4^{*(2)}$ of the following shape:

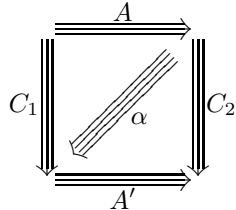
\[
\xymatrix @R=5em @C=5em{
\ar@3 [r] ^{A} 
\ar@3 [d] _{C_1} 
&
\ar@3 [d] ^{C_2} 
\ar@4 []!<-10pt,-10pt>;[ld]!<10pt,10pt> ^{\alpha}
\\
\ar@3 [r] _{A'} &
}
\]

\end{lem}
\begin{proof}
Let us write $A = A_1^{-1} \star_2 B_1 \star_2 A_2^{-1} \ldots \star_2 A_n^{-1} \star_2 B_n$, where the $A_i$ and $B_i$ are in $\TPN[\f,\g]_3^*$. For every $i \leq n$, we chose a $3$-cell $D_i \in \TPN^u[\f,\g]_3^*$ of source $\s(A_i) = \s(B_i)$ and of target a normal form for $\TPN^u[\f,\g]$.

According to Lemma \ref{lem:norm_TPN}, there exist for every $i$ some $3$-cells $A'_i$, $B'_i$ in $\TPN[f,g]^*$, $D'_i \in \TPN^u[\f,\g]_3^*$ and $D''_i \in \TPN^u[\f,\g]_3^*$ and some $4$-cells $\alpha_i$ and $\beta_i$ in $\TPN[\f,\g]^{*(2)}$ of the form:

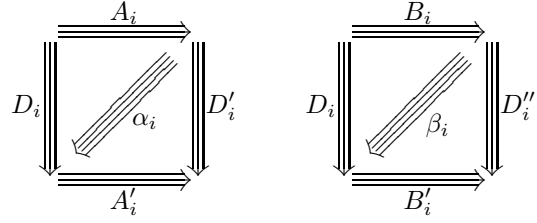
\[
\xymatrix @R=5em @C=5em{
\ar@3 [r] ^{A_i} 
\ar@3 [d] _{D_i} 
&
\ar@3 [d] ^{D'_i} 
\ar@4 []!<-10pt,-10pt>;[ld]!<10pt,10pt> ^{\alpha_i}
\\
\ar@3 [r] _{A'_i} &
}
\qquad
\xymatrix @R=5em @C=5em{
\ar@3 [r] ^{B_i} 
\ar@3 [d] _{D_i} 
&
\ar@3 [d] ^{D''_i} 
\ar@4 []!<-10pt,-10pt>;[ld]!<10pt,10pt> ^{\beta_i}
\\
\ar@3 [r] _{B'_i} &
}
\]

The following is a consequence of the target of $D_i$ being a normal form for $\TPN^u[\f,\g]$:
\begin{itemize}
\item Using Lemma \ref{lem:sourceFN}, $A'_i$ and $B'_i$ are in $\TPN^+[\f,\g]^*$, 
\item Using Lemma \ref{lem:FN_stables}, the target $A'_i$ and $B'_i$ (thus of  $D'_i$ and $D''_i$) are normal forms for $\TPN^u[\f,\g]$.
\item Since $\TPN^u[\f,\g]$ is $3$-convergent, for any $i < n$, the cells $D''_i$ and $D'_{i+1}$ are parallel.
\end{itemize}

Since $\TPN^u[\f,\g]$ is a sub-polygraph of $\PF[\f,\g]$ which is $3$-coherent, there exists, for every $i < n$, a $4$-cell $\gamma_{i}: D''_i \qfl D'_i$ in  $\PF[\f,\g]^{*(2)}_4$. 

We can now conclude the proof of this Lemma by taking $C_1 = D'_1$, $C_2 = D''_n$ and $A' = (A'_1)^{-1} \star_2 B'_2 \star_2 \ldots \star_2 (A'_n)^{-1} \star_2 B'_n$, and by defining $\alpha$ as the following composite:

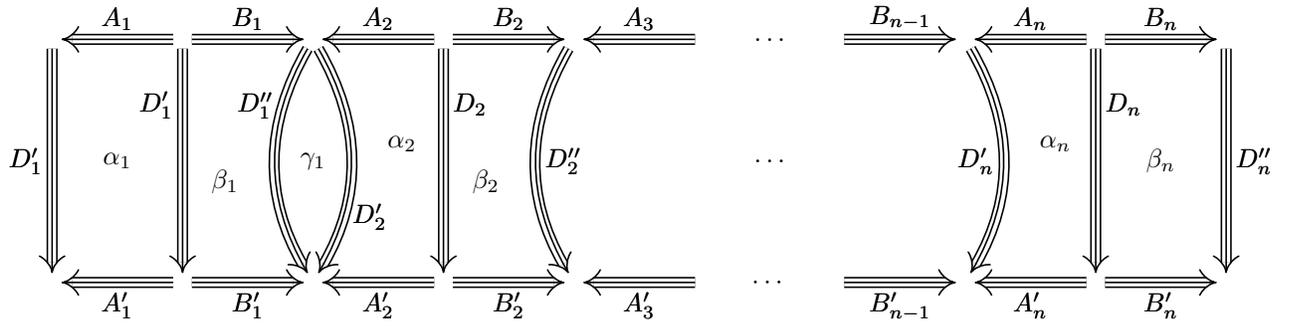
\[
\begin{tikzpicture}
\matrix(m) [matrix of math nodes, 
			nodes in empty cells,
			column sep = 1.5cm, 
			row sep = 3cm] 
{
& & & & & & & & & \\
& & & & & & & & & \\
};
\triplearrow{->}
{(m-1-2) -- node [above] {$A_1$} (m-1-1)}
\triplearrow{->}
{(m-1-2) -- node [above] {$B_1$} (m-1-3)}
\triplearrow{->}
{(m-1-4) -- node [above] {$A_2$} (m-1-3)}
\triplearrow{->}
{(m-1-4) -- node [above] {$B_2$} (m-1-5)}
\triplearrow{->}
{(m-1-6) -- node [above] {$A_3$} (m-1-5)}
\triplearrow{->}
{(m-1-7) -- node [above] {$B_{n-1}$} (m-1-8)}
\triplearrow{->}
{(m-1-9) -- node [above] {$A_n$} (m-1-8)}
\triplearrow{->}
{(m-1-9) -- node [above] {$B_n$} (m-1-10)}

\triplearrow{->}
{(m-1-1) -- node [left] {$D'_1$} (m-2-1)}
\triplearrow{->}
{(m-1-2) -- node [left, near start] {$D'_1$} (m-2-2)}
\triplearrow{->, bend right}
{(m-1-3) to node [left, near start] {$D''_1$} (m-2-3)}
\triplearrow{->, bend left}
{(m-1-3) to node [right, near end] {$D'_2$} (m-2-3)}
\triplearrow{->}
{(m-1-4) -- node [right, near start] {$D_2$} (m-2-4)}
\triplearrow{->, bend right}
{(m-1-5) to node [right] {$D''_2$} (m-2-5)}
\triplearrow{->, bend left}
{(m-1-8) to node [left] {$D'_n$} (m-2-8)}
\triplearrow{->}
{(m-1-9) -- node [right, near start] {$D_n$} (m-2-9)}
\triplearrow{->}
{(m-1-10) -- node [right] {$D''_n$} (m-2-10)}

\triplearrow{->}
{(m-2-2) -- node [below] {$A'_1$} (m-2-1)}
\triplearrow{->}
{(m-2-2) -- node [below] {$B'_1$} (m-2-3)}
\triplearrow{->}
{(m-2-4) -- node [below] {$A'_2$} (m-2-3)}
\triplearrow{->}
{(m-2-4) -- node [below] {$B'_2$} (m-2-5)}
\triplearrow{->}
{(m-2-6) -- node [below] {$A'_3$} (m-2-5)}
\triplearrow{->}
{(m-2-7) -- node [below] {$B'_{n-1}$} (m-2-8)}
\triplearrow{->}
{(m-2-9) -- node [below] {$A'_n$} (m-2-8)}
\triplearrow{->}
{(m-2-9) -- node [below] {$B'_n$} (m-2-10)}

\path (m-1-6) -- node {$\ldots$} (m-1-7);
\path (m-1-6) -- node {$\ldots$} (m-2-7);
\path (m-2-6) -- node {$\cdots$} (m-2-7);

\path (m-1-1) -- node {$\alpha_1$} (m-2-2);
\path (m-1-2) -- node [below left] {$\beta_1$} (m-2-3);
\path (m-1-3) -- node {$\gamma_1$} (m-2-3);
\path (m-1-3) -- node [above right] {$\alpha_2$} (m-2-4);
\path (m-1-4) -- node [below left] {$\beta_2$} (m-2-5);
\path (m-1-8) -- node [above right] {$\alpha_n$} (m-2-9);
\path (m-1-9) -- node {$\beta_n$} (m-2-10);

\end{tikzpicture}
\]

\end{proof}

We can now conclude the proof Theorem \ref{thm:TPNat_coh}.
\begin{repthm}{thm:TPNat_coh}[Coherence for pseudonatural transformations]
Let $\CC$ and $\DD$ be sets, and $\f,\g: \CC \to \DD$ applications.

Let $A,B \in \TPN[\f,\g]^{*(2)}_3$ be two parallel $3$-cells whose $1$-target is of weight $1$.

There is a $4$-cell $\alpha:A \qfl B \in \TPN[\f,\g]^{*(2)}_4$.
\end{repthm}
\begin{proof}
Let $A,B \in \TPN[\f,\g]^{*(2)}_3$ be two parallel $3$-cells whose $1$-target is ${}_a \twocell{1} {}_{\f(a)} \twocell{1} {}_b$. We are going to build a $4$-cell $\alpha: A \qfl B \in \TPN[\f,\g]^{*(2)}_4$.

According to Lemma \ref{lem:nom_TPN_32}, there exist $C_1,C_2,C'_1,C'_2 \in \TPN^u[\f,\g]^*$ whose targets are normal forms for $\TPN^u[\f,\g]$, $A',B' \in \TPN^+[\f,\g]^{*(2)}$ and $\alpha_1,\alpha_2 \in TPN[\f,\g]^{*(2)}_4$ such that we have the diagrams:

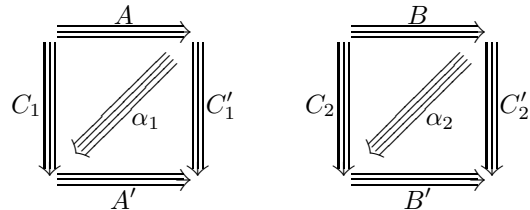
\[
\xymatrix @R=5em @C=5em{
\ar@3 [r] ^{A} 
\ar@3 [d] _{C_1} 
&
\ar@3 [d] ^{C'_1} 
\ar@4 []!<-10pt,-10pt>;[ld]!<10pt,10pt> ^{\alpha_1}
\\
\ar@3 [r] _{A'} &
}
\qquad
\xymatrix @R=5em @C=5em{
\ar@3 [r] ^{B} 
\ar@3 [d] _{C_2} 
&
\ar@3 [d] ^{C'_2} 
\ar@4 []!<-10pt,-10pt>;[ld]!<10pt,10pt> ^{\alpha_2}
\\
\ar@3 [r] _{B'} &
}
\]

The $3$-cells $A$ and $B$ are parallel, and the $3$-cells $C_1$ and $C_2$ (resp. $C'_1$ and $C'_2$) have the same source and have a normal form for $\TPN^u[\f,\g]$ as target. Since $\TPN^u[\f,\g]$ is $3$-convergent, this implies that the $3$-cells $C_1$ and $C_2$ (resp. $C'_1$ and $C'_2$) are parallel. This has two consequences:
\begin{itemize}
\item The critical pairs of $\TPN^u[\f,\g]$ already appeared in $\PF[\f,\g]$, and we showed that they admit fillers. Hence there exist cells $\beta_1 : C_1 \qfl C_2$ and $\beta_2: C'_1 \qfl C'_2$ in $\TPN[\f,\g]^{*(2)}_4$. 
\item The $3$-cells $A'$ and $B'$ are parallel, their $1$-target is still ${}_a \twocell{1} {}_{\f(a)} \twocell{1} {}_b$, and their $2$-source is a normal form for $\TPN^u[\f,\g]$. So by Proposition \ref{prop:restriction_TPN++} there exists a $4$-cell $\gamma: A' \qfl B'$.
\end{itemize} 

To conclude, we define $\alpha$ as the following composite (where we omit the context of the $4$-cells):
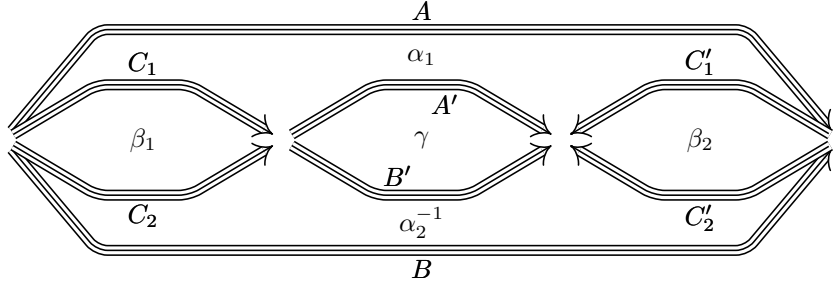
\[
\begin{tikzpicture}
\matrix (m) [matrix of math nodes, 
			nodes in empty cells,
			column sep = 1cm, 
			row sep = .5cm] 
{
 &  &  &  & & & & & & \\
 &  &  &  & & & & & & \\
 &  &  &  & & & & & & \\
 &  &  &  & & & & & & \\
 &  &  &  & & & & & & \\
};
\triplearrow{arrows = {-Implies}, rounded corners}{(m-3-1) -- (m-5-2.base) -- node [below] {$B$} (m-5-9.base) -- (m-3-10)}
\triplearrow{arrows = {-Implies}, rounded corners}{(m-3-1) -- (m-1-2.base) -- node [above] {$A$} (m-1-9.base) -- (m-3-10)}
\triplearrow{arrows = {-Implies}, rounded corners}{(m-3-1) -- (m-4-2.base) -- node [below] {$C_2$} (m-4-3.base) -- (m-3-4)}
\triplearrow{arrows = {-Implies}, rounded corners}{(m-3-1) -- (m-2-2.base) -- node [above] {$C_1$} (m-2-3.base) -- (m-3-4)}
\triplearrow{arrows = {-Implies}, rounded corners}{(m-3-4) -- (m-2-5.base) -- node [below right] {$A'$} (m-2-6.base) -- (m-3-7)}
\triplearrow{arrows = {-Implies}, rounded corners}{(m-3-4) -- (m-4-5.base) -- node [above left] {$B'$} (m-4-6.base) -- (m-3-7)}
\triplearrow{arrows = {Implies-}, rounded corners}{(m-3-7) -- (m-4-8.base) -- node [below] {$C'_2$} (m-4-9.base) -- (m-3-10)}
\triplearrow{arrows = {Implies-}, rounded corners}{(m-3-7) -- (m-2-8.base) -- node [above] {$C'_1$} (m-2-9.base) -- (m-3-10)}
\path  (m-3-1) -- node {$\beta_1$} (m-3-4);
\path  (m-3-4) -- node {$\gamma$} (m-3-7);
\path  (m-3-7) -- node {$\beta_2$} (m-3-10);
\path  (m-2-1) -- node {$\alpha_1$} (m-1-10);
\path  (m-4-1) -- node {$\alpha_2^{-1}$} (m-5-10);
\end{tikzpicture}
\]
\end{proof}

\newpage

\section{Transformation of a $(4,2)$-polygraph into a $(4,3)$-white-polygraph}
\label{sec:transfo_polygraph}

The proof of Theorem \ref{thm:main_theory} will occupy the rest of this article. We start with a $(4,2)$-polygraph $\A$ satisfying the hypotheses of Theorem \ref{thm:main_theory}. Let $S_\A$ be the set of all $2$-cells in $\A_2^*$ whose target is a normal form. Then proving Theorem \ref{thm:main_theory} consists in showing that $\A$ is $S_\A$-coherent.

In this section we successively transform $\A$ four times, leading to five pointed $(4,3)$-white-categories, namely $(\A^{*(2)},S_\A)$, $(\B^{\w(2)},S_\B)$, $(\C^{\w(3)},S_\C)$, $(\D^{\w(3)},S_\D)$ and $(\E^{\w(3)},S_\E)$, and we show each time that the new pointed $(4,3)$-white-category is stronger than the previous one. A brief description of each pointed $(4,3)$-white-category can be seen in Table \ref{table:Liste_transfo}.  Finally in Section \ref{subsec:retournement}, we perform a number of Tietze-transformations on the $4$-white-polygraph $\E$, leading to a $4$-white-polygraph $\F$. 

Thanks to Lemma \ref{lem:translation_coh} and Proposition \ref{prop:tietze_invariance}, we know that in order to show that $\A^{*(2)}$ is $S_\A$-coherent, it is enough to show that $\F^{\w(3)}$  is $S_\E$-coherent. This will be done in Section \ref{sec:proof_final}.

\begin{table}[H]
\centering
\begin{tabular}{|c|c|c|}
\hline
  Name & Description  & Commentary \\
  \hline
  & $\A_2$ &  \\
  $(\A^{*(2)}, S_\A)$ & $\A_3$   &  \\
  & $\A_4$  &  \\
  \hline
  &  $\A_2$  &  Weakening of the \\
  $(\B^{\w(2)},S_\B)$ & $\A_3 \cup K$   &  exchange law \footnotemark \\
  & $\A_4 \cup L$  & \\
  \hline
  & $\A_2 $  & Weakening \\
  $(\C^{\w(3)},S_\C)$ & $\A_3 \cup \A_3^{op} \cup K \cup K^{op}$  &   of the invertibility  \\
  & $\A_4 \cup L \cup \{\rho_A, \lambda_A \}$ &  of $3$-cells \\
  \hline
  & $\A_2 \cup \A_2^{op}$ &  Adjunction of \\
  $(\D^{\w(3)},S_\D)$ & $\A_3 \cup \A_3^{op} \cup K \cup K^{op}$  &  formal inverses  \\
  & $\A_4 \cup L \cup \{\rho_A, \lambda_A\}$ &  to $2$-cells \\
  \hline
  & $\A_2 \cup \A_2^{op}$ &  Adjunction \\
  $(\E^{\w(3)},S_\E)$ & $\A_3 \cup  \A_3^{op} \cup K \cup  K^{op} \cup \{ \eta_f, \epsilon_f \} $ &  of connections  \\
  & $\A_4 \cup L \cup \{\rho_A, \lambda_A\} \cup \{\tau_f, \sigma_f\}$ & between $2$-cells \\
  \hline
\end{tabular}
\caption{List of the successive transformations of $\A$.} 
\label{table:Liste_transfo}
\end{table}
\footnotetext{The sets $K$ and $L$ will be defined in Section \ref{subsec:affaiblissement}}

\begin{ex}
We have already shown in Section \ref{sec:application} that for every sets $\CC$, $\DD$ and for every applications $\f,\g:\CC \to \DD$, the $(4,2)$-polygraph $\TPN^{++}[\f,\g]$ satisfies the hypothesis of Theorem \ref{thm:main_theory}.

In what follows, we will use as a running example the polygraph $\A = \Ass$ which consists of one $0$-cell, one $1$-cell $\twocell{1}$, one $2$-cell $\twocell{prodC}: \twocell{2} \Rightarrow \twocell{1}$, one $3$-cell $\twocell{assocC}: \twocell{(prodC *0 1) *1 prodC} \Rrightarrow \twocell{(1 *0 prodC) *1 prodC}$, and one $4$-cell $\twocell{pentaC}$:

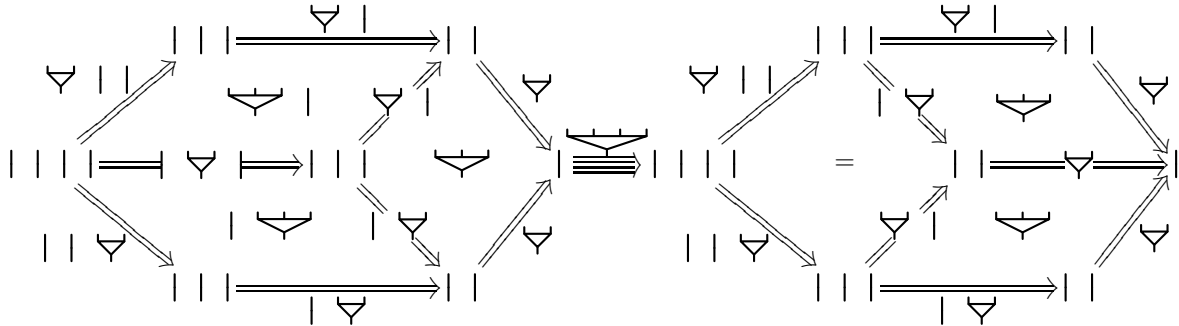
\[
\xymatrix @C = 2.5em @R = 3em {
&
\twocell{3}
\ar@2 [rr] ^-{\twocell{prodC *0 1}}
\ar@{} [rd] |-{\twocell{assocC *0 1}}
& &
\twocell{2}
\ar@2 [rd] ^-{\twocell{prodC}}
\ar@{} [dd] |-{\twocell{assocC}}
& & &
\twocell{3}
\ar@2 [rr] ^-{\twocell{prodC *0 1}}
\ar@{} [rrrd] |-{\twocell{assocC}}
\ar@2 [rd] |-{\twocell{1 *0 prodC}}
\ar@{} [dd] |-{=}
& &
\twocell{2}
\ar@2 [rd] ^-{\twocell{prodC}}
&
\\
\twocell{4}
\ar@2 [ru] ^-{\twocell{prodC *0 2}}
\ar@2 [rr] |-{\twocell{1 *0 prodC *0 1}}
\ar@2 [rd] _-{\twocell{2 *0 prodC}}
& &
\twocell{3}
\ar@{} [ld] |-{\twocell{1 *0 assocC}}
\ar@2 [ru] |-{\twocell{prodC *0 1}}
\ar@2 [rd] |-{\twocell{1 *0 prodC}}
&
&
\twocell{1} \:
\ar@4 [r] ^-{\twocell{pentaC}}
&
\: \twocell{4}
\ar@2 [ru] ^-{\twocell{prodC *0 2}}
\ar@2 [rd] _-{\twocell{2 *0 prodC}}
& & 
\twocell{2}
\ar@2 [rr] |-{\twocell{prodC}}
\ar@{} [rd] |-{\twocell{assocC}}
& &
\twocell{1}
\\
&
\twocell{3}
\ar@2 [rr] _-{\twocell{1 *0 prodC}}
& &
\twocell{2}
\ar@2 [ru] _-{\twocell{prodC}}
& & &
\twocell{3}
\ar@2 [rr] _-{\twocell{1 *0 prodC}}
\ar@2 [ru] |-{\twocell{prodC *0 1}}
& &
\twocell{2}
\ar@2 [ru] _-{\twocell{prodC}}
&
}
\]

In particular, $\Ass$ satisfies the $2$-Squier condition of depth $2$. The $2$-category $\Ass^*_2$ is $2$-convergent and its only normal form is the $1$-cell $\twocell{1}$.

The corresponding set $S_\A$ is then the set of $2$-cells in $\Ass^*_2$ from any $1$-cell $\twocell{dots}$ to $\twocell{1}$.
\end{ex}

\subsection{Weakening of the exchange law}
\label{subsec:affaiblissement}

We construct dimension by dimension a $(4,2)$-white-polygraph $\B$, together with a white-functor $F : \B^{\w(2)} \to \A^{*(2)}$. We then define a subset $S_\B$ of $\B^{\w(2)}$ and show (Proposition \ref{prop:b_stronger_a}) using $F$ that $(\B^{\w(2)},S_\B)$ is stronger than $(\A^{*(2)},S_\A)$.

In low dimensions, we set  $\B_i = \A_i$, for every $i \leq 2$, and the functor $F$ is the identity on generators.

\begin{lem}\label{lem:F_2_surj}
The functor $F : \B^\w \to \A^*$ is $2$-surjective. 
\end{lem}
\begin{proof}
By construction, $\A_2^*$ is the quotient of $\B_2^\w$ by the equivalence relation generated by:

\[
(f \star_0 v) \star_i (u' \star_0 g ) = (u \star_0 g ) \star_i (f \star_0 v').
\]
And $F$ is the canonical projection induced by the quotient.
\end{proof}

In what follows, we suppose chosen a section $i: \A^* \to \B^\w$ of $F$, which is possible thanks to Lemma \ref{lem:F_2_surj}.

We extend $\B$ into a $3$-white-polygraph and $F: \B^{\w} \to \A^{*}$ into a $3$-white-functor by setting $\B_3 := \A_3 \cup K$:
\begin{itemize} 
\item For every $3$-cell $A \in \A_3$, the source and target of $A$ in $\B_2^\w$ are respectively $\s^\B(A) := i(\s^\A(A))$ and $\t^\B(A):=i(\t^\A(A))$.
\item The set $K$ is the set of $3$-cells $A_{fv,ug}$, of shape:
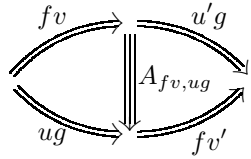
\[
\xymatrix @C = 4em @R = 1.5em{
& 
\ar@2 @/^/ [rd] ^{u'g}
\ar@3 [dd] ^{A_{fv,ug}}
& \\
\ar@2 @/^/ [ru] ^{fv}
\ar@2 @/_/ [rd] _{ug} 
& 
& \\
&  
\ar@2 @/_/ [ru] _{fv'}
& \\
}
\]
for every strict Peiffer branching $(fv,ug)$, where $f:u \Rightarrow u'$ and $g : v \Rightarrow v'$ are rewriting steps.
\end{itemize}

The image of a cell of $\B_3$ under $F$ is defined as follows:
\begin{itemize}
\item For every strict Peiffer branching $(fv , ug )$,  $F(A_{fv,ug}) := 1_{f \star_0 g}$
\item For every $3$-cell $A$ in $\A_3$, $F(A) := A$.
\end{itemize}

\begin{lem}\label{lem:carac_surjectivite_F}
Let $f,g \in \B^{\w}_2$. There exists a $3$-cell $A : f \Rrightarrow g$ in $K_3^{\w(2)}$ if and only if the equality $F(f) = F(g)$ holds in $\A_2^*$.
\end{lem}
\begin{proof}
Let $f,g \in \B^\w_2$. The image of any cell in $K^{\w(2)}_3$ by $F$ is an identity. So if there exists a $3$-cell $A:f \Rrightarrow g$ in $K^{\w(2)}_3$, necessarily $F(f) = F(g)$.

Conversely, the set $\A_2^*$ is the quotient of $\B_2^\w$ by the equivalence relation generated by:
\[ 
f \s(g) \star_1 \t(f) g =  \s(f) g \star_1 f \t(g), \] 
for $f,g \in \B_2^\w$.  The $3$-cells $A_{fu,vg}$, where $(fu,vg)$ is a strict Peiffer branching, generate this relation, and they are in $K$. Hence the result.
\end{proof}

\begin{lem}\label{lem:F_3_surj}
The functor $F : \B^{\w(2)} \to \A^{\w(2)}$ is $3$-surjective.
\end{lem}
\begin{proof}
Let $E$ be the set of $3$-cells $A \in \A_3^{*(2)}$ such that, for every $f,g \in \B_2^\w$ in the preimage of $\s(A)$ and $\t(A)$ under $F$, there exists a $3$-cell $B:f \Rrightarrow g \in \B_3^{\w(2)}$ satisfying $F(B) = A$. Let us show that $E = \A_3^{*}$. We already know that $E$ contains the identities thanks to Lemma \ref{lem:carac_surjectivite_F}.

The $3$-cells of length $1$ in $A^*_3$ are in $E$. Indeed, let $A \in \A_3^{*}$ be a $3$-cell of length $1$, and $f,g \in \B_2^\w$  such that $F(f) = \s(A)$ et $F(g) = \t(A)$. There exist $u,v \in \A_1^*$, $f',g' \in \A_2^*$, and $A' \in \A_3$ such that
\[
A= f' \star_1 (uA'v) \star_1 g'.
\]
Let $\tilde u$, $\tilde v$, $\tilde f$, $\tilde g$ be in the preimages  respectively of $u$, $v$, $f'$, $g'$ under $F$ (they exist thanks to Lemma \ref{lem:F_2_surj}), and let $B_1 := \tilde f \star_1 (\tilde u A' \tilde v) \star_1 \tilde g \in \B_3^{\w(2)}$. By construction,  $F(B_1) = A$, which leads to the equalities:
\[
F(\s(B_1)) =  F(f) \qquad F(\t(B_1)) = F(g).
\]
Thus, according to Lemma \ref{lem:carac_surjectivite_F}, there exist $3$-cells $C_1 : f \Rrightarrow \s(B_1) \in K_3^{\w(2)}$ and $C_2 : \t(B_1) \Rrightarrow g \in K_3^{\w(2)}$. Let $B := C_1 \star_2 B_1 \star_2 C_2$: by  construction, $B$ has the required source and target, and moreover:
\[
F(B) = F(C_1) \star_2 F(B_1) \star_2 F(C_2) = 1_{F(f)} \star_2 A \star_2 1_{F(g)} = A.
\]

The set $E$ is stable under composition. Indeed let $A_1,A_2 \in E$ such that $\t(A_1) = \s(A_2)$, and $f,g \in \B_2^\w$ satisfying $F(f) = \s(A_1)$ and $F(g) = \t(A_2)$. Since $F$ is $2$-surjective, there exists $h \in \B_2^\w$ in the inverse image of $\t(A_1)$ under $F$. Since $A_1$ (resp. $A_2$) is in $E$, there exists a cell $B_1$ (resp. $B_2$) in $\B_3^{\w(2)}$ such that $F(B_1) = A_1$ (resp. $F(B_2) = A_2$), $\s(B_1) = f$ (resp. $\s(B_2) = h$) and $\t(B_1) = h$ (resp. $\t(B_2) = g$). Let $B := B_1 \star_2 B_2$: we get:
\[
\s(B) = f \qquad F(B) = A_1 \star_2 A_2 \qquad \t(B) = g
\] 

The set $E$ is stable under $2$-composition. Indeed let $A \in E$ and $f,g \in \B_2^\w$ such that $F(f) = \s(A^{-1})$ and $F(g) = \t(A^{-1})$. There exists $B \in \B^{\w(2)}$ such that:
\[
\s(B) = g \qquad F(B) = A \qquad \t(B) =  f.
\]
Hence the cell $B^{-1}$ satisfies the required property.
\end{proof}

We now extend $\B$ into a $(4,2)$-white-polygraph and $F: \B^{\w(2)} \to \A^{*(2)}$ into a $4$-white-functor by setting  $\B_4 = \A_4 \cup L$:
\begin{itemize}
\item For every $3$-cell $A \in \A_4$, the source and target of $A$ in $\B_3^{\w(2)}$ are respectively $\s^\B(A) := i(\s^\A(A))$ and $\t^\B(A):=i(\t^\A(A))$, where $i$ is a chosen section of the application $F_3 : \B^{\w(2)}_3 \to \A^{\w(2)}_3$ (which exists since $F$ is $3$-surjective). And we set $F(A) := A$.
 
\item For every $3$-fold strict Peiffer branching $(f,g,h)$, the set $L$ contains a $4$-cell $A_{f,g,h}$, whose shape depends on the form of the branching $(f,g,h)$. If $(f,g,h) = (f'v,g'v,uh')$, with $(f',g')$ a critical pair, and $h' : v \Rightarrow v'$ then $A_{f,g,h}$ is of the following shape:

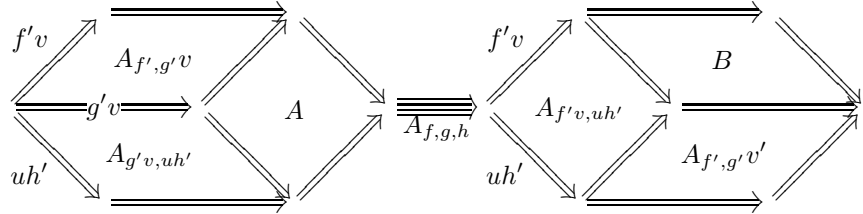
\[
\xymatrix @C = 3em @R = 3em {
&
\ar@2 [rr]
\ar@{} [rd] |-{A_{f',g'}v}
& &
\ar@2 [rd]
\ar@{} [dd] |-{A}
& & &
\ar@2 [rr]
\ar@{} [rrrd] |-{B}
\ar@2 [rd]
\ar@{} [dd] |-{A_{f'v,uh'}}
& &
\ar@2 [rd]
&
\\
\ar@2 [ru] ^{f'v}
\ar@2 [rr] |{g'v}
\ar@2 [rd] _{uh'}
& &
\ar@{} [ld] |-{A_{g'v,uh'}}
\ar@2 [ru]
\ar@2 [rd]
&
&
\ar@4 [r] _{A_{f,g,h}}
&
\ar@2 [ru] ^{f'v}
\ar@2 [rd] _{uh'}
& & 
\ar@2 [rr]
\ar@{} [rd] |-{A_{f',g'}v'}
& &
\\
&
\ar@2 [rr]
& &
\ar@2 [ru]
& & &
\ar@2 [ru]
\ar@2 [rr]
& &
\ar@2 [ru]
&
}
\]
where $A$ and $B$ are in $K^{\w(2)}_3$. And we define $F(A_{f,g,h}) := 1_{A_{f',g'} \star_0 h'}$.

If $(f,g,h) = (f'v,ug',uh')$, with $(g,h)$ a critical pair, and $f' : u \Rightarrow u'$ then $A_{f,g,h}$ is of the following shape:
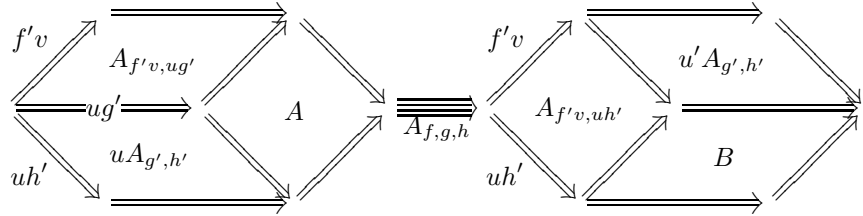
\[
\xymatrix @C = 3em @R = 3em {
&
\ar@2 [rr]
\ar@{} [rd] |-{A_{f'v,ug'}}
& &
\ar@2 [rd]
\ar@{} [dd] |-{A}
& & &
\ar@2 [rr]
\ar@{} [rrrd] |-{u'A_{g',h'}}
\ar@2 [rd]
\ar@{} [dd] |-{A_{f'v,uh'}}
& &
\ar@2 [rd]
&
\\
\ar@2 [ru] ^{f'v}
\ar@2 [rr] |{ug'}
\ar@2 [rd] _{uh'}
& &
\ar@{} [ld] |-{uA_{g',h'}}
\ar@2 [ru]
\ar@2 [rd]
&
&
\ar@4 [r] _{A_{f,g,h}}
& 
\ar@2 [ru] ^{f'v}
\ar@2 [rd] _{uh'}
& & 
\ar@2 [rr]
\ar@{} [rd] |-{B}
& &
\\
&
\ar@2 [rr]
& &
\ar@2 [ru]
& & &
\ar@2 [ru]
\ar@2 [rr]
& &
\ar@2 [ru]
&
}
\]
where $A$ and $B$ are in $K^{\w(2)}_3$. And we define $F(A_{f,g,h}) := 1_{f' \star_0 A_{g',h'}}$.

If $(f,g,h) = (f'vw,ug'w,uh'w)$, then $A_{f,g,h}$ is of the following shape, where $A$ and $B$ are in $K^{\w(2)}_3$:
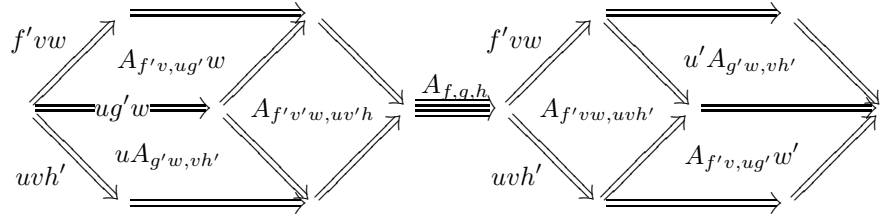
\[
\xymatrix @C = 3em @R = 3em {
&
\ar@2 [rr]
\ar@{} [rd] |-{A_{f'v,ug'}w}
& &
\ar@2 [rd]
\ar@{} [dd] |-{A_{f'v'w,uv'h}}
& & &
\ar@2 [rr]
\ar@{} [rrrd] |-{u'A_{g'w,vh'}}
\ar@2 [rd]
\ar@{} [dd] |-{A_{f'vw,uvh'}}
& &
\ar@2 [rd]
&
\\
\ar@2 [ru] ^{f'vw}
\ar@2 [rr] |{ug'w}
\ar@2 [rd] _{uvh'}
& &
\ar@{} [ld] |-{uA_{g'w,vh'}}
\ar@2 [ru]
\ar@2 [rd]
&
&
\ar@4 [r] ^{A_{f,g,h}}
& 
\ar@2 [ru] ^{f'vw}
\ar@2 [rd] _{uvh'}
& & 
\ar@2 [rr]
\ar@{} [rd] |-{A_{f'v,ug'}w'}
& &
\\
&
\ar@2 [rr]
& &
\ar@2 [ru]
& & &
\ar@2 [ru]
\ar@2 [rr]
& &
\ar@2 [ru]
&
}
\]
And we define $F(A_{f,g,h}) := 1_{f' \star_0 g' \star_0 h'}$.
\end{itemize}

Let now $S_\B$ be the set of all $2$-cells in $\B^\w$ whose $1$-target is a normal form. 

\begin{prop}\label{prop:b_stronger_a}
The pointed $(4,3)$-white-category $(\B^{\w(2)},S_\B)$ is stronger than $(\A^{*(2)},S_\A)$.
\end{prop}
\begin{proof}
The functor $F$ sends normal forms on normal forms. Hence by restriction it induces a $2$-functor  $F \restriction S_\B : \B^{\w(2)} \restriction S_\B \to \A^{*(2)} \restriction S_\A$. 

Lemmas \ref{lem:F_2_surj} and \ref{lem:F_3_surj} show that it is  $k$-surjective  for every $k < 2$. Hence we can conclude using Lemma \ref{lem:translation_coh}.
\end{proof}

\begin{ex}\label{ex:assoc}
In the case where $\A = \Ass$, the set $K$ contains in particular the following $3$-cells, associated respectively to the strict Peiffer branchings $(\twocell{prodC *0 2} \quad , \quad \twocell{2 *0 prodC})$ and $(\twocell{prodC *0 3}\quad , \quad \twocell{3 *0 prodC})$:
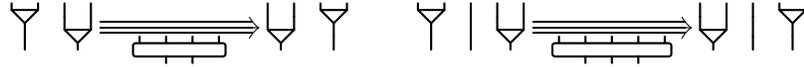
\[
\xymatrix @C = 6em{
\twocell{(prodC *0 2) *1 (1 *0 prodC)} 
\ar@3 [r] _{\twocell{exchange_z}}
&
\twocell{(2 *0 prodC) *1 (prodC *0 1)}
}
\qquad
\xymatrix @C = 6em{
\twocell{(prodC *0 3) *1 (2 *0 prodC)} 
\ar@3 [r] _{\twocell{exchange_u}}
&
\twocell{(3 *0 prodC) *1 (prodC *0 2)}
}
\]
In $L$, the $4$-cell associated to the strict Peiffer branching $ (\twocell{prodC *0 3}\quad , \quad \twocell{1 *0 prodC *0 2} \quad , \quad \twocell{3 *0 prodC})$ is the following:

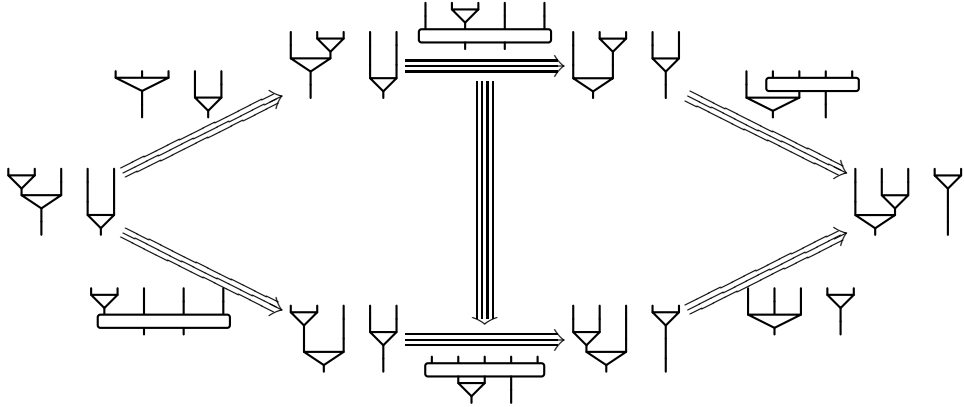
\[
\xymatrix @C = 6em @R = 2em {
&
\twocell{(1 *0 prodC *0 2) *1 (prodC *0 2) *1 (1 *0 prodC)}
\ar@3 [r] ^*+{\twocell{(1 *0 prodC *0 2) *1 exchange_z}}
\ar@{} [r] _*+{}="src"
&
\twocell{(1 *0 prodC *0 2) *1 (2 *0 prodC) *1 (prodC *0 1)}
\ar@3 [rd] ^*+{\twocell{(1 *0 exchange_z) *1 (prodC *0 1)}}
&
\\
\twocell{(prodC *0 3) *1 (prodC *0 2) *1 (1 *0 prodC)} 
\ar@3 [ru] ^*+{\twocell{(assocC *0 2) *1 (1 *0 prodC)}}
\ar@3 [rd] _*+{\twocell{(prodC *0 3) *1 exchange_z}}
& & &
\twocell{(3 *0 prodC) *1 (1 *0 prodC *0 1) *1 (prodC *0 1)}
\\
&
\twocell{(prodC *0 3) *1 (2 *0 prodC) *1 (prodC *0 1)}
\ar@3 [r] _*+{\twocell{exchange_u *1 (prodC *0 1)}}
\ar@{} [r] ^*+{}="tgt"
&
\twocell{(3 *0 prodC) *1 (prodC *0 2) *1 (prodC *0 1)}
\ar@3 [ru] _*+{\twocell{(3 *0 prodC) *1 (assocC *0 1)}}
&
\ar@4 "src";"tgt"
}
\]
\end{ex}

\subsection{Weakening of the invertibility of $3$-cells}
\label{subsec:weak_inverse_3}

We construct dimension by dimension a $4$-white-polygraph $\C$, together with a $3$-white-functor $G : \C^{\w(3)} \to \B^{\w(2)}$. We then define a subset $S_\C$ of $\C^{\w(3)}$ and show (Proposition \ref{prop:c_stronger_b}) using $G$ that $(\C^{\w(3)},S_\C)$ is stronger than $(\B^{\w(2)},S_\B)$.

In low dimensions, we set $\C_i = \B_i$ for $i \leq 2$, with the functor $G$ being the identity.

We extend $\C$ into a $3$-white-polygraph by setting $\C_3 := \B_3 \cup \B_3^{op}$, where the set $\B_3^{op}$ contains, for every $A \in \B_3$, a cell denoted by $A^{op}$, whose source and target are given by the equalities:
\[
\s(A^{op}) = \t(A) \qquad \t(A^{op}) = \s(A)
\]
And the functor $G:\C^{\w} \to \B^{\w(2)}$ is defined as follows for every $A \in \B_3$: 
\[
G(A) = A \qquad G(A^{op}) = A^{-1}.
\]

\begin{lem}\label{lem:G_3_surj}
The functor $G: \C^{\w(3)} \to \B^{\w(2)}$ is $3$-surjective.
\end{lem}
\begin{proof}
By definition, $\B_3^{\w(2)}$ is the quotient of $\C_3^\w$ by the relations $A^{op} \star_2 A = 1$ and $A \star_2 A^{op} = 1$, and $G$ is the corresponding canonical projection.
\end{proof}

We extend $\C$ into a $4$-white-polygraph by setting $\C_4: = \B_4 \cup \{\rho_A,\lambda_A |A \in \B_3\}$, where the applications source and target $\s,\t : \C_4 \to \C^\w_3$ are defined as follows:
\begin{itemize}
\item For $A \in \B_4$, the cell $\s^\C(A)$ (resp. $\t^\C(A)$) is any cell in the preimage of  $\s^\B(A)$ under $G$, which is non-empty thanks to Lemma \ref{lem:G_3_surj}. And we set $G(A) := A$.
\item For every $A \in \B_3$, the cells $\rho_A$ and $\lambda_A$ have the following shape:
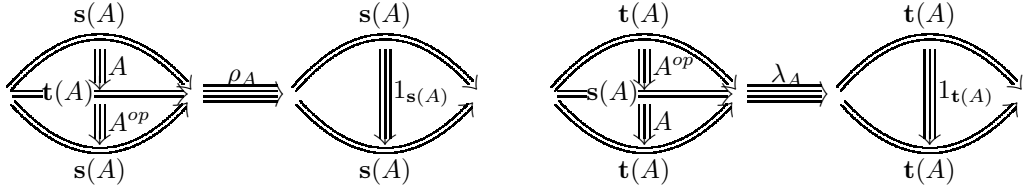
\[
\xymatrix @C = 3em @R = 1.5em{
& 
\ar@3 [d] ^{A}
& & & 
\ar@3 [dd] ^{1_{\s(A)}}
&
\\
\ar@2 @/^5ex/ [rr] ^{\s(A)}
\ar@2 @/_5ex/ [rr] _{\s(A)} 
\ar@2 [rr] |<<<<<<<<{\t(A)}
& 
\ar@3 [d] ^{A^{op}}
& 
\ar@4 [r] ^{\rho_A}
&
\ar@2 @/^5ex/ [rr] ^{\s(A)}
\ar@2 @/_5ex/ [rr] _{\s(A)} 
& &
\\
& 
& & & & \\
}
\qquad
\xymatrix @C = 3em @R = 1.5em{
& 
\ar@3 [d] ^{A^{op}}
& & & 
\ar@3 [dd] ^{1_{\t(A)}}
&
\\
\ar@2 @/^5ex/ [rr] ^{\t(A)}
\ar@2 @/_5ex/ [rr] _{\t(A)} 
\ar@2 [rr] |<<<<<<<<{\s(A)}
& 
\ar@3 [d] ^{A}
& 
\ar@4 [r] ^{\lambda_A}
&
\ar@2 @/^5ex/ [rr] ^{\t(A)}
\ar@2 @/_5ex/ [rr] _{\t(A)} 
& &
\\
& 
& & & & \\
}
\]
And we set $G(\rho_A) := 1_{\s(A)}$ and $G(\lambda_A) := 1_{\t(A)}$.
\end{itemize}

Let $S_\C$ be the set of all $2$-cells in $\C^\w$ whose $2$-target is a normal form.

\begin{prop}\label{prop:c_stronger_b}
The pointed $(4,3)$-white-category $(\C^{\w(3)},S_\C)$ is stronger than $(\B^{\w(2)},S_\B)$.
\end{prop}
\begin{proof}
The functor $G$ restricts into a functor $G \restriction S_\C :\C^{\w(3)} \restriction S_\C \to \B^{\w(2)}\restriction S_\B $, which is $i$-surjective for $i <2$ thanks to Lemma \ref{lem:G_3_surj}. Hence we can conclude thanks to Lemma \ref{lem:translation_coh}.
\end{proof}

\begin{ex}
In the case where $\A = \Ass$, let $A = \twocell{assocC}$. The set $\C_3$ contains the following $3$-cell:
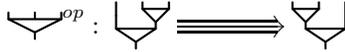
\[
\twocell{assocC}^{op}:
\xymatrix @C = 4em{
\twocell{(1 *0 prodC) *1 (prodC)} 
\ar@3 [r]
&
\twocell{(prodC *0 1) *1 (prodC)}
}
\]
And the following cells lie in $\C_4$, where $A = \twocell{assocC}$:

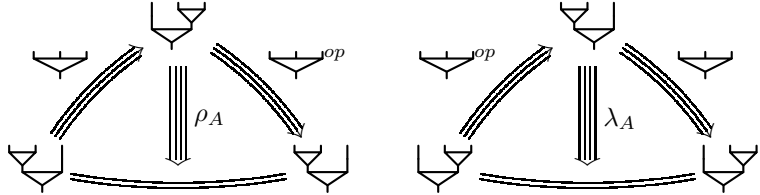
\[
\xymatrix @!C = 4em @!R = 3em{
&
\twocell{(1 *0 prodC) *1 (prodC)}
\ar@3 @/^/ [rd] ^{\twocell{assocC}^{op}}
\ar@4 []!<0pt,-15pt>;[d] ^*+{\rho_A}
&
\\
\twocell{(prodC *0 1) *1 (prodC)}
\ar@2{-} @/_/ [rr] _{}
\ar@3 @/^/ [ru] ^{\twocell{assocC}}
& &
\twocell{(prodC *0 1) *1 (prodC)}
}
\qquad
\xymatrix @!C = 4em @!R = 3em{
&
\twocell{(prodC *0 1) *1 (prodC)}
\ar@3 @/^/ [rd] ^{\twocell{assocC}}
\ar@4 []!<0pt,-15pt>;[d] ^*+{\lambda_A}
&
\\
\twocell{(1 *0 prodC) *1 (prodC)}
\ar@2{-} @/_/ [rr] _{}
\ar@3 @/^/ [ru] ^{\twocell{assocC}^{op}}
& &
\twocell{(1 *0 prodC) *1 (prodC)}
}
\]
\end{ex}

\subsection{Adjunction of formal inverses to $2$-cells}
\label{subsec:adjunc_2cell}
Let $\D$ be the $4$-white-polygraph defined as follows:
\[
\text{for every $i \neq 2$, } \D_i := \C_i \qquad 
 \D_2 := \C_2 \cup \bar \C_2,
 \]
where for every $f \in \C_2$, the set $\bar \C_2$ contains a cell $\bar f$ with source $\t(f)$ and with target $\s(f)$. Let $S_\D$ be the set of all $2$-cells of the sub-$2$-white-category $\C_2^\w$ of $\D_2^\w$ whose target is a normal form.

\begin{nota}
The application $\C_2 \to \bar{\C_2}$ extends into an application $\C_2^\w \to \bar{\C_2}^\w$ which exchanges the source and targets of the $2$-cells. 

We denote a $2$-cell $f$ by
$\xymatrix{
\ar@2 [r] ^{f}
& }
$ if $f$ is in $\B_2^\w$, by $\xymatrix{
 & \ar@2 [l] _{f}}
$ if  $\bar f$ is in $\B_2^\w$, and by
$\xymatrix{
\ar@{<=>} [r] ^{f}
& }
$ if $f$ is any cell in $\C_2^\w$.
\end{nota}

\begin{prop}
The pointed $(4,3)$-white-category $(\D^{\w(3)},S_\D)$ is stronger than $(\C^{\w(3)},S_\C)$.
\end{prop}
\begin{proof}
Let us show that $\D^{\w(3)} \restriction S_\D =\C^{\w(3)}\restriction S_\C $. Let $\iota : \C^{\w(3)} \to \D^{\w(3)}$ be the canonical inclusion functor. Since the only cells added are in dimension $2$, $\iota$ satisfies the hypotheses of Proposition \ref{prop:inject_libres_gen}, thus $\C^{\w}$ is a sub-$4$-white-category of $\D^\w$, which gives us an inclusion $\C^{\w(3)}\restriction S_\C 
 \subseteq \D^{\w(3)} \restriction S_\D$.
 
Let us show the reverse inclusion. Let $f \in \D^{\w(3)}$ be an $i$-cell ($i \geq 2$), and suppose that $f$ is in $\D^{\w(3)} \restriction S_\D$. In particular $\t_2(f)$ and $\s_2(f)$ are in $\C_2^\w$. Since $\iota$ also satisfies the hypotheses of Proposition \ref{prop:img_foncteur_gen}, with $k_0 = 2$, it is $2$-discriminating on $\D_i^{\w(3)}$. Thus $f$ is in $\C^{\w(3)}$, and in $\C^{\w(3)}\restriction S_\C$ since its $1$-target is a normal form.
\end{proof}

\begin{ex}
In the case where $\A = \Ass$, the set $\D_2$ contains one additional $2$-cell:
\[
\twocell{inv_prodC} := \overline{\twocell{prodC}}
\]
And the following cells are composites in $\D^\w$:
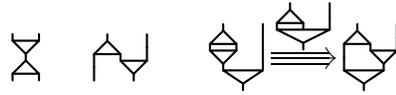
\[
\twocell{prodC *1 inv_prodC}
\qquad
\twocell{(inv_prodC *0 1) *1 (1 *0 prodC)}
\qquad
\xymatrix{
\twocell{(inv_prodC *0 1) *1 (prodC *0 1) *1 prodC}
\ar@3 [r] ^{\twocell{(inv_prodC *0 1) *1 assocC}}
&
\twocell{(inv_prodC *0 1) *1 (1 *0 prodC) *1 prodC}
}
\]
Note that the equality $\D^{\w(3)} \restriction S_\D =\C^{\w(3)}\restriction S_\C$ implies that none of these composites belongs to $\D^{\w(3)} \restriction S_\D$.
\end{ex}

\subsection{Adjunction of connections between $2$-cells}

Let $\E$ be the following $4$-white-polygraph:
\begin{itemize}
\item For $i = 0,1,2$, $\E_i = \D_i$,
\item For $i = 3$, $\E_3 = \D_3 \cup \{\eta_f,\epsilon_f | f \in \C_2 \}$.
\item For $i = 4$, $\E_4 = \D_4 \cup \{\tau_f, \sigma_f |f \in \C_2 \}$.
\end{itemize}
The cells $\eta_f$, $\epsilon_f$, $\tau_f$ and $\sigma_f$ have the following shape:
\begin{multicols}{2}
\begin{itemize}
\item $\epsilon_f : \bar f \star_1 f \Rrightarrow 1_{\t(f)}$ 

\[
\xymatrix @!C = 4em @!R = 3em{
&
\ar@2 @/_/ [ld] _f
\ar@2 @/^/ [rd] ^f
\ar@3 []!<0pt,-15pt>;[d] ^{\epsilon_f}
&
\\
\ar@2{-} @/_/ [rr] _{1_{\t(f)}}
& &
}
\]

\item $\eta_f : 1_{\s(f)} \Rrightarrow f \star_1 \bar f$

\[
\xymatrix @!C=4em @!R = 3em{
\ar@2 @/^/ [rr] ^{1_{\s(f)}}
\ar@2 @/_/ [rd] _f
&
\ar@3 [];[d]!<0pt,15pt> ^{\eta_f}
&
\ar@2 @/^/ [ld] ^f
\\
& &
}
\]
\end{itemize}
\end{multicols}
\begin{itemize}
\item $\tau_f : (\bar f \star_1 \eta_f) \star_2 (\epsilon_f \star_1 \bar f) \qfl 1_{\bar f}$ 
\[
\xymatrix @C = 4em @R = 1.5em{
&
\ar@2 @/_/ [ldd] _f
\ar@2{-} [rr]
\ar@2 [rdd] |-{f}
\ar@3  [dd] _{\epsilon_f}
&
\ar@3 [dd] _{\eta_f}
&
\ar@2 @/^/ [ldd] ^{f}
&
&
\ar@2 @/_/ [ldd] _f
\ar@2{-} [rr]
\ar@{} [rdd] |-{1_{\bar f}}
&
&
\ar@2 @/^/ [ldd] ^f
\\
& & &
\ar@4 [r] ^{\tau_f}
& & & &
\\
\ar@2{-} [rr]
& &
& & 
\ar@2{-} [rr]
& & &
}
\]
\item
$\sigma_f : (\eta_f \star_1 f) \star_2 (f \star_1 \epsilon_f) \qfl 1_f$.
\[
\xymatrix @C = 4em @R = 1.5em{
\ar@2{-} [rr]
\ar@2 @/_/ [rdd] _f
&
\ar@3  [dd] _{\eta_f}
&
\ar@3 [dd] ^{\epsilon_f}
\ar@2 [ldd] |-f
\ar@2 @/^/ [rdd] ^f
&
&
\ar@2{-} [rr]
\ar@2 @/_/ [rdd] _f
\ar@{} [rrrdd] |-{1_f}
&
&
\ar@2 @/^/ [rdd] ^f
&
\\
& & &
\ar@4 [r] ^{\sigma_f}
& & & &
\\
&
\ar@2{-} [rr]
&
& & 
& 
\ar@2{-} [rr]
& &
}
\]
\end{itemize}

\begin{nota}
Let us denote by $\twocell{epsilon}$ the $3$-cell $\epsilon_f$ and $\twocell{eta}$ the $3$-cell $\eta_f$. In a similar fashion, we denote by $\twocell{sigma}$ for  $\sigma_f$ and $\twocell{tau}$ for $\tau_f$:
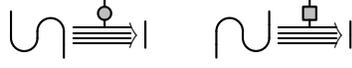
\[
\xymatrix {
\twocell{(1 *0 eta) *1 (epsilon *0 1)}
\ar@4 [r] ^-{\twocell{tau}}
&
\twocell{1}
}
\qquad
\xymatrix {
\twocell{(eta *0 1) *1 (1 *0 epsilon)}
\ar@4 [r] ^-{\twocell{sigma}}
&
\twocell{1}
}
\]

Let $\R := \{\sigma_f, \tau_f\}$, and $\R^\w$ (resp. $\R^{\w(3)}$) be the sub-$4$-white-category (resp. sub-$(4,3)$-white-category) of $\E^{\w(3)}$ generated by the cells in $\R$. A $4$-cell of length $1$ in $\R^\w$ is called an \emph{$\R$-rewriting step}.
\end{nota}

Let $S_\E$ be the set of all $2$-cells of the sub-$2$-white-category $\C_2^\w$ of $\E_2^\w$ whose target is a normal form. Using properties of the rewriting system induced by $\R^\w$, we are going to define a functor $K : \E^{\w(3)} \restriction S_\E \to \D^{\w(3)} \restriction S_\D$. 

\begin{lem}\label{lem:R_confluence}
Let $\alpha \in \E^{\w}_4$ and $\beta \in \R^{\w}$ of length $1$ with the same source. There exist $\alpha'\in \E^{\w}_4$ and $\beta' \in \R^{\w}$ of maximum length $1$, such that:

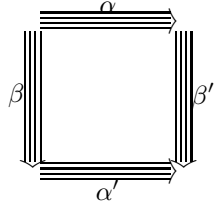
\[
\xymatrix @R=5em @C=5em{
\ar@4 [r] ^{\alpha} 
\ar@4 [d] _{\beta} 
&
\ar@4 [d] ^{\beta'} 
\\
\ar@4 [r] _{\alpha'} &
}
\]

\end{lem}
\begin{proof}

The result holds whenever $(\alpha,\beta)$ is a Peiffer or aspherical branching. 

If $(\alpha,\beta)$ is an overlapping branching, then the source of $\alpha$ must contain an $\eta_f$ or an $\epsilon_f$. The only cells of length $1$ in $\E_4^\w$ that satisfy this property are those in $\R^\w$. Hence $\alpha$ is in $\R^\w$. Thus the branching $(\alpha,\beta)$ is one of the following two, and both of them satisfy the required property:

\[
\xymatrix {
\twocell{ (eta *0 eta) *1 (1 *0 epsilon *0 1)}
\ar@4 @/^10pt/ [rr] ^-{\twocell{eta *1 (1 *0 tau)}}
\ar@4 @/_10pt/ [rr] _-{\twocell{eta *1 (sigma *0 1)}}
& &
\twocell{eta}
}
\qquad
\xymatrix {
\twocell{(1 *0 eta *0 1) *1 (epsilon *0 epsilon)}
\ar@4 @/^10pt/ [rr] ^-{\twocell{(tau *0 1) *1 epsilon}}
\ar@4 @/_10pt/ [rr] _-{\twocell{(1 *0 sigma) *1 epsilon}}
& &
\twocell{eta}
}
\]

\end{proof}

\begin{lem}
The rewriting system generated by $\R$ is $4$-convergent.
\end{lem}
\begin{proof}
Using Lemma \ref{lem:R_confluence}, the rewriting system generated by $\R$ is locally $4$-confluent. Moreover, the cells $\sigma_f$ and $\tau_f$ decrease the length of the $3$-cells, hence the $4$-termination. 
\end{proof}

Let $A \in \E_3^{\w}$: we denote by $\hat A \in \E_{3}^{\w}$ its normal form for $\R$. Remark in particular that if $A$ is in $\D_3^\w$, then $\hat A = A$.

\begin{lem}\label{lem:norm_3_cell}
Let $A$ be a $3$-cell of $\E_3^\w$ whose target is in $\C_2^\w$. 
\begin{itemize}
\item If the source of $A$ is in $\C_2^\w$, then $\hat A$ is in $\D_3^\w$.
\item Otherwise, for every factorization of $A$ into $f_1 \star_1 \bar f \star_1 f_2$, where $f$ is a rewriting step, there exists a factorisation of $A$ into:

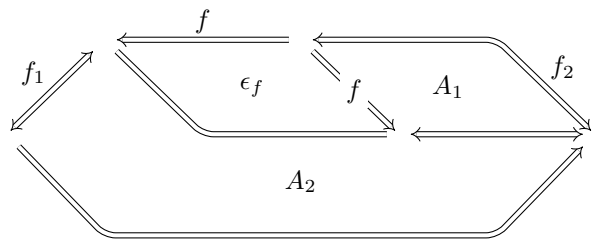
\[
\begin{tikzcd}[nodes in empty cells, 
			  arrows = Rightarrow,
			  column sep = 1cm, 
			  row sep = 1cm]
&
\ar[rrrd,phantom, "\epsilon_f"]
\ar[rrrd, rounded corners, no head, to path = { -- (B.center) -- (\tikztotarget)}] 
& 
& 
\ar[ll, "f"']
\ar[rd, "f" description] 
\ar[rrrd, Leftrightarrow, rounded corners,"f_2", to path = { -- (Z.center)  -- (\tikztotarget) \tikztonodes}] 
\ar[rrrd, phantom, "A_1"] 
& &
|[alias = Z]| 
& \\
\ar[rrrrrr, rounded corners, to path = { -- (Z1.base) -- (Z2.base) -- (\tikztotarget)}, "f_2"]
\ar[ru,Leftrightarrow, "f_1"] 
&  & 
|[alias = B]|
 & &
\ar[rr, Leftrightarrow]
& & \\
\ar[rrrrrru, phantom, "A_2"] 
& |[alias=Z1]| &  &  & & |[alias=Z2]| &  
\end{tikzcd}
\]

\end{itemize}
\end{lem}
\begin{proof}
We reason by induction on the length of $A$. If $A$ is of length $0$, then the source of $A$ is in $\C_2^\w$, and $\hat A = A $ is in $\D_3^\w$.

If $A$ is of length $n>0$, let us write $A = B_1 \star_1 B_2$, where $B_1$ is of length $1$. We can then apply the induction hypothesis to $B_2$. We distinguish three cases:
\begin{itemize}
\item If both the sources of $A$ and $B_2$ are in $\C_2^\w$, then $B_1$ is in $\D_3^\w$, and so is $\hat A = B_1 \star_2 \hat B_2$.
\item If the source of $A$ is in $\C_2^\w$ but not that of $B_2$, then $B_1$ is of the form $g_1 \star_1 \eta_f \star_1 g_2$. There hence exists a factorisation $(g_1 \star_1 f) \star_1 \bar f \star_1 g_2$ of the source of $B_2$. Applying the induction hypothesis to $B_2$, we deduce the following factorisation of $A$:

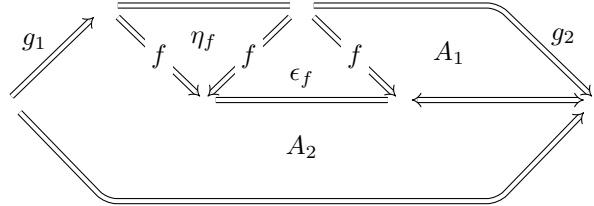
\[
\begin{tikzcd}[nodes in empty cells, 
			  arrows = Rightarrow,
			  column sep = 1cm, 
			  row sep = 1cm]
&
\ar[rd, "f" description] 
\ar[rr, no head]
& 
\ar[d, phantom, "\eta_f", near start]
& 
\ar[ld, "f" description]
\ar[rd, "f" description] 
\ar[d, phantom, "\epsilon_f", near end]
\ar[rrrd, rounded corners,"g_2", to path = { -- (Z.center)  -- (\tikztotarget) \tikztonodes}] 
\ar[rrrd, phantom, "A_1"] 
& &
|[alias = Z]| 
& \\
\ar[rrrrrr, rounded corners, to path = { -- (Z1.base) -- (Z2.base) -- (\tikztotarget)}]
\ar[ru, "g_1"] 
&  & 
\ar[rr, no head]
& &
\ar[Leftrightarrow, rr]
& & \\
\ar[rrrrrru, phantom, "A_2"] 
& |[alias=Z1]| &  &  & & |[alias=Z2]| &  
\end{tikzcd}
\]

In particular, $A$ is the source of an $\R$-rewriting step. Let $A'$ be its target, which is thus of length smaller than $A$. Applying the  induction hypothesis to $A'$, we get that $\hat A = \hat A'$ is in $\D_3^\w$.
\item There remains the case where the source of $A$ is not an element of $\C_2^\w$. 
\end{itemize}

In order to treat this last case, let us fix a factorisation  $f_1 \star_1 \bar f \star_1 f_2$ of the source of $A$, where $f$ is of length $1$. We distinguish three cases depending on the form of $B_1$.
\begin{itemize}
\item If $B_1 = f_1 \star_1 \bar f \star_1 B'_1$, where $B'_1$ is a $3$-cell of length $1$ from $f_2$ to $g_2 \in \D_2^\w$, then we get a factorisation of the source of $B$ into $f_1 \star_1 \bar f \star_1 g_2$. Let us apply the induction hypothesis to $B_2$: there exist $A'_1, A'_2 \in \E_3^\w$ and $g'_2 \in \D_2^\w$ such that: 
\[
B_2 = (f_1 \star_1 \bar f \star_1 A'_1) \star_2 (f_1 \star_1 \epsilon_f \star_1 g'_2) \star_2 A'_2
\]
Thus $A$ factorises as follows, which is of the required form by setting $A_1 = B'_1 \star_2 A'_1$ and $A_2 = A'_2$:

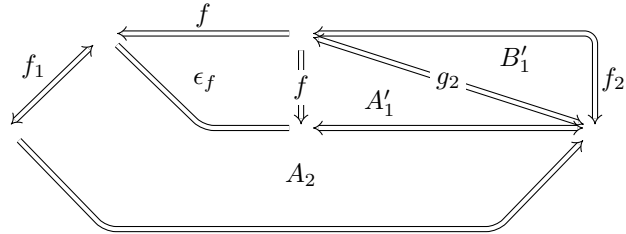
\[
\begin{tikzcd}[nodes in empty cells, 
			  arrows = Rightarrow,
			  column sep = 1cm, 
			  row sep = 1cm]
&
\ar[rrd, no head, rounded corners, to path = {-- (B.center) -- (\tikztotarget)}] 
\ar[rrd, "\epsilon_f", phantom]
& 
& 
\ar[ll, "f"']
\ar[d, "f" description]
\ar[rrrd, Leftrightarrow, rounded corners,"f_2", to path = { -- (Z.center)  -- (\tikztotarget) \tikztonodes}] 
\ar[rrrd, Leftrightarrow, "g_2" description]
& &
& |[alias = Z]|  \\
\ar[rrrrrr, rounded corners, to path = { -- (Z1.base) -- (Z2.base) -- (\tikztotarget)}]
\ar[ru, Leftrightarrow, "f_1"] 
&  & 
|[alias = B]|
& 
\ar[Leftrightarrow, rrr]
\ar[rrru, "A'_1" near start, "B'_1" near end, phantom]
&
& & \\
\ar[rrrrrru, phantom, "A_2"] 
& |[alias=Z1]| &  &  & & |[alias=Z2]| &  
\end{tikzcd}
\]

\item If $B_1 = B'_1 \star_1 \bar f \star_1 f_2$, where $B'_1$ is a $3$-cell of length $1$ from $f_1$ to $g_1 \in \D_2^\w$. Then the source of $B$ factorises into $g_1 \star_1 \bar f \star_1 f_2$. Applying the induction hypothesis to  $B_2$, there exist $A'_1, A'_2 \in \E_3^\w$ and $f'_2 \in \D_2^\w$ such that: 
\[
B_2 = (g_1 \star_1 \bar f \star_1 A'_1) \star_2 (g_1 \star_1 \epsilon_f \star_1 f'_2) \star_2 A'_2
\]
We get the required factorisation of $A$ by setting $A_1 = A'_1$ and $A_2 = (B'_1 \star_1 f'_2) \star_2 A'_2$.

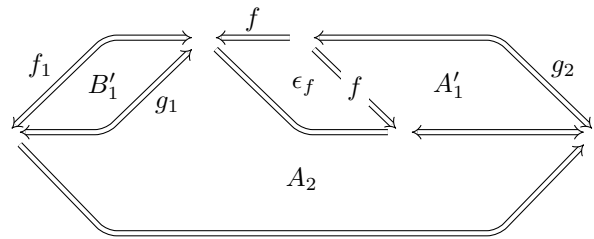
\[
\begin{tikzcd}[nodes in empty cells, 
			  arrows = Rightarrow,
			  column sep = 1cm, 
			  row sep = 1cm]
&
|[alias = Z0]| 
& 
\ar[rrd, no head, rounded corners, to path = {-- (B.center) -- (\tikztotarget)}]
\ar[rrd, phantom, "\epsilon_f"]
& 
\ar[l, "f"']
\ar[rd, "f" description] 
\ar[rrrd, rounded corners,Leftrightarrow, "g_2", to path = { -- (Z.center)  -- (\tikztotarget) \tikztonodes}] 
\ar[rrrd, phantom, "A'_1"] 
& &
|[alias = Z]| 
& \\
\ar[rrrrrr, rounded corners, to path = { -- (Z1.base) -- (Z2.base) -- (\tikztotarget)}]
\ar[rru, Leftrightarrow, "f_1", rounded corners, to path = { -- (Z0.center) \tikztonodes -- (\tikztotarget) }] 
\ar[rru, Leftrightarrow, "g_1"', rounded corners, to path = { -- (A.center) -- (\tikztotarget) \tikztonodes }] 
\ar[rru, "B'_1", phantom]
& 
|[alias = A]| 
& 
& 
|[alias = B]|
&
\ar[Leftrightarrow, rr]
& & \\
\ar[rrrrrru, phantom, "A_2"] 
& |[alias=Z1]| &  &  & & |[alias=Z2]| &  
\end{tikzcd}
\]

\item Otherwise, we have $B_1 = f_1 \star_1 \epsilon_f \star_1 f'_2$, with $f_2 = f \star_1 f'_2$. We then get the required factorisation of $A$ by setting $A_1 = 1_{f'_2}$ and $A_2 = B_2$.

\[
\begin{tikzcd}[nodes in empty cells, 
			  arrows = Rightarrow,
			  column sep = 1cm, 
			  row sep = 1cm]
&
\ar[rrrr, no head, rounded corners, to path = { -- (A1.base) -- (A2.base) -- (\tikztotarget)}]
& 
\ar[rrd, phantom, "\epsilon_f"]
& 
\ar[ll, "f"']
\ar[rr, "f"]
& 
&
\ar[rd, Leftrightarrow, "f'_2"] 
& \\
\ar[rrrrrr, rounded corners, to path = { -- (Z1.base) -- (Z2.base) -- (\tikztotarget)}]
\ar[ru, Leftrightarrow, "f_1"] 
& 
& 
|[alias = A1]|
& &
|[alias = A2]|
& & \\
\ar[rrrrrru, phantom, "B_2"] 
& |[alias=Z1]| &  &  & & |[alias=Z2]| &  
\end{tikzcd}
\]

\end{itemize}
\end{proof}

\begin{lem}\label{lem:simpl_4_cell}
Let $\beta \in \R^\w$, and $\alpha$ be a $4$-cell $\E^\w_4$ of same source. There exist $\alpha'\in \E^\w_4$ and $\beta' \in \R^\w$ of maximum length that of $\beta$ such that we have the following square:

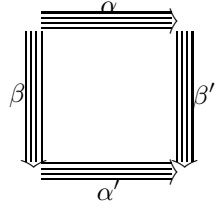
\[
\xymatrix @R=5em @C=5em{
\ar@4 [r] ^{\alpha} 
\ar@4 [d] _{\beta} 
&
\ar@4 [d] ^{\beta'} 
\\
\ar@4 [r] _{\alpha'} &
}
\]

\end{lem}
\begin{proof}
We reason using a double induction on the lengths of $\beta$ and $\alpha$. If $\beta$ (resp. $\alpha$) is an identity, then the result holds by setting $\alpha' = \alpha$ (resp. $\beta' = \beta$).

Otherwise, let us write $\alpha = \alpha_1 \star_3 \alpha_2$ and $\beta = \beta_1 \star_3 \beta_2$, where $\alpha_1$ and $\beta_1$ are of length $1$. We can then construct the following diagram:

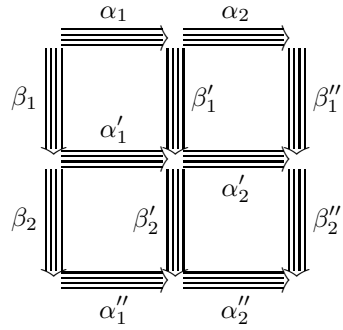
\[
\xymatrix @R=4em @C=4em{
\ar@4 [r] ^*+{\alpha_1}
\ar@4 [d] _*+{\beta_1}
&
\ar@4 [r] ^*+{\alpha_2}
\ar@4 [d] ^*+{\beta'_1}
&
\ar@4 [d] ^*+{\beta''_{1}}
\\
\ar@4 [r] ^*+{\alpha'_1}
\ar@4 [d] _*+{\beta_2}
&
\ar@4 [r] _*+{\alpha'_2}
\ar@4 [d] _*+{\beta'_2}
&
\ar@4 [d] ^*+{\beta''_{2}}
\\
\ar@4 [r] _*+{\alpha''_1}
&
\ar@4 [r] _*+{\alpha''_2}
&
}
\]
The $4$-cells $\alpha'_1$ and $\beta'_1$ exist thanks to Lemma \ref{lem:R_confluence}. We can then apply the induction hypothesis to the $4$-cells $\alpha_2$ and $\beta'_1$ (resp. $\alpha'_1$ and $\beta_2$) and we construct this way the cells $\alpha'_2$ and $\beta''_1$ (resp. $\alpha''_1$ and $\beta'_2$). Lastly, we apply the induction hypothesis to $\alpha'_2$ et $\beta'_2$ in order to construct $\alpha''_2$ and $\beta''_2$.
\end{proof}

\begin{lem}
The application $A \mapsto \hat A$ extends into a $1$-functor $K : \E^{\w(3)} \restriction S_\E \to \D^{\w(3)} \restriction S_\D$, which is the identity on objects.
\end{lem}
\begin{proof}
The application $A \mapsto \hat A$ does not change the source or target. Moreover, given a $3$-cell $A \in \E^{\w(3)}$, if $A$ is in $\E^{\w(3)} \restriction S_\E$ then in particular the source and target of $A$ are in $\C_2^\w$. Thus $\hat A$ is in $\D_3^\w \restriction S_\D$ (Lemma \ref{lem:norm_3_cell}).

Let $A$, $B$ be $3$-cells in $\E^{\w(3)}$ which belong to $\E^{\w(3)} \restriction S_\E$. We just showed that $\hat A$ and $\hat B$ are in $\D_3^\w \restriction S_\D$, hence so is $\hat A \star_2 \hat B$. So $\hat A \star_2 \hat B$ is a normal form for $\R$ which is attainable from $A \star_2 B$. Since $\R$ is $4$-convergent, this means that $\widehat{A \star_2 B} = \hat A \star_2 \hat B$. So $A \mapsto\hat A$ does indeed define a functor.
\end{proof}

\begin{prop}
The pointed $(4,3)$-category $(\E,S_\E)$ is stronger than $(\D,S_\D)$.
\end{prop}
\begin{proof}
Let us show that $K$ induces a functor $\bar K: \overline{\E^{\w(3)} \restriction S_\E} \to \overline{\D^{\w(3)} \restriction S_\D}$. Let $A,B$ be $1$-cells in $\E^{\w(3)} \restriction S_\E$, and suppose $\bar A = \bar B$. Let us show that $\overline{K(A)} = \overline{K(B)}$, that is that there exists a $4$-cell $\alpha' : \hat A \qfl \hat B \in \D^{\w(3)}_4$.

Since $\bar A = \bar B$ there exists a $4$-cell $\alpha:A  \qfl B \in \E_4^{\w(3)}$. Suppose that $\alpha$ lies in $\E_4^{\w}$. Let $\beta \in \R^\w$ be a cell from $A$ to $\hat A$. Applying Lemma \ref{lem:simpl_4_cell} to $\alpha$ and $\beta$, we get cells $\alpha'$ and $\beta'$ of sources respectively $\hat A$ and $B$. Let $B'$ be their common target. By hypothesis $\hat A$ is in $\D_3^\w$, and the only cells in $\E_4^\w$ whose source is in $\D_3^\w$ are the cells in $\D_4^\w$. Thus $\alpha'$ is in $\D_4^\w$, and so is $B'$. So $B'$ is a normal form for $\R^\w$ which is attainable from $B$. By unicity of the $\R^\w$-normal-form, $B' = \hat B$, and so $\alpha'$ is a cell in $\D_4^\w$ of source $K(A)$ and of target $K(B)$, hence $\overline{K(A)} = \overline{K(B)}$.

In general if $\bar A = \bar B$, there exist $A_1,\ldots ,A_n \in \E_3^{\w}$ with $A_1 = A$, $B_n = B$ and for every $i$ there exist cells $\alpha_i : A_{2i} \qfl A_{2i-1}$ and $\beta_i : A_{2i} \to A_{2i+1}$ in $\E_4^\w$. Hence using the previous case $\overline{K(A_1)} = \ldots = \overline{K(A_n)}$, that is $K(A_1) = K(A_n)$. 

So $\bar K : \overline{\E^{\w(3)} \restriction S_\E} \to \overline{\D^{\w(3)} \restriction S_\D}$ is well defined, and it is $0$ and $1$-surjective because $K$ is. Hence $(\E,S_\E)$ is stronger than $(\D,S_\D)$.
\end{proof}

\begin{ex}
In the case where $\A = \Ass$, let $A = \twocell{prodC}$. The set $\E_3$ contains the following $3$-cells:
\[
\xymatrix{
\twocell{inv_prodC *1 prodC}
\ar@3 [r] ^{\epsilon_A}
&
\twocell{1}}
\qquad
\xymatrix{
\twocell{2}
\ar@3 [r] ^{\eta_A}
&
\twocell{prodC *1 inv_prodC}}
\]

And the set $\E_4$ the following $4$-cells:
\[
\xymatrix{
& 
\twocell{inv_prodC *1 prodC *1 inv_prodC}
\ar@3 [rd] ^{\epsilon_A}
\ar@4 [d] ^-{\tau_A}
&
\\
\twocell{inv_prodC}
\ar@3 [ru] ^{\eta_A}
\ar@{=} [rr]
& &
\twocell{inv_prodC}}
\qquad
\xymatrix{
& 
\twocell{prodC *1 inv_prodC *1 prodC}
\ar@3 [rd] ^{\epsilon_A}
\ar@4 [d] ^-{\sigma_A}
&
\\
\twocell{prodC}
\ar@3 [ru] ^{\eta_A}
\ar@{=} [rr]
& &
\twocell{prodC}}
\]
\end{ex}

\subsection{Reversing the presentation of a $(4,3)$-white-category}\label{subsec:retournement}

We start by collecting some results on the cells of $\E$.

\begin{lem}
The set $\E_3$ is composed exactly of the following cells:
\begin{itemize}
\item For every $f \in \A_2$, $3$-cells $\eta_f$ and $\epsilon_f$.
\item For every non-aspherical minimal branching $(f,g)$, a $3$-cell $A_{f,g}$ of shape:
\[
\xymatrix @C = 4em @R = 1.5em{
& 
\ar@2 @/^/ [rd] ^{ f' }
\ar@3 [dd] ^{A_{f,g}}
& \\
\ar@2 @/^/ [ru] ^{f}
\ar@2 @/_/ [rd] _{g} 
& 
& \\
&  
\ar@2 @/_/ [ru] _{g'}
& \\
}
\]
\end{itemize}
And in particular for every non-aspherical minimal branching $(f,g)$, we have $A_{f,g}^{op} = A_{g,f}$.
\end{lem}
\begin{proof}
If $(f,g)$ is a critical pair: if it was associated to a $3$-cell in $\A$ then $A_{f,g}$ is this corresponding cell. Otherwise $A_{f,g}$ is in fact the cell $A_{g,f}^{op}$ from Section \ref{subsec:weak_inverse_3}.

If $(f,g)$ is a strict Peiffer branching, then $A_{f,g}$ is the cell defined in Section \ref{subsec:affaiblissement}. Otherwise, $(g,f)$ is a strict Peiffer branching, and we set $A_{f,g} := A_{g,f}^{op}$ from Section \ref{subsec:weak_inverse_3}.
\end{proof}

\begin{lem}\label{lem:triplet_remplis}
For every minimal non-aspherical branching $(f,g,h)$, there exists a $4$-cell $A_{f,g,h} \in \E_4^{\w(3)}$ of the following shape:
\[
\xymatrix @C = 3em @R = 3em {
&
\ar@2 [rr]
\ar@{} [rd] |-{A_{f,g}}
& &
\ar@2 [rd]
\ar@{} [dd] |-{A}
& & &
\ar@2 [rr]
\ar@{} [rrrd] |-{B_1}
\ar@2 [rd]
\ar@{} [dd] |-{A_{f,h}}
& &
\ar@2 [rd]
&
\\
\ar@2 [ru] ^{f}
\ar@2 [rr] |{g}
\ar@2 [rd] _{h}
& &
\ar@{} [ld] |-{A_{g,h}}
\ar@2 [ru]
\ar@2 [rd]
&
&
\ar@4 [r] _{A_{f,g,h}}
&
\ar@2 [ru] ^{f}
\ar@2 [rd] _{h}
& & 
\ar@2 [rr]
\ar@{} [rd] |-{B_2}
& &
\\
&
\ar@2 [rr]
& &
\ar@2 [ru]
& & &
\ar@2 [ru]
\ar@2 [rr]
& &
\ar@2 [ru]
&
}
\]
\end{lem}
\begin{proof}
Let us first start by showing that, for every non-aspherical $3$-fold minimal symmetrical branching $b$, there exists a representative $(f,g,h)$ of $b$ for which the property holds. If $b$ is an overlapping branching then, using the fact that $\A$ satisfies the $2$-Squier condition of depth $2$, the cell $A_{f,g,h}$ exists for some representative $(f,g,h)$ of $b$.
Otherwise $b$ is a Peiffer branching, and we conclude using the cells defined in Section \ref{subsec:affaiblissement}.

It remains to show that the set of all branchings satisfying the property is closed under the action of the symmetric group.
\begin{itemize}
\item If $(f_1,f_2,f_3)$ satisfies the property, then so does $(f_3,f_2,f_1)$. Indeed, let $A := A_{f_1,f_2,f_3}$, and let us denote its source by $s$ and its target by $t$, all we need to construct is a $4$-cell from $s^{op}$ to $t^{op}$. This is given by the following composite:
\[
\xymatrix @C = 6em @R = 3em {
s^{op}
\ar@4 [r] _-*+{s^{op} \star_2 \lambda_{t}^{-1}}
&
t^{op} \star_2 t \star_2 s^{op}
\ar@4 [r] _-*+{t^{op} \star_2 A^{-1} \star_2 s^{op}}
&
t^{op} \star_2 s \star_2 s^{op}
\ar@4 [r] _-*+{t^{op} \star_2 \rho_s}
&
t^{op}
}
\]

\item If $(f_1,f_2,f_3)$ satisfies the property, then so does $(f_2,f_1,f_3)$. Indeed, given a cell $A_{f_1,f_2,f_3}$, we can construct the following composite: 
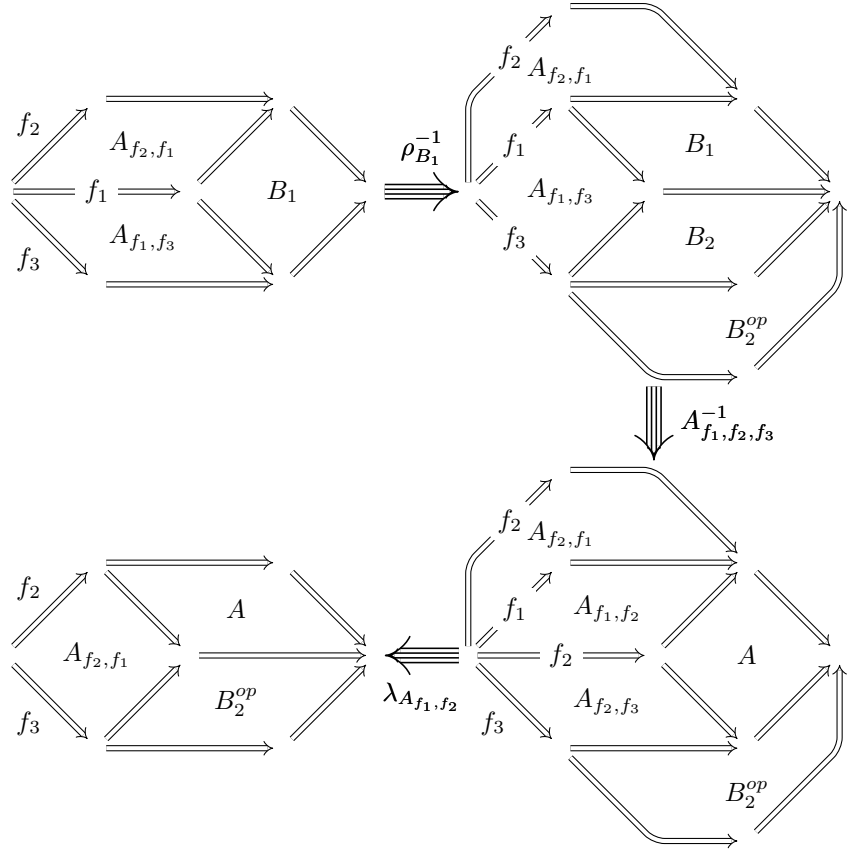
\[
\begin{tikzpicture}

\matrix (m) [matrix of math nodes, 
			nodes in empty cells,
			column sep = 1cm, 
			row sep = 1cm] 
{
& & & & & & & & & \\
& & & & & & & & & \\
& & & & & & & & & \\
& & & & & & & & & \\
& & & & & & & & & \\
& & & & & & & & & \\
& & & & & & & & & \\
& & & & & & & & & \\
& & & & & & & & & \\
& & & & & & & & & \\
}; 
\doublearrow{arrows={-Implies}}
{(m-3-1) -- node [above left] {$f_2$} (m-2-2)}
\doublearrow{arrows={-Implies}}
{(m-3-1) -- node [fill = white] {$f_1$} (m-3-3)}
\doublearrow{arrows={-Implies}}
{(m-3-1) -- node [below left] {$f_3$} (m-4-2)}
\doublearrow{arrows={-Implies}}
{(m-2-2) -- (m-2-4)}
\doublearrow{arrows={-Implies}}
{(m-4-2) -- (m-4-4)}
\doublearrow{arrows={-Implies}}
{(m-3-3) -- (m-2-4)}
\doublearrow{arrows={-Implies}}
{(m-3-3) -- (m-4-4)}
\doublearrow{arrows={-Implies}}
{(m-2-4) -- (m-3-5)}
\doublearrow{arrows={-Implies}}
{(m-4-4) -- (m-3-5)}
\path (m-3-1) -- node {$A_{f_2,f_1}$} (m-2-4);
\path (m-3-1) -- node {$A_{f_1,f_3}$} (m-4-4);
\path (m-3-3) -- node {$B_1$} (m-3-5);

\quadarrow{arrows={-Implies}}
{(m-3-5) -- node [label={[label distance=.1cm]90:$\rho_{B_1}^{-1}$}] {} (m-3-6)}

\doublearrow{arrows={-Implies}, rounded corners}
{(m-3-6) -- (m-2-6.center) -- node [fill = white] {$f_2$} (m-1-7)}
\doublearrow{arrows={-Implies}, rounded corners}
{(m-1-7) -- (m-1-8.center) -- (m-2-9)}
\doublearrow{arrows={-Implies}, rounded corners}
{(m-4-7) -- (m-5-8.center) -- (m-5-9)}
\doublearrow{arrows={-Implies}, rounded corners}
{(m-5-9) -- (m-4-10.center) -- (m-3-10)}
\doublearrow{arrows={-Implies}}
{(m-3-6) -- node [fill = white] {$f_1$} (m-2-7)}
\doublearrow{arrows={-Implies}}
{(m-3-6) -- node [fill = white] {$f_3$} (m-4-7)}
\doublearrow{arrows={-Implies}}
{(m-2-7) -- (m-2-9)}
\doublearrow{arrows={-Implies}}
{(m-2-7) -- (m-3-8)}
\doublearrow{arrows={-Implies}}
{(m-4-7) -- (m-3-8)}
\doublearrow{arrows={-Implies}}
{(m-4-7) -- (m-4-9)}
\doublearrow{arrows={-Implies}}
{(m-3-8) -- (m-3-10)}
\doublearrow{arrows={-Implies}}
{(m-2-9) -- (m-3-10)}
\doublearrow{arrows={-Implies}}
{(m-4-9) -- (m-3-10)}
\path (m-1-7) -- node [near end] {$A_{f_2,f_1}$} (m-2-7);
\path (m-3-6) -- node {$A_{f_1,f_3}$} (m-3-8);
\path (m-2-7) -- node {$B_1$} (m-3-10);
\path (m-4-7) -- node {$B_2$} (m-3-10);
\path (m-4-9) -- node {$B_2^{op}$} (m-5-9);

\quadarrow{arrows={-Implies}}
{(m-5-8) -- node [label={[label distance=.1cm]0:$A^{-1}_{f_1,f_2,f_3}$}] {} (m-6-8)}

\doublearrow{arrows={-Implies}}
{(m-8-1) -- node [above left] {$f_2$} (m-7-2)}
\doublearrow{arrows={-Implies}}
{(m-8-1) -- node [below left] {$f_3$} (m-9-2)}
\doublearrow{arrows={-Implies}}
{(m-7-2) -- (m-7-4)}
\doublearrow{arrows={-Implies}}
{(m-7-2) -- (m-8-3)}
\doublearrow{arrows={-Implies}}
{(m-9-2) -- (m-8-3)}
\doublearrow{arrows={-Implies}}
{(m-9-2) -- (m-9-4)}
\doublearrow{arrows={-Implies}}
{(m-8-3) -- (m-8-5)}
\doublearrow{arrows={-Implies}}
{(m-7-4) -- (m-8-5)}
\doublearrow{arrows={-Implies}}
{(m-9-4) -- (m-8-5)}
\path (m-8-1) -- node {$A_{f_2,f_1}$} (m-8-3);
\path (m-7-2) -- node {$A$} (m-8-5);
\path (m-9-2) -- node {$B_2^{op}$} (m-8-5);

\doublearrow{arrows={-Implies}, rounded corners}
{(m-8-6) -- (m-7-6.center) -- node [fill = white] {$f_2$} (m-6-7)}
\doublearrow{arrows={-Implies}, rounded corners}
{(m-6-7) -- (m-6-8.center) -- (m-7-9)}
\doublearrow{arrows={-Implies}, rounded corners}
{(m-9-7) -- (m-10-8.center) -- (m-10-9)}
\doublearrow{arrows={-Implies}, rounded corners}
{(m-10-9) -- (m-9-10.center) -- (m-8-10)}
\doublearrow{arrows={-Implies}}
{(m-8-6) -- node [fill = white] {$f_1$} (m-7-7)}
\doublearrow{arrows={-Implies}}
{(m-8-6) -- node [fill = white] {$f_2$} (m-8-8)}
\doublearrow{arrows={-Implies}}
{(m-8-6) -- node [below left] {$f_3$} (m-9-7)}
\doublearrow{arrows={-Implies}}
{(m-7-7) -- (m-7-9)}
\doublearrow{arrows={-Implies}}
{(m-9-7) -- (m-9-9)}
\doublearrow{arrows={-Implies}}
{(m-8-8) -- (m-7-9)}
\doublearrow{arrows={-Implies}}
{(m-8-8) -- (m-9-9)}
\doublearrow{arrows={-Implies}}
{(m-7-9) -- (m-8-10)}
\doublearrow{arrows={-Implies}}
{(m-9-9) -- (m-8-10)}
\path (m-6-7) -- node [near end] {$A_{f_2,f_1}$} (m-7-7);
\path (m-8-6) -- node {$A_{f_1,f_2}$} (m-7-9);
\path (m-8-6) -- node {$A_{f_2,f_3}$} (m-9-9);
\path (m-8-8) -- node {$A$} (m-8-10);
\path (m-9-9) -- node {$B_2^{op}$} (m-10-9);

\quadarrow{arrows={-Implies}}
{(m-8-6) -- node [label={[label distance=.1cm]-90:$\lambda_{A_{f_1,f_2}}$}] {} (m-8-5)}

\end{tikzpicture}
\]
\end{itemize}

Since the transpositions $(1 \quad 2)$ and $(1 \quad 3)$ generate the symmetric group, the set of all branchings satisfying the property is closed under the action of the symmetric group.
\end{proof}

We are now going to apply a series of Tietze-transformations to $\E$ in order to mimic a technique known as \emph{reversing}. Reversing is a combinatorial tool to study presented monoids \cite{D11}. Reversing is particularly adapted to monoids whose presentation contains no relation of the form $su = sv$, where $s$ is a generator and $u$ and $v$ words in the free monoid, and at most one relation of the form $su = s'v$, for $s$ and $s'$ generators. The $(4,2)$-polygraph $\A$ satisfies those properties, but only up to a dimensional shift: there are no $3$-cell in $\A_3$ of the form $f \star_2 g \Rrightarrow f \star_2 h$, where $f$ is of length $1$ and $g$ and $h$ are in $\A_2^\w$, and there is at most one $3$-cell in $\A_3$ of the form $f \star_2 g \Rrightarrow f' \star_2 h$, where $f$ and $f'$ are of length $1$. Hence we adapt this method to our higher-dimensional setting.

\paragraph*{Adjunction of $3$-cells $C_{f,g}$ with its defining $4$-cell $X_{f,g}$.}
For every non-aspherical branching $(f,g)$, we add a $3$-cell $C_{f,g}$ of the following shape:
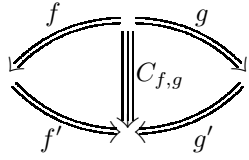
\[
\xymatrix @C = 4em @R = 1.5em{
& 
\ar@2 @/^/ [rd] ^{ g}
\ar@2 @/_/ [ld] _{f}
\ar@3 [dd] ^{C_{f,g}}
& \\
\ar@2 @/_/ [rd] _{f'}
& 
& 
\ar@2 @/^/ [ld] ^{g'} 
\\
&  
& \\
}
\]
using as defining $4$-cell a cell $X_{f,g}$ whose target is $C_{f,g}$ and whose source is the composite:
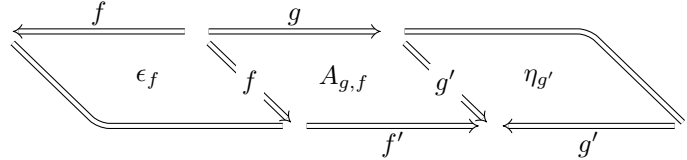
\[
\begin{tikzcd}[nodes in empty cells, 
			  arrows = Rightarrow,
			  column sep = 1cm, 
			  row sep = 1cm]
\ar[rrrd,rounded corners, no head, to path = { -- (B.center)  -- (\tikztotarget) \tikztonodes}]
\ar[rrrd, phantom, "\epsilon_f" description]
& & 
\ar[ll,"f"']
\ar[rr,"g"]
\ar[rd,"f" description]
\ar[rrrd, phantom, "A_{g,f}" description]
& & 
\ar[rd,"g'" description]
\ar[rrrd,rounded corners, no head, to path = { -- (A.center)  -- (\tikztotarget) \tikztonodes}]
\ar[rrrd, phantom, "\eta_{g'}" description]
& & 
|[alias = A]| 
& \\
& 
|[alias = B]| 
& & 
\ar[rr,"f'"']
& & & &
\ar[ll,"g'"]
\end{tikzcd}
\]

\paragraph*{Adjunction of a superfluous $4$-cell $Y_{f,g}$.}
We add a $4$-cell $Y_{f,g}$ of target $A_{g,f}$, parallel to the following $4$-cell (where the second step consists in the parallel application of $\sigma_f$ and $\sigma_{g'}$):
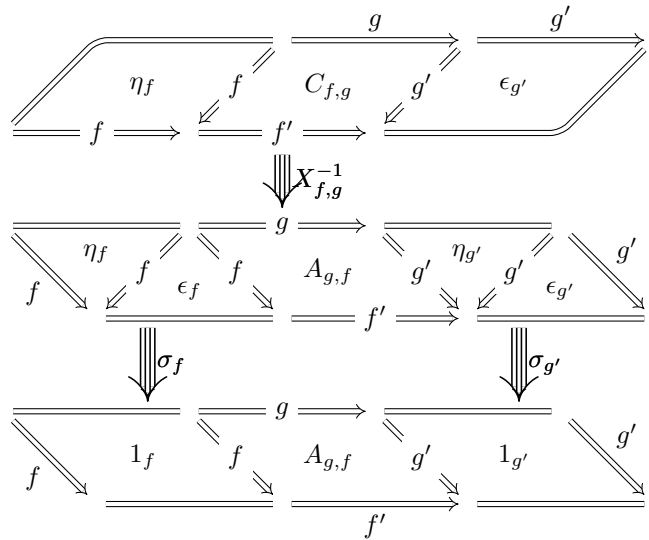
\[
\begin{tikzpicture}
\matrix (m) [matrix of math nodes, 
			nodes in empty cells,
			column sep = 1cm, 
			row sep = 1cm] 
{
& & & & & & & \\
& & & & & & & \\
& & & & & & & \\
& & & & & & & \\
& & & & & & & \\
& & & & & & & \\
};
\doublearrow{arrows ={-Implies}}
{(m-2-1) -- node [fill = white] {$f$} (m-2-3)}

\doublearrow{arrows ={-Implies}}
{(m-2-3) -- node (s41) [fill = white] {$f'$} (m-2-5)}

\doublearrow{arrows ={-}, rounded corners}
{(m-2-5) --  (m-2-7.center) -- (m-1-8)}

\doublearrow{arrows ={-}, rounded corners}
{(m-2-1) --  (m-1-2.center) -- (m-1-4)}

\doublearrow{arrows ={-Implies}}
{(m-1-4) -- node [above] {$g$} (m-1-6)}

\doublearrow{arrows ={-Implies}}
{(m-1-6) -- node [above] {$g'$} (m-1-8)}

\doublearrow{arrows ={-Implies}}
{(m-1-4) -- node [fill = white] {$f$} (m-2-3)}

\doublearrow{arrows ={-Implies}}
{(m-1-6) -- node [fill = white] {$g'$} (m-2-5)}

\path (m-2-1) -- node {$\eta_f$} (m-1-4);
\path (m-1-4) -- node {$C_{f,g}$} (m-2-5);
\path (m-1-6) -- node {$\epsilon_{g'}$} (m-2-7);

\doublearrow{arrows ={-}}
{(m-3-1) -- (m-3-3)}

\doublearrow{arrows ={-Implies}}
{(m-3-3) -- node (t41) [fill = white] {$g$} (m-3-5)}

\doublearrow{arrows ={-}}
{(m-3-5) -- (m-3-7)}

\doublearrow{arrows ={-Implies}}
{(m-3-1) to node [below left] {$f$} (m-4-2)}

\doublearrow{arrows ={-Implies}}
{(m-3-3) -- node [fill = white] {$f$} (m-4-2)}

\doublearrow{arrows ={-Implies}}
{(m-3-3) -- node [fill = white] {$f$} (m-4-4)}

\doublearrow{arrows ={-Implies}}
{(m-3-5) -- node [fill = white] {$g'$} (m-4-6)}

\doublearrow{arrows ={-Implies}}
{(m-3-7) -- node [fill = white] {$g'$} (m-4-6)}

\doublearrow{arrows ={-Implies}}
{(m-3-7) to node [above right] {$g'$} (m-4-8)}

\doublearrow{arrows ={-}}
{(m-4-2) -- node (s42) [near start] {} (m-4-4)}

\doublearrow{arrows ={-Implies}}
{(m-4-4) -- node [fill = white] {$f'$} (m-4-6)}

\doublearrow{arrows ={-}}
{(m-4-6) -- node (s43) [near start] {} (m-4-8)}

\path (m-3-2) -- node [near start] {$\eta_f$} (m-4-2);
\path (m-3-3) -- node [near end] {$\epsilon_f$} (m-4-3);
\path (m-3-3) -- node {$A_{g,f}$} (m-4-6);
\path (m-3-6) -- node [near start] {$\eta_{g'}$} (m-4-6);
\path (m-3-7) -- node [near end] {$\epsilon_{g'}$} (m-4-7);

\doublearrow{arrows ={-}}
{(m-5-1) -- node (t42) [near end] {} (m-5-3)}

\doublearrow{arrows ={-Implies}}
{(m-5-3) -- node [fill = white] {$g$} (m-5-5)}

\doublearrow{arrows ={-}}
{(m-5-5) -- node (t43) [near end] {} (m-5-7)}

\doublearrow{arrows ={-Implies}}
{(m-5-1) to node [below left] {$f$} (m-6-2)}

\doublearrow{arrows ={-Implies}}
{(m-5-3) -- node [fill = white] {$f$} (m-6-4)}

\doublearrow{arrows ={-Implies}}
{(m-5-5) -- node [fill = white] {$g'$} (m-6-6)}

\doublearrow{arrows ={-Implies}}
{(m-5-7) to node [above right] {$g'$} (m-6-8)}

\doublearrow{arrows ={-}}
{(m-6-2) -- (m-6-4)}

\doublearrow{arrows ={-Implies}}
{(m-6-4) -- node [below] {$f'$} (m-6-6)}

\doublearrow{arrows ={-}}
{(m-6-6) -- (m-6-8)}

\path (m-5-1) -- node {$1_f$} (m-6-4);
\path (m-5-3) -- node {$A_{g,f}$} (m-6-6);
\path (m-5-5) -- node {$1_{g'}$} (m-6-8);

\quadarrow{arrows ={-Implies}}
{(s41) -- node [right] {$X_{f,g}^{-1}$} (t41)}

\quadarrow{arrows ={-Implies}}
{(s42) -- node [right] {$\sigma_f$} (t42.north-|s42)}

\quadarrow{arrows ={-Implies}}
{(s43) -- node [right] {$\sigma_{g'}$} (t43.north-|s43)}

\end{tikzpicture}
\]
\paragraph*{Removal of the superfluous $4$-cell $X_{f,g}$.}

We remove the $4$-cell $X_{f,g}$, using the fact that it is parallel to the following composite:
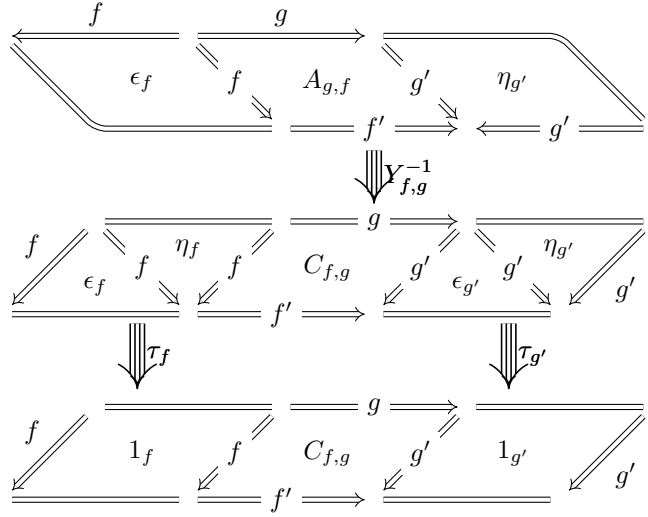
\[
\begin{tikzpicture}
\matrix (m) [matrix of math nodes, 
			nodes in empty cells,
			column sep = 1cm, 
			row sep = 1cm] 
{
& & & & & & & \\
& & & & & & & \\
& & & & & & & \\
& & & & & & & \\
& & & & & & & \\
& & & & & & & \\
};
\doublearrow{arrows = {-Implies}}
{(m-1-3) -- node [above] {$f$} (m-1-1)}

\doublearrow{arrows = {-Implies}}
{(m-1-3) -- node [above] {$g$} (m-1-5)}

\doublearrow{arrows = {-}, rounded corners}
{(m-1-5) -- (m-1-7.center) -- (m-2-8)}

\doublearrow{arrows = {-}, rounded corners}
{(m-1-1) -- (m-2-2.center) -- (m-2-4)}

\doublearrow{arrows = {-Implies}}
{(m-1-3) -- node [fill = white] {$f$} (m-2-4)}

\doublearrow{arrows = {-Implies}}
{(m-1-5) -- node [fill = white] {$g'$} (m-2-6)}

\doublearrow{arrows = {-Implies}}
{(m-2-4) -- node (s41) [fill = white] {$f'$} (m-2-6)}

\doublearrow{arrows = {-Implies}}
{(m-2-8) -- node [fill = white] {$g'$} (m-2-6)}

\path (m-1-1) -- node {$\epsilon_f$} (m-2-4);
\path (m-1-3) -- node {$A_{g,f}$} (m-2-6);
\path (m-1-5) -- node {$\eta_{g'}$} (m-2-8);

\doublearrow{arrows ={-}}
{(m-3-2) -- (m-3-4)}

\doublearrow{arrows ={-Implies}}
{(m-3-4) -- node (t41)  [fill = white] {$g$} (m-3-6)}

\doublearrow{arrows ={-}}
{(m-3-6) -- node (s43) [near start] {} (m-3-8)}

\doublearrow{arrows ={Implies-}}
{(m-4-1) to node [above left] {$f$} (m-3-2)}

\doublearrow{arrows ={Implies-}}
{(m-4-3) -- node [fill = white] {$f$} (m-3-2)}

\doublearrow{arrows ={Implies-}}
{(m-4-3) -- node [fill = white] {$f$} (m-3-4)}

\doublearrow{arrows ={Implies-}}
{(m-4-5) -- node [fill = white] {$g'$} (m-3-6)}

\doublearrow{arrows ={Implies-}}
{(m-4-7) -- node [fill = white] {$g'$} (m-3-6)}

\doublearrow{arrows ={Implies-}}
{(m-4-7) to node [below right] {$g'$} (m-3-8)}

\doublearrow{arrows ={-}}
{(m-4-1) -- node (s42) [near end] {} (m-4-3)}

\doublearrow{arrows ={-Implies}}
{(m-4-3) -- node [fill = white] {$f'$} (m-4-5)}

\doublearrow{arrows ={-}}
{(m-4-5) -- node (s43) [near end] {} (m-4-7)}

\path (m-3-2) -- node [near end] {$\epsilon_f$} (m-4-2);
\path (m-3-3) -- node [near start] {$\eta_f$} (m-4-3);
\path (m-3-3) -- node {$C_{f,g}$} (m-4-6);
\path (m-3-6) -- node [near end] {$\epsilon_{g'}$} (m-4-6);
\path (m-3-7) -- node [near start] {$\eta_{g'}$} (m-4-7);

\doublearrow{arrows ={-}}
{(m-6-1) --  (m-6-3)}

\doublearrow{arrows ={-Implies}}
{(m-6-3) -- node [fill = white] {$f'$} (m-6-5)}

\doublearrow{arrows ={-}}
{(m-6-5) -- (m-6-7)}

\doublearrow{arrows ={Implies-}}
{(m-6-1) to node [above left] {$f$} (m-5-2)}

\doublearrow{arrows ={Implies-}}
{(m-6-3) -- node [fill = white] {$f$} (m-5-4)}

\doublearrow{arrows ={Implies-}}
{(m-6-5) -- node [fill = white] {$g'$} (m-5-6)}

\doublearrow{arrows ={Implies-}}
{(m-6-7) to node [below right] {$g'$} (m-5-8)}

\doublearrow{arrows ={-}}
{(m-5-2) -- (m-5-4)}

\doublearrow{arrows ={-Implies}}
{(m-5-4) -- node (t42) [fill = white] {$g$} (m-5-6)}

\doublearrow{arrows ={-}}
{(m-5-6) -- (m-5-8)}

\path (m-5-1) -- node {$1_f$} (m-6-4);
\path (m-5-3) -- node {$C_{f,g}$} (m-6-6);
\path (m-5-5) -- node {$1_{g'}$} (m-6-8);

\quadarrow{arrows ={-Implies}}
{(s41) -- node [right] {$Y_{f,g}^{-1}$} (t41)}

\quadarrow{arrows ={-Implies}}
{(s42) -- node [right] {$\tau_f$} (t42.north-|s42)}

\quadarrow{arrows ={-Implies}}
{(s43) -- node [right] {$\tau_{g'}$} (t42.north-|s43)}

\end{tikzpicture}
\]

\paragraph*{Removal of the $3$-cell $A_{g,f}$ with its defining $4$-cell $Y_{f,g}$.}

This last step is possible because $A_{g,f}$ is the target of $Y_{f,g}$ and does not appear in its source.

We denote by $\F$ the $4$-white-polygraph obtained after performing this series of Tietze-transformations for every non-aspherical branching $(f,g)$, and $\Pi : \E^{\w(3)} \to \F^{\w(3)}$ the $3$-white-functor induced by the Tietze-transformations. We still denote by $A_{g,f}$ the composite in $\F_4^{\w(3)}$, image by $\Pi$ of $A_{f,g} \in \E_4$.

\begin{ex}
In the case where $\A = \Ass$, the cells $\twocell{assocC}$ and $\twocell{assocC}^{op}$ respectively associated to the branchings $(\twocell{prodC *0 1} \quad , \quad \twocell{1 *0 prodC})$ and $(\twocell{1 *0 prodC} \quad , \quad \twocell{prodC *0 1})$ have been replaced by cells of the following shape:

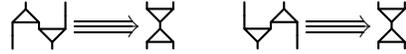
\[
\xymatrix{
\twocell{(inv_prodC *0 1) *1 (1 *0 prodC)}
\ar@3 [r]
&
\twocell{prodC *1 inv_prodC}
}
\qquad
\xymatrix{
\twocell{(1 *0 inv_prodC) *1 (prodC *0 1)}
\ar@3 [r]
&
\twocell{prodC *1 inv_prodC}
}
\]
\end{ex}

\newpage

\section{Proof of Theorem \ref{thm:main_theory}}
\label{sec:proof_final}

This Section concludes the proof of Theorem \ref{thm:main_theory}. We keep the notations from Section \ref{sec:transfo_polygraph}. In Section \ref{subsec:local_coherence}, we study the $4$-cells of the $(4,3)$-white-category $\F^{\w(3)}$, and in particular study the consequences of $\A$ satisfying the $2$-Squier condition of depth $2$.

In Section \ref{subsec:ordering}, we define a well-founded ordering on $\mathbb N[\F_1^\w]$, the free commutative monoid on $\F_1^\w$. Using this ordering together with two applications $\p : \F_2^\w \to \mathbb N[\F_1^\w]$ and $\w_\eta : \F_3^\w \to \mathbb N[\F_1^\w]$, we proceed to complete the proof by induction in Section \ref{subsec:proof_final}.

\subsection{Local coherence}
\label{subsec:local_coherence}

\begin{defn}
We extend the notation $C_{f,g}$ from Section \ref{subsec:retournement} by defining, for every local branching $(f,g)$ of $\B_2^\w$, a $3$-cell of the form $C_{f,g}: \bar f \star_1 g \Rrightarrow f' \star_1 \bar g' \in \F^\w_3$, where $f'$ and $g'$ are in $\B^\w_2$.
\begin{itemize}
\item If $(f,g)$ is a minimal overlapping or Peiffer branching, then $C_{f,g}$ is already defined.
\item If $(f,g)$ is aspherical, that is $f=g$, then we set $C_{f,f} = \epsilon_f$.
\item If $(f,g)$ is not minimal, then let us write $(f,g) = (u \tilde f v, u \tilde g v)$, with $(\tilde f, \tilde g)$ a minimal branching, and we set $C_{f,g} := uC_{\tilde f, \tilde g}v$. 
\end{itemize}
\end{defn}

\begin{defn}
We say that a $3$-fold local branching $(f,g,h)$ of $\A_2$ is \emph{coherent} if there exists a $4$-cell $C_{f,g,h} \in \F_4^{\w(3)}$ of the following shape, where $A$ and $B$ are $4$-cells in $\F_3^{\w}$.

\[
\begin{tikzpicture}
\matrix (m) [matrix of math nodes, 
			nodes in empty cells,
			column sep = 1cm, 
			row sep = .7cm] 
{
& & & & & & & \\
& & & & & & & \\
& & & & & & & \\
& & & & & & & \\
};
\doublearrow{arrows = {-}}
{(m-1-2) -- (m-1-4)}

\doublearrow{arrows = {-Implies}}
{(m-1-2) -- node [above left] {$f$} (m-2-1)}
\doublearrow{arrows = {-Implies}}
{(m-1-2) -- node [fill=white] {$g$} (m-2-3)}
\doublearrow{arrows = {-Implies}}
{(m-1-4) -- node [fill = white] {$g$} (m-2-3)}
\doublearrow{arrows = {-Implies}}
{(m-1-4) -- node [above right] {$h$} (m-2-5)}

\doublearrow{arrows = {-Implies}}
{(m-2-1) -- (m-3-2)}
\doublearrow{arrows = {-Implies}}
{(m-2-3) -- (m-3-2)}
\doublearrow{arrows = {-Implies}}
{(m-2-3) -- (m-3-4)}
\doublearrow{arrows = {-Implies}}
{(m-2-5) -- node (s4) [ below right] {} (m-3-4)}

\doublearrow{arrows = {-Implies}}
{(m-3-2) -- (m-4-3)}
\doublearrow{arrows = {-Implies}}
{(m-3-4) -- (m-4-3)}

\path (m-1-3) -- node [near start] {$\eta_g$} (m-2-3);
\path (m-1-2) -- node  {$C_{f,g}$} (m-3-2);
\path (m-1-4) -- node  {$C_{g,h}$} (m-3-4);
\path (m-2-3) -- node  {$A$} (m-4-3);

\doublearrow{arrows = {-Implies}}
{(m-1-7) -- node [above left] {$f$} (m-2-6)}
\doublearrow{arrows = {-Implies}}
{(m-1-7) -- node [above right] {$h$} (m-2-8)}

\doublearrow{arrows = {-Implies}}
{(m-2-6) -- node (t4) [left] {} (m-3-6)}
\doublearrow{arrows = {-Implies}}
{(m-2-6) -- (m-3-7)}
\doublearrow{arrows = {-Implies}}
{(m-2-8) -- (m-3-7)}
\doublearrow{arrows = {-Implies}}
{(m-2-8) -- (m-3-8)}

\doublearrow{arrows = {-Implies}}
{(m-3-6) -- (m-4-7)}
\doublearrow{arrows = {-Implies}}
{(m-3-8) -- (m-4-7)}

\path (m-1-7) -- node  {$C_{f,h}$} (m-3-7);
\path (m-3-7) -- node [near start] {$B$} (m-4-7);

\quadarrow{arrows = {-Implies}}
{(s4-|m-2-5.west) -- node [below] {$C_{f,g,h}$} (s4-|t4.west)}

\end{tikzpicture}
\]
\end{defn}

\begin{lem}\label{lem:triple_coherence_local}
Every $3$-fold local branching of $\B_2^\w$ is coherent.
\end{lem}
\begin{proof}
Let $(f,g,h)$ be a minimal local branching. We first treat the case where $(f,g,h)$ is an aspherical branching. If $f=g$, then $C_{f,g} = \epsilon_f$, and the following cell shows that the branching is coherent:

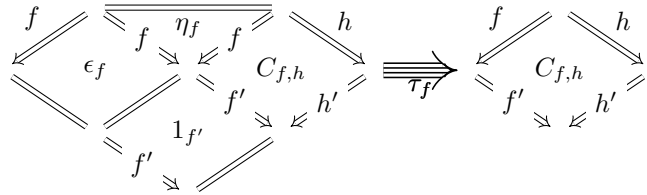
\[
\begin{tikzpicture}
\matrix (m) [matrix of math nodes, 
			nodes in empty cells,
			column sep = 1cm, 
			row sep = .6cm] 
{
& & & & & & &  \\
& & & & & & &  \\
& & & & & & &  \\
& & & & & & &  \\
};
\doublearrow{arrows = {-}}
{(m-1-2) -- (m-1-4)}

\doublearrow{arrows = {-Implies}}
{(m-1-2) -- node [above left] {$f$} (m-2-1)}
\doublearrow{arrows = {-Implies}}
{(m-1-2) -- node [fill=white] {$f$} (m-2-3)}
\doublearrow{arrows = {-Implies}}
{(m-1-4) -- node [fill = white] {$f$} (m-2-3)}
\doublearrow{arrows = {-Implies}}
{(m-1-4) -- node [above right] {$h$} (m-2-5)}

\doublearrow{arrows = {-}}
{(m-2-1) -- (m-3-2)}
\doublearrow{arrows = {-}}
{(m-2-3) -- (m-3-2)}
\doublearrow{arrows = {-Implies}}
{(m-2-3) -- node [fill = white] {$f'$} (m-3-4)}
\doublearrow{arrows = {-Implies}}
{(m-2-5) -- node (s4) [below right] {} 
            node [fill = white] {$h'$} (m-3-4)}

\doublearrow{arrows = {-Implies}}
{(m-3-2) -- node [fill = white] {$f'$} (m-4-3)}
\doublearrow{arrows = {-}}
{(m-3-4) -- (m-4-3)}

\path (m-1-3) -- node [near start] {$\eta_f$} (m-2-3);
\path (m-1-2) -- node  {$\epsilon_f$} (m-3-2);
\path (m-1-4) -- node  {$C_{f,h}$} (m-3-4);
\path (m-2-3) -- node  {$1_{f'}$} (m-4-3);

\doublearrow{arrows = {-Implies}}
{(m-1-7) -- node [above left] {$f$} (m-2-6)}
\doublearrow{arrows = {-Implies}}
{(m-1-7) -- node [above right] {$h$} (m-2-8)}

\doublearrow{arrows = {-Implies}}
{(m-2-6) -- node [fill = white] {$f'$} (m-3-7)}
\doublearrow{arrows = {-Implies}}
{(m-2-8) -- node [fill = white] {$h'$} (m-3-7)}

\path (m-1-7) -- node  {$C_{f,h}$} (m-3-7);

\quadarrow{arrows = {-Implies}}
{(m-2-5) -- node [below] {$\tau_f$} (m-2-6)}
\end{tikzpicture}
\]
The case where $g = h$ is symmetrical. Assume now $g \neq f,h$ and $f = h$. Then $(f,g)$ is either an overlapping or a Peiffer branching. In any case there exists either a cell $A_{f,g}$ or $A_{g,f}$ in $\E_3^\w$. In the former case, we can construct the following cell in $\F_4^{\w(3)}$.

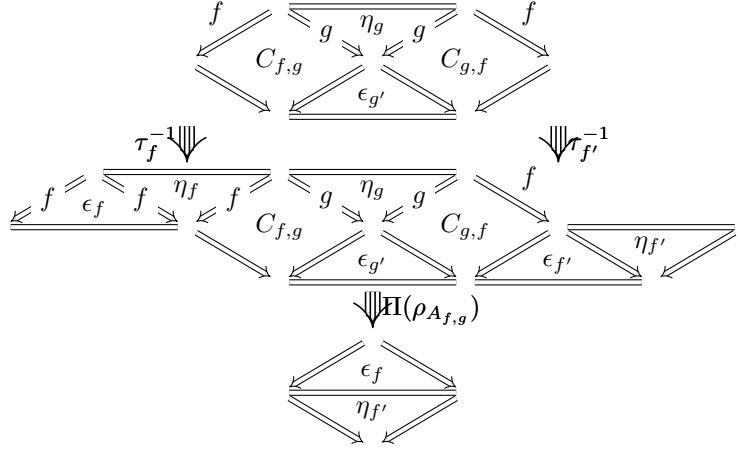
\[
\begin{tikzpicture}
\matrix (m) [matrix of math nodes, 
			nodes in empty cells,
			column sep = 1cm, 
			row sep = .5cm] 
{
& & & & & & & & \\
& & & & & & & & \\
& & & & & & & & \\
& & & & & & & & \\
& & & & & & & & \\
& & & & & & & & \\
& & & & & & & & \\
& & & & & & & & \\
& & & & & & & & \\
};

\doublearrow{arrows = {-}}
{(m-1-4) -- (m-1-6)}

\doublearrow{arrows = {-Implies}}
{(m-1-4) -- node [above left] {$f$} (m-2-3)}
\doublearrow{arrows = {-Implies}}
{(m-1-4) -- node [fill=white] {$g$} (m-2-5)}
\doublearrow{arrows = {-Implies}}
{(m-1-6) -- node [fill = white] {$g$} (m-2-5)}
\doublearrow{arrows = {-Implies}}
{(m-1-6) -- node [above right] {$f$} (m-2-7)}

\doublearrow{arrows = {-Implies}}
{(m-2-3) -- (m-3-4)}
\doublearrow{arrows = {-Implies}}
{(m-2-5) -- (m-3-4)}
\doublearrow{arrows = {-Implies}}
{(m-2-5) -- (m-3-6)}
\doublearrow{arrows = {-Implies}}
{(m-2-7) -- (m-3-6)}

\doublearrow{arrows = {-}}
{(m-3-4) -- (m-3-6)}

\path (m-1-5) -- node [near start] {$\eta_g$} (m-2-5);
\path (m-1-4) -- node  {$C_{f,g}$} (m-3-4);
\path (m-1-6) -- node  {$C_{g,f}$} (m-3-6);
\path (m-2-5) -- node [near end] {$\epsilon_{g'}$} (m-3-5);

\doublearrow{arrows = {-}}
{(m-4-2) -- (m-4-4)}
\doublearrow{arrows = {-}}
{(m-4-4) -- (m-4-6)}

\doublearrow{arrows = {-Implies}}
{(m-4-2) -- node [fill = white] {$f$} (m-5-1)}
\doublearrow{arrows = {-Implies}}
{(m-4-2) -- node [fill = white] {$f$} (m-5-3)}
\doublearrow{arrows = {-Implies}}
{(m-4-4) -- node [fill = white] {$f$} (m-5-3)}
\doublearrow{arrows = {-Implies}}
{(m-4-4) -- node [fill=white] {$g$} (m-5-5)}
\doublearrow{arrows = {-Implies}}
{(m-4-6) -- node [fill = white] {$g$} (m-5-5)}
\doublearrow{arrows = {-Implies}}
{(m-4-6) -- node [above right] {$f$} (m-5-7)}

\doublearrow{arrows = {-}}
{(m-5-1) -- (m-5-3)}
\doublearrow{arrows = {-}}
{(m-5-7) -- (m-5-9)}

\doublearrow{arrows = {-Implies}}
{(m-5-3) -- (m-6-4)}
\doublearrow{arrows = {-Implies}}
{(m-5-5) -- (m-6-4)}
\doublearrow{arrows = {-Implies}}
{(m-5-5) -- (m-6-6)}
\doublearrow{arrows = {-Implies}}
{(m-5-7) -- (m-6-6)}
\doublearrow{arrows = {-Implies}}
{(m-5-7) -- (m-6-8)}
\doublearrow{arrows = {-Implies}}
{(m-5-9) -- (m-6-8)}

\doublearrow{arrows = {-}}
{(m-6-4) -- (m-6-6)}
\doublearrow{arrows = {-}}
{(m-6-6) -- (m-6-8)}

\path (m-4-2) -- node [near end] {$\epsilon_f$} (m-5-2);
\path (m-4-3) -- node [near start] {$\eta_f$} (m-5-3);
\path (m-4-5) -- node [near start] {$\eta_g$} (m-5-5);
\path (m-4-4) -- node  {$C_{f,g}$} (m-6-4);
\path (m-4-6) -- node  {$C_{g,f}$} (m-6-6);
\path (m-5-5) -- node [near end] {$\epsilon_{g'}$} (m-6-5);
\path (m-5-7) -- node [near end] {$\epsilon_{f'}$} (m-6-7);
\path (m-5-8) -- node [near start] {$\eta_{f'}$} (m-6-8);

\doublearrow{arrows = {-Implies}}
{(m-7-5) -- (m-8-4)}
\doublearrow{arrows = {-Implies}}
{(m-7-5) -- (m-8-6)}

\doublearrow{arrows = {-}}
{(m-8-4) -- (m-8-6)}

\doublearrow{arrows = {-Implies}}
{(m-8-4) -- (m-9-5)}
\doublearrow{arrows = {-Implies}}
{(m-8-6) -- (m-9-5)}

\path (m-7-5) -- node [near end] {$\epsilon_f$} (m-8-5);
\path (m-8-5) -- node [near start] {$\eta_{f'}$} (m-9-5);

\quadarrow{arrows = {-Implies}}
{(m-3-3) -- node [left] {$\tau^{-1}_f$} (m-4-3)}
\quadarrow{arrows = {-Implies}}
{(m-3-7) -- node [right] {$\tau^{-1}_{f'}$} (m-4-7)}
\quadarrow{arrows = {-Implies}}
{(m-6-5) -- node [right] {$\Pi(\rho_{A_{f,g}})$} (m-7-5)}

\end{tikzpicture}
\]
In the latter, we can construct the same cell, only replacing  $\Pi(\rho_{A_{f,g}})$ by $\Pi(\lambda_{A_{f,g}})$.

Suppose now that $(f,g,h)$ is not aspherical. Using the cell  $A_{f,g,h}$ described in Lemma \ref{lem:triplet_remplis}, we build the following composite in $\F^{\w(3)}_4$:

\[
\begin{tikzpicture}
\matrix (m) [matrix of math nodes, 
			nodes in empty cells,
			column sep = 1cm, 
			row sep = .5cm] 
{
& & & & & & & & & \\
& & & & & & & & & \\
& & & & & & & & & \\
& & & & & & & & & \\
& & & & & & & & & \\
& & & & & & & & & \\
& & & & & & & & & \\
& & & & & & & & & \\
& & & & & & & & & \\
& & & & & & & & & \\
& & & & & & & & & \\
& & & & & & & & & \\
& & & & & & & & & \\
& & & & & & & & & \\
& & & & & & & & & \\
& & & & & & & & & \\
& & & & & & & & & \\
& & & & & & & & & \\
};

\doublearrow{arrows = {-}}
{(m-1-4) -- (m-1-6)}

\doublearrow{arrows = {-Implies}}
{(m-1-4) -- (m-2-3)}
\doublearrow{arrows = {-Implies}}
{(m-1-4) -- (m-2-5)}
\doublearrow{arrows = {-Implies}}
{(m-1-6) -- (m-2-5)}
\doublearrow{arrows = {-Implies}}
{(m-1-6) -- (m-2-7)}

\doublearrow{arrows = {-Implies}}
{(m-2-3) -- (m-3-4)}
\doublearrow{arrows = {-Implies}}
{(m-2-5) -- (m-3-4)}
\doublearrow{arrows = {-Implies}}
{(m-2-5) -- (m-4-5)}
\doublearrow{arrows = {-Implies}}
{(m-2-5) -- (m-3-6)}
\doublearrow{arrows = {-Implies}}
{(m-2-7) -- (m-3-6)}

\doublearrow{arrows = {-}}
{(m-3-6) -- (m-3-8)}

\doublearrow{arrows = {-}}
{(m-3-4) -- (m-4-5)}
\doublearrow{arrows = {-Implies}}
{(m-3-6) -- (m-4-7)}
\doublearrow{arrows = {-Implies}}
{(m-3-8) -- (m-4-7)}

\doublearrow{arrows = {-Implies}}
{(m-4-5) -- (m-4-7)}

\path (m-1-4) -- node {$C_{f,g}$} (m-3-4);
\path (m-1-5) -- node [near start]{$\eta_g$} (m-2-5);
\path (m-1-6) -- node {$C_{g,h}$} (m-3-6);
\path (m-3-4) -- node [near end] {$\epsilon$} (m-3-5);
\path (m-3-6) -- node {$\Pi(A)$} (m-4-5);
\path (m-3-7) -- node [near start] {$\eta$} (m-4-7);

\doublearrow{arrows = {-}}
{(m-5-2) -- (m-5-4)}
\doublearrow{arrows = {-}}
{(m-5-4) -- (m-5-6)}

\doublearrow{arrows = {-Implies}}
{(m-5-2) -- (m-6-1)}
\doublearrow{arrows = {-Implies}}
{(m-5-2) -- (m-6-3)}
\doublearrow{arrows = {-Implies}}
{(m-5-4) -- (m-6-3)}
\doublearrow{arrows = {-Implies}}
{(m-5-4) -- (m-6-5)}
\doublearrow{arrows = {-Implies}}
{(m-5-6) -- (m-6-5)}
\doublearrow{arrows = {-Implies}}
{(m-5-6) -- (m-6-7)}

\doublearrow{arrows = {-}}
{(m-6-1) -- (m-6-3)}
\doublearrow{arrows = {-}}
{(m-6-7) -- (m-6-9)}

\doublearrow{arrows = {-Implies}}
{(m-6-3) -- (m-7-4)}
\doublearrow{arrows = {-Implies}}
{(m-6-5) -- (m-7-4)}
\doublearrow{arrows = {-Implies}}
{(m-6-5) -- (m-8-5)}
\doublearrow{arrows = {-Implies}}
{(m-6-5) -- (m-7-6)}
\doublearrow{arrows = {-Implies}}
{(m-6-7) -- (m-7-6)}
\doublearrow{arrows = {-Implies}}
{(m-6-7) -- (m-7-8)}
\doublearrow{arrows = {-Implies}}
{(m-6-9) -- (m-7-8)}

\doublearrow{arrows = {-}}
{(m-7-6) -- (m-7-8)}
\doublearrow{arrows = {-}}
{(m-7-8) -- (m-7-10)}

\doublearrow{arrows = {-}}
{(m-7-4) -- (m-8-5)}
\doublearrow{arrows = {-Implies}}
{(m-7-6) -- (m-8-7)}
\doublearrow{arrows = {-Implies}}
{(m-7-8) -- (m-8-9)}
\doublearrow{arrows = {-Implies}}
{(m-7-10) -- (m-8-9)}

\doublearrow{arrows = {-Implies}}
{(m-8-5) -- (m-8-7)}
\doublearrow{arrows = {-}}
{(m-8-7) -- (m-8-9)}

\path (m-5-2) -- node [near end] {$\epsilon_{f}$} (m-6-2);
\path (m-5-3) -- node [near start] {$\eta_{f}$} (m-6-3);
\path (m-5-4) -- node {$C_{f,g}$} (m-7-4);
\path (m-5-5) -- node [near start]{$\eta_g$} (m-6-5);
\path (m-5-6) -- node {$C_{g,h}$} (m-7-6);
\path (m-6-7) -- node [near end] {$\epsilon$} (m-7-7);
\path (m-6-8) -- node [near start] {$\eta$} (m-7-8);
\path (m-7-4) -- node [near end] {$\epsilon$} (m-7-5);
\path (m-7-6) -- node {$\Pi(A)$} (m-8-5);
\path (m-7-8) -- node {$1$} (m-8-7);
\path (m-7-9) -- node [near start] {$\eta$} (m-8-9);

\doublearrow{arrows = {-}}
{(m-9-3) -- (m-9-5)}

\doublearrow{arrows = {-Implies}}
{(m-9-3) -- (m-10-2)}
\doublearrow{arrows = {-Implies}}
{(m-9-3) -- (m-10-4)}
\doublearrow{arrows = {-Implies}}
{(m-9-5) -- (m-10-4)}
\doublearrow{arrows = {-Implies}}
{(m-9-5) -- (m-10-6)}

\doublearrow{arrows = {-}}
{(m-10-2) -- (m-10-4)}
\doublearrow{arrows = {-}}
{(m-10-6) -- (m-10-8)}

\doublearrow{arrows = {-Implies}, rounded corners}
{(m-10-4) -- (m-12-4.center) -- (m-13-5)}
\doublearrow{arrows = {-Implies}}
{(m-10-4) -- (m-11-5)}
\doublearrow{arrows = {-Implies}}
{(m-10-6) -- (m-11-5)}
\doublearrow{arrows = {-Implies}}
{(m-10-6) -- (m-12-6)}
\doublearrow{arrows = {-Implies}}
{(m-10-6) -- (m-11-7)}
\doublearrow{arrows = {-Implies}}
{(m-10-8) -- (m-11-7)}

\doublearrow{arrows = {-}}
{(m-11-7) -- (m-11-9)}

\doublearrow{arrows = {-}}
{(m-11-5) -- (m-12-6)}
\doublearrow{arrows = {-Implies}}
{(m-11-7) -- (m-12-8)}
\doublearrow{arrows = {-Implies}}
{(m-11-9) -- (m-12-8)}

\doublearrow{arrows = {-Implies}}
{(m-12-6) -- (m-12-8)}

\doublearrow{arrows = {-Implies}, rounded corners}
{(m-13-5) -- (m-13-7.center) -- (m-12-8)}

\path (m-9-3) -- node [near end] {$\epsilon_{f}$} (m-10-3);
\path (m-9-4) -- node [near start] {$\eta_{f}$} (m-10-4);
\path (m-9-5) -- node {$C_{f,h}$} (m-11-5);
\path (m-11-5) -- node [near end] {$\epsilon$} (m-11-6);
\path (m-10-7) -- node [near start] {$\eta$} (m-11-7);
\path (m-11-8) -- node [near start] {$\eta$} (m-12-8);
\path (m-11-7) -- node {$\Pi(B_1)$} (m-12-6);
\path (m-11-5) -- node {$\Pi(B_2)$} (m-13-5);

\doublearrow{arrows = {-Implies}}
{(m-14-5) -- (m-15-4)}
\doublearrow{arrows = {-Implies}}
{(m-14-5) -- (m-15-6)}

\doublearrow{arrows = {-}}
{(m-15-6) -- (m-15-8)}

\doublearrow{arrows = {-Implies}, rounded corners}
{(m-15-4) -- (m-17-4.center) -- (m-18-5)}
\doublearrow{arrows = {-Implies}}
{(m-15-4) -- (m-16-5)}
\doublearrow{arrows = {-Implies}}
{(m-15-6) -- (m-16-5)}
\doublearrow{arrows = {-Implies}}
{(m-15-6) -- (m-17-6)}
\doublearrow{arrows = {-Implies}}
{(m-15-6) -- (m-16-7)}
\doublearrow{arrows = {-Implies}}
{(m-15-8) -- (m-16-7)}

\doublearrow{arrows = {-}}
{(m-16-7) -- (m-16-9)}

\doublearrow{arrows = {-}}
{(m-16-5) -- (m-17-6)}
\doublearrow{arrows = {-Implies}}
{(m-16-7) -- (m-17-8)}
\doublearrow{arrows = {-Implies}}
{(m-16-9) -- (m-17-8)}

\doublearrow{arrows = {-Implies}}
{(m-17-6) -- (m-17-8)}

\doublearrow{arrows = {-Implies}, rounded corners}
{(m-18-5) -- (m-18-7.center) -- (m-17-8)}

\path (m-14-5) -- node {$C_{f,h}$} (m-16-5);
\path (m-16-5) -- node [near end] {$\epsilon$} (m-16-6);
\path (m-15-7) -- node [near start] {$\eta$} (m-16-7);
\path (m-16-8) -- node [near start] {$\eta$} (m-17-8);
\path (m-16-7) -- node {$\Pi(B_1)$} (m-17-6);
\path (m-16-5) -- node {$\Pi(B_2)$} (m-18-5);

\quadarrow{arrows = {-Implies}}
{(m-4-3) -- node [label={[label distance=.1cm]180:$\tau_f^{-1}$}] {} (m-5-3)}
\quadarrow{arrows = {-Implies}}
{(m-4-8) -- node [label={[label distance=.1cm]0:$\tau^{-1}$}] {} (m-5-8)}
\quadarrow{arrows = {-Implies}}
{(m-8-5) -- node [label={[label distance=.1cm]0:$\Pi(A_{f,g,h})$}] {} (m-9-5)}
\quadarrow{arrows = {-Implies}}
{(m-13-5) -- node [label={[label distance=.1cm]0:$\tau_f$}] {} (m-14-5)}
\end{tikzpicture}
\]

Finally, if $(f,g,h)$ is not aspherical, then there exists a $3$-fold minimal branching $(\tilde f, \tilde g, \tilde h)$ of $\B_2^\w$ and $1$-cells $u,v \in \B_1^\w$ such that $(f,g,h) = (u\tilde f v, u\tilde g v,u \tilde h v)$. Then the cell $uC_{\tilde f, \tilde g, \tilde h}v$ shows that $(f,g,h)$ is coherent.
\end{proof}

\subsection{Orderings on the cells of $\F^\w$}
\label{subsec:ordering}
\begin{defn}
Let $E$ be a set. The \emph{set of all finite multi-sets on $E$} is $\mathbb N[E]$, the free commutative monoid over $E$. For every $e \in E$, let $\v_e:\mathbb N[E] \to \mathbb N$ be the morphism of monoids that sends $e$ to $1$ and every other elements of $E$ to $0$.

If $E$ is equiped with a strict ordering $>$, we denote by $>_{m}$ the strict ordering on $\mathbb N[E]$ defined as follows: for every $f,g \in \mathbb N[E]$, one has $f >_m g$ if
\begin{itemize}
\item $f \neq g$
\item For every $e \in E$, if $\v_e(f) < \v_e(g)$, then there exists $e' > e$ such that $\v_{e'}(f) > \v_{e'}(g)$.
\end{itemize}
\end{defn}

\begin{lem}\label{lem:somme_multiset}
Let $E$ be a set and $a \in E$. The set of all $f \in \mathbb N[E]$ such that $f < a$ is equal to the set of all $f \in \mathbb N[E]$ satisfying the following implication for every $b \in E$: \[\v_b(f) > 0 \Rightarrow b < a.\]

In particular, this set is a sub-monoid of $\mathbb N[E]$.
\end{lem}
\begin{proof}
Let $f \in \mathbb N[E]$ such that for every $b \in E$ the implication $\v_b(f) > 0 \Rightarrow b < a$ is verified. Let us prove that $f <_m a$. Necessarily $\v_a(f) = 0$, otherwise we would have $a < a$. Thus in particular $f \neq a$. Moreover, let $b \in E$ such that $\v_b(f) > \v_b(a) \geq 0$. By definition of $f$ this implies that $b < a$, and since $0 = \v_a(f) < \v_a(a) = 1$ we get that $f < m_a$.

Conversely, let $f <_m a$. Let us show by contradiction that $\v_a(f) = 0$. If $\v_a(f) \neq 0$, we distinguish two cases:
\begin{itemize}
\item If $\v_a(f) = 1$, then since $f \neq a$, there exists   $b \neq a \in E$ such that $\v_b(f) >0$. Thus because $f < a$, there exists $c > b \in E$ such that $\v_c(f) < \v_c(a)$. So we necessarily have $\v_c(a) \geq 1$, which implies that $c = a$. The condition $\v_c(f) < \v_c(a)$ thus becomes $\v_a (f) < 1$, which contradicts the hypothesis that $\v_a(f) = 1$.
\item If $\v_a(f) > 1$, then there exists $b > a$ such that $\v_b(a) > \v_b(f)$, which is impossible. 
\end{itemize}
Hence necessarily $\v_a(f) = 0$.

Let $b \in E$ such that $\v_b(f) >0$, and let us show that $b < a$. We just showed that $b \neq a$, and so $\v_b(f) > \v_b(a)$. Thus there exists $c>b$ such that $\v_c(a) > \v_c(f)$. In particular this implies $\v_c(a) > 0$. So $c = a$ and finally $a > b$. 
\end{proof}

\begin{lem}
Let $(E,<)$ be a set equipped with a strict ordering. The relation $>_{m}$ is compatible with the monoidal structure on  $\mathbb N (E)$, that is, for every $f,f',g \in \mathbb N(E)$, if $f >_m f'$, then $f + g >_m f' + g$.
\end{lem}
\begin{proof}
Let $f,f',g \in \mathbb N(E)$, and suppose that $f >_m f'$. Let us show that $f + g >_m f' + g$. Firstly, $f \neq f'$, hence $f + g \neq f' + g$. 

Let $e \in E$ such that $\v_e(f+g) < \v_e(f'+g)$. Since $\v_e$ is a morphism of monoids, this implies that $\v_e(f) < \v_e(f')$. Hence there exists $e' >e$ such that $\v_{e'}(f) > \v_{e'}(f')$, and so $\v_{e'}(f+g) > \v_{e'}(f'+g)$
\end{proof}

The proof of the following theorem can be found in \cite{BN99}.

\begin{thm}
Let $(E,>)$ be a set equipped with a strict ordering. Then $>_{m}$ is a well-founded ordering if and only if $>$ is.
\end{thm}

Since $\A$ is $2$-terminating, the set $\A_1^*$ is equipped with a well-founded ordering $\Rightarrow$. This induces a well founded ordering $\Rightarrow_m$ on $\mathbb N[\A_1^*]$. We now define two applications $\p : \F_2^\w \to \mathbb N[\A_1^*]$ and $\w_\eta : \F_3^\w \to \mathbb N[\A_1^*]$. Using $\Rightarrow_m$, those applications induce well-founded orderings on $\F_2^\w$ and $\F_3^\w$. We then show a number of properties of these applications in preparation for Section \ref{subsec:partial_coh}.

\begin{defn}
We define an application $\p : \F_2^\w \to \mathbb N[\A_1^*]$:
\begin{itemize}
\item for every $f \in \F_2^\w$ of length $1$, we set $\p(f) := \s(f) + \t(f)$,
\item for every composable $f_1,f_2 \in \F_2^\w$, we set $\p(f_1 \star_1 f_2) := \p(f_1) + \p(f_2)$.
\end{itemize} 

For every $f,g \in \F_2^\w$, we set $f > g$ if $\p(f) \Rightarrow_m \p(g)$. The relation $>$ is a well-founded ordering of $\F_2^\w$.
\end{defn}

\begin{defn}
We define an application $\w_\eta : \F_3^\w \to \mathbb N[\A_1^*]$ by setting:
\begin{itemize}
\item For every $f \in \B_2^\w$ of length $1$, $\w_\eta(\eta_f) = \s(f)$.
\item For every $3$-cell $A \in \F_3$ and $u,v \in \A_1^*$, if $A$ is not an $\eta_f$ then $\w_\eta(uAv) = 0$.
\item For every $f_1,f_2 \in \F_2^\w$ and $A \in \F_3^\w$, $\w_\eta(f_1 \star_1 A \star_1 f_2) = \w_\eta(A)$.
\item For every $A_1,A_2 \in \F_3^\w$, $\w_\eta(A_1 \star_2 A_2) = \w_\eta(A_1) + \w_\eta(A_2)$.
\end{itemize}
\end{defn}

\begin{defn}
A product of the form $\bar f \star_1 g \in \F^\w_2$, where $f$ and $g$ are nonempty cells in $\B_2^\w$ is called a \emph{cavity}. It is a \emph{local cavity} if $f$ and $g$ are of length $1$. Let $C_\F$ be the set of all cavities.
\end{defn}

\begin{lem}\label{lem:source_greater}
Let $f,g \in \B_2^\w$. Suppose $f$ is not an identity and $\t(f) = \s(g)$. The following inequality holds: 
\[
\s(f) > \p(g)
\]
\end{lem}
\begin{proof}
We reason by induction on the length of $g$. If $g$ is empty, then $\p(g) = 0 < \s(f)$.

Otherwise, let us write $g = g_1 \star_1 g_2$, with $g_1$ of length $1$. Then $\p(g) = \p(g_1) + \p(g_2)$ and by induction hypothesis $\p(g_2) < \s(f \star_1 g_1) = \s(f)$. Moreover we have $f: \s(f) \Rightarrow \s(g_1)$ and $f \star_1 g_1 : \s(f) \Rightarrow \t(g_1)$. Hence $\s(f) > \p(g_1), \s(g_2), \t(g_2)$ and, by Lemma \ref{lem:somme_multiset}, we get $\s(f) > \p(g_1)+ \s(g_2)+ \t(g_2) = \p(g)$.

\end{proof}

\begin{lem}\label{lem:completion_decroissante}
Let $f_1,f_2,g_1,g_2 \in \B_2^\w$, with $f_1$ and $f_2$ non-empty and of same source $u$. For every $3$-cell $A : \bar f_1 \star_1 f_2 \Rrightarrow g_1 \star_1 \bar g_2 \in \F_3^\w$, the following inequalities hold:
\[
\p(\s(A)) > u > \p(\t(A)).
\]

In particular for every cell $C_{f,g}$, we have $\s(C_{f,g}) > \t(C_{f,g})$.
\end{lem}
\begin{proof}
Considering the first inequality, we have $\p(\s(A)) = \p(f_1) + \p(g_2) \geq 2u > u$.

Considering the second one, using Lemma \ref{lem:source_greater}, we have the inequalities  $u = \s(f_1) > \p(g_1)$ and $u = \s(f_2) > \p(g_2)$. By \ref{lem:somme_multiset}, we then have $u > \p(g_1) + \p(g_2) = \p(\t(A))$.
\end{proof}

\begin{defn}
Let $h \in \F_2^\w$. A factorisation $h = h_1 \star_1 \bar f_1 \star_1 f_2 \star_1 h_2$ of $h$, with $f_1,f_2 \in \B_2^\w$ of length $1$ and $h_1,h_2 \in \F_2^\w$ is called \emph{a cavity-factorisation} of $h$. Thus a cavity-factorisation is represented as follows:

\[
\begin{tikzpicture}
\matrix (m) [matrix of math nodes, 
			nodes in empty cells,
			column sep = 1cm, 
			row sep = .5cm] 
{
& & & & & & \\
& & & & & & \\
};
\doublearrow{arrows={Implies-Implies}}
{(m-2-1) -- node [above] {$h_1$} (m-2-3)}
\doublearrow{arrows={-Implies}}
{(m-1-4) -- node [above left] {$f_1$} (m-2-3)}
\doublearrow{arrows={-Implies}}
{(m-1-4) -- node [above right] {$f_2$} (m-2-5)}
\doublearrow{arrows={Implies-Implies}}
{(m-2-5) -- node [above] {$h_2$} (m-2-7)}

\end{tikzpicture}
\]
\end{defn}

\begin{lem}\label{lem:facto_creux}
Let $h \in \F_2^\w$ be a $2$-cell which is not an identity, and whose source and target are a normal form for $\A_2$. Then there exists a cavity-factorisation of $h$. 
\end{lem}
\begin{proof}
By definition of $\F_2^\w$, there exist $n \in \mathbb N^*$ and $g_1,\ldots ,g_{2n} \in \B_2^\w$ all not identities, except possibly $g_1$ or $g_{2n}$, such that $h = \bar g_1 \star_1 g_2 \star_1 \ldots \star_1 \bar g_{2n-1} \star_1 g_{2n}$. 

Let us show that $g_1$ and $g_{2n}$ are not identities:
\begin{itemize}
\item If $g_1$ is an identity, then since $h$ isn't, either $n \geq 2$ or $n=1$ and $g_{2n}$ is not an identity. In both cases $g_2$ is of length at least $1$, and has $\s(h)$ as target, which contradicts the fact that $\s(h)$ is a normal form for $\A_2$.
\item The case where $g_{2n}$ is an identity is symmetric.
\end{itemize}

Therefore the $2$-cells $g_1$ and $g_2$ are of length at least $1$. So we can write $g_1 = f_1 \star_1 g'_1$ and $g_2 = f_2 \star_1 g'_2$, with $f_1,f_2 \in \B_2^\w$ of length $1$. Let $h_1 := \bar g'_1 $ and $h_2 := g'_2 \star_1 \bar g_3 \star_1 g_4 \star_1 \ldots \star_1 \bar g_{2n-1} \star_1 g_{2n}$. We finally get: $h = h_1 \star_1 \bar f_1 \star_1 f_2 \star_1 h_2$.
\end{proof}

\begin{lem}\label{lem:confluence_eta_zero}
Let $h \in \F_2^\w$ be a $2$-cell of source and target $\hat u$, a normal form for $\A_2$. There exists a $3$-cell $A : h \Rrightarrow 1_{\hat u}$ such that $\w_\eta(A) = 0$.
\end{lem}
\begin{proof}
We reason by induction on $h$ using the ordering $>$. If $h$ is minimal, then $h = 1_{\hat u}$ and we can set $A := 1_h$.

Otherwise by Lemma \ref{lem:facto_creux} there exists a cavity-factorisation $h = h_1 \star_1 f_1 \star_1 f_2 \star_1 h_2$ of $h$. Let $A_1:= C_{f_1,f_2}$: we have $\w_\eta(A_1) = 0$ and by Lemma \ref{lem:completion_decroissante}, $\s(A_1) > \t(A_1)$. Since the ordering is compatible with composition, we get $h > h_1 \star_1 \t(A_1) \star_1 h_2$. By induction hypothesis, there exists a $3$-cell $A_2:h_1 \star_1 \t(A_1) \star_1 h_2\Rrightarrow 1_{\hat u} \in \F_3^\w$ such that $\w_\eta(A_2) = 0$.

Let $A := (h_1 \star_1 A_1 \star_1 h_2) \star_1 A_2$. We have $\w_\eta(A) = \w(h_1 \star_1 A_1 \star_1 h_2) + \w(A_2) = \w(A_1) + 0 = 0$.
\end{proof}

\begin{lem}\label{lem:facto_completion}
Let $h \in \F_2^\w$ of source and target $\hat u$ a normal form for $\A_2$, and $A: h \Rrightarrow 1_{\hat u} \in \F_3^\w$. For every cavity-factorisation $h = h_1 \star_1 \bar f_1 \star_1 f_2 \star_1 h_2$, there exists a factorisation of $A = (h_1 \star_1 A_1 \star_1 h_2) \star_2 A_2$, with $A_1,A_2 \in \F_3^\w$, and either $A_1 = C_{f_1,f_2}$ or $A_1 = \bar f_1 \star_1 \eta_{f_3} \star_1 f_2$, with $f_3 \in \B_2^\w$ of length $1$.
\end{lem}
\begin{proof}
We reason by induction on the length of $A$. If $A$ is of length $0$, then there is no cavity-factorisation of $h  = 1_{\hat u}$ and the result holds.

If $A$ is not of length $0$, let $h = h_1 \star_1 \bar f_1 \star_1 f_2 \star_1 h_2$ be a cavity-factorisation of $h$. Let us write $A = B \star_1 C$, where $B$ is of length $1$. If $B$ is not of the required form, then either $B = B' \star_1 \bar f_1 \star_1 f_2 \star_1 h_2$, or $B = h_1 \star_1 \bar f_1 \star_1 f_2 \star_1 B'$. Let us treat the first case, the second being symmetrical. The source of $C$ admits a cavity-factorisation $\s(C) = \t(B') \star_1 \bar f_1 \star_1 f_2 \star_1 h_2$. By induction hypothesis, we can factorise $C$ as follows:
\[
C = (h'_1 \star_1 A_1 \star_1 h_2) \star_2 C',
\]
with $A_1 = C_{f_1,f_2}$ or $A_1 = \bar f_1 \star_1 \eta_{f_3} \star_1 f_2$. Let $A_2 := (B'_1 \star_1 \t(A_1) \star_1 h_2) \star_2 C_2$: we then have $A = (h_1 \star_1 A_1 \star_1 h_2) \star_2 A_2$.

\[
\begin{tikzpicture}
\matrix (m) [matrix of math nodes, 
			nodes in empty cells,
			column sep = 1cm, 
			row sep = .5cm] 
{
& & & & & & \\
& & & & & & \\
& & & & & & \\
& & & & & & \\
};
\doublearrow{arrows={Implies-Implies}}
{(m-2-1) to [bend left] node [above] {$h_1$} (m-2-3)}
\doublearrow{arrows={Implies-Implies}}
{(m-2-1) to [bend right] (m-2-3)}
\doublearrow{arrows={-Implies}}
{(m-1-4) -- node [above left] {$f_1$} (m-2-3)}
\doublearrow{arrows={-Implies}}
{(m-1-4) -- node [above right] {$f_2$} (m-2-5)}
\doublearrow{arrows={Implies-Implies}}
{(m-2-5) -- node [above] {$h_2$} (m-2-7)}

\doublearrow{arrows={Implies-Implies}, rounded corners}
{(m-2-3) -- (m-3-4.center) -- (m-2-5)}

\doublearrow{arrows={Implies-Implies}, rounded corners}
{(m-2-1) -- (m-4-2.center) --(m-4-6.center) -- (m-2-7)}

\path (m-2-1) -- node {$B$} (m-2-3);
\path (m-1-4) -- node {$A_1$} (m-3-4);
\path (m-3-4) -- node {$C'$} (m-4-4);

\end{tikzpicture}
\]
\end{proof}

\begin{lem}\label{lem:eta_borne}
Let $h \in \F_2^\w$ and $u \in \A_1^*$ such that $u > \p(h)$, $u > \s(h)$ and $u > \t(h)$. For every $3$-cell $A \in \F_3^\w$ of source $h$, the inequality $u > \w_{\eta}(A)$ holds.
\end{lem}
\begin{proof}
We reason by induction on the length of $A$. If $A$ is of length $0$, $\w_\eta(A) = 0$ and the result holds.

Otherwise, let us write $A = A_1 \star_2 A_2$, with $A_1$ of length $1$. We distinguish two cases depending on the shape of $A_1$.
\begin{itemize}
\item If $A_1 = h_1 \star_1 \eta_f \star_1 h_2$, with $h_1,h_2 \in \F_2^\w$ and $f \in \B_2^\w$ of length $1$. 

If $h_1$ and $h_2$ are empty, then $\s(A_2) = f \star_1 \bar f$. Thus $\p(\s(A_2)) = 2 \s(f) + 2 \t(f) \leq 4 \s(f) = 4 \s(h)$. Since $\s(h) < u$, using Lemma \ref{lem:somme_multiset}, we get that $\p(\s(A_2)) < u$. Applying the induction hypothesis to $A_2$, we get $\w_\eta(A_2) < u$. Moreover, $\w_\eta(A) = \w_\eta(A_1) + \w_\eta(A_2) = \s(f) + \w(A_2)$, and we showed that $\w(A_2) < u$ and $\s(f) = \s(h) <u$. Thus according to Lemma \ref{lem:somme_multiset}, we get $\w_\eta(A) < u$.

Otherwise, suppose for example that $h_1$ is not an identity (the case where $h_2$ is not an identity  being symmetrical). Then we have $\v_{\t(h_1)}(\p(h_1))>0$, so $\v_{\t(h_1)}(\p(h))>0$. Since $\p(h) < u$, we have by Lemma \ref{lem:somme_multiset} that $\s(f) = \t(h_1) < u$. So $\p(\s(A_2)) = \p(h_1) + \p(h_2) + 2 \s(f) + 2 \t(f) < \p(h) + 4 \s(f) < u$. By induction hypothesis, we thus have $\w_\eta(A_2) < u$, and finally $\w_\eta(A) = \s(f) + \w_\eta(A_2) < u$.
\item Otherwise, we have on the one hand that $\w_\eta(A_1) = 0$, and on the other hand that $\s(A_2) = \t(A_1) < \s(A_1) = h <u$ by Lemma \ref{lem:completion_decroissante}. Thus $\w_\eta(A) = \w_\eta(A_2) <u$.
\end{itemize}
\end{proof}

\begin{lem}\label{lem:eta_local_conf}
Let $(f_1,f_2,f_3)$ be a $3$-fold local branching, $u \in \A_1^*$, and $A,B\in \F_3^{\w}$ two $3$-cells  such that there exists a $4$-cell:
\[C_{f_1,f_3,f_2} : \bar f_1 \star_1 \eta_{f_3} \star_1 f_2 \star_2 (C_{f_1,f_3} \star_1 C_{f_3,f_2}) \star_2 A \qfl C_{f_1,f_2} \star_2 B.\] 
Then $\w_\eta(A),\w_\eta(B) < u$.
\end{lem}
\begin{proof}
Using Lemma \ref{lem:completion_decroissante}, we have $\p(\t(C_{f_1,f_2})),\p(\t(C_{f_2,f_3})),\p(\t(C_{f_1,f_3})) < u$. So $\p(\s(A)) = \p(\t(C_{f_1,f_2}) + \p(\t(C_{f_2,f_3})) < u$ et $\p(\s(B)) = \p(\t(C_{f_1,f_3})) <u$, and using \ref{lem:eta_borne}, we get $\w_\eta(A),\w_\eta(B) < u$
\end{proof}

\subsection{Partial coherence of $\F^{\w(3)}$}
\label{subsec:proof_final}
\begin{prop}\label{prop:coh_dans_F3}
For every $2$-cell $h \in \F_2^\w$ with source and target $\hat u$ a normal form for $\A_2$, and for every $3$-cells $A,B : h \Rrightarrow 1_{\hat u} \in \F_3^\w$, there exists a  $4$-cell $\alpha : A \qfl B \in \F_4^{\w(3)}$.
\end{prop}
\begin{proof}
We reason by induction on the couple $(\w_\eta(A)+ \w_\eta(B), \p(h))$, using the lexicographic order. If $h = 1_{\hat u}$, then $A=B = 1_h$. Thus setting $\alpha = 1_A = 1_B$ shows that the property is verified.

Suppose now that $h$ is not an identity. Using Lemma \ref{lem:facto_creux}, there exists a cavity-factorisation $h = h_1 \star_1 \bar f_1 \star_1 f_2 \star_1 h_2$. By Lemma \ref{lem:facto_completion}, there exist $A_1,A_2,B_1,B_2 \in \F_3^\w$, such that $A = (h_1 \star_1 A_1 \star_1 h_2) \star_2 A_2$ and $B = (h_1 \star_1 B_1 \star_1 h_2) \star_2 B_2$. Using this Lemma, we distinguish four cases depending on the shape of $A_1$ and $A_2$.

\paragraph*{If $A_1=B_1=C_{f_1,f_2}$.}
Then in particular we have:
\[
\s(A_2) = \s(B_2) \qquad \w_\eta(A) = \w_\eta(A_2) \qquad \w_\eta(B) = \w_\eta(B_2) \qquad \t(A_1) < \s(A_1),
\]
where the last inequality is a consequence of Lemma \ref{lem:completion_decroissante}. Hence we get $\p(\s(A_2)) = \p(h_1) + \p(\t(A_1)) + \p(h_2) < \p(h_1) + \p(\s(A_1)) + \p(h_2) = \p(h)$, and finally $(\w_\eta(A_2) + \w_\eta(B_2),\p(\s(A_2))) < (\w_\eta(A) + \w_\eta(B),h))$. Using the induction hypothesis there exists $\alpha : A_2 \qfl B_2 \in \F_4^{\w(3)}$, and by composition we construct $A_1 \star_2 \alpha : A \to B$.
\paragraph*{If $A_1= \bar f_1 \star_1 \eta_{f_3} \star_1 f_2$ and $B_1 = C_{f_1,f_2}$.}
We are going to construct the following composite:
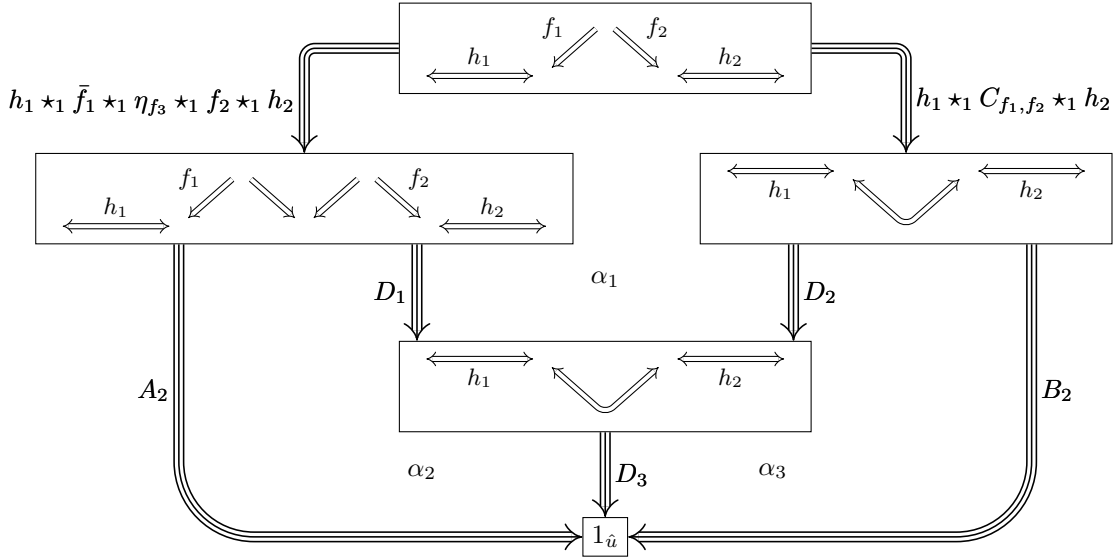
\[
\begin{tikzpicture}
\matrix (m1) [draw = black,
			matrix of math nodes, 
			nodes in empty cells,
			column sep = .6cm, 
			row sep = .5cm] 
{
& & & & & & \\
& & & & & & \\
};
\doublearrow{arrows={Implies-Implies}}
{(m1-2-1) -- node [above] {\small $h_1$} (m1-2-3)}
\doublearrow{arrows={-Implies}}
{(m1-1-4) -- node [above left] {\small $f_1$} (m1-2-3)}
\doublearrow{arrows={-Implies}}
{(m1-1-4) -- node [above right] {\small $f_2$} (m1-2-5)}
\doublearrow{arrows={Implies-Implies}}
{(m1-2-5) -- node [above] {\small $h_2$} (m1-2-7)}

\matrix (m2) [draw = black,
			matrix of math nodes, 
			nodes in empty cells,
			column sep = .6cm, 
			row sep = .5cm,
			xshift = 4cm,
			yshift = -2cm] 
{
& & & & & & \\
& & & & & & \\
};
\doublearrow{arrows={Implies-Implies}}
{(m2-1-1) -- node [below] {\small $h_1$} (m2-1-3)}
\doublearrow{arrows={Implies-Implies}, rounded corners}
{(m2-1-3) -- (m2-2-4.center) -- (m2-1-5)}
\doublearrow{arrows={Implies-Implies}}
{(m2-1-5) -- node [below] {\small $h_2$} (m2-1-7)}

\matrix (m3) [draw = black,
			matrix of math nodes, 
			nodes in empty cells,
			column sep = .6cm, 
			row sep = .5cm,
			xshift = -4cm,
			yshift = -2 cm] 
{
& & & & & & & & \\
& & & & & & & & \\
};
\doublearrow{arrows={Implies-Implies}}
{(m3-2-1) -- node [above] {\small $h_1$} (m3-2-3)}
\doublearrow{arrows={-Implies}}
{(m3-1-4) -- node [above left] {\small $f_1$} (m3-2-3)}
\doublearrow{arrows={-Implies}}
{(m3-1-4) -- (m3-2-5)}
\doublearrow{arrows={-Implies}}
{(m3-1-6) -- (m3-2-5)}
\doublearrow{arrows={-Implies}}
{(m3-1-6) -- node [above right] {\small $f_2$} (m3-2-7)}
\doublearrow{arrows={Implies-Implies}}
{(m3-2-7) -- node [above] {\small $h_2$} (m3-2-9)}

\matrix (m4) [draw = black,
			matrix of math nodes, 
			nodes in empty cells,
			column sep = .6cm, 
			row sep = .5cm,
			yshift = -4.5cm] 
{
& & & & & & \\
& & & & & & \\
};
\doublearrow{arrows={Implies-Implies}}
{(m4-1-1) -- node [below] {\small $h_1$} (m4-1-3)}
\doublearrow{arrows={Implies-Implies}, rounded corners}
{(m4-1-3) -- (m4-2-4.center) -- (m4-1-5)}
\doublearrow{arrows={Implies-Implies}}
{(m4-1-5) -- node [below] {\small $h_2$} (m4-1-7)}

\node (end) [draw = black, yshift = -6.5cm] {$1_{\hat u}$};

\triplearrow{arrows={-Implies}, rounded corners}
{(m1.east) -| node [right, near end] {$h_1 \star_1 C_{f_1,f_2} \star_1 h_2$} (m2)}
\triplearrow{arrows={-Implies}, rounded corners}
{(m1.west) -| node [left, near end] {$h_1 \star_1 \bar f_1 \star_1 \eta_{f_3} \star_1 f_2 \star_1 h_2$} (m3)}
\triplearrow{arrows={-Implies}}
{(m3.south-|m4-2-1) to node [left] {$D_1$} (m4.north-|m4-2-1)}
\triplearrow{arrows={-Implies}}
{(m2.south-|m4-2-7) to node [right] {$D_2$} (m4.north-|m4-2-7)}
\triplearrow{arrows={-Implies}, rounded corners = 1cm}
{(m3.south-|m3-2-3) -- node [left] {$A_2$} (end-|m3-2-3) -- (end)}
\triplearrow{arrows={-Implies}, rounded corners = 1cm}
{(m2.south -| m2-2-6) -- node [right] {$B_2$} (end-|m2-2-6) -- (end)}
\triplearrow{arrows={-Implies}}
{(m4) to node [right] {$D_3$} (end)}

\path (m1) -- node [near end] {$\alpha_1$} (m4);
\path (m3.south west) -- node [near end, below left] {$\alpha_2$} (end);
\path (m2.south east) -- node [near end, below right] {$\alpha_3$} (end);
\end{tikzpicture}
\]
According to Lemma \ref{lem:triple_coherence_local}, there exists a $4$-cell \[C_{f_1,f_3,f_2} : \bar f_1 \star_1 \eta_{f_3} \star_1 f_2 \star_2 (C_{f_1,f_3} \star_1 C_{f_3,f_2}) \star_2 D'_1 \qfl C_{f_1,f_2} \star_2 D'_2,\] with $D'_1,D'_2 \in \F_4^{\w(3)}$. Let us define $D_1 := h_1 \star_1 (C_{f_1,f_3} \star_1 C_{f_3,f_2}) \star_2 D'_1) \star_1 h_2$,  $D_2 := h_1 \star_1 D'_2 \star_1 h_2$, and $\alpha_1 := h_1 \star_1 C_{f_1,f_3,f_2} \star_1 h_2$. The existence of $D_3$ is guaranteed by Lemma \ref{lem:confluence_eta_zero}, which also proves that we can choose $D_3$ such that $\w_\eta(D_3) = 0$.

In order to construct the $4$-cells $\alpha_1$ and $\alpha_2$, let us show that we can apply the induction hypothesis to the couples $(A_2, D_1 \star_2 D_3)$ and $(D_2 \star_2 D_3,B_2)$. Let $v$ be the common source of $f_1$ and $f_2$.
\begin{itemize}
\item Using Lemma \ref{lem:eta_local_conf}, $\w_\eta(D_1 \star_2 D_3)= \w_\eta(D_1) = \w_\eta(D'_1) < v$, and so:
\[
\w_\eta(A_2) + \w_\eta(D_1 \star_2 D_3) < \w_\eta(A_2) + \w(\eta_3) = \w_\eta (A) \leq \w_\eta(A) + \w_\eta(B).
\]
\item As previously $\w_\eta(D_2 \star_2 D_3)= \w_\eta(D_2) = \w_\eta(D'_2) < v$, and so:
\[
\w_\eta(B_2) + \w_\eta(D_2 \star_2 D_3) < \w_\eta(B_2) + \w(\eta_3) \leq \w_\eta(B) + \w_\eta(A).
\]
\end{itemize}

\paragraph*{If $A_1= C_{f_1,f_2}$ and $B_1 = \bar f_1 \star_1 \eta_{f_3} \star_1 f_2$.}
This case is similar to the previous one, only using $C_{f_1,f_3,f_2}^{-1}$ rather than $C_{f_1,f_3,f_2}$.

\paragraph*{If $A_1= \bar f_1 \star_1 \eta_{f_3} \star_1 f_2$ and $B_1 = \bar f_1 \star_1 \eta_{f_4} \star_1 f_2$.}
We are going to construct the following composite:
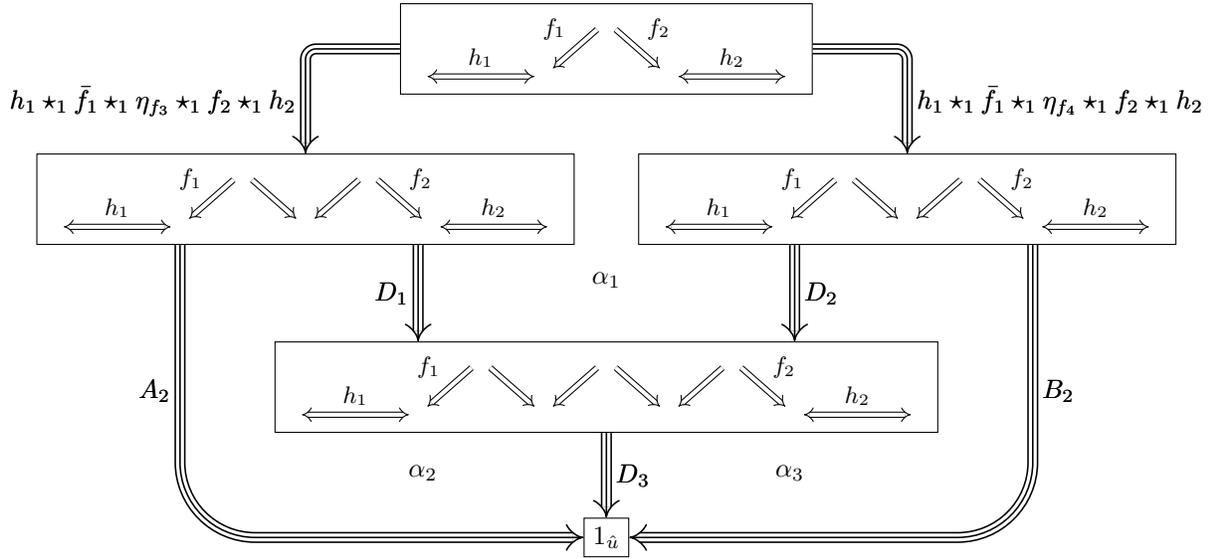
\[
\begin{tikzpicture}
\matrix (m1) [draw = black,
			matrix of math nodes, 
			nodes in empty cells,
			column sep = .6cm, 
			row sep = .5cm] 
{
& & & & & & \\
& & & & & & \\
};
\doublearrow{arrows={Implies-Implies}}
{(m1-2-1) -- node [above] {\small $h_1$} (m1-2-3)}
\doublearrow{arrows={-Implies}}
{(m1-1-4) -- node [above left] {\small $f_1$} (m1-2-3)}
\doublearrow{arrows={-Implies}}
{(m1-1-4) -- node [above right] {\small $f_2$} (m1-2-5)}
\doublearrow{arrows={Implies-Implies}}
{(m1-2-5) -- node [above] {\small $h_2$} (m1-2-7)}

\matrix (m2) [draw = black,
			matrix of math nodes, 
			nodes in empty cells,
			column sep = .6cm, 
			row sep = .5cm,
			xshift = -4cm,
			yshift = -2cm] 
{
& & & & & & & & \\
& & & & & & & & \\
};
\doublearrow{arrows={Implies-Implies}}
{(m2-2-1) -- node [above] {\small $h_1$} (m2-2-3)}
\doublearrow{arrows={-Implies}}
{(m2-1-4) -- node [above left] {\small $f_1$} (m2-2-3)}
\doublearrow{arrows={-Implies}}
{(m2-1-4) -- (m2-2-5)}
\doublearrow{arrows={-Implies}}
{(m2-1-6) -- (m2-2-5)}
\doublearrow{arrows={-Implies}}
{(m2-1-6) -- node [above right] {\small $f_2$} (m2-2-7)}
\doublearrow{arrows={Implies-Implies}}
{(m2-2-7) -- node [above] {\small $h_2$} (m2-2-9)}

\matrix (m3) [draw = black,
			matrix of math nodes, 
			nodes in empty cells,
			column sep = .6cm, 
			row sep = .5cm,
			xshift = 4cm,
			yshift = -2 cm] 
{
& & & & & & & & \\
& & & & & & & & \\
};
\doublearrow{arrows={Implies-Implies}}
{(m3-2-1) -- node [above] {\small $h_1$} (m3-2-3)}
\doublearrow{arrows={-Implies}}
{(m3-1-4) -- node [above left] {\small $f_1$} (m3-2-3)}
\doublearrow{arrows={-Implies}}
{(m3-1-4) -- (m3-2-5)}
\doublearrow{arrows={-Implies}}
{(m3-1-6) -- (m3-2-5)}
\doublearrow{arrows={-Implies}}
{(m3-1-6) -- node [above right] {\small $f_2$} (m3-2-7)}
\doublearrow{arrows={Implies-Implies}}
{(m3-2-7) -- node [above] {\small $h_2$} (m3-2-9)}

\matrix (m4) [draw = black,
			matrix of math nodes, 
			nodes in empty cells,
			column sep = .6cm, 
			row sep = .5cm,
			yshift = -4.5cm] 
{
& & & |[alias=v1]| & & & & & & & \\
& & & & & & & & & & \\
};
\doublearrow{arrows={Implies-Implies}}
{(m4-2-1) -- node [above] {\small $h_1$} (m4-2-3)}
\doublearrow{arrows={-Implies}}
{(m4-1-4) -- node [above left] {\small $f_1$} (m4-2-3)}
\doublearrow{arrows={-Implies}}
{(m4-1-4) -- (m4-2-5)}
\doublearrow{arrows={-Implies}}
{(m4-1-6) -- (m4-2-5)}
\doublearrow{arrows={-Implies}}
{(m4-1-6) -- (m4-2-7)}
\doublearrow{arrows={-Implies}}
{(m4-1-8) -- (m4-2-7)}
\doublearrow{arrows={-Implies}}
{(m4-1-8) -- node [above right] {\small $f_2$} (m4-2-9)}
\doublearrow{arrows={Implies-Implies}}
{(m4-2-9) -- node [above] {\small $h_2$} (m4-2-11)}

\node (end) [draw = black] at (0,-6.5) {$1_{\hat u}$};

\triplearrow{arrows={-Implies}, rounded corners}
{(m1.west) -| node [left, near end] {$h_1 \star_1 \bar f_1 \star_1 \eta_{f_3} \star_1 f_2 \star_1 h_2$} (m2)}
\triplearrow{arrows={-Implies}, rounded corners}
{(m1.east) -| node [right, near end] {$h_1 \star_1 \bar f_1 \star_1 \eta_{f_4} \star_1 f_2 \star_1 h_2$} (m3)}
\triplearrow{arrows={-Implies}}
{(m2.south-|m4-1-3) to node [left] {$D_1$} (m4.north-|m4-1-3)}
\triplearrow{arrows={-Implies}}
{(m3.south -| m4-1-9) to node [right] {$D_2$} (m4.north -| m4-1-9)}
\triplearrow{arrows={-Implies}, rounded corners = 1cm}
{(m2.south -| m2-2-3) -- node [left] {$A_2$} (m2-2-3 |- end) -- (end)}
\triplearrow{arrows={-Implies}, rounded corners = 1cm}
{(m3.south -| m3-2-7) -- node [right] {$B_2$} (m3-2-7 |- end) -- (end)}
\triplearrow{arrows={-Implies}}
{(m4.south) to node [right] {$D_3$} (end)}

\path (m1) -- node [near end] {$\alpha_1$} (m4);
\path (m2.south west) -- node [near end, below left] {$\alpha_2$} (end);
\path (m3.south east) -- node [near end, below right] {$\alpha_3$} (end);
\end{tikzpicture}
\]
Let us set
\[D_1 := h_1 \star_1 \bar f_1 \star_1 f_3 \star_1 \bar f_3 \star_1 \eta_{f_4} \star_1 f_2 \star_1 h_2
\]
\[
D_2 := h_1 \star_1 \bar f_1 \star_1  \eta_{f_3} \star_1 f_4 \star_1 \bar f_4 \star_1 f_2 \star_1 h_2.
\] 
We then have
\[(h_1 \star_1 A_1 \star_1 h_2) \star_2 D_1 = h_1 \star_1 \bar f_1 \star_1 \eta_{f_3} \star_1 \eta_{f_4} \star_1 f_2 \star_1 h_2 = (h_1 \star_1 B_1 \star_1 h_2) \star_1 D_2.\]
Hence we define $\alpha_1$ as an identity. Let now $D_3$ be as in Lemma \ref{lem:confluence_eta_zero}, with $\w_\eta(D_3) = 0$, and $v$ be the common source of $f_1,f_2,f_3$ and $f_4$. We then have the inequalities:
\[
\w_\eta(A_2) + \w_\eta(D_1) + \w_\eta(D_3) = \w_\eta(A_2) + v < \w_\eta(A_2) + \w_\eta(B_2) + 2v = \w_\eta(A) + \w_\eta(B),
\]
\[
\w_\eta(B_2) + \w_\eta(D_2) + \w_\eta(D_3) = \w_\eta(B_2) + v < \w_\eta(B_2) + \w_\eta(A_2) + 2v = \w_\eta(A) + \w_\eta(B).
\]
Hence we can apply the induction hypothesis to the couples $(A_2,D_1 \star_2 D_3)$ and $(D_2 \star_2 D_3,B_2)$, which provides $\alpha_2$ and $\alpha_3$.
\end{proof}

\begin{prop}
The $(4,3)$-white-category $\F^{\w(3)}$ is $S_\E$-coherent.
\end{prop}
\begin{proof}
Let $A,B: f \Rrightarrow h \in \F_3^\w$ whose $1$-target is a normal form $\hat u$, with $f,g \in \B_2^\w$.

The $3$-cells $(\bar h \star_1 A) \star_2 \epsilon_h$ and $(\bar h \star_1 B) \star_2 \epsilon_h$ are parallel, and their target is $1_{\hat u}$. In particular they verify the hypothesis of Proposition \ref{prop:coh_dans_F3}. So there exists $\alpha :(\bar h \star_1 A) \star_2 \epsilon_h \qfl (\bar h \star_1 B) \star_2 \epsilon_h$. Then the following composite is the required cell from $A$ to $B$:

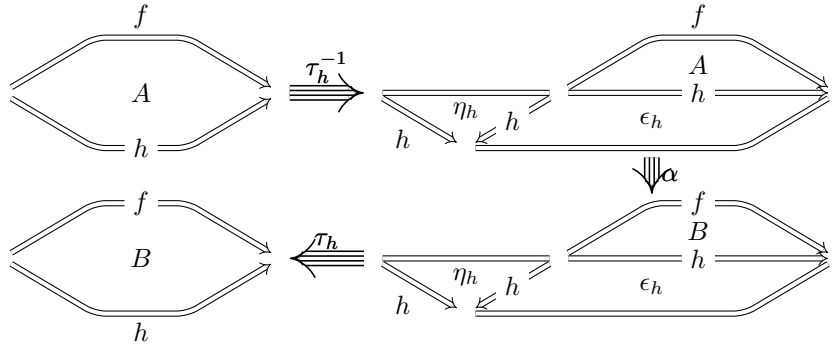
\[
\begin{tikzpicture}

\matrix (m) [matrix of math nodes, 
			nodes in empty cells,
			column sep = 1cm, 
			row sep = .5cm] 
{
& & & & & & & & & \\
& & & & & & & & & \\
& & & & & & & & & \\
& & & & & & & & & \\
& & & & & & & & & \\
& & & & & & & & & \\
}; 
\doublearrow{arrows={-Implies}, rounded corners}
{(m-2-1) -- (m-1-2.center) -- node [above] {$f$} (m-1-3.center) -- (m-2-4)}
\doublearrow{arrows={-Implies}, rounded corners}
{(m-2-1) -- (m-3-2.center) -- node [fill = white] {$h$}(m-3-3.center) -- (m-2-4)}

\path (m-2-1) -- node {$A$} (m-2-4);

\doublearrow{arrows={-Implies}, rounded corners}
{(m-2-7) -- (m-1-8.center) -- node [above] {$f$} (m-1-9.center) -- (m-2-10)}
\doublearrow{arrows={-}}
{(m-2-5) -- (m-2-7)}
\doublearrow{arrows={-Implies}}
{(m-2-5) -- node [below left] {$h$} (m-3-6)}
\doublearrow{arrows={-Implies}}
{(m-2-7) -- node [fill = white] {$h$} (m-3-6)}
\doublearrow{arrows={-Implies}}
{(m-2-7) -- node [fill = white] {$h$} (m-2-10)}
\doublearrow{arrows={-}, rounded corners}
{(m-3-6) -- (m-3-9.center) -- (m-2-10)}

\path (m-2-6) -- node [near start] {$\eta_h$} (m-3-6);
\path (m-3-6) -- node {$\epsilon_h$} (m-2-10);
\path (m-1-8) -- node {$A$} (m-2-9);

\doublearrow{arrows={-Implies}, rounded corners}
{(m-5-7) -- (m-4-8.center) -- node [fill = white] {$f$} (m-4-9.center) -- (m-5-10)}
\doublearrow{arrows={-}}
{(m-5-5) -- (m-5-7)}
\doublearrow{arrows={-Implies}}
{(m-5-5) -- node [below left] {$h$} (m-6-6)}
\doublearrow{arrows={-Implies}}
{(m-5-7) -- node [fill = white] {$h$} (m-6-6)}
\doublearrow{arrows={-Implies}}
{(m-5-7) -- node [fill = white] {$h$} (m-5-10)}
\doublearrow{arrows={-}, rounded corners}
{(m-6-6) -- (m-6-9.center) -- (m-5-10)}

\path (m-5-6) -- node [near start] {$\eta_h$} (m-6-6);
\path (m-6-6) -- node {$\epsilon_h$} (m-5-10);
\path (m-4-8) -- node {$B$} (m-5-9);

\doublearrow{arrows={-Implies}, rounded corners}
{(m-5-1) -- (m-4-2.center) -- node [fill = white] {$f$} (m-4-3.center) -- (m-5-4)}
\doublearrow{arrows={-Implies}, rounded corners}
{(m-5-1) -- (m-6-2.center) -- node [below] {$h$}(m-6-3.center) -- (m-5-4)}

\path (m-5-1) -- node {$B$} (m-5-4);

\quadarrow{arrows={-Implies}}
{(m-2-4) -- node [above] {$\tau_h^{-1}$} (m-2-5)}
\quadarrow{arrows={-Implies}}
{(m-3-8) -- node [right] {$\alpha$} (m-4-8)}
\quadarrow{arrows={-Implies}}
{(m-5-5) -- node [above] {$\tau_h$} (m-5-4)}
\end{tikzpicture}
\]
\end{proof}

We can now complete the proof of Theorem \ref{thm:main_theory}. Indeed we showed that $\F^{\w(3)}$ is $S_\E$-coherent. Using Proposition \ref{prop:tietze_invariance}, that means that $\E^{\w(3)}$ is $S_\E$-coherent, and finally using Lemma \ref{lem:translation_coh} that $\A^{*(2)}$  is $S_\A$-coherent, that is that for every $3$-cells $A,B \in \A^{*(2)}_3$, whose $1$-target is a normal form, there exists a $4$-cell $\alpha : A \qfl B \in \A^{*(2)}_4$.

\newpage
\bibliographystyle{plain}
\bibliography{maitre}

\end{document}